\documentclass[reqno]{amsart}
\usepackage{hyperref}
\usepackage{amssymb}
\usepackage{mathrsfs}
\usepackage{esint}
\usepackage{graphicx}
\usepackage{wrapfig}
\usepackage{pifont}
\usepackage{float}
\usepackage[numbers]{natbib}
\usepackage{inputenc}
\usepackage{amsmath}
\numberwithin{equation}{section}
\newtheorem{theorem}{Theorem}[section]
\newtheorem{lemma}[theorem]{Lemma}
\newtheorem{proposition}[theorem]{Proposition}
\newtheorem{remark}[theorem]{Remark}

\numberwithin{equation}{section}
\newtheorem{thm}{Theorem}[section]

\title[]
{Hydrodynamic limit of the Vlasov-Poisson-Boltzmann system for gas mixture}

\author[
]{Yeping Li, Gaofeng Wang  and Tianfang Wu}

\address{Y.P. Li. School of Mathematics and Statistics,
    Nantong University, Nantong 226019, P. R. China.}
\email{ypleemei@aliyun.com}

\address{G.F. Wang. (corresponding author). School of Mathematics and Statistics,
    Nantong University, Nantong 226019, P. R. China.}
\email{gfwang@ntu.edu.cn}

\address{T.F. Wu. Wenzhou Business College, Wenzhou 325000, P. R. China.}
\email{20249237@wzbc.edu.cn}

\thanks{}

\begin{document}

    \maketitle
\begin{abstract}
In this paper, we study the hydrodynamic limit of the
Vlasov-Poisson-Boltzmann system for a gas mixture in the whole space
$(x \in \mathbb{R}^3)$ with the potential range of $\gamma \in
\left(-3, 1\right]$. Using the method of Hilbert expansion, we first
derive a bi-Maxwellian determined by the Euler-Poisson system of two
fluids. To justify the convergence of the solution rigorously as the
Knudsen number tends to zero, we sequentially calculate the first
$2k-1$ terms of the expansion series $(k \geq 6)$, and then truncate
it, and express the solution as the sum of these first $2k-1$ terms
and a remainder term. Within the framework of the
$L_{x,v}^2-W_{x,v}^{1,\infty}$ interplay established by Guo and Jang
\cite{[ininp]Guo2010CMP}, we construct a new weight function to
estimate the remainder term in four different cases regarding the
potential $\gamma$. Here, the particle masses $m^A, m^B > 0$
and their charges $e^A, e^B$ can be given arbitrarily. This causes
the collision operator to exhibit asymmetric effects ($m^A \neq m^B$), rendering the
system of equations impossible to decouple. So, it adds difficulties
to both $L^2$, $L^{\infty}$ estimates for the remainder. Therefore,
we adopt the framework of vector-valued functions and analyze the velocity decay rate of the operator $K_{M,2,w}^{\alpha,c}$  to
eliminate the singularity induced by small parameters in
characteristic line iterations.   Our results  show that the
validity time of the solution is $O(\varepsilon^{-y})$, where $y$ is
$-\frac{2k-3}{2(2k-1)}$ when $-1 \leq \gamma \leq 1$, and it becomes
$-\frac{2k-3}{(1-\gamma)(2k-1)}$, when $-3 < \gamma < -1$. These results possess strong physical realism and can be applied to analyze gas flow dynamics in the daytime ionosphere at high altitudes above the Earth.\\

         \vspace{5pt}
    \end{abstract}

    \allowdisplaybreaks

    \section{Introduction and main result}

    \subsection{Introduction}
This paper studies the hydrodynamic limit of the
Vlasov-Poisson-Boltzmann (denoted as VPB in the sequel) system for a
gas mixture containing two types of molecules with different
molecular masses. We work with the following non-dimensionalized
form of the VPB system, which describes the evolution of two species
of particles interacting under the influence of a self-consistent
electric field:
    \begin{equation}\label{MAMAMAMAzhuyaotuimox}
        \left\{
        \begin{gathered}
            St \, \partial_tF^A+v\cdot \nabla_xF^A+\frac{e^{A}}{m^A}\nabla_{x}\phi \cdot \nabla_{v}F^A=\frac{1}{Kn}\Big[Q^{AA}(F^A,F^A)+Q^{AB}(F^A,F^B)\Big],
            \\
            St \, \partial_tF^B+v\cdot \nabla_xF^B+\frac{e_{B}}{m^B}\nabla_{x}\phi \cdot \nabla_{v}F^B=\frac{1}{Kn}\Big[Q^{BA}(F^B,F^A)+Q^{BB}(F^B,F^B)\Big],\\
            \Delta \phi = \int_{\mathbb{R}^3} e^{A}F^A+ e^{B}F^Bdv-\bar{n}_e.
        \end{gathered}
        \right.
    \end{equation}
Here, we denote particle species by the Greek letters
$\alpha$ and $\beta$ ($\alpha,\beta\in\{A,B\}$). The density
distribution function for particle species $\alpha$ is given by
$F^\alpha(t,x,v)$, defined at time $t>0$, position
$x=(x_1,x_2,x_3)\in\mathbb{R}^3$, and velocity
$v=(v_1,v_2,v_3)\in\mathbb{R}^3$. Particle mass, charge and particle
diameter are denoted by $m^\alpha>0$, $e^\alpha$  and
$\sigma_{\alpha}>0$, respectively. The self-consistent electric
potential $\phi(t,x)$ is coupled with the distribution function $F^\alpha(t,x,v)$ through the
Poisson equation$\eqref{MAMAMAMAzhuyaotuimox}_3$. The constant ion background charge is denoted by $\bar{n}_e$ (defined in \eqref{CCXXTJ25}). $St$ and $Kn$ are two dimensionless parameters. More precisely, the Strouhal number $St$ characterizes the frequency of oscillatory flow phenomena, such as vortex shedding, in relation to flow velocity and a characteristic length. The Knudsen number $Kn$ is defined as the ratio of the gas molecular mean free path to a characteristic physical length scale. Interactions are modeled by the Boltzmann collision operator for
$\alpha$-$\beta$ particle pairs
    $$ Q^{\alpha \beta}(F^{\alpha},F^{\beta})=\frac{(\sigma_{\alpha}+\sigma_{\beta})^2}{4}\int_{\mathbb{R}^3}\int_{\mathbb{S}^2}B^{\alpha \beta}(|v-v_*|,\theta) \Big[F^{\alpha}{'} F_*^{\beta }{'}-F^{\alpha}F_*^{\beta} \Big] d\omega dv_*, $$
    where $\omega \in \mathbb{S}^2$ and we use the shorthands  $F^{\alpha} = F^{\alpha}(v)$, $F_*^{\alpha'} = F^{\alpha}(v'_*)$, $F^{\alpha'} = F^{\alpha}(v')$, $F^{\alpha}_* = F^{\alpha}(v_*)$. The velocity pairs $(v_*,v)$ and $(v'_*,v')$ represent the pre- and post-collision states of particles belonging to species $\alpha$ and $\beta$, respectively. They obey the laws of conservation of momentum and energy:
    \begin{equation*}
        m^{\alpha} v_*' +m^{\beta}v' =m^{\alpha} v_*+m^{\beta} v, \quad
        m^{\alpha}|v_*'|^{2}+m^{\beta}|v'|^2 =m^{\alpha}|v_*|^2+m^{\beta}|v|^2.
    \end{equation*}
    The velocity $ v', v'_* $ after the collision are expressed by
    \begin{equation*}
        \left\{
        \begin{gathered}
            v'=v-\frac{2m^{\beta}}{m^\alpha+m^\beta}[(v-v_*)\cdot \omega]\omega,
            \\
            v'_{*}=v_{*}+\frac{2m^{\alpha}}{m^\alpha+m^\beta}[(v-v_*)\cdot \omega]\omega.
        \end{gathered}
        \right.
    \end{equation*}

    Throughout this paper, we denote the Knudsen number by $Kn= \varepsilon$, where
    $\varepsilon$ is assumed to be small.
    The collision kernels satisfy the following assumptions:

    \noindent (1) The symmetry of collision kernels:
    \begin{equation*}
        B^{\alpha \beta}(|v-v_*|,\theta) =B^{\beta \alpha}(|v-v_*|,\theta) \,\quad  \alpha,  \beta\in \{A,B\}.
    \end{equation*}

    \noindent (2) The collision kernels could be decomposed into
    \begin{equation*}
        B^{\alpha \beta}(|v-v_*|,\theta) =\Phi^{\alpha \beta} (|v-v_*|) b^{\alpha \beta}(\cos \theta)\,\quad  \alpha,  \beta\in \{A,B\},
    \end{equation*}
    \begin{itemize}
        \item  The kinetic part has the following form
        \begin{equation*}
            \Phi^{\alpha \beta}(|v-v_*|) =C_{\alpha \beta}^{\Phi} |v-v_*|^{\gamma} , \,  C_{\alpha \beta}^{\Phi} >0 , \, \gamma \in
            ( -3,1 ].
        \end{equation*}
        \item Angular part has a strong form of Grad's angular cutoff \cite{[61]Grad1958TG}: For $\cos\theta=\frac{\omega\cdot(v-v_*)}{|v-v_*|}$, there exist a constant $C_{b}  >0 $, for $\forall \alpha,  \beta\in \{A,B\}$ and $\theta \in [0,\pi]$, such that
        \begin{equation*}
            0< b^{\alpha \beta}(\cos \theta)\leq C_{b}  |\cos \theta |.
        \end{equation*}
    \end{itemize}

The rigorous derivation of hydrodynamic limits from kinetic theory is a fundamental challenge in mathematical physics, directly addressing Hilbert's sixth problem \cite{[23]Hilbert} concerning the mathematical unification of microscopic and macroscopic descriptions. In the asymptotic limit where the Knudsen number $\varepsilon \to 0$, the system transitions from kinetic to hydrodynamic regimes, with characteristic macroscopic scales by far exceeding the molecular mean free path.

Maxwell \cite{[37]Maxwell1867JCM} and Boltzmann
\cite{[9]Boltzmann1872AWW} first showed that the limiting solution of the
Boltzmann equation approaches a local Maxwellian, and its density,
momentum, and temperature correspond to a fluid system. These
contributions laid the foundation for further investigations into
the Boltzmann equation. Justifying the hydrodynamic limits that
arise from the Boltzmann equations presents significant challenges
and has been a topic of research for over five decades.

Caflisch \cite{[12]Caflisch1980CPAM} validated the compressible Euler limit via truncated Hilbert expansion but imposed constraints on initial data. Guo, Jang, and Jiang \cite{[21]Guo2009KRM} refined this using the $L_{x,v}^2 - L_{x,v}^{\infty}$ technique \cite{[19]Guo2010ARMA}, later extending to the acoustic limit \cite{[22]Guo2010CPAM}. Note that Hilbert expansion requires initial values to match its form, however, for general cases, initial layers are needed \cite{[33]Lachowicz1987MMAS}. The VPB system's various scalings yield different fluid models, e.g., incompressible Navier-Stokes-Fourier-Poisson \cite{[32]Jiang2018IUM}, \cite{[42]Raymond2009JMPA}. For further details, see \cite{[4]BGLJSP 1991}, \cite{[5]BGLCPAM1993}, and \cite{[40]Raymond2009BOOK}.
Boundary effects complicate Euler limits: Sone \cite{[DA]Sobook} demonstrated viscous/Knudsen layers' necessity. Boundary treatments include: specular reflection in \cite{[20]Guo2021ARMA}, diffuse/Maxwell conditions in \cite{[29]Jiang2021} and \cite{[30]Jiang2021}.

While monatomic gases dominate hydrodynamic limit studies, multi-species Boltzmann systems receive less attention. Aoki, Bardos, and Takata \cite{[1]Aoki2003JSP} proved the existence of a Knudsen layer for Boltzmann equations in gas mixtures with hard sphere interactions (zero bulk velocity assumption), later extended by Bardos and Yang \cite{[7]Bardos2012CPDE} to include drifting velocities. Wang \cite{YJWangSIMA} studied the VPB system's diffusive limit, which decouples for equal molecular collision masses $(m^A=m^B)$ in periodic domains, and later analyzed its decay rate \cite{[49]WangJDE2013}. Fang and Qi \cite{[FQ]ARC} examined the hydrodynamic transition from Boltzmann equations for gas mixtures $(m^A=m^B)$ to two-fluid systems. Guo \cite{[17]Guo2003Invention} proved the Euler-Maxwell limit for hard spheres, while Jiang, Lei, and Zhao \cite{[2025X]Jiang} investigated this limit, ($m_{\pm}=e_{\pm}=1$). Duan and Liu \cite{[DL]VPB} discussed the Euler-Poisson limit for 1D gas mixtures. 

However, in the real world, collisions between gas molecules often occur between two particles with different molecular masses. For example, the mass ratio between oxygen molecules and nitrogen molecules in the atmosphere is 8 to 7.

When considering collisions involving charged particles, the study of ionized gas mixtures is of paramount importance in atmospheric physics, especially in the ionosphere situated 60 km above Earth's surface. This atmospheric layer exhibits a highly ionized state due to intense ultraviolet radiation and cosmic ray bombardment. ‌For example‌, in the atmospheric E layer  or F1 layer , some nitrogen and oxygen are  converted into nitrosyl ions (NO$^+$, $m_{NO^+}$ $\approx$ 30$m_p$, $q_{NO^+}$ $= +1e$) and oxygen ions (O$^+$, $m_{O^+}$ $\approx$ 16$m_p$, $q_{O^+}$ $= +1e$) under strong solar radiation during the daytime, leading to collisions between these ion species. In addition, with increasing high-altitude radiation, the atmosphere exhibits ions of different species, with their primary distribution as follows \cite{Appleton1933},\cite{Hanson172}:
\\
$\bullet$\hspace{0.06cm} D layer (60--90 km):\, $NO^+$, $O_{2^+}$, $H_2O^+$, $H_3O^+$, $NO_{2^-}$, $O_{2^-}$. \\
$\bullet$\hspace{0.06cm} E layer (90--140 km): \,  $N_{2^+}$, $O_{2^+}$, $NO^+$, $O^+$. \\
$\bullet$ \hspace{0.07cm}F layer = F1 layer + F2 layer (140--1000 km):
\begin{itemize}
    \item F1 layer (140--210 km): \,  $O^+$,  $NO^+$, $O_{2^+}$.
    \item F2 layer (210--1000 km): \, $O^+$, $NO^+$,  $He^+$, $H^+$, $N^+$.
\end{itemize}
The particle masses and mass-to-charge ratios of these ion species exhibit variations. The dominant ions in both the D and E layers are NO$^+$, with concentration ranges of $10^5$\text{--}$10^6$~\text{cm}$^{-3}$ for the D layer and $10^6$\text{--}$10^7$~\text{cm}$^{-3}$ for the E layer. In contrast, O$^+$ is the primary ion in both the F1 and F2 layers, with concentrations ranging approximately from $10^6$\text{--}$10^7$~\text{cm}$^{-3}$.

Current research exhibits a notable gap in mathematically formalizing these collision dynamics within the Boltzmann equation framework, especially when the assumption of equal masses or the normalization of both mass and charge  becomes invalid. Our work systematically examines scenarios involving arbitrary mass and charge configurations, thereby establishing a generalized theoretical framework for particle collision phenomena in the ionosphere. Additionally, we construct appropriate weightings that enable both soft and hard potentials to be addressed within a unified framework.

\begin{figure}{}
    \centering
    \includegraphics[width=11.5cm,height=7cm]{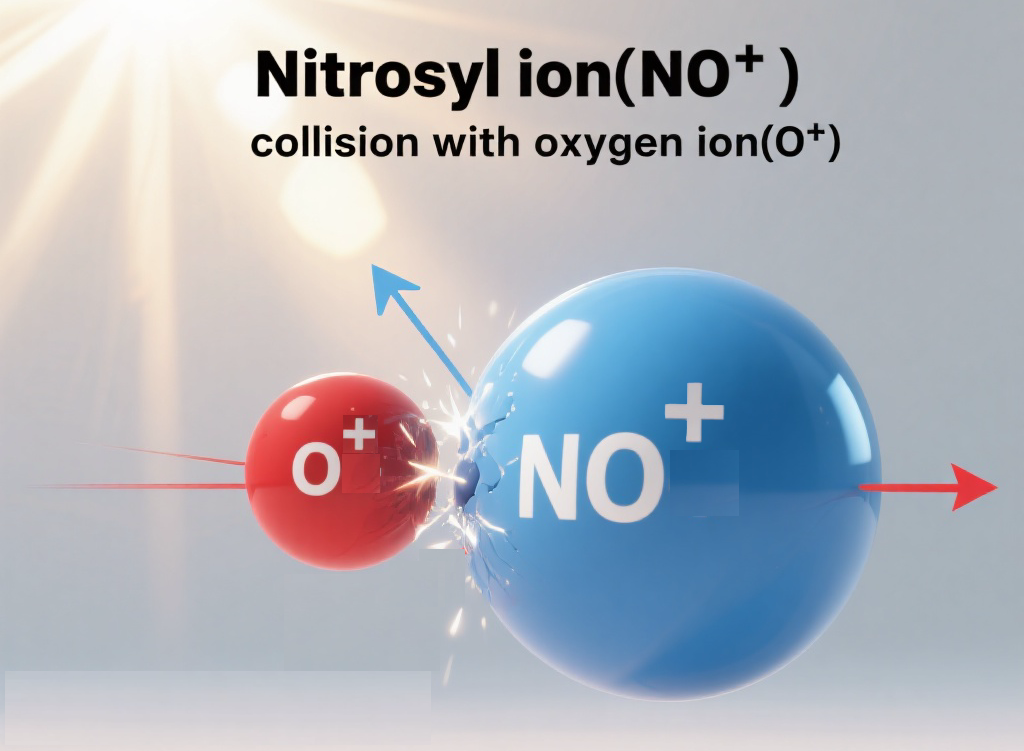}
    \caption{The collision between the nitrosyl ion and the oxygen ion }
\end{figure}

This paper investigates the hydrodynamic limit of a VPB system for binary gas mixtures with:
\begin{itemize}
    \item [\ding{72}]Potential range: $\gamma \in (-3, 1]$,
    \item [\ding{72}]Non-identical molecular masses: $m^A \neq m^B$,
    \item [\ding{72}]Arbitrary charge configurations: $e^A, e^B \in \mathbb{R}$,
    \item[\ding{72}] Characteristic timescale $O(1)$,
    requiring Strouhal number $St$
    to be $O(1)$.
\end{itemize}
The inherent coupling of this system necessitates analysis within a vector-valued function framework. Building on the $L_{x,v}^2 - W_{x,v}^{1,\infty}$ technique developed by \cite{[ininp]Guo2010CMP}, we establish remainder estimates through weighted function analysis across four regimes:
\begin{align*}
    \gamma &= 1 \text{ (hard sphere)}, \\
    0 &\leq \gamma < 1 \text{ (hard potentials \& Maxwell molecules)}, \\
    -1 &\leq \gamma < 0 \text{ (moderately soft potentials)}, \\
    -3 &< \gamma < -1 \text{ (very soft potentials)}.
\end{align*}
Key technical innovations include:
\begin{enumerate}
    \item Well-designed weight function \eqref{WWWEIGHTFUC}-\eqref{WWDWEIGHTFUC}.
    \item  The velocity decay rate analysis of the operator $K_{M,2,w}^{\alpha,c}$ ($m^A \neq m^B$ case, see Lemma \ref{LeMK2ker4444}).
    \item Resolution of the $(1+|v|)^\gamma$ degeneracy as $|v|\to\infty$ (soft potential case, see \eqref{resollowerbbd}).
    \item Control of remainder support propagation (soft potential case, see \eqref{SOFSUPGT59},\eqref{SOFSUPGT510}).
\end{enumerate}
 According to the Hilbert expansion
 \begin{equation*}
    F_{\varepsilon}^\alpha= \sum_{i=0}^{2k-1}\varepsilon^i F^\alpha_i + \varepsilon^kF_R^{\alpha}, \qquad \nabla_{x}\phi_\epsilon= \sum_{i=0}^{2k-1}\varepsilon^i \nabla_{x}\phi_i + \varepsilon^k\nabla_{x}\phi_R,\quad k\geq6,
 \end{equation*}
our results also compare the effects of different potential parameters $\gamma$ on the weight function within the same theoretical framework. The analysis reveals that the solution's validity time scales as $O(\varepsilon^{-y})$ (see Figure~2), where the exponent $y$ is determined by:
\[
y =
\begin{cases}
    -\dfrac{2k-3}{2(2k-1)} & \text{for } -1 \leq \gamma \leq 1, \\[10pt]
    -\dfrac{2k-3}{(1-\gamma)(2k-1)} & \text{for } -3 < \gamma < -1.
\end{cases}
\]
\begin{figure}{}
    \centering
    \includegraphics[width=12cm,height=5.4cm]{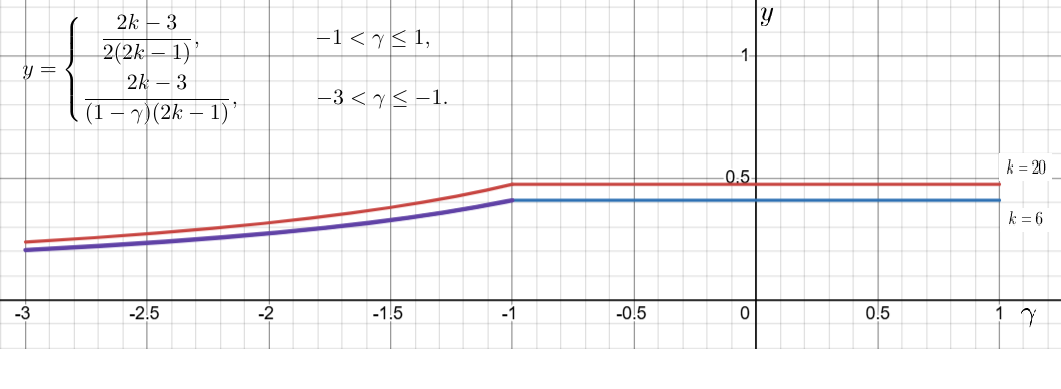}
    \caption{The validity time of the solution is $O(\varepsilon^{-y})$}
\end{figure}
\subsection{Basic properties of collision operator (see\cite{[60]Briant2016ARM})}
    Firstly, we  present some properties of the collision operators for later use. For the vector field $\mathbf{F} = (F^A, F^B)^T$,  the vector-form bilinear Boltzmann collision operator $\mathscr{C}$ is  defined as:
    \begin{equation}\label{equation1.4}
        \mathscr{C} \mathbf{F} =
        \left(
        \begin{array}{cccc}
            Q^{AA}(F^A,F^A)+Q^{AB}(F^A,F^B)\vspace{3pt}\\
            Q^{BA}(F^B,F^A)+Q^{BB}(F^B,F^B)
        \end{array}
        \right).
    \end{equation}
    \begin{itemize}
        \item (Collision invariants) A function $\mathbf{\Psi} = (\Psi^{A}, \Psi^{B})^{T}$
        is called a collision invariant of the operator $\mathscr{C}$ with respect to the inner product in $(L^2_v(\mathbb{R}^3))^2$,
        if it satisfies
        \begin{equation}\label{Collision invariants} \langle\mathscr{C}\mathbf{F}, \mathbf{\Psi} \rangle_{(L_v^2(\mathbb{R}^3))^2} = \sum_{\alpha,\beta \in {A,B}} \left\langle Q^{\alpha \beta}(F^{\alpha}, F^{\beta}), \Psi^{\alpha} \right\rangle_{L_v^2} = 0 \quad \text{for any vector } \mathbf{F}.
        \end{equation}
        The equality $\langle\mathscr{C}\mathbf{F}, \mathbf{\Psi} \rangle_{(L_v^2(\mathbb{R}^3))^2} = 0$
        holds, if and only if   $$\mathbf{\Psi} \in \text{Span} \{\mathbf{e}_1, \mathbf{e}_2, v_1 \mathbf{m}, v_2 \mathbf{m}, v_3 \mathbf{m}, |v|^2 \mathbf{m}\},$$where $\mathbf{m} = (m^A, m^B)^{T}$
        and $\mathbf{e}_j$
        is the $j^{\text{th}}$
        unit vector in $\mathbb{R}^2$.
        \item  (H-theorem) 
        The entropy function of the Boltzmann equations for binary gas mixtures satisfies
        \begin{equation}\label{equation1.5}
            \left \langle
            \mathscr{C}\mathbf{F},  \log \mathbf{F}
            \right \rangle_{(L_v^2(\mathbb{R}^3))^2}
            = \sum_{\alpha,\beta\in \{A,B\}}\left\langle Q^{\alpha  \beta}(F^{\alpha},F^{\beta}) , ~~\log F^{\alpha} \right\rangle_{L_v^2(\mathbb{R}^3)}   \leq 0,
        \end{equation}
        and this equality holds, if and only if
        \begin{equation}\label{HHHHCCCRM}
            \log \mathbf{F} \in {\rm Span} \{\mathbf{e}_1, \mathbf{e}_2, v_1 \mathbf{m}, v_2 \mathbf{m}, v_3 \mathbf{m}, |v|^2 \mathbf{m}\}.
        \end{equation}
    \end{itemize}
    \subsection{Hilbert Expansion}
    We take $St = 1$ to rewrite the VPB system \eqref{MAMAMAMAzhuyaotuimox}:
    \begin{equation}\label{Meqeps}
        \left\{
        \begin{gathered}
            \partial_{t} F_\varepsilon^A+v\cdot \nabla_xF_\varepsilon^A+\frac{e^{A}}{m^A}\nabla_{x}\phi_\epsilon \cdot \nabla_{v}F_\epsilon^A=\frac{1}{\varepsilon}\Big[Q^{AA}(F_\varepsilon^A,F_\epsilon^A)+Q^{AB}(F_\varepsilon^A,F_\varepsilon^B)\Big],
            \\
            \partial_{t} F_\varepsilon^B+v\cdot \nabla_xF_\varepsilon^B+\frac{e_{B}}{m^B}\nabla_{x}\phi_\varepsilon \cdot \nabla_{v}F_\varepsilon^B=\frac{1}{\varepsilon}\Big[Q^{BA}(F_\varepsilon^B,F_\varepsilon^A)+Q^{BB}(F_\varepsilon^B,F_\epsilon^B)\Big],\\
            \Delta \phi_\varepsilon = \int_{\mathbb{R}^3} e^{A}F_\varepsilon^A+ e^{B}F_\varepsilon^Bdv-\bar{n}_e.
        \end{gathered}
        \right.
    \end{equation}
    We assume the solutions of \eqref{Meqeps} have the form of Hilbert expansion:
    \begin{equation}\label{MAINEXPFPHIKKK}
        F_{\varepsilon}^\alpha= \sum_{i=0}^{2k-1}\varepsilon^i F^\alpha_i + \varepsilon^kF_R^{\alpha}, \qquad \nabla_{x}\phi_\epsilon= \sum_{i=0}^{2k-1}\varepsilon^i \nabla_{x}\phi_i + \varepsilon^k\nabla_{x}\phi_R,\quad k\geq6.
    \end{equation}
    Substituting \eqref{MAINEXPFPHIKKK} into the system \eqref{Meqeps}, and compare the coefficients about order of $\varepsilon$, one derives that
    \begin{equation}\label{ORDERZONG}
        \begin{aligned}
            & \text{Order ~$O(\varepsilon^{-1})$:} \qquad\quad ~~ \qquad\qquad ~\,~~\, 0= \displaystyle{\sum_{\beta=A,B}}Q^{ \alpha \beta}(F_0^{\alpha}, F_0^{\beta}), \\
            & \text{Order ~$O(1)$:} \qquad\quad ~~~~( \partial_t +v \cdot \nabla_x+\nabla_x\phi_0\cdot\nabla_v)F_0^\alpha =\displaystyle{\sum_{i'+i''=1}\sum_{\beta=A,B}}Q^{ \alpha \beta}(F_{i'}^{\alpha},F_{i''}^{\beta}),\\
            & \qquad \qquad\quad \qquad\quad\qquad~ \qquad\quad ~   \Delta \phi_0 = \int_{\mathbb{R}^3} e^{A}F_0^A+ e^{B}F_0^Bdv-\bar{n}_e,\\
            & \text{Order ~$O(\varepsilon)$:} \quad \quad\qquad ~~~~    (\partial_t +v \cdot \nabla_x+\nabla_x\phi_0\cdot\nabla_v)F_1^\alpha +\nabla_x\phi_1\cdot\nabla_vF_0^\alpha\\
            &\quad \qquad\qquad \qquad ~~~~~\quad \quad\qquad ~~~~\quad \quad\qquad ~~~~=\displaystyle{\sum_{i'+i''=2} \sum_{\beta=A,B}} Q^{\alpha \beta}(F_{i'}^{\alpha},F_{i''}^{\beta}),\\
            & \qquad \qquad\quad \qquad\quad\qquad~ \qquad\quad ~   \Delta \phi_1 = \int_{\mathbb{R}^3} e^{A}F_1^A+ e^{B}F_1^Bdv,\\
            &.\,.\,.\\
            & \text{Order ~$O(\varepsilon^{i})$:} \,\quad\quad ~~~~~~
            (\partial_t +v \cdot \nabla_{x}+\nabla_x\phi_0\cdot\nabla_v)F_{j}^{\alpha}+\nabla_x\phi_{j}\cdot\nabla_vF_0^\alpha\\
            &\quad \qquad \qquad\qquad ~~~~~\quad \quad\qquad ~~~~\quad \quad\qquad ~~~~=\displaystyle{\mathop{\sum}_{i'+i''=i+1\atop i',i''\geq0} \sum_{\beta=A,B}} Q^{ \alpha\beta}(F_{i'}^{\alpha},F_{i''}^{\beta})-\mathop{\sum}_{i'+i''=i\atop i',i''\geq1}\nabla_x\phi_{i'}\cdot\nabla_vF_{i''}^\alpha,  \\
            & \qquad \qquad\quad \qquad\quad\qquad~ \qquad\quad ~   \Delta \phi_{i} = \int_{\mathbb{R}^3} e^{A}F_{i}^A+ e^{B}F_{i}^Bdv,  \qquad{i}\in \mathbf{Z}^+,\quad {i}\geq2.
        \end{aligned}
    \end{equation}
    Then, the equations governing the remainder term can be expressed as
    \begin{equation}\label{reeqmain}
        \left\{
        \begin{gathered}
            (\partial_t +v \cdot \nabla_{x}+\nabla_x\phi_0\cdot\nabla_v)F_R^{\alpha}+\nabla_x\phi_R\cdot\nabla_vF_0^\alpha-\frac{1}{\varepsilon}\displaystyle{\sum_{\beta=A,B}}(Q^{ \alpha \beta}(F_0^{\alpha}, F_R^{\beta})+Q^{ \alpha \beta}(F_R^{\alpha}, F_0^{\beta}))\\
            \quad \hspace{1.3cm}=\varepsilon^{k-1}\displaystyle{\sum_{\beta=A,B}}Q^{ \alpha \beta}(F_R^{\alpha}, F_R^{\beta})+\displaystyle{\sum^{2k-1}_{i=1}}\displaystyle{\sum_{\beta=A,B}}\varepsilon^{i-1}(Q^{ \alpha \beta}(F_i^{\alpha}, F_R^{\beta})+Q^{ \alpha \beta}(F_R^{\alpha}, F_i^{\beta}))\\
            \hspace{1.8cm}-\varepsilon^k \nabla_x\phi_R\cdot\nabla_vF_R^\alpha-\displaystyle{\sum^{2k-1}_{i=1}}\varepsilon^{i}(\nabla_x\phi_R\cdot\nabla_vF_i^\alpha+\nabla_x\phi_i\cdot\nabla_vF_R^\alpha)+\varepsilon^{k-1}E^{\alpha}, \\
            \Delta \phi_R = \int_{\mathbb{R}^3} e^{A}F_R^A+ e^{B}F_R^Bdv,
        \end{gathered}
        \right.
    \end{equation}
    where
    \begin{equation}\label{dyjjieshuzkep}
        \begin{aligned}
            E^{\alpha}:=&-(\partial_t +v \cdot \nabla_{x})F_{2k-1}^{\alpha}+\mathop{\sum}_{i'+i'' \geq 2k\atop 2k-1 \geq i',i'' \geq1}\displaystyle{\sum_{\beta=A,B}} \varepsilon^{i'+i''-2k} Q^{ \alpha \beta}(F_{i'}^{\alpha},F_{i''}^{\beta})\\
            &-\mathop{\sum}_{i'+i'' \geq 2k-1\atop 2k-1 \geq i',i'' \geq0}\varepsilon^{i'+i''-2k+1}\nabla_x\phi_{i'}\cdot\nabla_vF_{i''}^\alpha.
        \end{aligned}
    \end{equation}
    Finally, from the equations $\eqref{ORDERZONG}_1$ and \eqref{HHHHCCCRM}, we deduce that $\mathbf{F}_0$ is a local bi-Maxwellian:
    \begin{equation}\label{dyjhlibsxF0}
        \mathbf{F}_0 =  \left(
        \begin{array}{cccc}
            F_0^{A}\\
            F_0^{B}
        \end{array}
        \right)
        =
        \left(
        \begin{array}{cccc}
            \frac{n^A(m^A)^{3/2} }{(2\pi \theta)^{3/2}}e^{-\frac{m^A |v-\mathbf{u}|^2}{2\theta}} \vspace{3pt}\\
            \frac{n^B(m^B)^{3/2} }{(2\pi \theta)^{3/2}}e^{-\frac{m^B |v-\mathbf{u}|^2}{2\theta}}
        \end{array}
        \right)= \left(
        \begin{array}{cccc}  \mu^A \\ \mu^B\end{array}
        \right),
    \end{equation}
    where $ n^{\alpha}, \mathbf{u}, \theta$ are the mass, velocity and the temperature,
    which will be determined later.
    \begin{remark}
        The reasons for the truncated Hilbert expansion taking the form of \eqref{MAINEXPFPHIKKK} are as follows. Based on the Hilbert expansion
        \begin{equation*}
            F_{\varepsilon}^\alpha= \sum_{i=0}^{+\infty}\varepsilon^i F^\alpha_i,
        \end{equation*}
        the truncated expansion is assumed to be
        \begin{equation}
            F_{\varepsilon}^\alpha= \sum_{i=0}^{n}\varepsilon^i F^\alpha_i + \varepsilon^{r} F_R^{\alpha},\quad n\geq r\geq 1.
        \end{equation}
        It can be seen from Sections 5 and 6 that the validity time of the solution increases as $\frac{r}{n}$
        decreases, as noted in Remark \ref{FXPINGSHIJIAN} . However, the decrease in $\frac{r}{n}$
        necessitates the expansion of additional terms, requiring $r \geq 6$, see Remark \ref{GYKGET6RE}. We consider both the validity time and term expansion requirements.  To balance their influences, we choose $r = \frac{n+1}{2}$
        , which implies $n = 2k - 1$ and $r = k$. This also establishes a balance between the nonlinear term $\sum_{\beta=A,B} Q^{\alpha \beta}(F_R^{\alpha}, F_R^{\beta})$
        and \eqref{dyjjieshuzkep}, as they share the same coefficient $\varepsilon^{k-1}$. For other coefficients in the remainder expansion, corresponding conclusions can be drawn by following the analysis presented in this paper.
    \end{remark}

    \subsection{Notations}  Let $1\leq p \leq \infty$, we use notation $\Vert \cdot\Vert_{L^{p}_{x,v} } $ to denote the norm of $L^p$ space with respect to the variables $(x,v) \in \mathbb{R}^3  \times \mathbb{R}^3$. $| \cdot|_{L^p_v} $ and $| \cdot|_{L^p_x} $ are the $L^p$ norms   w.r.t. the variable $v$ and $x$ respectively. For vector functions $\mathbf{f}(t,x,v)=(f^A,f^B)^T$, its $L^p$ norms are denoted by
    \begin{equation*}
        \Vert \mathbf{f} \Vert_{L^{p}_{x,v}}=\Vert f^A\Vert_{L^{p}_{x,v}}+\Vert f^B \Vert_{L^{p}_{x,v}},\qquad | \mathbf{f} |_{L^{p}_{v}}=| f^A|_{L^{p}_{v}}+| f^B|_{L^{p}_{v}}.
    \end{equation*}
    The notations  $\left \langle \cdot, \cdot  \right \rangle_{x,v}$ and $\left\langle \cdot, \cdot\right \rangle_v$ are the inner product of the function space $L^2_{x,v}$ and $L^2_{v}$ respectively. For vector functions $\mathbf{f}$ and $\mathbf{g}$ ,  their inner product in $L^2_v$ and $L^2_{x,v}$ are denoted as
    \begin{gather*}
        \left \langle \mathbf{f}, \mathbf{g}\right \rangle_v = \left \langle f^A, g^A   \right \rangle_v + \left \langle f^B, g^B   \right \rangle_v,
        ~~~~
        \left \langle \mathbf{f}, \mathbf{g}\right \rangle_{x,v} = \left \langle f^A, g^A   \right \rangle_{x,v} + \left \langle f^B, g^B   \right \rangle_{x,v}.
    \end{gather*}
Let $\ell=(\ell_1,\ell_2,\ell_3)$ be a multi-index. The
$s-$th order derivative is $\partial_{x}^{\ell}
=\partial_{x_1}^{\ell_1}
\partial_{x_2}^{\ell_2} \partial_{x_3}^{\ell_3}$, with $|\ell|=s $.
$W_x^{s,p}$ ‌represents the standard Sobolev space with corresponding
norm $\Vert \cdot\Vert_{W_x^{s,p}} $. Moreover, when $p=2$, we denote
$W_x^{s,2}$ as $H_x^s$. The letter $C>1$, which may vary from line to line is used to
denote a constant. $J \lesssim K$ means $J \leq C K$. If both hold
$J \lesssim K$ and $K \lesssim J$, we say
$J\thicksim K$.
    We define the polynomial weight function $\left \langle v \right\rangle= 1+|v|$, and  the collision frequency function
    \begin{equation*}
        \upsilon^{\alpha}=   \sum_{\beta=A,B} \int_{\mathbb{R}^3\times \mathbb{S}^2}
        B^{ \alpha \beta}(|v-v_*|, \cos \theta) \mu^{\beta}(v_*) d\sigma dv_*.
    \end{equation*}
   Then for $-3<\gamma\leq1$, it can be proved that $\upsilon^{\alpha} \thicksim \left \langle v \right\rangle^{\gamma}$, so we introduce the following weighted $L^2$ norm
    \begin{gather*}
        \Vert f \Vert_{\nu}^2 = \iint_{\mathbb{R}^3 \times \mathbb{R}^3} |f|^2 \left \langle v \right\rangle^{\gamma} dxdv, \qquad  | f |_{\nu}^2 = \int_{\mathbb{R}^3 } |f|^2 \left \langle v \right\rangle^{\gamma}dv.
    \end{gather*}

   Next, we introduce a global bi-Maxwellian
    \begin{gather*}\label{def1.35}
        \mathbf{\mu}_{M}=
        \left(
        \begin{array}{cccc}
            \mu_{M}^A\\
            \mu_{M}^B
        \end{array}
        \right)
        =\frac{1}{(2\pi \theta_M)^{3/2}}
        \left(
        \begin{array}{cccc}
            e^{- \frac{m_A|v|^2}{2\theta_M}}\\
            e^{- \frac{m_B|v|^2}{2\theta_M}}
        \end{array}
        \right).
    \end{gather*}
    From \eqref{CCXXTJ25} and \eqref{quedwenducc}, the constant $\bar{\theta}$ is defined by
     \begin{equation}
        \bar{\theta}=[c_1+c_2(C_p)^{\frac{2}{3}}](\bar{n}_1)^{\frac{2}{3}}.
    \end{equation}
     Without loss of generality, we assume $m^A\geq m^B$. There exists
    $\theta_M=\frac{2m^A}{2m^A+m^B}\bar{\theta}$
    such that
    \begin{equation*}\label{globalMaxwellian}
        \theta_M\leq\mathop{\rm min}_{(t,x)\in [0,\tau]\times \mathbb{R}^3} \theta(t,x)\leq\mathop{\rm max}_{(t,x)\in [0,\tau]\times \mathbb{R}^3} \theta(t,x)\leq \frac{m^A+m^B}{m^A} \theta_M.
    \end{equation*}
   Moreover, based on the bound \eqref{PTWDAY}, there exist constants $C>0 $ and some
    $\frac{m^A}{m^A+m^B}<\tilde{q}<1$, such that for each $(t,x,v)\in \mathbb{R}^{+} \times \mathbb{R}^3\times \mathbb{R}^3$, it holds
    \begin{equation*}\label{1.36}
        C^{-1} \mu^{\alpha}_M \leq \mu^{\alpha}  \leq C  (\mu^{\alpha}_M)^{\tilde{q}}.
    \end{equation*}
     Then, we can propose the following weight function:
    \begin{equation}\label{WWWEIGHTFUC}
        w_{\gamma}(v)=
        \begin{cases}
            &\left\langle v\right\rangle^{l}   \qquad \,\, \,\gamma=1, \\
            &\omega_{\gamma} \,\,\,\quad -1\leq \gamma<1, \\
            &\omega_{-1} \quad  -3<\gamma<-1 .
        \end{cases}
    \end{equation}
     With fixed $k$ from \eqref{MAINEXPFPHIKKK}, the weight function satisfies:
    \begin{equation}\label{weightFUVS}
        \omega_{\gamma}=e^{\tilde{\kappa}\left \langle v \right\rangle^{\frac{3-\gamma}{2}}},\quad \tilde{\kappa}=\kappa_0 (1+(1+t)^{-\frac{2}{2k-1}}), \quad 0<\kappa_0=\kappa_0(m^A,m^B)\ll 1.
    \end{equation}
    The rationale for constructing the weight function \eqref{weightFUVS} is detailed in Section 5.

    Finally, from the remainder term equations \eqref{reeqmain}, the $L^{\infty}$ weighted functions are given by:
    \begin{equation}\label{WWDWEIGHTFUC}
        \mathbf{R}_{\gamma}=(R_{\gamma}^A,R_{\gamma}^B)^T=\begin{cases}
            &\mathbf{h}=(h^A,h^B)^T     \qquad \,\qquad  \,\gamma=1, \\
            &\mathbf{S}_{\gamma}=(S^A_{\gamma},S^B_{\gamma})^T \,\quad \quad -1\leq \gamma<1, \\
            &\mathbf{S}_{-1}=(S^A_{-1},S^B_{-1})^T \quad  -3<\gamma<-1,
        \end{cases}
    \end{equation}
    here
    \begin{equation}\label{Threedefweight1}
        h^\alpha=  \frac{\left \langle v \right\rangle^{l} }{\sqrt{\mu^\alpha_M}} F_R^\alpha, \quad l\geq7, \qquad  S^\alpha_{\gamma}=  \frac{\omega_{\gamma} }{\sqrt{\mu^\alpha_M}} F_R^\alpha, \quad -1\leq \gamma<1, \quad {\rm for} ~~\alpha=A,B.
    \end{equation}
    The validity time is defined by
    \begin{equation}\label{TTMSJ}
    T_\gamma=
    \begin{cases}
        &T_L \,\,\,\quad -1\leq \gamma \leq 1, \\
        &T_S \quad  -3<\gamma<-1,
    \end{cases}
\end{equation}
where
\begin{equation}
    T_L=O(\varepsilon^{-\frac{2k-3}{2(2k-1)}}),\qquad T_S=O(\varepsilon^{-\frac{2k-3}{(1-\gamma)(2k-1)}}).
\end{equation}
The initial total energy is defined as
\begin{equation*}
    \mathcal{E}^{\rm in}=\sum_{\alpha=A,B} \Big(\Big\|\frac{F_{R}^{\rm \alpha, in}}{\sqrt{ \mu^{\rm \alpha, in}}} \Big\|_{L^{2}_{x,v} } +|\nabla_{x}\phi^{\rm in}_R|_{{L^{2}_{x} }}+ \varepsilon^{\frac{3}{2}} \Big\Vert \frac{\left \langle v \right\rangle^{2-\gamma}w_{\gamma} F_{R}^{\rm \alpha,in}}{\sqrt{\mu^{\rm \alpha, in}_M}} \Big\Vert_{L^{\infty}_{x,v} } +\varepsilon^5 \Big\Vert \nabla_{x,v}(\frac{w_{\gamma} F_{R}^{\rm \alpha,in}}{\sqrt{\mu^{\rm \alpha, in}_M}}) \Big\Vert_{L^{\infty}_{x,v} }\Big).
\end{equation*}
    \subsection{The main result} Now, we state the main result of this article.
    \begin{thm}\label{maintheorem}
        Let $\mathbf{F}_0$
        be defined in \eqref{dyjhlibsxF0}. $(n^A, n^B, \mathbf{u}, \theta)$
        is the smooth solution to the hyperbolic system \eqref{EQF0EPSION}-\eqref{IID24} established in Lemma \ref{CPCO}. We denote $F_j^{\alpha} = f_j^{\alpha} \sqrt{\mu}$
        for $j \geq 1$ and $\alpha = A, B$, as obtained in Section 2.2. Suppose that the initial data are given as
        \begin{equation}\label{initialdatatogether}
            \left(
            \begin{array}{ccc}
                F_\varepsilon^{ A}(0,x,v)\\[2mm]
                F_\varepsilon^{ B}(0,x,v)
            \end{array}
            \right)=
            \sum_{i=0}^{2k-1}\varepsilon^{i}
            \left(
            \begin{array}{ccc}
                F_{k}^{A, \rm in}(x,v)\\[2mm]
                F_{k}^{B, \rm in}(x,v)
            \end{array}
            \right)  +\varepsilon^k
            \left(
            \begin{array}{ccc}
                F_{R}^{A, \rm in}(x,v)\\[2mm]
                F_{R}^{B, \rm in}(x,v)
            \end{array}
            \right) \geq 0.
        \end{equation}
        Then, there exists a small $\varepsilon_0>0$ such that for all $0<\varepsilon \leq \varepsilon_0$,  the initial value problem \eqref{Meqeps}-\eqref{initialdatatogether} admits a unique solution. The remainder term satisfies:
        \begin{gather}\label{mainthemSOFTCASD}
            \sup_{0\leq t \leq T_\gamma} \Big(  \Big\|\frac{F_{R}^{\rm \alpha }}{\sqrt{ \mu^{\rm \alpha}}} \Big\|_{L^{2}_{x,v} } + |\nabla_{x}\phi^{\rm }_R|_{L^{2}_{x} } +\varepsilon^{\frac{3}{2}} \Big\Vert \frac{\left \langle v \right\rangle^{2-\gamma}w_{\gamma} F_{R}^{\rm \alpha}}{\sqrt{\mu^{\rm \alpha}_M}} \Big\Vert_{L^{\infty}_{x,v} } +\varepsilon^5 \big\Vert \nabla_{x,v}(\frac{w_{\gamma} F_{R}^{\rm \alpha}}{\sqrt{\mu^{\rm \alpha}_M}}) \Big\Vert_{L^{\infty}_{x,v} }\Big) \leq \mathcal{E}^{\rm in}.
        \end{gather}
    \end{thm}
    \begin{remark}[$\mathbf{Criteria\,\, for\,\, classification}$]
        This theorem is divided into three cases: the critical case $(\gamma=1)$, the moderate range $(-1 \leq \gamma < 1)$, and the strongly singular case $(-3 < \gamma < -1)$, with the soft potential analysis requiring the additional constraint of compact initial values.
    \end{remark}

    \begin{remark}[$\mathbf{Different \,\,weight \,\,functions\,\,for\,\,the \,\,remainder \,\,term}$]
        ‌The rationale for treating the hard sphere model separately stems from the following estimate:
        $$
        \Big|\int_{0}^{t}\mathcal{F}_0^{\alpha}(t,s) \mathcal{H}^\alpha_3(s,X(s),V(s))ds \Big|
        \lesssim \varepsilon(1+\varepsilon\mathcal{I}_1) e^{-\frac{\nu_0 t}{2\varepsilon}}
        \sup_{0\leq s \leq t}(e^{\frac{\nu_0 s}{2\varepsilon}}\Vert \langle v \rangle^{1-\gamma}\mathbf{h}_R (s)\Vert_{L^{\infty}_{x,v}}).
        $$
        Polynomial weights suffice for the $L_{x,v}^2$ estimation. However, for $-3 < \gamma < 1$, standard nonlinear estimation techniques fail to handle the linear term $\mathcal{H}^\alpha_3$ \eqref{h3h3h3ddq} adequately. This necessitates the construction of our new weight function $\omega_{\gamma}$ \eqref{weightFUVS} to transform and control this term.
    \end{remark}

    \begin{remark}[$\mathbf{Shortening \,\,of \,\,the\,\,validity\,\,time \,\,for \,-3<\gamma<-1}$]
        Section 5 initially establishes that the weight function
        $\omega_{\gamma} = e^{\tilde{\kappa} \left\langle v \right\rangle^{\frac{3-\gamma}{2}}}$
        preserves the time of validity
        $T_L = O(\varepsilon^{-\frac{2k-3}{2(2k-1)}})$. Notably, the specific case
       $\gamma = -1$ yields
        $e^{\tilde{\kappa} \left\langle v \right\rangle^{2}}$. As emphasized in Remark \ref{WLIGHTERANALYSIS}, the constraint
        $\frac{3-\gamma}{2}\leq 2$ implies that the solution's time of validity is necessarily reduced to
        $T_S = O(\varepsilon^{-\frac{2k-3}{(1-\gamma)(2k-1)}})$.
    \end{remark}

\noindent \textbf{The paper is structured as follows}.
  Section 2 derives the macroscopic components $\mathbf{F}_i$ ($i=0,\dots,2k-1$) evolution equations and establishes uniform a priori estimates.
Section 3 obtains the $L^2$-norm estimates of the remainder term through energy method analysis.
Section 4 provides $W_{x,v}^{1,\infty}$ estimates for $h^\alpha$ via Duhamel's principle,  addressing the $\gamma=1$ case.
 Section 5 extends the analysis to $-1\leq\gamma<1$ cases for $S^\alpha_\gamma$ in $W_{x,v}^{1,\infty}$ framework.
 Section 6 investigates the more singular $-3<\gamma<-1$ case, analyzing the shortening of the validity time‌.
Within this $L^2$--$W_{x,v}^{1,\infty}$ hybrid analytical framework, we ultimately prove Theorem~\ref{maintheorem}, which rigorously establishes the hydrodynamic limit for multi-component VPB systems.

    \subsection{Innovations of this article}
    The innovations and challenges discussed in this paper arise primarily from three aspects: the physical model, the technical aspects, and the conclusions.
    \\
    \noindent $\mathbf{1.\,\,The \,\,physical\,\,model.}\,$
    \begin{itemize}
\item [\ding{72}]The existing results for Vlasov-Poisson-Boltzmann (VPB) fluid equations concerning gas mixtures often adopt the normalization assumption where both mass and charge magnitude are set to unity, which inherently carries limitations. In contrast, the physical model considered in this work maintains greater fidelity by preserving the actual ionic masses and charge states. Consequently, our conclusions not only enable accurate modeling of ion collision dynamics in atmospheric thermosphere but also demonstrate broader applicability across various physical and chemical domains.
\\
\item [\ding{72}]The Euler-Poisson system \eqref{EQF0EPSION} derived from the VPB system for gas mixtures ($m^A \neq m^B$) includes $\rho = m^A n^A + m^B n^B$, $\tilde{n} = n^A + n^B$, and $n_e = e^A n^A + e^B n^B$. These quantities make it more difficult to prove the existence of solutions for this system.
Furthermore, the equations of the macroscopic part for higher-order systems are linear equations with a coefficient matrix of size $6 \times 6$, which is more complex than the equations derived from the VPB system for a single type of gas. It only has a $5 \times 5$ coefficient matrix. \\
\item [\ding{72}]Compared to a single-gas model, gas mixtures with unequal masses exhibit distinct remainder estimates, where the difference in the $L^2$
        estimate is caused by the mathematical structure of the terms $\frac{e^{\alpha}}{m^{\alpha}} \nabla_{x} \phi \cdot \nabla_{v} F^{\alpha}$
        and $\Delta \phi = \int_{\mathbb{R}^3} e^{A} F^A + e^{B} F^B \, dv$.  While, the structure of the characteristic lines in the $L^{\infty}$
        estimate is different (if $\frac{e^A}{m^A} \neq \frac{e^B}{m^B}$
        ), featuring a combination of two different types of characteristic lines in \eqref{doublelineqianT}.
    \end{itemize}

    \noindent $\mathbf{2.\,\,The \,\,technical \,\,aspects.}\,$
    \begin{itemize}
        \item [\ding{72}]  When conducting $L^{\infty}$ estimation, we will encounter the following nonlinear term:
        \begin{equation*}
            \mathcal{H}^\alpha_3 = -\frac{e^{\alpha}}{m^{\alpha}}\nabla_{x}\phi^{\varepsilon} \cdot  \frac{\left \langle v \right\rangle^{l}}{\sqrt{\mu^{\alpha}_M}} \nabla_{v}(\frac{\sqrt{\mu^{\alpha}_M}}{\left \langle v \right\rangle^{l}})h^\alpha, \, \quad  |\mathcal{H}^\alpha_3|  \thicksim \left \langle v \right\rangle \Vert\mathbf{h} \Vert_{L^{\infty}_{x,v}}.
        \end{equation*}
        If $\gamma \neq 1$, the decay provided by the characteristic lines will be insufficient to control its growth in velocity, since
        \begin{equation*}
            \left|\int_{0}^{t}\mathcal{F}_0^{\alpha}(t,s) \mathcal{H}^\alpha_3(s,X(s),V(s))ds \right|
            \lesssim \varepsilon(1+\varepsilon\mathcal{I}_1) e^{-\frac{\nu_ot}{2\varepsilon}} \displaystyle{\sup_{0\leq s \leq t}}(e^{\frac{\nu_os}{2\varepsilon}}\Vert \left \langle v \right\rangle^{1-\gamma}\mathbf{h} (s)\Vert_{L^{\infty}_{x,v}}).
        \end{equation*}
        To address this, we propose a  weight function $\omega_{\gamma}(t,v)$(see \eqref{WWWEIGHTFUC} and \eqref{weightFUVS}) to transform it via the equation:
        \begin{equation*}
            \begin{split}
                \hat{\upsilon_{\varepsilon}^\alpha}(t,x,v):&=   \frac{\upsilon^\alpha}{\varepsilon} +\frac{e^{\alpha}}{m^{\alpha}}\nabla_{x}\phi^{\varepsilon} \cdot  \frac{\omega_{\kappa}(t,v)}{\sqrt{\mu^{\alpha}_M}} \nabla_{v}(\frac{\sqrt{\mu^{\alpha}_M}}{\omega_{\kappa}(t,v)})-\frac{\partial_t\omega_{\kappa}(t,v)}{\omega_{\kappa}(t,v)}\\
                &\geq \frac{\left \langle v \right\rangle^{\gamma}}{C\varepsilon}- C\left \langle v \right\rangle \Vert\nabla_{x}\phi^{\varepsilon} \Vert_{L^{\infty}_{x,v} } +\frac{\kappa_0\kappa_1}{C}\frac{\left \langle v \right\rangle^{\kappa_2}}{(1+t)^{1+\kappa_1}}.
            \end{split}
        \end{equation*}
        Young's inequality is used to transform $\left \langle v \right\rangle$:
        \begin{equation*}
            \begin{split}
                \left \langle v \right\rangle &= \Big(\left \langle v \right\rangle^{\gamma}\varepsilon^{\iota-1}(1+t)^{\frac{(1+\kappa_1)(1-\gamma)}{\kappa_2-1}}\Big)^{\frac{\kappa_2-1}{\kappa_2-\gamma}}    \Big(\left \langle v \right\rangle^{\kappa_2}\varepsilon^{\frac{(1-\iota)(\kappa_2-1)}{1-\gamma}}(1+t)^{-(1+\kappa_1)}\Big)^{\frac{1-\gamma}{\kappa_2-\gamma}}\\
                & \leq \frac{1}{\varepsilon}\left \langle v \right\rangle^{\gamma}\varepsilon^{\iota}(1+t)^{\frac{(1+\kappa_1)(1-\gamma)}{\kappa_2-1}} + \left \langle v \right\rangle^{\kappa_2}\varepsilon^{\frac{(1-\iota)(\kappa_2-1)}{1-\gamma}}(1+t)^{-(1+\kappa_1)}.
            \end{split}
        \end{equation*}
        Through subsequent calculations, with appropriate selection of parameters $\kappa_1, \kappa_2$, and $\iota$ (detailed in Scetion 5 \eqref{xxbfANLmainlower1}-\eqref{gghk5sctm}), one obtains
        \begin{equation*}
            \begin{split}
                \hat{\upsilon_{\varepsilon}^\alpha}
                &\geq \frac{2\left \langle v \right\rangle^{\gamma}}{3C\varepsilon} +\frac{1}{4Ck}\frac{\left \langle v \right\rangle^{\frac{3-\gamma}{2}}}{(1+t)^{\frac{2k+1}{2k-1}}}> \frac{2\left \langle v \right\rangle^{\gamma}}{3C\varepsilon}.
            \end{split}
        \end{equation*}
        This effectively helps us overcome the difficulties caused by the electric field effect.‌‌

        \item [\ding{72}]Through a rigorous analysis of the velocity decay properties of the operator $K_{M,w}^{\alpha,c}$ (Lemma \ref{LeMK2ker4444}, under the condition $m^A \neq m^B$), in the local $L_{x,v}^{\infty}$ framework, we are able to establish the following estimate:
        \begin{equation*}
            \begin{split}
                &\int_{\mathbb{R}^3} \int_{\mathbb{R}^3} |k_{M,\omega_{\gamma}}^{\alpha\beta}(s,X^{\alpha}(s),V^{\alpha}(s),v_*)k_{M,\omega_{\gamma}}^{\beta\beta^{'}}(s_1,X^{\alpha,\beta}_1(s_1),V^{\alpha,\beta}_1(s_1),v_{**})|dv_{**}\, dv_*\\
                &\hspace{6.3cm}\leq \frac{C\left \langle V(s) \right\rangle^{\gamma}}{1+|V(s)|}  \times \left \langle V_1(s_1) \right\rangle^{\gamma} .
            \end{split}
        \end{equation*}
        When dealing with soft potentials, we need to ‌squeeze out an $\varepsilon$ from the certain nonlinear terms to eliminate‌ the intrinsic
        $\frac{1}{\varepsilon}$, (detailed in Scetion 5 \eqref{resollowerbbd}).
        \begin{gather*}
            \left|\int_{0}^{t} \mathcal{T}_{0}^{\alpha}(t,s) f(t,x,v) ds \right|\lesssim  \left|\int_{0}^{t} \left \langle v \right\rangle^{\gamma}\mathcal{T}_{0}^{\alpha}(t,s)  ds \right| \sup_{0\leq s \leq t} \|\frac{f(s)}{\left \langle v \right\rangle^{\gamma}}\|_{L^{\infty}_{x,v}}\lesssim \varepsilon \sup_{0\leq s \leq t} \|\frac{f(s)}{\left \langle v \right\rangle^{|\gamma|}}\|_{L^{\infty}_{x,v}}.
        \end{gather*}
    \end{itemize}

    \noindent $\mathbf{3.\,\,The \,\,conclusions.}\,$
    \begin{itemize}
        \item [\ding{72}]The potential range that we consider is $(-3, 1]$, which includes four different cases: $\gamma = 1$, $0 \leq \gamma < 1$, $-1 \leq \gamma < 0$, and $-3 < \gamma < -1$. Moreover, this approach provides a unification of methods for handling different potentials, as previously addressed by others using various techniques. Our findings further quantify the degree of change in polynomial-weighted initial data, which is always $\langle v \rangle^{2 - \gamma}$
        for $\gamma \in \left(-3, 1\right]$.\\
        \item[\ding{72}] We discuss the validity time that depends on the Knudsen number $\varepsilon$. The validity time is $O(\varepsilon^{-\frac{2k-3}{2(2k-1)}})$
        when $-1 \leq \gamma \leq 1$, and it can be $O(\varepsilon^{-\frac{2k-3}{(1-\gamma)(2k-1)}})$
        when $-3 < \gamma < -1$. This conclusion applies to all $k \geq 6$. We note that the validity time is also continuous with respect to the potential $\gamma \in (-3, 1]$, see Figure 2.
    \end{itemize}

    \section{The determination of $\mathbf{F_i}~~(i=0,\cdots,2k-1)$.}
    \subsection{The equations for the equilibrium $\mathbf{F}_0$}
    From \eqref{dyjhlibsxF0}, we know $\mathbf{F}_0$ is a local  bi-Maxwellian
    \begin{equation}\label{F0PHT2121}
        \mathbf{F}_0 =  \left(
        \begin{array}{cccc}
            F_0^{A}\\
            F_0^{B}
        \end{array}
        \right)
        =
        \left(
        \begin{array}{cccc}
            \frac{n^A(m^A)^{3/2} }{(2\pi \theta)^{3/2}}e^{-\frac{m^A |v-\mathbf{u}|^2}{2\theta}} \vspace{3pt}\\
            \frac{n^B(m^B)^{3/2} }{(2\pi \theta)^{3/2}}e^{-\frac{m^B |v-\mathbf{u}|^2}{2\theta}}
        \end{array}
        \right).
    \end{equation}
    By a direct calculation, $(n^A,n^B,\mathbf{u},\theta)$ satisfies the following relations:
    \begin{equation*}
        \int_{\mathbb{R}^3} F_0^{\alpha} dv = n^{\alpha}, \quad \int_{\mathbb{R}^3} v F_0^{\alpha} dv = n^{\alpha}\mathbf{u},\quad \int_{\mathbb{R}^3}|v|^2 F_0^{\alpha} dv = n^{\alpha} |\mathbf{u}|^2+3n^{\alpha}\theta,
    \end{equation*}
     and we denote
    \begin{equation*}
        \rho=m^An^A+m^Bn^B, \quad \tilde{n}=n^A+n^B, \quad n_e=e^An^A+e^Bn^B.
    \end{equation*}
    The first line of the equation $\eqref{ORDERZONG}_1$ can be written in a vector form
    \begin{equation*}
        \left(
        \begin{array}{cccc}
            ( \partial_t +v \cdot \nabla_x+\nabla_x\phi_0\cdot\nabla_v)F_0^A\\
            ( \partial_t +v \cdot \nabla_x+\nabla_x\phi_0\cdot\nabla_v)F_0^B
        \end{array}
        \right)=\sum_{i+j=1}
        \left(
        \begin{array}{cccc}
            Q^{AA}(F_i^A,F_j^A)+Q^{AB}(F_i^A,F_j^B)\vspace{3pt}\\
            Q^{BA}(F_i^B,F_j^A)+Q^{BB}(F_i^B,F_j^B)
        \end{array}
        \right).
    \end{equation*}

Moreover, taking inner product with six collision invariants $
\{\mathbf{e}_1, \mathbf{e}_2, v_1 \mathbf{m}, v_2 \mathbf{m}, v_3
\mathbf{m}, |v|^2 \mathbf{m}\}$(see \eqref{YZJVEPDB}), one can
obtain the system of $(n^{A},n^{B},\mathbf{u},\theta)$:
\begin{equation}\label{EQF0EPSION}
        \left \{
        \begin{array}{lll}
            \partial_t n^{A} +{\rm div}(n^{B}\mathbf{u})=0, \\
            \partial_t n^{B} +{\rm div}(n^{B}\mathbf{u})=0, \\
            \partial_t ( \rho \mathbf{u})
            +{\rm div}(\rho\mathbf{u}\otimes \mathbf{u})
            +\nabla (\tilde{n}\theta)-n_e\nabla\phi_0=0, \\
\partial_t[\frac{\rho}{2}|\mathbf{u}|^2+ \frac{3}{2}\tilde{n}\theta]
            +{\rm div}\left[(\frac{\rho}{2}|\mathbf{u}|^2+ \frac{3}{2}\tilde{n}\theta)\mathbf{u}+ \tilde{n}\theta\mathbf{u}\right]-n_e\nabla\phi_0\cdot\mathbf{u}=0,\\
            \Delta \phi_0 = n_e -\bar{n}_e,
        \end{array}
        \right.
    \end{equation}
    with initial data
    \begin{equation}\label{IID24}
        [n^{A}(x),n^{B}(x),\mathbf{u}(x),\theta(x)]|_{t=0}=[n^{A,in}(x),n^{B,in}(x),\mathbf{u}^{in}(x),\theta^{in}(x)].
    \end{equation}
Further, by replacing $\eqref{EQF0EPSION}_4$ with
$(\eqref{EQF0EPSION}_4-\eqref{EQF0EPSION}_3 \cdot \mathbf{u})$, and
utilizing $\eqref{EQF0EPSION}_1$ and $\eqref{EQF0EPSION}_2$, the
system $\eqref{EQF0EPSION}$ can be written as
    \begin{equation}\label{REWEUPSIO}
        \left \{
        \begin{array}{lll}
            \partial_t n^{A} +{\rm div}(n^{A}\mathbf{u})=0, \\[2mm]
            \partial_t n^{B} +{\rm div}(n^{B}\mathbf{u})=0, \\[2mm]
            \rho \partial_t  \mathbf{u}
            +\rho (\mathbf{u}\cdot\nabla)\mathbf{u}
            +\nabla (\tilde{n}\theta)-n_e\nabla\phi_0=0, \\[2mm]
            \partial_t \theta +(\mathbf{u}\cdot\nabla)\theta+\frac{2}{3}\theta {\rm div} \mathbf{u}=0,\\
            \Delta \phi_0 =  n_e-\bar{n}_e.
        \end{array}
        \right.
    \end{equation}
Denote $ \mathbf{U}_0 =( n^A, n^B, \mathbf{u}, \theta )^T$, and $
\mathbf{V}_0 =( 0, 0, \nabla\phi_0, 0 )^T$. The system
\eqref{REWEUPSIO} could be changed into the following form
    \begin{gather}\label{leadingordervector}
        \partial_t \mathbf{U}_0-\frac{n_e}{\rho}\mathbf{V}_0 + \sum_{j=1}^{3}   \mathbf{A}_j(\mathbf{U}_0) \partial_{x_j}  \mathbf{U}_0 =  0,
    \end{gather}
    where
    \begin{equation*}
        \mathbf{A}_j(\mathbf{U}_0) =    \begin{pmatrix}
            u^j  & 0 & n^A\mathbf{e}_j^{T} & 0  \\
            0 & u^j & n^B\mathbf{e}_j^{T}  & 0 \\
            \frac{\theta}{\rho}\mathbf{e}_j  & \frac{\theta}{\rho}\mathbf{e}_j &  u^j \mathbf{I}_3 &  \frac{\tilde{n}}{\rho}\mathbf{e}_j \\
            0  & 0 & \frac{2}{3}\theta\mathbf{e}_j^{T}  & u^j
        \end{pmatrix},
    \end{equation*}
    with the standard vectors $e_1=(1,0,0)^T,~e_2=(0,1,0)^T,~e_3=(0,0,1)^T$ and $\mathbf{I}_3={\rm diag}(1\,1\,1)$. As stated in \cite{[62]A.Majda},  the system \eqref{leadingordervector} is well-posedness in local time if it is a symmetrizable hyperbolic system. We can find a positive defined symmetric matrix:
    \begin{equation*}
        \mathbf{A}_0(\mathbf{U}_0) =    \begin{pmatrix}
            \frac{\theta}{n^A}  & 0 & 0 & 0  \\
            0 & \frac{\theta}{n^B} & 0  & 0 \\
            0  & 0 & \rho \mathbf{I}_3 & 0 \\
            0  & 0 & 0  & \frac{3}{2}\frac{\tilde{n}}{\theta}
        \end{pmatrix},
    \end{equation*}
    and multiply \eqref{leadingordervector} by $\mathbf{A}_0(\mathbf{U}_0)$ to obtain
    \begin{gather}\label{leadisssngordervector}
        \mathbf{A}_0(\mathbf{U}_0)(\partial_t \mathbf{U}_0-\frac{n_e}{\rho}\mathbf{V}_0) + \sum_{j=1}^{3} \mathbf{A}_0(\mathbf{U}_0)\mathbf{A}_j(\mathbf{U}_0) \partial_{x_j}  \mathbf{U}_0 =  0.
    \end{gather}
    It can be verified that
    \begin{equation*}
        \begin{gathered}
            \mathbf{A}_0(\mathbf{U}_0)\mathbf{A}_j(\mathbf{U}_0) =  \begin{pmatrix}
                \frac{u^j\theta}{n^{A}}  & 0 & \theta\mathbf{e}_j^{T}  & 0  \\
                0 & \frac{u^j\theta}{n^{B}} & \theta\mathbf{e}_j^{T}  & 0 \\
                \theta\mathbf{e}_j  & \theta\mathbf{e}_j & \rho u^j \mathbf{I}_3 & \tilde{n} \mathbf{e}_j \\
                0  & 0 & \tilde{n} \mathbf{e}_j^{T}  & \frac{3}{2}\frac{\tilde{n}}{\theta}u^j \\
            \end{pmatrix}
        \end{gathered} \quad j=1,2,3.
    \end{equation*}

Since the matrices in \eqref{leadisssngordervector} are symmetric,
according to the classical mathematical theory of hyperbolic systems
(see \cite{[62]A.Majda}), there exists a local in-time solution in
the corresponding Sobolev space. However, finding a global solution
to the Euler-Poisson system is a very difficult task. Even for the
fluid system derived from the VPB equations of a single gas particle
remains a challenging issue. In \cite{[lc]TM1986}, T. Makino worked
out a local solution with compact support for the initial data.
However, what we need is a global-in-time solution. The equilibrium
cannot be vacuum. The existing conclusions regarding global-in-time
solutions without compact support of initial data, including the
case of $p = \rho^{\frac{5}{3}}$, always impose additional
structural conditions, such as the irrotational condition in
\cite{[ivR]Guo1998} and the spherical symmetry condition in
\cite{[8]CGQSSD}. Unfortunately, due to the quality and arbitrary
selection of the charge $e^{\alpha}$, the non-homogeneous system of
equations cannot maintain these additional structural conditions
like a single equation. Thus, the existing relevant results cannot
be utilized. Currently, the only feasible method is to consider the
initial data for $n^A$ and $n^B$ in a certain proportion. This
allows us to transform the full Euler-Poisson system into the
isentropic Euler-Poisson system, thereby utilizing the results in
\cite{[ivR]Guo1998}.
    \begin{lemma}\label{CPCO}
        Let initial data $n^{\rm B,in}=C_p n^{\rm A,in}$ and $n^{\rm A,in}-\bar{n}_1,n^{\rm B,in}-\bar{n}_2, \mathbf{u}^{in}$  be sufficiently small, where $\bar{n}_1, \bar{n}_2=C_p\bar{n}_1, \bar{n}_e=e^A\bar{n}_1+e^B\bar{n}_2,$ are constants. They satisfy following conditions:
        \begin{equation}\label{CCXXTJ25}
            \nabla \times \mathbf{u}^{in}=0, \qquad \quad \int_{\mathbb{R}^3}(n_e^{in}-\bar{n}_e)dx=0.
        \end{equation}
        Then, there exist a unique global solution
        $(n^{A},n^{B},\mathbf{u},\theta)$, where $\frac{1}{C}<\theta<C$, and $n^{A},n^{B},\mathbf{u},\phi_0$ satisfy the following pointwise decay
        \begin{equation}\label{PTWDAY}
            \Vert n^{A}-\bar{n}_1\Vert_{W^{s,\infty}}+\Vert n^{B}-\bar{n}_2\Vert_{W^{s,\infty}}+\Vert \mathbf{u} \Vert_{W^{s,\infty}}+\Vert \nabla\phi_0 \Vert_{W^{s,\infty}}\leq\frac{C}{(1+t)^q},
        \end{equation}
        for each $s\geq0$ and $1<q<\frac{1}{3}$.
    \end{lemma}
    \begin{proof}
        Let $n^B(t,x)=P(t,x)n^A(t,x)$. From $\eqref{REWEUPSIO}_2$, one can get
        \begin{equation*}\label{PPFq}
            \partial_t P + u \cdot \nabla_{x}P=0, \\[2mm]
        \end{equation*}
        with initial data  $P(0,x)=C_p$. It is easy to solve $P(t,x)\equiv C_p$.
        We note that
        the continuity equation for $n^A,n^B$ and the temperature equation for $\theta$ are equivalent when
        \begin{equation}\label{quedwenducc}
            \theta=c_1(n^A)^{\frac{2}{3}}+c_2(n^B)^{\frac{2}{3}}=[c_1+c_2(C_p)^{\frac{2}{3}}](n^A)^{\frac{2}{3}},
        \end{equation}
        where $c_1$ and $c_2$ are two constants. We plunge \eqref{quedwenducc} into \eqref{REWEUPSIO},
        then they satisfy the following isentropic Euler-Poisson system:
        \begin{equation*}\label{REWISPeq}
            \left \{
            \begin{array}{lll}
                \partial_t n^{A} +{\rm div}(n^{A}\mathbf{u})=0, \\[2mm]
                \partial_t  \mathbf{u}
                + (\mathbf{u}\cdot\nabla)\mathbf{u}
                +\frac{\nabla[(1+C_p)(c_1+c_2C_p^{\frac{2}{3}})(n^{A})^{\frac{5}{3}}] }{(m^A+C_pm^B) n^{A}}-\frac{e^A+e^BC_p}{m^A+C_pm^B }\nabla\phi_0=0, \\[2mm]
                \Delta \phi_0 =  (e^{A}+ e^{B}C_p)n^A-(e^{A}+ e^{B}C_p)\bar{n}_1.
            \end{array}
            \right.
        \end{equation*}
        Since $c_1,c_2$ and $C_p$ are constants, it holds that
        \begin{equation*}\label{WXJGHDSBXCL}
            \nabla\times\Big\{\frac{\nabla[(1+C_p)(c_1+c_2C_p^{\frac{2}{3}})(n^{A})^{\frac{5}{3}}] }{(m^A+C_pm^B) n^{A}}-\frac{e^A+e^BC_p}{m^A+C_pm^B }\nabla\phi_0\Big\} \equiv 0.
        \end{equation*}
        Therefore, we can use the proof in \cite{[ivR]Guo1998}. \end{proof}

    \begin{remark}
        There are also some results concerning the Euler-Poisson equations for mixed fluids, such as \cite{[iJMP]Guo}. However, these fluid equations are not directly derived from Vlasov–Poisson-Boltzmann equations. The fundamental reason lies in the fact that the equilibrium state of gas mixture derived from momentum conservation and energy conservation, which determines special forms of collision invariants. So it is hard to ‌decouple‌ the momentum equations and energy equations for the two different fluids.
    \end{remark}
    To accommodate additional symbols in the following section, we temporarily repurpose the letter $l$, distinct from its usage in \eqref{Threedefweight1}.
    \subsection{The construction of  $F^{\alpha}_{l}$ ($\alpha\in\{A,B\}$, $ l = 1, 2, \cdots, 2k-1$)}
    First, we introduce the linearized collision operator and discuss some of its necessary properties.

    Let $\mathbf{F}_{l}=(F^A_{l},F^B_{l})^{T}=(\sqrt{\mu^A}f^A_l,\sqrt{\mu^B}f^B_l)^{T}$, $\mathbf{f}_l=(f^A_l,f^B_l)^{T}$, and we can define the linearized bi-collision operator $\mathbf{L}$  as
    \begin{gather}\label{LNOTATIONV}
        \mathbf{L} \mathbf{f}_l=\left(
        \begin{array}{lll}
            \mathbf{L}^{A} \mathbf{f}_l
            \\
            \mathbf{L}^{B} \mathbf{f}_l
        \end{array}
        \right)=\left(
        \begin{array}{lll}
            -\frac{1}{\sqrt{\mu^A}}  \Big(Q^{A A}(\mu^A,\sqrt{\mu^A}f_l^A)+ Q^{A A}(\sqrt{\mu^A}f_l^A, \mu^A)\\ \hspace{1cm}+Q^{A B}(\sqrt{\mu^A} f_l^A, \mu^B)+ Q^{A B}(\mu^A,  \sqrt{\mu^B} f_l^B) \Big)
            \\
            -\frac{1}{\sqrt{\mu^B}}  \Big(Q^{B A}(\mu^B,  \sqrt{\mu^A} f_l^A)+Q^{B A}(\sqrt{\mu^B} f_l^B, \mu^A) \\
            \hspace{1cm} + Q^{B B}(\mu^B,\sqrt{\mu^B}f_l^B)+ Q^{B B}(\sqrt{\mu^B}f_l^B, \mu^B) \Big)
        \end{array}
        \right).
    \end{gather}
    Here $\mathbf{L}$ is a self-adjoint operator \cite{[60]Briant2016ARM} in $(L^2_v(\mathbb{R})^3)^2$. It could be decomposed as $\mathbf{L}=\mathbf{\nu} + K$,
    \begin{gather*}
        \mathbf{\nu} \mathbf{f}_l = \left(\begin{array}{ccc}\nu^A(v) f_l^A  \\ \nu^B(v) f_l^B  \end{array}\right), \quad K\mathbf{f}_l =\left(\begin{array}{ccc}K^{A} \mathbf{f}_l \\ K^{B} \mathbf{f}_l \end{array}\right),
    \end{gather*}
    where the collision frequency $\nu$ and the $K$ operator are defined by
    \begin{equation*}
        \begin{split}
            &\nu^{\alpha}=  \sum_{\beta=A,B} \int_{\mathbb{R}^3\times \mathbb{S}^2}
            B^{ \alpha \beta}(|v-v_*|, \cos \theta) \mu^{\beta}_* d\sigma dv_*,
            \\
            &                \mathcal{K}^{\alpha} \mathbf{f}_l = - \frac{1}{\sqrt{\mu^\alpha}} \sum_{\beta=A,B}\int_{\mathbb{R}^3 \times \mathbb{S}^2}B^{\alpha\beta}(|v-v_*|, \cos \theta)  \Big[\mu^\alpha(v')\sqrt{\mu^\beta(v_*')}f_l^\beta(v_*') \\
            &\hspace{2cm}  + \mu^{\beta}(v_*') \sqrt{\mu^\alpha(v')}\, f_l^\alpha(v')- \mu^\alpha(v)\sqrt{\mu^\beta(v_*)}f_l^\beta(v_*) \Big] d\sigma dv_*. \\
        \end{split}
    \end{equation*}
    (The kernel of $\mathbf{L}$ \cite{[60]Briant2016ARM}).\,  $\langle \mathbf{f}_l, \mathbf{L}\mathbf{f}_l\rangle_{(L^2_v(\mathbb{R})^3)^2}=0$ if and only if $\mathbf{f}_l$ belongs to the kernel of $ \mathbf{L}$, which is denoted as $ {\rm Ker}(\mathbf{L})$. It is spanned by the following six vector functions
    \begin{equation*}
        {\rm Ker}(\mathbf{L})={\rm span}\{ X_0,X_1,\cdots, X_5\},
    \end{equation*}
    with
    \begin{equation}\label{YZJVEPDB}
        \begin{split}
            &\mathbf{X}_0= \sqrt{\frac{\mu^A}{n^A }}  \mathbf{e}_1, \, \mathbf{X}_1= \sqrt{\frac{\mu^B}{n^B }} \mathbf{e}_2,\, \mathbf{X}_j=\frac{v_{j-1}-u^{j-1}}{\sqrt{\theta n}}
            \left(
            \begin{array}{ccc}
                m^A\sqrt{\mu^A}\\[2mm]
                m^B \sqrt{\mu^B}
            \end{array}
            \right) ~(j=2,3,4),\\
            & \mathbf{X}_5=\frac{1}{\sqrt{6n}}
            \left(
            \begin{array}{ccc}
                (\frac{m^A |v-\mathbf{u}|^2}{\theta}-3)\sqrt{\mu^A}\\[2mm]
                (\frac{m^B |v-\mathbf{u}|^2}{\theta}-3)\sqrt{\mu^B}
            \end{array}
            \right).
        \end{split}
    \end{equation}
    One can check that
    $
    \langle \mathbf{X}_i\cdot \mathbf{X}_j\rangle
    =\delta_{ij},\,\, 0\leq i,j \leq 5$, where $\delta_{ij}$ is the Kronecker symbol.
    The space can be decomposed by $(L^2_v(\mathbb{R}^3))^2 = {\rm Ker}\mathbf{L} \oplus \mathcal{N}(\mathbf{L})^{\bot}$, where $ \mathcal{N}(\mathbf{L})^{\bot}$ is the orthogonal set of ${\rm Ker}\mathbf{L}$.
    For any $ \mathbf{g}=(g^A,g^B)^T \in (L^2_v(\mathbb{R}^3))^2 $, it could be decomposed  as
    \begin{gather}\label{MACMICfenjie}
        \mathbf{g} = \mathcal{P} \mathbf{g} + (\mathbf{I}- \mathcal{P}) \mathbf{g}, \quad {\rm with}\quad \mathcal{P} \mathbf{g}= \sum_{j=0}^5 \langle \mathbf{g}, \mathbf{X}_j\rangle \mathbf{X}_j,
    \end{gather}
    where $\mathcal{P} \mathbf{g}$ and $(\mathbf{I}-\mathcal{P})\mathbf{g} = \mathbf{g}-\mathcal{P}\mathbf{g} $ are the macro-micro parts of $\mathbf{g}$ respectively.

    \noindent (The solvability of $\mathbf{L}$). For any ${\mathbf{\bar{R}}} =(\bar{R}^A,\bar{R}^B)^T \in (L^2_v(\mathbb{R}^3)^2)^2 $, there is a solution of
    $\mathbf{L}\mathbf{g}=\mathbf{\bar{R}}$
    if and only if $\mathbf{\bar{R}} \in \mathcal{N}(\mathbf{L})^{\bot}$, that is,
    \begin{gather}\label{SOBCD222}
        \langle\mathbf{\bar{R}}, \mathbf{X}_j\rangle_{(L^2_v(\mathbb{R})^3)^2} =0, \quad {\rm for~ all}\quad j=0,\cdots,5.
    \end{gather}
    The solution $\mathbf{g} \in \mathcal{N}(\mathbf{L})^{\bot}$ is formulated as
    \begin{eqnarray*}\label{1.23}
        \mathbf{g}= \mathbf{L}^{-1} \mathbf{\bar{R}}.
    \end{eqnarray*}

    Now, we can determine the terms  $\mathbf{F}_l (l=1,2,\cdots,2k-1)$ step by step.  By macro-micro decomposition in \eqref{MACMICfenjie}, we write $\mathbf{f}_l=(f^A_l,f^B_l)^{T}$ as follows
    \begin{eqnarray*}\label{solveklarge1deco}
        \mathbf{f}_l = \mathcal{P}\mathbf{f}_l + (\mathbf{I}-\mathcal{P})\mathbf{f}_l.
    \end{eqnarray*}
    The macroscopic part of $\mathbf{f}_l $ is  expressed as
    \begin{eqnarray*}\label{macropart1.23}
        \begin{split}
            \mathcal{P} \mathbf{f}_l
            = &\frac{n^{A}_{l}}{\sqrt{n^{A}}}\mathbf{X}_0+\frac{n^{B}_{l}}{\sqrt{n^{B}}}\mathbf{X}_1
            +\sum_{j=2,3,4}\sqrt{\frac{n}{\theta}} u_l^{j-1} \mathbf{X}_j
            +\sqrt{\frac{n}{6}}\frac{\theta_l}{\theta}\mathbf{X}_5 \\
            =& \left(\begin{array}{lll} \Big(\frac{n^A_{l}}{n^A} + \mathbf{u}_{l} \cdot \frac{m^A(v-\mathbf{u})}{\theta} +
                \frac{\theta_l}{6\theta}(\frac{m^A |v-\mathbf{u}|^2}{\theta} -3 )\Big) \sqrt{\mu^A}\vspace{3pt}\\
                \Big( \frac{n^B_{l}}{n^B} + \mathbf{u}_{l} \cdot \frac{m^B(v-\mathbf{u})}{\theta} +
                \frac{\theta_l}{6\theta}(\frac{m^B|v-\mathbf{u}|^2}{\theta} -3 )\Big) \sqrt{\mu^B}\end{array}\right),
        \end{split}
    \end{eqnarray*}
    where $(n_l^A, n_l^B, \mathbf{u}_l, \theta_l)$ are density of the mass, velocity and temperature, which satisfy

    \begin{eqnarray*}\label{fluid def}
        \begin{split}
            &n_l= n_l^A+n_l^B, \, n_l^{\alpha} =\int_{\mathbb{R}^3} F_l^{\alpha} dv = \int_{\mathbb{R}^3} f_l^{\alpha} \sqrt{\mu^\alpha} dv, \quad l\geq 1,\\
            &  \rho_l= m^A n_l^A+m^B n_l^B, \,\,n_l^{\alpha} u^j +n^{\alpha} u^j_l= \int_{\mathbb{R}^3} v_j F_l^\alpha  dv= \int_{\mathbb{R}^3} v_j f_l^\alpha \sqrt{\mu^\alpha} dv,\\
            & n^\alpha \theta_{l} + 3 \theta n^{\alpha}_l = \int_{\mathbb{R}^3}m^{\alpha} |v-\mathbf{u}|^2 F_l^\alpha dv= \int_{\mathbb{R}^3}m^{\alpha} |v-\mathbf{u}|^2 f_l^\alpha \sqrt{\mu^\alpha} dv.
        \end{split}
    \end{eqnarray*}
  Moreover, applying \eqref{LNOTATIONV} to $\eqref{ORDERZONG}_3$, one has
    \begin{eqnarray*}\label{RNDDYTH}
        \begin{aligned}
            \mathbf{L}^{\alpha} \mathbf{f}_{l+1}=&-(\partial_t +v \cdot \nabla_{x}+\nabla_x\phi_0\cdot\nabla_v)F_l^{\alpha}-\nabla_x\phi_l\cdot\nabla_vF_0^\alpha-\mathop{\sum}_{i+j=l\atop i,j\geq1}\nabla_x\phi_i\cdot\nabla_vF_j^\alpha\\
            &\hspace{1.5cm} +\displaystyle{\sum_{i+j=l+1\atop i,j\geq1} \sum_{\beta=A,B}} Q^{ \alpha\beta}(F_i^{\alpha},F_j^{\beta}):=\mathbf{\bar{R}}^{\alpha}_l,
        \end{aligned}
    \end{eqnarray*}
    where $\mathbf{\bar{R}}_l=(\mathbf{\bar{R}}^{A}_l,\mathbf{\bar{R}}^{B}_l)^{T}$, the solvability condition \eqref{SOBCD222} means that
    \begin{eqnarray}\label{solvabiliCDCD}
        \left\langle \mathbf{\bar{R}}_l, \mathbf{X}_{j} \right\rangle_v =0, \quad {\rm for~ all}\quad j=0,\cdots,5.
    \end{eqnarray}
    Then, these six equations in \eqref{solvabiliCDCD} lead to
    \begin{equation}\label{linarFLsys}
        \left\{
        \begin{array}{lllll}
            \partial_t n^{A}_{{l}} +{\rm div}_x(n^{A}_{l} \mathbf{u}+n^{A}\mathbf{u}_{l})=0,\vspace{3pt} \\
            \partial_t n^{B}_{l} +{\rm div}_x(n^{B}_{l}\mathbf{u}+n^{B}\mathbf{u}_{l})=0,
            \\
            \rho\left(\partial_t \mathbf{u}_{l} +\mathbf{u}_{l} \cdot \nabla_{x}\mathbf{u}+\mathbf{u} \cdot \nabla_{x}\mathbf{u}_{l}\right)
            -\Big[\frac{n^{A}_{l}}{n^{A}}\nabla_{x}(n^{A}\theta) + \frac{n^{B}_{l}}{n^{B}}\nabla_{x}(n^{B}\theta)\Big] \\
            \hspace{1.5cm}-n_e\nabla_{x}\phi_l+\nabla \Big(\frac{n^A \theta_l+3\theta n_l^A}{3}\Big)+\nabla \Big(\frac{n^B \theta_l+3\theta n_l^B}{3}\Big)=\mathbf{H}_{l-1},\vspace{3pt} \\
            n \Big(\partial_t \theta_{l} + \frac{2}{3}(\theta_{l} {\rm div} \mathbf{u} + 3\theta {\rm div} \mathbf{u}_{l})  +\mathbf{u}\cdot \nabla_{x} \theta_{l} + 3\mathbf{u}_{l} \cdot \nabla_{x} \theta \Big) = {\rm g}_{l-1},
        \end{array}
        \right.
    \end{equation}
    where
    \begin{eqnarray}\label{espq1.30}
        \begin{aligned}
            (\mathbf{H}_{l-1})_i= &- \sum_{\alpha=A,B}m^{\alpha}\sum_{j=1}^{3} \partial_{x_j}\Big(\frac{\theta}{m^\alpha} \int_{\mathbb{R}^3} \mathbf{A}^{\alpha}_{i,  j} f_l^{\alpha} dv\Big)+\mathop{\sum}_{i+j=l\atop i,j\geq1}\rho_j\nabla_x\phi_i,   \\
            {\rm g}_{l-1}=& - 2 \sum_{\alpha=A,B} m^{\alpha} \sum_{i=1}^{3} \partial_{x_i} \Big((\frac{\theta}{m^\alpha})^{\frac{3}{2}}\int_{\mathbb{R}^3} \mathbf{B}^{\alpha}_{i }f_k^{\alpha} dv + \sum_{j=1}^{3} u_j \frac{\theta}{m^\alpha} \int_{\mathbb{R}^3} \mathbf{A}^{\alpha}_{i,  j}f_l^{\alpha} dv\Big)\\
            &  -2\mathbf{u}_l \cdot \mathbf{H}_{l-1}+\mathop{\sum}_{i+j=l\atop i,j\geq1}(\rho_j\mathbf{u}+\rho \mathbf{u}_j)\nabla_x\phi_i.
        \end{aligned}
    \end{eqnarray}
    Here the the Burnett functions  $\mathbf{A}^{\alpha}_{i,  j},  \mathbf{B}^{\alpha}_{i } (\alpha=A, B) $ are defined as follows
    \begin{equation*}
        \begin{split}
            \mathbf{A}^{\alpha}_{i,  j}:= &\Big(\frac{m^{\alpha}(v_i-\mathbf{u}_i)(v_j-\mathbf{u}_j)}{\theta} -\delta_{ij} \frac{m^{\alpha}|v-\mathbf{u}|^2}{3\theta}\Big)\sqrt{\mu^\alpha}, \\
            \mathbf{B}^{\alpha}_{i } :=& \frac{v_i-\mathbf{u}_i}{2}\sqrt{\frac{m^\alpha}{\theta}}(\frac{m^{\alpha}|v-\mathbf{u}|^2}{\theta}-5)\sqrt{\mu^\alpha}.
        \end{split}
    \end{equation*}
    It should be pointed out that $\mathbf{A}^{\alpha}_{i,  j} \in \mathcal{N}^{\bot}$ and $\mathbf{B}^{\alpha}_{i }\in \mathcal{N}^{\bot}$, so the source term $\mathbf{H}_{l-1}$ and $\rm{g}_{l-1}$ in \eqref{espq1.30} depend only on the microscopic part of $(f_l^A, f_l^B)^T$.
    Finally, the micro-part of $\mathbf{f}_l$ could be expressed as
    \begin{eqnarray*}\label{micropart1.301}
        (\mathbf{I}-\mathcal{P})\mathbf{f}_{l+1} = \mathbf{L}^{-1} \mathbf{\bar{R}}_{l}.
    \end{eqnarray*}
    \begin{lemma} [\cite{[61]Jiang}, \cite{[50]Wu2023JDE}] \label{L-1estJ}
        Let $-3<\gamma\leq 1$ and $\mathbf{g} \in \mathcal{N}(\mathbf{L})^{\perp}$. \\
        (\uppercase\expandafter{\romannumeral1})\, $\mathbf{The \, \, cases}$ $-\frac{3}{2}<\gamma\leq 1:$
        Let $0<q<1$ be arbitrarily fixed. Define the positive number $k_{\gamma}$ by
        \begin{equation*}\label{macroseries2}
            k_{\gamma} = \left \{
            \begin{split}
                &   \frac{3-\gamma}{2}, \quad  if \quad 0 \leq \gamma \leq 1,\\
                &   \frac{3-2\gamma}{2}, \quad   if \quad -\frac{3}{2} \leq \gamma \leq 0.
            \end{split}
            \right.
        \end{equation*}
        Assume that $(1+|v|)^{  k_{\gamma}} \mu^{-\frac{q}{2}} \in L^{\infty}$. Then the following statements hold: \\
        ({\romannumeral1}) \,  For the hard potential cases $0 \leq \gamma \leq 1$,
        \begin{eqnarray}\label{247}
            |\mathbf{L}^{-1}\mathbf{g}(v)| \leq C \Vert  (1+|v|)^{k_{\gamma}} \mu^{-\frac{q}{2}} \mathbf{g} \Vert_{\mathbf{L}_{v}^{\infty}} \mu^{\frac{q}{2}}(v), \quad v \in \mathbb{{R}}^3.
        \end{eqnarray}
        ({\romannumeral2}) \, For the part of soft potential cases  $-\frac{3}{2} < \gamma < 0$,
        \begin{eqnarray}\label{248}
            |\mathbf{L}^{-1}\mathbf{g}(v)| \leq C \Vert  (1+|v|)^{k_{\gamma}} \mu^{-\frac{q}{2}} \mathbf{g} \Vert_{\mathbf{L}_{v}^{\infty}} \nu^{-1}(v) \mu^{\frac{q}{2}}(v), \quad v \in \mathbb{{R}}^3.
        \end{eqnarray}
        (\uppercase\expandafter{\romannumeral2})\, $\mathbf{The \, \, cases}$ $-3<\gamma\leq -\frac{3}{2}:$
        Assume that
        \begin{eqnarray*}
            \sum_{|\beta|\leq2}\Vert (1+|v|)^{\gamma(|\beta|-3)} \mu^{-q_0} \partial^{\beta}_{v}\mathbf{g}\Vert_{L_{v}^2}<\infty,
        \end{eqnarray*}
        for a sufficiently small $0<q_0<1$. Then
        \begin{eqnarray}\label{250}
            |\mathbf{L}^{-1}\mathbf{g}(v)| \leq C \sum_{|\beta|\leq2}\Vert (1+|v|)^{\gamma(|\beta|-3)} \mu^{-q_0} \partial^{\beta}_{v}\mathbf{g}\Vert_{L_{v}^2} \mu^{q_0}(v), \quad  v \in \mathbb{{R}}^3.
        \end{eqnarray}
        The positive constant $C>0$ above depends on the solution of the compressible Euler systems.
    \end{lemma}

    \begin{remark}\label{278}
        If we take $\mathbf{\bar{R}}_{l}$ in \eqref{RNDDYTH} to replace $\mathbf{g}$ in \eqref{250}, then there exists a series of number $10<h_1 <  h_2 < \cdots <h_{2k-1}$, such that for $|\beta|<2$,
        \begin{eqnarray*}
            |   \partial^{\beta}_{v} \mathbf{\bar{R}}_{l} | \leq C \{1+|v|\}^{h_l} |\mathbf{\bar{R}}_{l}| \qquad  for \quad v \in \mathbb{{R}}^3.
        \end{eqnarray*}
        Since $\max\{\gamma(|\beta|-3),k_{\gamma}\}\leq 10  $, and $\nu^{-1} \leq C \{1+|v|\}^3$ from \eqref{247},\eqref{248} and \eqref{250}, there always holds for all the case  $-3<\gamma\leq1$
        \begin{eqnarray}\label{togetheresL-1}
            |\mathbf{L}^{-1}\mathbf{\bar{R}}_{l}(v)| \leq C \Vert  (1+|v|)^{    h_l} \mu^{-q_l} \mathbf{\bar{R}}_{l} \Vert_{L^{\infty}_v} \{1+|v|\}^3 \mu^{q_l}(v) \quad v \in \mathbb{{R}}^3,
        \end{eqnarray}
        where $0<q_0<1$ is small enough and $0<q_{2k-1}<q_4<...<q_1<q_0$.
    \end{remark}
    \begin{lemma}\label{FI2KM1EST}
        Let $(n^A, n^B, \mathbf{u}, \theta, \nabla\phi_0 )$ be the smooth solution of  established in Lemma \ref{CPCO} . $(\rho^{in}_l, \mathbf{u}^{in}_l, \theta^{in}_{l})$  is the initial data of the linear system \eqref{linarFLsys}, then it has a unique solution such that
        \begin{eqnarray}\label{decayLxvFf}
            \begin{split}
                |\mathbf{f}_{l}|\leq C (1+t)^{l-1}(1+|v|)^{h_l}\sqrt{\mu}, \quad |\nabla_{x}\phi_1|\leq C (1+t)^{l-1}, \qquad \qquad\\
                |\nabla_{v}\mathbf{f}_{l}|\leq C  (1+t)^{l-1}(1+|v|)^{h_l+1}\sqrt{\mu},\quad |\nabla_{x}\mathbf{f}_{l}|\leq C  (1+t)^{l-1}(1+|v|)^{h_l+2}\sqrt{\mu}, \\
                |\nabla^2_{v}\mathbf{f}_{l}|\leq C  (1+t)^{l-1}(1+|v|)^{h_l+2}\sqrt{\mu},\quad |\nabla_{x}\nabla_{v}\mathbf{f}_{l}|\leq C  (1+t)^{l-1}(1+|v|)^{h_l+3}\sqrt{\mu}.
            \end{split}
        \end{eqnarray}
    \end{lemma}
    \noindent \proofname. The proof of the decay estimate could be divided into two parts: the estimate of the macro part $\mathcal{P}\mathbf{f}_{l}$ and the estimate of micro part $(\mathbf{I}-\mathcal{P})\mathbf{f}_{l}$. By Lemma \ref{L-1estJ} , if \eqref{decayLxvFf} holds for the case $l-1$, one gets
    \begin{equation*}
        |(\mathbf{I}-\mathcal{P})\mathbf{f}_{l+1}|\leq C (1+t)^{l}(1+|v|)^{h_{l+1}}\sqrt{\mu}.
    \end{equation*}
To estimate $\mathcal{P}\mathbf{f}_{l}$, noting that its coefficients $(n_l^A, n_l^B, \mathbf{u}_l, \theta_l,\nabla\phi_l )$ are solutions of system \eqref{linarFLsys}, and recalling the notations:
\begin{gather*}
    \rho=m^An^A+m^Bn^B , \quad \tilde{n}=n^A+n^B,\quad  n_e=e^An^A+e^Bn^B,
\end{gather*}
we can define
\begin{gather*}
    p^{\alpha}_k= \frac{n^{\alpha} \theta_k +3 \theta n^{\alpha}_k}{3}   , \quad p^{\alpha}=n^{\alpha}\theta.
\end{gather*}
    Then the system \eqref{linarFLsys} is equivalent to
    \begin{equation}\label{BBKLLL}
        \left \{
        \begin{split}
            &   \partial_t p^{A}_l +\mathbf{u} \cdot \nabla_x p^{A}_l + \frac{5}{3} \Big(p^{A} {\rm div}_x \mathbf{u}_l + p^{A}_l {\rm div}_x \mathbf{u} \Big)+  \nabla_x p^{A} \cdot \mathbf{u}_l =\frac{n^{A}}{3\tilde{n}} g_{l-1},\\
            &   \partial_t p^{B}_l +\mathbf{u} \cdot \nabla_x p^{B}_l + \frac{5}{3} \Big( p^{B}{\rm div}_x \mathbf{u}_l + p^{B}_k{\rm div}_x \mathbf{u}\Big) +  \nabla_x p^{B} \cdot \mathbf{u}_k=\frac{n^{B}}{3\tilde{n}} g_{l-1},\\
            &   \partial_t \mathbf{u}_l +\mathbf{u} \cdot \nabla_{x}\mathbf{u}_l +\mathbf{u}_l \cdot \nabla_{x}\mathbf{u} -
            \frac{\nabla_{x}p^{A}}{\rho p^{A}}p^{A}_l  - \frac{\nabla_{x}p^{B}}{\rho p^{B}}p^{B}_l \\
            &\hspace{3cm}+ \frac{1}{\rho}\nabla_{x}(p^{A}_l +p^{B}_l) -\frac{n_{e}}{\rho}\nabla_{x}\phi_l  - \frac{\nabla_{x}\tilde{p} }{3\rho \theta}\theta_k=\frac{\mathbf{H}_{l-1}}{\rho}, \\
            &   \partial_t \theta_l + \frac{2}{3}(\theta_l {\rm div}_x \mathbf{u} + 3\theta {\rm div}_x\mathbf{u}_l)
            +\mathbf{u}\cdot \nabla_{x}\theta_k + 3\mathbf{u}_l \cdot \nabla_{x} \theta = \frac{g_{l-1}}{\tilde{n}}.
        \end{split}
        \right.
    \end{equation}
    Further, setting  the two vectors $\mathbf{U}_l=(p_l^A, p_l^B, \mathbf{u}_l, \theta_l)^T$, $\mathbf{V}_l=(0,0,\nabla_x\phi_l,0)^T$, we symmetrize the linear system \eqref{BBKLLL} with the matrix $B_0$:
    \begin{eqnarray}\label{DCHDJZFC}
        \mathbf{B}_0 (\partial_t \mathbf{U}_l-\frac{n_e}{\rho} \mathbf{V}_l)+ \sum_{i=1}^{3}  \mathbf{B}_i \partial_i \mathbf{U}_l + \mathbf{F}_{l-1} \mathbf{U}_l=  \mathbf{G}_{l-1},
    \end{eqnarray}
    where these matrixes are
    \begin{equation*}
        \begin{gathered}
            \mathbf{B}_0 =  \begin{pmatrix}
                \frac{9}{5 n^{A}}  & 0 & 0 & 0 & 0 & -1 \\
                0 & \frac{9}{5 n^{B}} & 0 & 0 & 0 & -1\\
                0 & 0 & p & 0 & 0 & 0\\
                0  & 0 & 0 & p & 0 & 0\\
                0 & 0 & 0 & 0 & p & 0\\
                -1 & -1 & 0 & 0 & 0 & \frac{5}{6}\tilde{n}
            \end{pmatrix} ,  \quad
            \mathbf{B}_1 =  \begin{pmatrix}
                \frac{9u^1}{5 n^{A}}  & 0 & \theta & 0 & 0 & -u^1\\
                0 & \frac{9u^1}{5 n^{B}} & \theta & 0 & 0 & -u^1\\
                \theta & \theta & pu^1 & 0 & 0 & 0\\
                0  & 0 & 0 & pu^1 & 0 & 0\\
                0 & 0 & 0 & 0 & pu^1& 0\\
                -u^1 & -u^1 & 0 & 0 & 0 & \frac{5}{6} \tilde{n} u^1
            \end{pmatrix},
        \end{gathered}
    \end{equation*}
    \begin{equation*}
        \begin{gathered}
            \mathbf{B}_2 =  \begin{pmatrix}
                \frac{9u^2}{5 n^{A}}  & 0 & 0 & \theta & 0 & -u^2\\
                0 & \frac{9u^2}{5 n^{B}} & 0 & \theta & 0 & -u^2\\
                0 & 0 & pu^2 & 0 & 0 & 0\\
                \theta  & \theta & 0 & pu^2 & 0 & 0\\
                0 & 0 & 0 & 0 & pu^2& 0\\
                -u^2 & -u^2 & 0 & 0 & 0 & \frac{5}{6} \tilde{n} u^2
            \end{pmatrix} ,  \quad
            \mathbf{B}_3 =  \begin{pmatrix}
                \frac{9u^3}{5 n^{A}}  & 0 & 0 & 0 & \theta & -u^3\\
                0 & \frac{9u^3}{5 n^{B}} & 0 & 0 & \theta & -u^3\\
                0 & 0 & pu^3 & 0 & 0 & 0\\
                0  & 0 & 0 & pu^3 & 0 & 0\\
                \theta & \theta & 0 & 0 & pu^3& 0\\
                -u^3 & -u^3 & 0 & 0 & 0 & \frac{5}{6} \tilde{n} u^3
            \end{pmatrix}.
        \end{gathered}
    \end{equation*}

    Since $B_0, B_1, B_2$ and $B_3$
    are symmetrized matrices, and $B_0$
    is positive definite, its well-posedness follows from the linear theory in \cite{[62]A.Majda}. First, we consider the case $l = 1$. The matrix $\mathbf{F}_{0}$
    and the source term $\mathbf{G}_{0}$
    consist of $(n^{A}, n^{B}, \mathbf{u}, \theta)$
    and their first derivatives. In particular, for any $1 < q < \frac{3}{2}$
    and $s > 0$, there exist constants $C_s$
    such that:
    \begin{equation*}
        \Vert \mathbf{F}_{0} \Vert_{W^{s_0,\infty}}+\Vert \mathbf{G}_{0} \Vert_{W^{s_0,\infty}}\leq\frac{C_{s_0}}{(1+t)^{q}}.
    \end{equation*}
    Due to $\Delta\phi_1=e^An_1^A+e^Bn_1^B$,
    we can apply the standard energy method of the linear symmetric hyperbolic system to system \eqref{DCHDJZFC}. Thus, the following energy inequality yields:
    \begin{eqnarray}
        \frac{d}{dt}(\Vert \mathbf{U}_{1} \Vert^2_{H^{s_1}}+\Vert \mathbf{V}_{1} \Vert^2_{H^{s_1}})\leq \frac{C_{s_1}}{(1+t)^{q}} (\Vert \mathbf{U}_{1} \Vert^2_{H^{s_1}}+\Vert \mathbf{V}_{1} \Vert^2_{H^{s_1}}+\Vert \mathbf{U}_{1} \Vert_{H^{s_1}} +\Vert \mathbf{V}_{1} \Vert_{H^{s_1}}),
    \end{eqnarray}
    where we have used the fact that $\frac{1}{C}<\theta<C$ to deduce $\Vert \mathbf{U}_{1} \Vert_{H^{s_1}} +\Vert \mathbf{V}_{1} \Vert_{H^{s_1}}\leq C$. Therefore, the inequality \eqref{decayLxvFf} has been proved for $l=1$. Next, we suppose that \eqref{decayLxvFf} holds for $1\leq l\leq n$, and consider the case $l=n+1$. By Lemma \ref{L-1estJ} and the induction hypothesis, the microscopic part of $\mathbf{f}_{n+1}$ is bounded by
    \begin{equation*}
        |(\mathbf{I}-\mathcal{P})\mathbf{f}_{n+1}|\leq C (1+t)^{n}(1+|v|)^{h_{n+1}}\sqrt{\mu}.
    \end{equation*}

While for the macroscopic part $\mathcal{P} \mathbf{f}_{n+1}$, we
 need to obtain $p_{n+1}^A, p_{n+1}^B, \mathbf{u}_{n+1}$ and
$\theta_{n+1}$ in equation \eqref{BBKLLL}. The structure of the
left-hand side of equation \eqref{BBKLLL} is same as that of
$p_{1}^A, p_{1}^B, \mathbf{u}_{1}, \theta_{1}$, except for the extra
term caused by $\mathop{\sum}_{i+j=l\atop i,j\geq1} \nabla_x \phi_i
\cdot \nabla_v F_j^\alpha$. This term is part of the source term
consisting of $\mathbf{f}_{n+1}$ for $1 \leq l \leq n$. By the
induction hypothesis, one obtains the following energy inequality
for $\mathbf{U}_{n+1}$ and $\mathbf{V}_{n+1}$:
    \begin{eqnarray*}
        \begin{split}
            \frac{d}{dt}(\Vert \mathbf{U}_{n+1} \Vert^2_{H^{s_{n+1}}}+\Vert \mathbf{V}_{n+1} \Vert^2_{H^{s_{n+1}}})\leq & \frac{C_{s_{n+1}}}{(1+t)^{q}} (\Vert \mathbf{U}_{n+1} \Vert^2_{H^{s_{n+1}}}+\Vert \mathbf{V}_{n+1} \Vert^2_{H^{s_{n+1}}})\\
            &+C_{s_{n+1}}(1+t)^{n-1}(\Vert \mathbf{U}_{n+1} \Vert_{H^{s_{n+1}}} +\Vert \mathbf{V}_{n+1} \Vert_{H^{s_{n+1}}}).
        \end{split}
    \end{eqnarray*}
    By the Gronwall inequality, one gets
    \begin{equation*}
        \Vert \mathbf{U}_{n+1} \Vert_{H^{s_{n+1}}} +\Vert \mathbf{V}_{n+1} \Vert_{H^{s_{n+1}}} \leq C_{s_{n+1}}(1+t)^{n}.
    \end{equation*}
   Therefore, the case for $l=n+1$ has been proved.

    \section{ The $L^2$ estimate of the remainder term}
In this section, we derive $L^2$-energy estimates for the remainder term in both hard and soft potential cases. Before presenting the proof, the following two lemmas are necessary.
    \begin{lemma}[The coercivity of  $\mathbf{L}$]\label{COVWEIGHTED}
       For the linearized bi-collision operator $\mathbf{L}$ defined in \eqref{LNOTATIONV},
        there exist $c_0>0$, such that
        \begin{equation}\label{WWCoercivity}
            \sum_{\alpha=A,B}   \left \langle \mathbf{L}^{\alpha}\mathbf{f}_i, f^{\alpha}_i \right \rangle_v \geq c_0 |
            (1+|v|)^{\frac{\gamma}{2}}(\mathbf{I}-\mathcal{P})\mathbf{f}_i |_{\nu}^2.
        \end{equation}
    \end{lemma}
    \begin{proof}
    	‌The proof of this theorem needs to be divided into two cases: hard potential and soft potential. The proof for the hard potential case is in Paper \cite{[60]Briant2016ARM}, while the proof for the soft potential case can be found in my latest Paper \cite{Wu20252pp}.
    \end{proof}
    \noindent The (non-symmetric) bilinear form  $\Gamma^{\alpha \beta}$
    is given by:
    \begin{eqnarray*}\label{gain loss part}
        \begin{split}
            \Gamma^{\alpha \beta }(g_1,g_2)\triangleq & \frac{1}{\sqrt{\mu^{\alpha }}}Q^{ \alpha \beta}(\sqrt{\mu^{\alpha }}g_1,  \sqrt{\mu^{\beta }}g_2)\\
            = &  \frac{1}{\sqrt{\mu^{\alpha }}}\int_{\mathbb{R}^3\times \mathbb{S}^2}|v-v_*|^{\gamma}b^{\alpha \beta}(\theta)\\
            &  \times (\sqrt{\mu^{\alpha }(v') \mu^{\beta }(v_*')}g_1(v')g_2(v_*')
            -\sqrt{\mu^{\alpha }(v)\mu^{\beta }(v_*)}g_1(v)g_2(v_*))d\sigma dv_*.
        \end{split}
    \end{eqnarray*}
    \begin{lemma}\label{L2estimatesofnonlinears}
        Let $-3 < \gamma \leq 1$. For the weight function $w_{\gamma}(v)$ defined in \eqref{WWWEIGHTFUC}, there exists a constant $C$ such that the operator $\Gamma^{\alpha \beta}$
        can be estimated as follows:
        \begin{gather}\label{Nonlineartermfullp2}
            \begin{split}
                \left|\left \langle \Gamma^{\alpha \beta}(g_1,g_2),  g_3    \right \rangle_{x,v} \right|\leq& C  \| w g_i\|_{L^{\infty}_{x,v} }\, \|g_{3-i}\|_{\nu} \,\Vert g_3\Vert_{\nu}, \quad {\rm for}~~i=1,2.
            \end{split}
        \end{gather}
        Moreover, there exists a constant $C$ such that
        \begin{gather*}\label{Nonlineartermhardp}
            \begin{split}
                \left|\left \langle \Gamma^{\beta\alpha }(g_1,g_2),  g_3    \right \rangle_{x,v} \right|\leq& C  (\|w g_1\|_{L^{\infty}_{x,v} }\, \|g_2\|_{L^{2}_{x,v} } + \| g_1\|_{L^{2}_{x,v} }\, \|w g_2\|_{L^{\infty}_{x,v} } )\,\Vert g_3\Vert_{L^{2}_{x,v} }.
            \end{split}
        \end{gather*}
    \end{lemma}
    \begin{proof}
        Since for any $l\geq 7$, and $-1\leq\gamma<1$, there exists some constant $C$ such that $C\omega_{\gamma}>\left \langle v \right\rangle^{l}$, the proof is same as Lemma 6.1 in \cite{[50]Wu2023JDE}, and can be omitted.
    \end{proof}
     From \eqref{reeqmain}, $\mathbf{f}_R=(f_R^A, f_R^B)^T$ satisfies the following equations:
    \begin{equation}\label{L2dequationoff}
        \begin{split}
            &(\partial_t+v\cdot\nabla_{x}+\frac{e^\alpha}{m^\alpha}\nabla_{x}\phi_0 \cdot\nabla_{v})f^\alpha_R+ \frac{1}{\varepsilon} \mathbf{L}^\alpha \mathbf{f}_R
            =  \sum_{i=1}^{7}E^\alpha_i.
        \end{split}
    \end{equation}
    Here $\mathbf{E}_i= (E_i^A,E_i^B)^T (i=1,2,3,...,7.)$,
    \begin{align*}
            E_1^{\alpha} = & -\frac{e^\alpha}{m^\alpha}\nabla_{x}\phi_R \cdot\frac{\nabla_{v}\mu^{\alpha}}{\sqrt{\mu^{\alpha}}},\quad
            E_2^{\alpha} =  -\frac{(\partial_t+v\cdot\nabla_{x}+\frac{e^\alpha}{m^\alpha}\nabla_{x}\phi_0 \cdot\nabla_{v} )\sqrt{\mu^\alpha}}{\sqrt{\mu^\alpha}}f_R^\alpha,\\
            E_3^{\alpha}=& \sum_{i=1}^{2k-1}\sum_{\beta=A,B}\varepsilon^{i-1}\Big[\Gamma^{ \alpha \beta}(f^\alpha_i,f^\beta_R)+\Gamma^{\alpha\beta }(f^\alpha_R, f^\beta_i)\Big], \quad E_4^{\alpha}= \varepsilon^{k-1} \displaystyle{\sum_{\beta=A,B} \Gamma^{ \alpha \beta}( f_R^{\alpha},  f_R^\beta ) },\\
            E_5^{\alpha}= & -\varepsilon^k \frac{e^\alpha}{m^\alpha} \Big(\nabla_{x}\phi_R\cdot\frac{\nabla_{v}\sqrt{\mu^{\alpha}}}{\sqrt{\mu^{\alpha}}}f_R^{\alpha}+\nabla_{x}\phi_R\cdot\nabla_{v}f_R^{\alpha}\Big), \\
            E_6^{\alpha}= & -\sum_{i=1}^{2k-1} \varepsilon^i \frac{e^\alpha}{m^\alpha} \Big[\frac{\nabla_{v}F_i^{\alpha}}{\sqrt{\mu^\alpha}}\cdot\nabla_{x}\phi_R+\nabla_{x}\phi_i\cdot(\frac{\nabla_{v}\sqrt{\mu^{\alpha}}}{\sqrt{\mu^{\alpha}}}f_R^{\alpha}+\nabla_{v}f_R^{\alpha})\Big],\\
            E_7^{\alpha}= & -\frac{(\partial_t+v\cdot \nabla_{x})}{\sqrt{\mu^\alpha}}(\sqrt{\mu^\alpha}f_{2k-1}^\alpha)-\mathop{\sum}_{i+j \geq 2k-1\atop 2k-1 \geq i,j \geq1} \frac{\varepsilon^{i+j-6}}{\sqrt{\mu^\alpha}}\frac{e^\alpha}{m^\alpha}\nabla_{x}\phi_i\cdot\nabla_{v}F_j^{\alpha}\\
            &\hspace{1.7cm}+\mathop{\sum}_{i+j \geq 2k\atop 2k-1 \geq i,j \geq1} \sum_{\beta=A,B}\varepsilon^{i+j-6}\Big[\Gamma^{\alpha\beta}(f^\alpha_i,f^\beta_j)
            +\Gamma^{ \alpha \beta}(f^\alpha_j,f^\beta_i)\Big].
        \end{align*}
Then we have the following proposition.
    \begin{proposition}\label{L2estimate}
Let $(n^{A},n^{B},\mathbf{u},\theta,\phi_0)$ be the solution of the
Euler-Poisson systems \eqref{REWEUPSIO} obtained in Lemma
\ref{CPCO}.  $\mathbf{f}_R=(f^A_R,f^B_R)^T$
satisfies \eqref{L2dequationoff}. Then, there exist constants
$\varepsilon_0 >0$  such that for all $0<\varepsilon <
\varepsilon_0$, it holds that
        \begin{equation}\label{MainresultL22}
            \begin{split}
                \frac{d}{dt}&\Big(\sum_{\alpha=A,B}\Vert \sqrt{\theta}f^{\alpha}_R\Vert_{L^{2}_{x,v} }^{2}+| \nabla_x\mathbf{\phi}_R|_{L^{2}_{x} }^{2}\Big)+\frac{c_0}{2\varepsilon}\Vert(\mathbf{I}-\mathcal{P})\mathbf{f}_R\Vert_{\nu}^2 \\
                &\lesssim
                \varepsilon^2\Vert\mathbf{R}_{\gamma}\Vert_{{L^{\infty}_{x,v} }}
                \Vert \mathbf{f}_R\Vert_{{L^{2}_{x,v} }}
                +\varepsilon^{k-1}\Vert\mathbf{R}_{\gamma}\Vert_{L^{\infty}_{x,v} }
                \Vert \mathbf{f}_R\Vert^2_{{L^{2}_{x,v} }}+\varepsilon^k\Vert\mathbf{R}_{\gamma}\Vert_{L^{\infty}_{x,v} }
                \Vert \mathbf{f}_R\Vert_{{L^{2}_{x,v} }}| \nabla_x\mathbf{\phi}_R|_{{L^{2}_{x} }}\\
                &\,\,\,+\frac{1}{(1+t)^q}(\Vert \mathbf{f}_R\Vert^2_{{L^{2}_{x,v} }}+| \nabla_x\mathbf{\phi}_R|^2_{{L^{2}_{x} }})+\varepsilon\mathcal{I}_1(\Vert \mathbf{f}_R\Vert^2_{{L^{2}_{x,v} }}+| \nabla_x\mathbf{\phi}_R|^2_{{L^{2}_{x} }})+\varepsilon^{k-1}\mathcal{I}_2\Vert \mathbf{f}_R\Vert_{{L^{2}_{x,v} }},
            \end{split}
        \end{equation}
        where $\mathcal{I}_1$, $\mathcal{I}_2$ are given by the following:
        \begin{gather}
            \begin{split}
                \mathcal{I}_1&=\sum_{i=1}^{2k-1}[\varepsilon(1+t)]^{i-1}+(\sum_{i=1}^{2k-1}[\varepsilon(1+t)]^{i-1})^2, \\
                \mathcal{I}_2&=\sum_{2k\leq i+j\leq 4k}\varepsilon^{i+j-2k}(1+t)^{i+j-2}.
            \end{split}
        \end{gather}
    \end{proposition}
    \noindent\begin{proof}
        Taking $L^2$ inner product with $\theta f^{\alpha}_R$ on both sides of equation \eqref{L2dequationoff}, one obtains
        \begin{eqnarray}\label{L2ESmaineq}
            \begin{split}
                \sum_{\alpha=A,B}&\Big[\frac{1}{2}\frac{d}{dt}\Vert \sqrt{\theta}f^{\alpha}_R\Vert_{L^2_{x,v}}^{2}-\frac{1}{2}\left \langle (\partial_t+v\cdot\nabla_{x})\theta f^{\alpha}_R, f^{\alpha}_R   \right \rangle_{x,v} + \frac{1}{\varepsilon}\left \langle L^{\alpha} \mathbf{f}_R, \theta f^{\alpha}_R \right \rangle_{x,v}\Big]\\
                =\sum_{\alpha=A,B}&\sum_{i=1}^{7}\left \langle E^{\alpha}_{i}, \theta f^{\alpha}_R \right \rangle_{x,v}.
            \end{split}
        \end{eqnarray}
        In the following, we shall show \eqref{MainresultL22} for the hard potential case and the soft potential with $-1\leq \gamma \leq 0$, respectively.

        First, let us focus on the hard sphere case. To begin with, from the
         identity
        $
        \frac{e^\alpha}{m^\alpha} \frac{\nabla_v \mu^\alpha}{\sqrt{\mu^\alpha}} = -e^\alpha \frac{v - \mathbf{u}}{\theta} \sqrt{\mu^\alpha}
        $,
        we see
        \begin{equation*}
            \sum_{\alpha=A,B}\left \langle E^{\alpha}_{1}, \theta f^{\alpha}_R \right \rangle_{x,v}=\sum_{\alpha=A,B}e^{\alpha}\displaystyle{\iint_{\mathbb{R}^3\times\mathbb{R}^3}}(v-\mathbf{u})\sqrt{\mu^\alpha}f^{\alpha}_R\cdot\nabla_{x}\phi_R \,dv \, dx.
        \end{equation*}
        Due to
        $
        \Delta \phi_R = \int_{\mathbb{R}^3} \left( e^{A}F_R^A + e^{B}F_R^B \right) dv
        $
        which implies
        \begin{equation*}
            \begin{split}
                -\Delta \partial_{t}\phi_R = \int_{\mathbb{R}^3} e^{A}\partial_{t}F_R^A+ e^{B}\partial_{t}F_R^Bdv=\sum_{\alpha=A,B}e^{\alpha}\int_{\mathbb{R}^3}v\cdot\nabla_{x}(\sqrt{\mu^\alpha}f^{\alpha}_R) dv.
            \end{split}
        \end{equation*}
        So, it yields that
        \begin{eqnarray*}
            \begin{split}
                -\frac{1}{2}\frac{d}{dt}| \nabla_x\mathbf{\phi}_R |_{L^{2}_{x}}^{2}=\int_{\mathbb{R}^3}     \Delta \partial_{t}\phi_R \, \phi_R\,dx&=-\sum_{\alpha=A,B}e^{\alpha}\int_{\mathbb{R}^3}\int_{\mathbb{R}^3}v\cdot\nabla_{x}(\sqrt{\mu^\alpha}f^{\alpha}_R) \phi_R \,dv\, dx\\
                &=\sum_{\alpha=A,B}e^{\alpha}\int_{\mathbb{R}^3}(\int_{\mathbb{R}^3}v\sqrt{\mu^{\alpha}}f^{\alpha}_Rdv)\cdot \nabla_{x}\phi_R dx.
            \end{split}
        \end{eqnarray*}
        The second part is bounded by
        \begin{eqnarray*}
            \begin{split}
                \Big|\sum_{\alpha=A,B}e^{\alpha}\int_{\mathbb{R}^3}\mathbf{u}(\int_{\mathbb{R}^3}\sqrt{\mu^{\alpha}}f^{\alpha}_Rdv)\cdot \nabla_{x}\phi_R dx\Big| &= \Big|\int_{\mathbb{R}^3}   \Delta \partial_{t}\phi_R \mathbf{u}\cdot \nabla_{x}\phi_Rdx\Big|
                \lesssim |\nabla_{x}\mathbf{u}|_{L^{\infty}_{x} }|\nabla_{x}\phi^{\rm }_R|^2_{L^{2}_{x} }\\
                & \lesssim \frac{1}{(1+t)^q} |\nabla_{x}\phi^{\rm }_R|^2_{L^{2}_{x}}.
            \end{split}
        \end{eqnarray*}

        Next, the term $\frac{1}{2}\left \langle (\partial_t+v\cdot\nabla_{x}) f^{\alpha}_R, \theta f^{\alpha}_R    \right  \rangle_{x,v} + \left \langle E^{\alpha}_{2}, \theta f^{\alpha}_R \right \rangle_{x,v}$ is a cubic polynomial in
        $v$, which increase as $\{1+|v|\}^3$. Let $a_r=(3-\gamma)^{-1}$, for any small
        constant $\kappa>0$, we devide $\mathbb{R}^3_v$ into $\{|v|<\frac{\kappa}{\varepsilon^{a_r}} \}$ and $\{ |v|\geq\frac{\kappa}{\varepsilon^{a_r}}\}$, and notice that
        \begin{equation*}
            (1+|v|^2)^{\frac{3}{2}}|f_R^{\alpha}| \leq  (1+|v|)^{l-3} |h^{\alpha}|, \qquad l\geq7.
        \end{equation*}
        It holds that
        \begin{align*}
            & \sum_{\alpha=A,B}\Big(|\frac{1}{2}\left \langle \{\partial_t+v\cdot\nabla_{x}\}\theta f^{\alpha}_R,  f^{\alpha}_R    \right  \rangle_{x,v}| + |\left \langle E^{\alpha}_{2}, \theta f^{\alpha}_R \right \rangle_{x,v}|\Big)\\
            \lesssim & | ( n^A,n^B, \mathbf{u}, \theta) |_{W_x^{1,\infty}}\Vert \mathbf{h} (1+|v|^2)^{\gamma-3}\mathbf{1}_{|v|\geq \frac{\kappa}{\varepsilon^{a_r}}}\Vert_{L^{\infty}_{x,v}}
            \Vert \mathbf{f}_R \Vert_{L^{2}_{x,v}} \Big(\int_{\mathbb{R}^3}  (1+|v|^2)^{-(l+\gamma-3)} dv\Big)^{1/2} \\
            & \hspace{1.5cm}+ |(n^A,n^B,\mathbf{u}, \theta) |_{W_x^{1,\infty}} \, \Vert(1+|v|^2)^{\frac{3}{4}}\mathbf{f}_R \mathbf{1}_{|v|\leq \frac{\kappa}{\varepsilon^{a_r}}}\Vert_{L^{2}_{x,v}}^2 \\
            \lesssim & C_{\kappa}\varepsilon^2 \Vert \mathbf{h} \Vert_{L^{\infty}_{x,v}}\Vert \mathbf{f}_R \Vert_{L^{2}_{x,v}} +
             \Vert(1+|v|^2)^{\frac{3}{4}}\mathcal{P}\mathbf{f}_R\, \mathbf{1}_{|v|\leq \frac{\kappa}{\varepsilon^{a_r}}}\Vert_{L^{2}_{x,v}}^2\\
             &\hspace{1.5cm} +\Vert(1+|v|^2)^{\frac{3}{4}}(\mathbf{I}-\mathcal{P})\mathbf{f}_R\, \mathbf{1}_{|v|\leq \frac{\kappa}{\varepsilon^{a_r}}}\Vert_{L^{2}_{x,v}}^2\\
            \lesssim & C_{\kappa}\varepsilon^2 \Vert \mathbf{h} \Vert_{L^{\infty}_{x,v}}\Vert \mathbf{f}_R \Vert_{L^{2}_{x,v}} + \Vert(1+|v|^2)^{\frac{3}{4}}\mathcal{P}\mathbf{f}_R\, \mathbf{1}_{|v|\leq \frac{\kappa}{\varepsilon^{a_r}}}\Vert_{L^{2}_{x,v}}^2 \\ &\hspace{1.5cm}+\int_{\mathbb{R}^3}\int_{\mathbb{R}^3}\left \langle v \right\rangle^{\gamma}|(\mathbf{I}-\mathcal{P})\mathbf{f}_R|^2\, \left \langle v \right\rangle^{3-\gamma}\mathbf{1}_{|v|\leq \frac{\kappa}{\varepsilon^{a_r}}} dv \, dx\\
            \lesssim & C_{\kappa}\varepsilon^2 \Vert \mathbf{h} \Vert_{L^{\infty}_{x,v}}\Vert \mathbf{f}_R \Vert_{L^{2}_{x,v}} +\Vert(1+|v|^2)^{\frac{3}{4}}\mathcal{P}\mathbf{f}_R\, \mathbf{1}_{|v|\leq \frac{\kappa}{\varepsilon^{a_r}}}\Vert_{L^{2}_{x,v}}^2+\frac{\kappa^{3-\gamma}}{\varepsilon}\Vert(\mathbf{I}-\mathcal{P})\mathbf{f}_R \Vert_{\nu}^2 \\
            \lesssim & C_{\kappa}\varepsilon^2 \Vert \mathbf{h} \Vert_{L^{\infty}_{x,v}}\Vert \mathbf{f}_R \Vert_{L^{2}_{x,v}} +\frac{1}{(1+t)^q}\Vert \mathbf{f}_R \Vert_{L^{2}_{x,v}}+\frac{\kappa^{3-\gamma}}{\varepsilon}\Vert(\mathbf{I}-\mathcal{P})\mathbf{f}_R \Vert_{\nu}^2.
        \end{align*}
        For $\gamma=1$ (hard sphere), \eqref{YZJVEPDB} shows that
        \begin{equation*}
            \Vert(1+|v|^2)^{\frac{3}{4}}\mathcal{P}\mathbf{f}_R\Vert_{L^{2}_{x,v}}^2 \leq \Vert \mathbf{f}_R \Vert_{L^{2}_{x,v}},
        \end{equation*}
        where $\frac{1}{(1+t)^q}$ comes from the decay estimates of $| ( n^A,n^B, \mathbf{u}, \theta) |_{W_x^{1,\infty}}$ in \eqref{PTWDAY}.\\

        Moreover, using Lemma \ref{L2estimatesofnonlinears}, it is easy to get
        \begin{eqnarray*}
            \begin{split}
                \sum_{\alpha=A,B}\Big|\left \langle E^{\alpha}_{3}, \theta f^{\alpha}_R \right \rangle_{x,v}\Big|
                =& \Big|\sum_{i=1}^{2k-1}\sum_{\alpha=A,B}\sum_{\beta=A,B}\varepsilon^{i-1}  \left \langle\Big(\Gamma^{ \alpha \beta}(f^\alpha_i,f^\beta_R)+\Gamma^{\alpha\beta }(f^\alpha_R, f^\beta_i)\Big), \theta f^{\alpha}_R\right \rangle_{x,v}\Big|\\
                \lesssim & \sum_{i=1}^{2k-1}[\varepsilon(1+t)]^{i-1}
                \Vert \mathbf{f}_R \Vert_{L^{2}_{x,v}}\Vert(\mathbf{I}-\mathcal{P})f^{\alpha}_R \Vert_{\nu} \\
                \lesssim & C_{\kappa}\varepsilon \Vert \mathbf{f}_R \Vert^2_{L^{2}_{x,v}} +\frac{\kappa}{\varepsilon}\Vert(\mathbf{I}-\mathcal{P})\mathbf{f}_R \Vert_{\nu}^2,
            \end{split}
        \end{eqnarray*}
        and
        \begin{eqnarray*}
            \begin{split}
                \sum_{\alpha=A,B}\Big|\left \langle E^{\alpha}_{4}, \theta f^{\alpha}_R \right \rangle_{x,v}\Big|
                =& \varepsilon^{k-1} \sum_{\alpha=A,B}\sum_{\beta=A,B}  \Big|\left \langle \Gamma^{ \alpha \beta}( f_R^{\alpha},  f_R^\beta ) , \theta f^{\alpha}_R \right \rangle_{x,v}\Big|\\
                \lesssim & \varepsilon^{k-1}\Vert \mathbf{h} \Vert_{L^{\infty}_{x,v}} \Vert \mathbf{f}_R \Vert^2_{L^{2}_{x,v}}.
            \end{split}
        \end{eqnarray*}
        Noting the orthogonality relation
        $
        \bigl\langle \nabla_x \phi_R \cdot \nabla_v f_R^\alpha, \theta f_R^\alpha \bigr\rangle_{x,v} = 0
        $,
        one obtains
        \begin{align*}
            \sum_{\alpha=A,B}\Big|\left \langle E^{\alpha}_{5}, \theta f^{\alpha}_R \right \rangle_{x,v}\Big|
            =& \,\varepsilon^{k} \sum_{\alpha=A,B} e^{\alpha}\iint_{\mathbb{R}^3\times\mathbb{R}^3} \nabla_{x}\phi_R\cdot\frac{v-\mathbf{u}}{2}(f_R^{\alpha})^2 dv dx \\
            \lesssim & \,\varepsilon^{k}\sum_{\alpha=A,B} |\nabla_{x}\phi^{\rm }_R|_{L^{2}_{x} }\Big(\int_{\mathbb{R}^3}|\int_{\mathbb{R}^3}|v-\mathbf{u}|f^{\alpha}_R\frac{\sqrt{\mu_{M}^{\alpha}}}{\left \langle v \right\rangle^{l}\sqrt{\mu^{\alpha}}}h^{\alpha} dv|^2 dx\Big)^{\frac{1}{2}}\\
            \lesssim & \,\varepsilon^{k}\sum_{\alpha=A,B}
            \Vert (\int_{\mathbb{R}^3}\frac{|v-\mathbf{u}|^2\mu_{M}^{\alpha}}{\left \langle v \right\rangle^{2l}\mu^{\alpha}}dv)^{\frac{1}{2}} \Vert_{{L^{\infty}_{x,v} }}\Vert h^{\alpha} \Vert_{{L^{\infty}_{x,v}}} \Vert f^{\alpha}_R \Vert_{L^{2}_{x,v}}|\nabla_{x}\phi^{\rm }_R|_{L^{2}_{x} } \\
            \lesssim & \,\varepsilon^{k}\Vert h^{\alpha} \Vert_{L^{\infty}_{x,v}} \Vert f^{\alpha}_R \Vert_{L^{2}_{x,v}}|\nabla_{x}\phi^{\rm }_R|_{L^{2}_{x}}.
        \end{align*}
        The decomposition
        $
        \displaystyle{\sum_{\alpha=A,B}} \bigl\langle E_6^\alpha, \theta f_R^\alpha \bigr\rangle_{x,v} =: I_{h1} + I_{h2}
        $
        leads to
        \begin{align*}
                I_{h1}  &=\sum_{i=1}^{2k-1} \varepsilon^i \sum_{\alpha=A,B} \iint_{\mathbb{R}^3\times\mathbb{R}^3} \frac{\nabla_{v}F_i^{\alpha}}{\sqrt{\mu^\alpha}}\cdot\nabla_{x}\phi^{\alpha}_R \theta f_R^{\alpha}\, dv\, dx \\
                & \lesssim \sum_{i=1}^{2k-1} \varepsilon^i \sum_{\alpha=A,B} \|\theta (\int_{\mathbb{R}^3}\frac{|\nabla_{v}F_i|^2}{\mu^{\alpha}}dv)^{\frac{1}{2}} \|_{{L^{\infty}_{x,v} }}\Vert f^{\alpha}_R \Vert_{L^{2}_{x,v}}|\nabla_{x}\phi^{\rm }_R|_{L^{2}_{x} }\\
                & \lesssim (\sum_{i=1}^{2k-1}[\varepsilon(1+t)]^{i-1})(|\nabla_{x}\phi^{\rm }_R|_{L^{2}_{x} }+\Vert \mathbf{f}_R  \Vert_{L^{2}_{x,v}}) \\
                & \lesssim \mathcal{I}_1 (|\nabla_{x}\phi^{\rm }_R|_{L^{2}_{x} }+\Vert \mathbf{f}_R \Vert_{L^{2}_{x,v}}),
        \end{align*}
        and
        \begin{align*}
                I_{h2}  &=\sum_{i=1}^{2k-1} \varepsilon^i\sum_{\alpha=A,B}\frac{e^{\alpha}}{m^{\alpha}} \iint_{\mathbb{R}^3\times\mathbb{R}^3} \nabla_{x}\phi_i\cdot(\frac{\nabla_{v}\sqrt{\mu^{\alpha}}}{\sqrt{\mu^{\alpha}}}f_R^{\alpha}-\nabla_{v}f_R^{\alpha}) \theta f_R^{\alpha} \, dv \,dx \\
                &=\sum_{i=1}^{2k-1} \varepsilon^i\sum_{\alpha=A,B}e^{\alpha} \iint_{\mathbb{R}^3\times\mathbb{R}^3} \nabla_{x}\phi_i\cdot(\frac{v-\mathbf{u}}{2}f_R^{\alpha}) f_R^{\alpha} \, dv \,dx \\
                & \lesssim (\sum_{i=1}^{2k-1} \varepsilon^i|\nabla_{x}\phi^{\rm }_R|_{L^{2}_{x} })^2 \frac{\varepsilon}{\kappa^2}\Vert f^{\alpha}_R \Vert^2_{L^{2}_{x,v}}+\frac{\kappa^2}{\varepsilon}\Vert(\mathbf{I}-\mathcal{P})f^A_R \Vert_{\nu}^2\\
                & \lesssim C_{\kappa}\varepsilon\mathcal{I}_1\Vert \mathbf{f}_R \Vert^2_{L^{2}_{x,v}}+\frac{\kappa^2}{\varepsilon}\Vert(\mathbf{I}-\mathcal{P})f^A_R \Vert_{\nu}^2.
        \end{align*}
          Considering the decomposition
        $
        \displaystyle{\sum_{\alpha=A,B}}  \bigl\langle E_7^\alpha, \theta f_R^\alpha \bigr\rangle_{x,v} =: I_{h3} + I_{h4} + I_{h5},
        $
        Lemma~\ref{FI2KM1EST} establishes the boundedness of the following norms:

        \begin{eqnarray*}
            \begin{split}
                &\Vert  \frac{\left \langle v \right\rangle^{l}F_i^{\alpha} }{\sqrt{\mu^{\alpha}}}\Vert_{L^{\infty}_{v}L^{\frac{6}{5}}_{x}},\quad   \Vert \left \langle v \right\rangle^{l} \frac{\nabla_{x,v}F_i^{\alpha} }{\sqrt{\mu^{\alpha}}}\Vert_{L^{\infty}_{v}L^{\frac{6}{5}}_{x}},\quad    \Vert \nabla_{x}\phi^{\rm }_i \Vert_{p} \lesssim (1+t)^{i-1}, \\
                &\Vert  \frac{\left \langle v \right\rangle^{l}F_i^{\alpha} }{\sqrt{\mu^{\alpha}}}\Vert_{L^{\infty}_{v}L^{\frac{6}{5}}_{x}},\quad   \Vert \left \langle v \right\rangle^{l} \frac{\nabla_{x,v}F_i^{\alpha} }{\sqrt{\mu^{\alpha}}}\Vert_{L^{\infty}_{v}L^{\frac{6}{5}}_{x}},\quad for\quad  1<p\leq \infty.
            \end{split}
        \end{eqnarray*}
        Consequently,
        \begin{align*}
            I_{h3}  &=\sum_{\alpha=A,B} \iint_{\mathbb{R}^3\times\mathbb{R}^3} -\frac{(\partial_t+v\cdot \nabla_{x})}{\sqrt{\mu^\alpha}}F_{2k-1}^{\alpha} \theta f_R^\alpha \, dv  \,dx
            \lesssim (1+t)^{i-1} \Vert \mathbf{f}_R \Vert_{L^{2}_{x,v}}, \\
            I_{h4}  &= \sum_{\alpha=A,B}\mathop{\sum}_{i+j \geq 2k-1\atop 0 \leq i,j \leq 2k-1}\frac{e^{\alpha}}{m^{\alpha}}\iint_{\mathbb{R}^3\times\mathbb{R}^3} \frac{\varepsilon^{i+j-2k+1}}{\sqrt{\mu^\alpha}}\nabla_{x}\phi^{\alpha}_i\cdot\nabla_{v}F_j^{\alpha}\theta f_R^\alpha \, dv  \,dx \\
            &\lesssim (1+t)^{i-1}\Vert \mathbf{f}_R \Vert_{L^{2}_{x,v}},\\
            I_{h5}  &=\sum_{\alpha=A,B}\mathop{\sum}_{i+j \geq 2k\atop 1 \leq i,j \leq 2k-1} \sum_{\beta=A,B}\iint_{\mathbb{R}^3\times\mathbb{R}^3}\frac{\varepsilon^{i+j-2k}}{\sqrt{\mu^\beta}}\Big[\Gamma^{\alpha\beta}(f^\alpha_i,f^\beta_j)
            +\Gamma^{ \alpha \beta}(f^\alpha_j,f^\beta_i)\Big] \theta f_R^\alpha \, dv \, dx\\
            & \lesssim \sum_{2k \leq i+j \leq 4k-2}\varepsilon^{i+j-k+1}(1+t)^{i+j-2}\Vert\mathbf{f}_R \Vert_{L^{2}_{x,v} }
            \lesssim \varepsilon^{k-1} \mathcal{I}_2 \Vert\mathbf{f}_R \Vert_{L^{2}_{x,v}}.
        \end{align*}

       Hence, insert the above estimates into \eqref{L2ESmaineq}, and using Lemma~\ref{COVWEIGHTED} we obtain
        \begin{align}
                \frac{d}{dt}&\Big(\sum_{\alpha=A,B}\Vert \sqrt{\theta}f^{\alpha}_R\Vert_{L^{2}_{x,v} }^{2}+| \nabla_x\mathbf{\phi}_R|_{L^{2}_{x} }^{2}\Big)+\frac{c_1}{\varepsilon}\Vert(\mathbf{I}-\mathcal{P})\mathbf{f}_R\Vert_{\nu}^2 \label{MainresultL2}\\
                &\lesssim
                \varepsilon^2\Vert\mathbf{h}\Vert_{{L^{\infty}_{x,v} }}
                \Vert \mathbf{f}_R\Vert_{{L^{2}_{x,v} }}
                +\varepsilon^{k-1}\Vert\mathbf{h}\Vert_{L^{\infty}_{x,v} }
                \Vert \mathbf{f}_R\Vert^2_{{L^{2}_{x,v} }}+\varepsilon^k\Vert\mathbf{h}\Vert_{L^{\infty}_{x,v} }
                \Vert \mathbf{f}_R\Vert_{{L^{2}_{x,v} }}| \nabla_x\mathbf{\phi}_R|_{{L^{2}_{x} }}\notag\\
                &\,\,\,+\frac{1}{(1+t)^q}(\Vert \mathbf{f}_R\Vert^2_{{L^{2}_{x,v} }}+| \nabla_x\mathbf{\phi}_R|^2_{{L^{2}_{x} }})+\varepsilon\mathcal{I}_1(\Vert \mathbf{f}_R\Vert^2_{{L^{2}_{x,v} }}+| \nabla_x\mathbf{\phi}_R|^2_{{L^{2}_{x} }})\notag\\
                &\,\,\,+\varepsilon^{k-1}\mathcal{I}_2\Vert \mathbf{f}_R\Vert_{L^{2}_{x,v} }+\frac{\kappa}{\varepsilon}\Vert(\mathbf{I}-\mathcal{P})\mathbf{f}_R\Vert_{\nu}^2.\notag
            \end{align}
        Choosing $\kappa<\frac{c_1}{2}$  small enough,  the last term on the right-hand side could be absorbed into the
        dissipation term, so that \eqref{MainresultL22} yields for $\gamma=1$.\\

        Next, let us turn to  potential $-1\leq\gamma<1$, in this case, the weight function takes the form
        $
        e^{\tilde{\kappa}\langle v \rangle^{(3-\gamma)/2}}.
        $
        The estimates
        $
        \displaystyle{\sum_{\alpha=A,B}}\bigl|\langle E_i^\alpha, \theta f_R^\alpha \rangle_{x,v}\bigr|, \quad i\in\{1,3,7\},
        $
        together with the $I_{h1}$ component of
        $
        \displaystyle{\sum_{\alpha=A,B}}\bigl|\langle E_6^\alpha, \theta f_R^\alpha \rangle_{x,v}\bigr|,
        $
        coincide with those in the hard sphere case.
        Following analogous arguments yields
        \begin{align*}
                \sum_{\alpha=A,B}\Big|\left \langle E^{\alpha}_{4}, \theta f^{\alpha}_R \right \rangle_{x,v}\Big|
                =& \varepsilon^{k-1} \sum_{\alpha=A,B}\sum_{\beta=A,B}  \Big|\left \langle \Gamma^{ \alpha \beta}( f_R^{\alpha},  f_R^\beta ) , \theta f^{\alpha}_R \right \rangle_{x,v}\Big|
                \lesssim  \varepsilon^{k-1}\Vert \mathbf{S}_{\gamma} \Vert_{L^{\infty}_{x,v}} \Vert \mathbf{f}_R \Vert^2_{L^{2}_{x,v}},\\
                \sum_{\alpha=A,B}\Big|\left \langle E^{\alpha}_{5}, \theta f^{\alpha}_R \right \rangle_{x,v}\Big|
                \lesssim & \varepsilon^{k}\sum_{\alpha=A,B} |\nabla_{x}\phi^{\rm }_R|_{L^{2}_{x} }\Big(\int_{\mathbb{R}^3}|\int_{\mathbb{R}^3}|v-\mathbf{u}|f^{\alpha}_R\frac{\sqrt{\mu_{M}^{\alpha}}}{\omega_{\gamma}\sqrt{\mu^{\alpha}}}S^{\alpha}_{\gamma} dv|^2 dx\Big)^{\frac{1}{2}}\\
                \lesssim & \,\varepsilon^{k}\Vert \mathbf{S}_{\gamma} \Vert_{L^{\infty}_{x,v}} \Vert \mathbf{f}_R \Vert_{L^{2}_{x,v}}|\nabla_{x}\phi^{\rm }_R|_{L^{2}_{x} }.
        \end{align*}
        The terms containing differences are analyzed as follows.

        Given $a_r=(3-\gamma)^{-1}$, the velocity space $\mathbb{R}^3_v$ is partitioned into
        $
        \left\{|v|<\frac{\kappa}{\varepsilon^{a_r}}\right\} \,\, \text{and} \,\, \left\{|v|\geq\frac{\kappa}{\varepsilon^{a_r}}\right\},
        $
        \begin{align*}
                & \sum_{\alpha=A,B} \Big|\left \langle E^{\alpha}_{2}, e^{\alpha}\theta f^{\alpha}_R \right \rangle_{x,v}\Big|\\
                \lesssim &\Vert ( n^A,n^B,\mathbf{u}, \theta) \Vert_{H^1}\Vert \left \langle v  \right \rangle^{\frac{3}{2}} \mathbf{f}_R\mathbf{1}_{|v|\geq \frac{\kappa}{\varepsilon^a}}\Vert_{L^{\infty}_{x,v}}
                \Vert \mathbf{f}_R \Vert_{L^{2}_{x,v}}
                + \Vert(n^A,n^B,\mathbf{u}, \theta) \Vert_{W_{x}^{1,\infty}} \, \Vert\left \langle v\right \rangle^{\frac{3}{4}}\mathbf{f}_R \mathbf{1}_{|v|\leq \frac{\kappa}{\varepsilon^a}}\Vert_{L^{2}_{x,v}}^2 \\
                \lesssim & \Vert  \frac{\omega_{\gamma}}{\left \langle v    \right \rangle^{3-\gamma}}\mathbf{1}_{|v|\geq \frac{\kappa}{\varepsilon^a}}\mathbf{f}_R\Vert_{L^{\infty}_{x,v}}\Vert \mathbf{f}_R \Vert_{L^{2}_{x,v}} + \frac{1}{(1+t)^q}\Vert\left \langle v   \right \rangle^{\frac{3}{4}}\mathcal{P}\mathbf{f}_R \Vert_{L^{2}_{x,v}}^2 +\Vert|v|^{\frac{3}{2}}\mathbf{1}_{|v|\leq \frac{\kappa}{\varepsilon^a}}(\mathbf{I}-\mathcal{P})\mathbf{f}_R \Vert_{L^{2}_{x,v}}^2\\
                \lesssim & C_{\kappa}\varepsilon^2 \Vert \mathbf{S}_{\gamma} \Vert_{L^{\infty}_{x,v}}\Vert \mathbf{f}_R \Vert_{L^{2}_{x,v}} +\frac{1}{(1+t)^q}\Vert \mathbf{f}_R \Vert^2_{L^{2}_{x,v}}+ \int_{\mathbb{R}^3} dx \int_{\mathbb{R}^3} \left \langle v  \right \rangle^{\frac{\gamma}{2}} |(\mathbf{I}-\mathcal{P})\mathbf{f}_R|^2\left \langle v   \right \rangle^{\frac{3-\gamma}{2}} \mathbf{1}_{|v|\leq \frac{\kappa}{\varepsilon^a}} dv\\
                \lesssim & C_{\kappa}\varepsilon^2 \Vert \mathbf{S}_{\gamma} \Vert_{L^{\infty}_{x,v}}\Vert \mathbf{f}_R \Vert_{L^{2}_{x,v}} +\frac{1}{(1+t)^q}\Vert \mathbf{f}_R \Vert^2_{L^{2}_{x,v}}+\frac{\kappa^{3-\gamma}}{\varepsilon}\Vert(\mathbf{I}-\mathcal{P})\mathbf{f}_R\Vert_{\nu}^2.
            \end{align*}
       Moreover, by \eqref{ORDERZONG} and Gagliardo-Nirenberg inequality, the potential function $\phi^{\rm }_i$ satisfies
        \begin{eqnarray*}
            \begin{split}
                |\nabla_{x}\phi^{\rm }_i|^2_{L^{2}_{x} }    &=-\int_{\mathbb{R}^3}  \Delta \phi_i  \phi_i\,dx=-\sum_{\alpha=A,B}e^{\alpha}\iint_{\mathbb{R}^3\times\mathbb{R}^3} F_i^{\alpha} \phi_i  \,dv \,dx \\
                &\lesssim \sum_{\alpha=A,B}\iint_{\mathbb{R}^3\times\mathbb{R}^3} \frac{|F_i^{\alpha}|}{\sqrt{\mu^{\alpha}}} |\phi_i|\sqrt{\mu^{\alpha}}  \,dv \,dx
                \lesssim |\phi^{\rm }_i|_{L^{6}_{x} }
                \Vert \frac{F_i^{\alpha}}{\sqrt{\mu^{\alpha}}}\Vert_{L^{\infty}_{v}L^{\frac{6}{5}}_{x}} \lesssim (1+t)^{i-1}|\nabla_{x}\phi^{\rm
                }_i|_{L^{2}_{x}},
            \end{split}
        \end{eqnarray*}
  which implies $|\nabla_{x}\phi^{\rm }_i|_{L^{2}_{x} }\lesssim (1+t)^{i-1}$, then the $I_{32}$ part of $\displaystyle{\sum_{\alpha=A,B}}\Big|\left \langle E^{\alpha}_{6}, \theta f^{\alpha}_R \right \rangle_{x,v}\Big|$ is bounded by
    \begin{align*}
                I_{h2}
                =&\sum_{i=1}^{2k-1} \varepsilon^i\sum_{\alpha=A,B}e^{\alpha} \iint_{\mathbb{R}^3\times\mathbb{R}^3} \nabla_{x}\phi_i\cdot(\frac{v-\mathbf{u}}{2}f_R^{\alpha}) f_R^{\alpha} \, dv \,dx \\
                \lesssim & \sum_{i=1}^{2k-1} \varepsilon^i \Big(\iint_{|v|\geq\frac{1}{\frac{\kappa}{1-\gamma}}}+\iint_{|v|\leq\frac{1}{\frac{\kappa}{1-\gamma}}}\Big)|\nabla_{x}\phi^{\rm }_i|\left \langle v    \right \rangle \mathbf{f}_R^2dv \, dx\\
                \lesssim &\sum_{i=1}^{2k-1} \varepsilon^i \Big(\Vert \left \langle v   \right \rangle \omega_{\gamma} \mathbf{f}_R \mathbf{1}_{|v|\geq\frac{1}{\frac{\kappa}{1-\gamma}}}\Vert_{L^{\infty}_{x,v}}|\nabla_{x}\phi^{\rm }_i|_{L^{2}_{x} }\Vert \mathbf{f}_R \Vert_{L^{2}_{x,v}}+ |\nabla_{x}\phi^{\rm }_i|_{L^{\infty}_{x} }\Vert \mathbf{P}\mathbf{f}_R \Vert^2_{L^{2}_{x,v}}\\
                & \hspace{1cm}+ |\nabla_{x}\phi^{\rm }_i|_{L^{\infty}_{x} }\iint_{\mathbb{R}^3\times\mathbb{R}^3}\left \langle v    \right \rangle^{\gamma}|(\mathbf{I}-\mathcal{P})\mathbf{f}_R|^2 \left \langle v \right \rangle^{1-\gamma}\mathbf{1}_{|v|\leq\frac{1}{\frac{\kappa}{1-\gamma}}} dv\, dx\Big)\\
                \lesssim & \sum_{i=1}^{2k-1} \varepsilon^i(1+t)^{i-1}\Big(\varepsilon^2\Vert \mathbf{S}_{\gamma} \Vert_{L^{\infty}_{x,v}} \Vert \mathbf{f}_R \Vert_{L^{2}_{x,v}}+\Vert \mathbf{f}_R \Vert^2_{L^{2}_{x,v}}+\frac{\kappa^{1-\gamma}}{\varepsilon}\Vert(\mathbf{I}-\mathcal{P})\mathbf{f}_R \Vert_{\nu}^2\Big)\\
                \lesssim & \,\varepsilon^2\Vert \mathbf{S}_{\gamma} \Vert_{L^{\infty}_{x,v}} \Vert \mathbf{f}_R \Vert_{L^{2}_{x,v}}+\varepsilon \mathcal{I}_1\Vert \mathbf{f}_R \Vert^2_{L^{2}_{x,v}}+\frac{\kappa^{1-\gamma}}{\varepsilon}\Vert(\mathbf{I}-\mathcal{P})\mathbf{f}_R \Vert_{\nu}^2.
        \end{align*}
        Consequently, whenever  $-1\leq\gamma<1$,  it follows that:
        \begin{equation}\label{MainresultL2}
            \begin{split}
                \frac{d}{dt}&\Big(\sum_{\alpha=A,B}\Vert \sqrt{\theta}f^{\alpha}_R\Vert_{L^{2}_{x,v} }^{2}+| \nabla_x\mathbf{\phi}_R|_{L^{2}_{x} }^{2}\Big)+\frac{c_0}{2\varepsilon}\Vert(\mathbf{I}-\mathcal{P})\mathbf{f}_R\Vert_{\nu}^2 \\
                &\lesssim
                \varepsilon^2\Vert\mathbf{S}_{\gamma}\Vert_{{L^{\infty}_{x,v} }}
                \Vert \mathbf{f}_R\Vert_{{L^{2}_{x,v} }}
                +\varepsilon^{k-1}\Vert\mathbf{S}_{\gamma}\Vert_{L^{\infty}_{x,v} }
                \Vert \mathbf{f}_R\Vert^2_{{L^{2}_{x,v} }}+\varepsilon^k\Vert\mathbf{S}_{\gamma}\Vert_{L^{\infty}_{x,v} }
                \Vert \mathbf{f}_R\Vert_{{L^{2}_{x,v} }}| \nabla_x\mathbf{\phi}_R|_{{L^{2}_{x} }}\\
                &\,\,\,+\frac{1}{(1+t)^q}(\Vert \mathbf{f}_R\Vert^2_{{L^{2}_{x,v} }}+| \nabla_x\mathbf{\phi}_R|^2_{{L^{2}_{x} }})+\varepsilon\mathcal{I}_1(\Vert \mathbf{f}_R\Vert^2_{{L^{2}_{x,v} }}+| \nabla_x\mathbf{\phi}_R|^2_{{L^{2}_{x} }})+\varepsilon^{k-1}\mathcal{I}_2\Vert \mathbf{f}_R\Vert_{{L^{2}_{x,v} }}.
            \end{split}
        \end{equation}
    \end{proof}
    \begin{remark}
        For the potential range $-3<\gamma<-1$, the weight function $\omega_{\gamma}$ should be replaced by $\omega_{-1}$. The remaining procedure follows the case of $-1\leq\gamma<1$, with $\mathbf{S}_{\gamma}$ replaced by $\mathbf{S}_{-1}$.
    \end{remark}

    \section{ The  $W^{1,\infty}_{x,v}$ estimate of the remainder term for hard sphere}
    To establish energy closure in the $L^2$ estimates \eqref{MainresultL22}, weighted $L^{\infty}$ estimates are required for either $\mathbf{h}$ or $\mathbf{S}_{\gamma}$, which is obtained through Duhamel's principle along the VPB system trajectory. Unlike the compressible Euler limit \cite{[22]Guo2010CPAM}, the interaction between $L^{2}$ and $L^{\infty}$ norms fails to achieve complete energy closure, owing to the self-consistent electric field $\frac{e^{\alpha}}{m^\alpha} \nabla_{x} \phi \cdot \nabla_{v} F^{\alpha}$. This analytical challenge necessitates additional estimation of the derivatives' $L^{\infty}$ norms.
    The subsequent analysis proceeds as follows: Section 4 presents essential lemmas and proves the hard sphere case ($\gamma=1$), while Section 5 addresses the range $-1 \leq \gamma < 1$. For $-3 < \gamma < -1$, Section 6 examines the corresponding reduction in the validity time.

First, recalling the global Maxwellian defined in \eqref{def1.35},
for the vector function $\mathbf{g}=(g^{A},g^{B})^{T}$, we introduce
the  linear operator
$\mathbf{L}_{M}=(\mathbf{L}_M^{A},\mathbf{L}_M^{B})^{T}$
    as following
    \begin{gather*}\label{4.2}
        \begin{split}
            \mathbf{L}^\alpha_M \mathbf{g}=- \frac{1}{\sqrt{\mu^{\alpha}_M}} \sum_{\beta=A,B} \left[Q^{ \alpha\beta}(\mu^\alpha,\sqrt{\mu^{\beta}_M} g^{\beta})+Q^{\alpha \beta}(\sqrt{\mu^\alpha_M}g^{\alpha},\mu^\beta)\right].
        \end{split}
    \end{gather*}
    It can be decomposed to
    \begin{gather*}\label{4.3}
        \begin{split}
            \left(\begin{array} {lll} \mathbf{L}_M^{A} \mathbf{g} \\  \mathbf{L}_M^{B} \mathbf{g} \end{array}\right) = \left(\begin{array} {lll} \upsilon^{A} g^A + K_M^A \mathbf{g}\\ \upsilon^{B} g^B +  K_M^B \mathbf{g}\end{array}\right),
        \end{split}
    \end{gather*}
    where
        \begin{align}
            &\upsilon^{\alpha}=   \sum_{\beta=A,B} \int_{\mathbb{R}^3\times \mathbb{S}^2}
            B^{ \alpha \beta}(|v-v_*|, \cos \theta) \mu^{\beta}(v_*) d\sigma dv_*, \quad  K_{M}^{\alpha} \mathbf{g}= K_{M,1}^{\alpha} \mathbf{g}+K_{M,2}^{\alpha} \mathbf{g},
            \notag\\
            &    K_{M,1}^{\alpha} \mathbf{g}=  \frac{1}{\sqrt{\mu_M^\alpha}} \sum_{\beta=A,B} \int_{\mathbb{R}^3 \times \mathbb{S}^2}
            B^{\alpha\beta}(|v-v_*|, \cos \theta)  \mu^\alpha(v) \sqrt{\mu_M^\beta(v_*)}g^\beta(v_*)   d\sigma dv_*,\notag\\
            &    K_{M,2}^{\alpha} \mathbf{g}= - \frac{1}{\sqrt{\mu_M^\alpha}} \sum_{\beta=A,B} \int_{\mathbb{R}^3 \times \mathbb{S}^2}
            B^{\alpha\beta}(|v-v_*|, \cos \theta)  \Big[\mu^\alpha(v')\sqrt{\mu_M^\beta(v_*')}g^\beta(v_*')\label{nuKsuanzidef}\\
            & \hspace{6.8cm} +\mu^{\beta}(v_*') \sqrt{\mu_{M}^{\alpha}(v')}\, g^\alpha(v')\Big] d\sigma dv_*\notag.
        \end{align}
    Moreover, there exists a constant $C$ such that
    \begin{eqnarray*}
        C^{-1} \left \langle v \right\rangle^{\gamma} \leq \upsilon^{\alpha} \leq C \left \langle v \right\rangle^{\gamma}, \qquad  {\rm for}\qquad \alpha=A,B.
    \end{eqnarray*}
A smooth cutoff function $\chi_{\delta}$ with $0 \leq \chi_{\delta} \leq 1$ is introduced to regularize the singular part of the operators $K^{\alpha \beta}_M$. That is for any $\delta>0$, the function $\chi_{\delta}(s)$ satisfies:
    \begin{equation}\label{Cutoffunctiondef}
        \chi_{\delta}(s) \equiv 1 \quad \textnormal{for}  \, s \leq \delta, \quad {\rm and} \quad   \chi_{\delta}(s) \equiv 0,  \quad \textnormal{for}\, s \geq 2\delta.
    \end{equation}
    Finally, let $B^{\alpha \beta}_{\chi}(|v-v_*|,\cos \theta) =: B^{\alpha \beta}(|v-v_*|,\cos \theta) \chi_{\delta}(|v-v_*|) $, and we define the following operators
    \begin{gather*}\label{4.51}
        \begin{split}
            &    K_{M,1}^{\alpha,\chi} \mathbf{g}=  \frac{1}{\sqrt{\mu_M^\alpha}} \sum_{\beta=A,B} \int_{\mathbb{R}^3 \times \mathbb{S}^2}
            B_{\chi}^{\alpha\beta}(|v-v_*|, \cos \theta)  \mu^\alpha(v) \sqrt{\mu_M^\beta(v_*)}g^\beta(v_*)   d\sigma dv_*,\\
            &    K_{M,2}^{\alpha,\chi} \mathbf{g}= - \frac{1}{\sqrt{\mu_M^\alpha}} \sum_{\beta=A,B} \int_{\mathbb{R}^3 \times \mathbb{S}^2}
            B_{\chi}^{\alpha\beta}(|v-v_*|, \cos \theta)  \Big[\mu^\alpha(v')\sqrt{\mu_M^\beta(v_*')}g^\beta(v_*')\\
            & \hspace{6.8cm} +\mu^{\beta}(v_*') \sqrt{\mu_{M}^{\alpha}(v')}\, g^\alpha(v')\Big] d\sigma dv_*,
        \end{split}
    \end{gather*}
    and
    \begin{eqnarray*}\label{ccccKCdef}
        \begin{split}
            &K_{M,1}^{\alpha,c} \mathbf{g} = ( K_{M,1}^{\alpha}-  K_{M,1}^{\alpha,\chi}) \mathbf{g},\qquad K_{M,2}^{\alpha,c} \mathbf{g} = ( K_{M,2}^{\alpha}-  K_{M,2}^{\alpha,\chi}) \mathbf{g},\\
            &K_{M}^{\alpha,\chi} \mathbf{g}= K_{M,1}^{\alpha,\chi} \mathbf{g}+K_{M,2}^{\alpha,\chi} \mathbf{g}, \qquad K_{M}^{\alpha,c} \mathbf{g}= K_{M,1}^{\alpha,c} \mathbf{g}+K_{M,2}^{\alpha,c} \mathbf{g}.
        \end{split}
    \end{eqnarray*}
Using the weight function $w_{\gamma}(v)$ defined in
\eqref{WWWEIGHTFUC}, we denote
    \begin{eqnarray}\label{MWMWMWimpmain}
        \begin{split}
            &K_{M,w}^{\alpha} (\mathbf{g})=wK_{M}^{\alpha} (\frac{\mathbf{g}}{w}),\quad K_{M,w}^{\alpha,\chi} (\mathbf{g})=wK_{M}^{\alpha,\chi} (\frac{\mathbf{g}}{w}), \quad K_{M,w}^{\alpha,c} (\mathbf{g})=wK_{M}^{\alpha,c} (\frac{\mathbf{g}}{w}),\\
            &K_{M,i,w}^{\alpha,c} (\mathbf{g})=wK_{M,i}^{\alpha,c} (\frac{\mathbf{g}}{w}),\quad K_{M,i,w}^{\alpha,\chi} (\mathbf{g})=wK_{M,i}^{\alpha,\chi} (\frac{\mathbf{g}}{w}),\quad \quad  \quad i=1,2.
        \end{split}
    \end{eqnarray}

    \begin{lemma}\label{kchiestmainle}
        For the weight function $w_{\gamma}(v)$ defined in \eqref{WWWEIGHTFUC}, $-3 < \gamma \leq 1$ and sufficiently small $\delta > 0$, there exist constants $C > 1$ (large) and $c > 0$ (small) such that the singular operator $K_{M,w}^{\alpha,\chi}$ satisfies the estimate:
        \begin{gather}\label{kchiestmain}
            | K_{M,w}^{\alpha,\chi}\mathbf{g}(v) |+| D_x(K_{M,w}^{\alpha,\chi})\mathbf{g}(v) |+| D_v(K_{M,w}^{\alpha,\chi})\mathbf{g}(v) | \leq C {\delta}^{3+\gamma} \left \langle v \right\rangle^{\gamma} e^{-c|v|^2} |\mathbf{g} |_{L^\infty_v},
        \end{gather}
        where
        \begin{eqnarray}
            \begin{split}
                (D_x  K_{M,w}^{\alpha,\chi})\mathbf{g}:=D_x (K_{M,w}^{\alpha,\chi}\mathbf{g})-    K_{M,w}^{\alpha,\chi}(D_x\mathbf{g}),\\
                (D_v  K_{M,w}^{\alpha,\chi})\mathbf{g}:=D_v(K_{M,w}^{\alpha,\chi}\mathbf{g})-     K_{M,w}^{\alpha,\chi}(D_v\mathbf{g} ).
            \end{split}
        \end{eqnarray}
    \end{lemma}
    \begin{proof} Without loss of generality, we take the  weight function $w(v)=\omega_{\gamma}$ to give the proof, since  it increases the fastest with respect to the velocity variable $v$. By \eqref{1.36}, (c.f \cite{[22]Guo2010CPAM}), we have
        \begin{eqnarray*}\label{Qwanmambum}
            c_1 \mu_M^{\alpha} \leq   \mu^{\alpha} \leq c_2  (\mu_M^{\alpha})^{\tilde{q}} , \quad {\rm for \,\, some }\quad  1>\tilde{q} > \frac{\max\{m^A,m^B\}}{m^A+m^B}\geq \frac{1}{2}.
        \end{eqnarray*}
        From the conservation law $m^{\alpha}|v'|^{2}+m^{\beta}|v_*'|^2 =m^{\alpha}|v|^2+m^{\beta}|v_*|^2$  for $\alpha , \beta \in \{A,B\}$, one can get
        \begin{gather*}\label{4.12}
            \begin{split}
                \mu^{\alpha}(v')\frac{\sqrt{\mu_M^{\beta}(v_*')}}{\sqrt{\mu_M^{\alpha}(v)}}\leq& c_2  [\mu_M^{\alpha}(v') ]^{\tilde{q}} \frac{\sqrt{\mu_M^{\beta}(v_*')}}{\sqrt{\mu_M^{\alpha}(v)}} \leq C [\mu_M^{\alpha}(v') ]^{\tilde{q}-\frac{1}{2}} \sqrt{\mu_M^{\beta}(v_*)}, \\
                \mu^{\beta}(v_*')\frac{\sqrt{\mu_M^{\alpha}(v')}}{\sqrt{\mu_M^{\alpha}(v)}}\leq& c_2  [\mu_M^{\beta}(v_*') ]^{\tilde{q}} \frac{\sqrt{\mu_M^{\alpha}(v')}}{\sqrt{\mu_M^{\alpha}(v)}} \leq C [\mu_M^{\beta}(v_*') ]^{\tilde{q}-\frac{1}{2}} \sqrt{\mu_M^{\beta}(v_*)},\\
                \mu^\alpha (v)\frac{\sqrt{\mu_M^{\beta}(v_*)}}{\sqrt{\mu_M^{\alpha}(v)}} \leq& c_2 [\mu_M^{\alpha}(v)]^{\tilde{q}-\frac{1}{2}}\sqrt{\mu_M^{\beta}(v_*)} .
            \end{split}
        \end{gather*}
        Since $|v-v_*| \leq 2\delta < 2$, then one has
        \begin{gather*}
            \begin{split}
                \mu_M^{\beta}(v_*) &=(2\pi \theta_M)^{-\frac{3}{2}} e^{-\frac{m^{\beta}|v_*|^2}{2\theta_M}} \leq (2\pi \theta_M)^{-\frac{3}{2}} e^{-\frac{m^{\beta}\{|v|^2-|v-v_*|^2\}}{2\theta_M}}  \leq C \mu_M^{\beta}(v),\\
                \frac{w(v)}{w(v_*)}&=e^{\tilde{\kappa}[\left \langle v \right\rangle^{2}-\left \langle v_* \right\rangle^{2}]} \leq C, \qquad   \{\mu_M^{\beta}(v)\}^{\frac{1}{4}} \leq C\left \langle v \right\rangle^{\gamma}.
            \end{split}
        \end{gather*}
        Thus, there exist a small constant $c>0$, such that for all $-3<\gamma\leq1$, we have
        \begin{gather*}
            \begin{split}
                |K_{M,w}^{\alpha,\chi}\mathbf{g}(v)|
                \leq C {\delta}^{3+\gamma} \left \langle v \right\rangle^{\gamma} e^{-c|v|^2} |\mathbf{g} |_{L^\infty_v}.
            \end{split}
        \end{gather*}
        Moreover, it can be directly calculated that
        \begin{equation*}
            | (D_x  K_{M,w}^{\alpha,\chi}) \mathbf{g}(v) |\leq C \left \langle v \right\rangle^{2}| K_{M,w}^{\alpha,\chi}\mathbf{g}(v) |,\quad | (D_v K_{M,w}^{\alpha,\chi}) \mathbf{g}(v) |\leq C \left \langle v \right\rangle | K_{M,w}^{\alpha,\chi}\mathbf{g}(v) |,
        \end{equation*}
        where we have used the fact that $\left \langle v_* \right\rangle^{2}\leq C \left \langle v \right\rangle^{2}$. Since $\{\mu_M^{\beta}(v)\}^{\frac{1}{4}} \leq C\left \langle v \right\rangle^{\gamma-2}$, we obtain
        \begin{gather*}
            \begin{split}
                | D_x(K_{M,w}^{\alpha,\chi})\mathbf{g}(v) |+| D_v ( K_{M,w}^{\alpha,\chi})\mathbf{g}(v) |
                \leq C {\delta}^{2+\gamma} \left \langle v \right\rangle^{\gamma} e^{-c|v|^2} |\mathbf{g} |_{L^\infty_v}.
            \end{split}
        \end{gather*}
    \end{proof}
    \begin{remark}\label{kejixingsuodaom2}
        The term $| (D_v K_{M,w}^{\alpha,\chi}) \mathbf{g}(v) |$ involves differentiating the collision kernel, yielding $\frac{d}{dv}|v-v_*|^{\gamma} = \gamma|v-v_*|^{\gamma-1}$ for $|v-v_*| \leq 2\delta$. This derivative introduces a stronger singularity near $v = v_*$, and seems to make the expression not integrable when $-3<\gamma \leq -2$. However, we can remove the singularity through ‌a change of variable `` $v+v_*\rightarrow v_*$". Consequently, the validity range of the conclusion remains  $-3<\gamma \leq 1$.
    \end{remark}

    \begin{lemma}\label{LeMK2ker4444}
        Under the hypotheses of Lemma \ref{kchiestmainle} and using the notation from \eqref{MWMWMWimpmain}, the regular component $K_{M,w}^{\alpha,c}$ admits the decomposition:
        $$
        K_{M,w}^{\alpha,c} = K_{M,1,w}^{\alpha,c} + K_{M,2,w}^{\alpha,c}.
        $$
        Both parts $K_{M,1,w}^{\alpha,\chi}$ and $K_{M,2,w}^{\alpha,\chi}$ can be represented as integral operators of the form:
        \begin{eqnarray}\label{lem4311}
            \begin{split}
                K_{M,1,w}^{\alpha,c} \mathbf{g}(v)=\sum_{\beta=A,B}\int_{\mathbb{R}^3} k_{M,1}^{\alpha \beta}(v,v_*)\frac{w(v)}{w(v_*)}g^{\beta}(v_*) d v_{*},
            \end{split}
        \end{eqnarray}
        \begin{eqnarray}\label{lem43122}
            \begin{split}
                K_{M,2,w}^{\alpha,c} \mathbf{g}(v)&=K_{M,2,w}^{\alpha,c,(1)} \mathbf{g}(v)+K_{M,2,w}^{\alpha,c,(2)} \mathbf{g}(v) \\
                &=\sum_{\beta=A,B} \Big[\int_{\mathbb{R}^3}k_{M,2}^{\alpha \beta(1)}(v,u_{\parallel})\frac{w(v)}{w(v+u_{\parallel})}g^{\beta}(v+u_{\parallel}) d u_{\parallel}\\
                &\hspace{1cm}+\int_{\mathbb{R}^3}\int_{\mathbb{R}^2}k_{M,2}^{\alpha \beta(2)}(v,u_{\perp},u_{\parallel})\frac{w(v)}{w(v+u_{\perp}+u_{\parallel})}g^{\beta}(v+u_{\perp}+u_{\parallel}) d u_{\perp} d u_{\parallel} \Big].
            \end{split}
        \end{eqnarray}
        Here, we adopt the shorthand $w = w_{\gamma}$. There exists a small constant $c>0$ such that
        \begin{eqnarray}\label{ZZZzhuanghuagezV}
            \begin{split}
                \frac{w(v)}{w(v_*)}|k_{M,1}^{\alpha \beta}(v,v_*)|&\leq C_{\delta} \,e^{-c\{|v|^2+|v_*|^2\}},\\
                \frac{w(v)}{w(v+u_{\parallel})}|k_{M,2}^{\alpha \beta(1)}(v,u_{\parallel})|&\leq C_{\delta} \frac{(1+|v|+|u_{\parallel}|)^{\gamma-1}}{|u_{\parallel}|} \, e^{-c\{|u_{\parallel}|^2+|v_{\parallel}|^2\}} \\
                \frac{w(v)}{w(v+u_{\perp}+u_{\parallel})}|k_{M,2}^{\alpha \beta(2)}(v,u_{\perp},u_{\parallel})|&\leq C_{\delta}\,
                e^{-c[|v|^2+|u_{\parallel}|^2+|u_{\perp}|^2]},
            \end{split}
        \end{eqnarray}
        where
        \begin{equation*}
            v=v_{\parallel}+v_{\perp},\quad   |v|^2=|v_{\parallel}|^2+|v_{\perp}|^2, \quad  |u|^2 \thicksim (|u_{\parallel}|^2+|u_{\perp}|^2).
        \end{equation*}
       These kernel operators exhibit distinct decay properties:
        \begin{eqnarray}\label{sec4inL1decayK12}
            \begin{split}
                &\int_{\mathbb{R}^3}|k_{M,2}^{\alpha \beta(1)}(v,u_{\parallel})|\frac{w(v)}{w(v+u_{\parallel})} d u_{\parallel} \leq C_{\delta} \left \langle v \right\rangle^{\gamma-2},\\
                &\int_{\mathbb{R}^3}\int_{\mathbb{R}^2}|k_{M,2}^{\alpha \beta(2)}(v,u_{\perp},u_{\parallel})|\frac{w(v)}{w(v+u_{\perp}+u_{\parallel})} d u_{\perp} d u_{\parallel}\leq C_{\delta}\,
                e^{-c|v|^2}.
            \end{split}
        \end{eqnarray}
       The preceding analysis yields the following estimates:
        \begin{eqnarray}\label{K1111XVE}
            | K_{M,1,w}^{\alpha,c}\mathbf{g}(v) |+| (D_x  K_{M,1,w}^{\alpha,c}) \mathbf{g}(v) |+| (D_v K_{M,1,w}^{\alpha,c}) \mathbf{g}(v) | \leq C_{\delta} \left \langle v \right\rangle^{\gamma-2}e^{-\frac{c}{2}|v|^2}  |\mathbf{g} |_{L^\infty_v},
        \end{eqnarray}
        and
        \begin{eqnarray}\label{K222XVE}
            \begin{split}
                | K_{M,2,w}^{\alpha,c}\mathbf{g}(v) |+| (D_x  K_{M,2,w}^{\alpha,c}) \mathbf{g}(v) | &\leq C_{\delta}  \left \langle v \right\rangle^{\gamma-2}  |\mathbf{g} |_{L^\infty_v},\\
                | (D_v  K_{M,2,w}^{\alpha,c} )\mathbf{g}(v) | &\leq C_{\delta} \left \langle v \right\rangle^{\gamma-1}  | \left \langle v \right\rangle\mathbf{g} |_{L^\infty_v}.
            \end{split}
        \end{eqnarray}
    \end{lemma}
    \begin{remark}\label{kernalcompwithsinglek}
        The linear operator $K_{M,2,w}^{\alpha,c}$ for  mixed gases $(m^A\neq m^B)$ consists of two distinct components:
         \begin{equation*}\label{mixed_kernel}
            K_{M,2,w}^{\alpha,c} = \underbrace{K_{M,2,w}^{\alpha,c,(1)}}_{\text{Partial decay part}} + \underbrace{K_{M,2,w}^{\alpha,c,(2)}}_{\text{Total decay part}}
        \end{equation*}
        where the second term is an additional part compared to the Boltzmann equation for single-species particles.‌ The first term can be reformulated through a variable substitution as
        \begin{equation*}
            \int_{\mathbb{R}^3}k_{M,2}^{\alpha \beta(1)}(v,u_{\parallel})\frac{w(v)}{w(v+u_{\parallel})}g^{\beta}(v+u_{\parallel}) d u_{\parallel}= \int_{\mathbb{R}^3}\tilde{k}_{M,2}^{\alpha \beta(1)}(v,v_*)\frac{w(v)}{w(v_*)}g^{\beta}(v_*) d v_*,
        \end{equation*}
      and satisfies the estimate identical to those for single-species particles:
        \begin{eqnarray}
            \begin{split}
                \frac{w(v)}{w(v_*)}|\tilde{k}_{M,2}^{\alpha \beta(1)}(v,v_*)|\leq C_{\delta} \frac{(1+|v|+|v_*|)^{\gamma-1}}{|v-v_*|} \, e^{-c(|v-v_*|^2+\frac{||v|^2-|v_*|^2|^2}{|v-v_*|^2})}.
            \end{split}
        \end{eqnarray}
    \end{remark}

\begin{proof}
The proof of \eqref{lem4311}-\eqref{sec4inL1decayK12} as well as
Remark \ref{kernalcompwithsinglek} can be found in my recent work
\cite{Wu2025pp}. We only prove \eqref{K1111XVE} and \eqref{K222XVE}
in the following. First, a direct calculation leads to
        \begin{equation*}
            | (D_x  K_{M,1,w}^{\alpha,c}) \mathbf{g}(v) |\leq C \left \langle v \right\rangle^{2}| K_{M,1,w}^{\alpha,c}\mathbf{g}(v) |,\quad | (D_v  K_{M,1,w}^{\alpha,c}) \mathbf{g}(v) |\leq C \left \langle v \right\rangle | K_{M,1,w}^{\alpha,c}\mathbf{g}(v) |.
        \end{equation*}
       The inequality estimation derived from \eqref{ZZZzhuanghuagezV} demonstrates:
        \begin{equation*}
            \begin{split}
                \left \langle v \right\rangle^{2}| K_{M,1,w}^{\alpha,c}\mathbf{g}(v) |\leq &C \sum_{\beta=A,B}\int_{\mathbb{R}^3} |v-v_*|^{\gamma}\,    \left \langle v \right\rangle^{2}\,e^{-c\{|v|^2+|v_*|^2\}} d v_* |\mathbf{g} |_{L^\infty_v}\\
                \leq &C \int_{\mathbb{R}^3} |v-v_*|^{\gamma}\,  e^{-\frac{3c}{4}\{|v|^2+|v_*|^2\}} d v_* |\mathbf{g} |_{L^\infty_v}
                \lesssim  \left \langle v \right\rangle^{\gamma-2}  e^{-\frac{c}{2}|v|^2} |\mathbf{g} |_{L^\infty_v}.
            \end{split}
        \end{equation*}
      Next, considering the definition in \eqref{MWMWMWimpmain}, we have
      \begin{equation*}
        \begin{split}
            & | (D_x  K_{M,2,w}^{\alpha,c})\mathbf{g}(v) |\leq C \frac{1}{\sqrt{\mu_M^\alpha}} \sum_{\beta=A,B} \int_{\mathbb{R}^3 \times \mathbb{S}^2}   \frac{w(v)}{w(v_*)} |v-v_*|^{\gamma} \Big[\left \langle v' \right\rangle^2 \mu^\alpha(v')\sqrt{\mu_M^\beta(v_*')}\\
            & \hspace{7cm}  +\left \langle v_*' \right\rangle^2\mu^{\beta}(v_*') \sqrt{\mu_{M}^{\alpha}(v')}\Big] d\sigma dv_*|\mathbf{g} |_{L^\infty_v}.
        \end{split}
      \end{equation*}
      From \eqref{Qwanmambum} with $\tilde{q}-\frac{1}{2}=2c_s>0$ and sufficiently small $\kappa_0$, it follows that
      \begin{eqnarray}
        \begin{split}
            | (D_x  K_{M,2,w}^{\alpha,c}) \mathbf{g}(v) |\leq C  \sum_{\beta=A,B} &\int_{\mathbb{R}^3 \times \mathbb{S}^2}  \frac{w(v)}{w(v_*)} |v-v_*|^{\gamma} \Big[\sqrt{\mu_M^{\beta}(v_*)} [\mu_M^{\alpha}(v') ]^{c_s}  \\
            &\hspace{1.1cm}  +  \sqrt{\mu_M^{\beta}(v_*)}[\mu_M^{\beta}(v_*') ]^{c_s}\Big]d\sigma dv_*|\mathbf{g} |_{L^\infty_v}.
        \end{split}
      \end{eqnarray}
    By utilizing Remark \ref{kernalcompwithsinglek}, it can be derived that
        \begin{eqnarray*}
            \begin{split}
                &| (D_x  K_{M,2,w}^{\alpha,c}) \mathbf{g}(v) |\\
                &\hspace{0.2cm}\leq C \sum_{\beta=A,B} \int_{\mathbb{R}^3 }       \frac{(1+|v|+|v_*|)^{\gamma-1}}{|v-v_*|} \,\, e^{-2c\{|v-v_*|^2+\frac{||v|^2-|v_*|^2|^2}{|v-v_*|^2}\}}  \frac{w(v)}{w(v_*)}  dv_*|\mathbf{g} |_{L^\infty_v}\\
                &\hspace{0.2cm}\leq C \left \langle v \right\rangle^{\gamma-2}  | \mathbf{g} |_{L^\infty_v}.
            \end{split}
        \end{eqnarray*}
        Since $|\frac{d v'}{dv}|+|\frac{d v_*'}{dv}|\leq C$, from the chain rule of differentiation, and the conservation law,  we follow the established reasoning‌ through a ‌parallel process‌ while ‌invoking Remark \ref{kernalcompwithsinglek}, the result emerges that
        \begin{eqnarray*}
            \begin{split}
                | (D_v  K_{M,2,w}^{\alpha,c}) \mathbf{g}(v) |
                \lesssim &\left \langle v \right\rangle\sum_{\beta=A,B} \int_{\mathbb{R}^3 }        \frac{w(v)(1+|v|+|v_*|)^{\gamma-1}}{w(v_*)|v-v_*|} \\
                &\hspace{3cm} \times e^{-2c\{|v-v_*|^2+\frac{||v|^2-|v_*|^2|^2}{|v-v_*|^2}\}}   dv_*|\left \langle v \right\rangle\mathbf{g} |_{L^\infty_v}\\
                \leq & C \left \langle v \right\rangle^{\gamma-1}  | \left \langle v \right\rangle\mathbf{g} |_{L^\infty_v}.
            \end{split}
        \end{eqnarray*}
        \end{proof}

    \begin{remark}
        The constant $C > 0$ and the small constant $c > 0$ depend on the particle masses $m^A$
        and $m^B$
        as well as the macroscopic quantities of the fluid $(n^{A}, n^{B}, \mathbf{u}, \theta)$, which are not only bounded but also possess positive lower bounds.
    \end{remark}
    \begin{remark}
        For the weight function $\omega_{\gamma} = e^{\tilde{\kappa} \left \langle v \right\rangle^{\frac{3 - \gamma}{2}}}, \tilde{\kappa} = \kappa_0 [1 + (1 + t)^{-\frac{2}{2k - 1}}]$. When $\frac{3 - \gamma}{2} = 2$, the parameter $\kappa_0$
        should be sufficiently small to ensure the inequalities \eqref{K1111XVE} and \eqref{K222XVE}.
    \end{remark}
    \noindent   The following lemma is used in the $W_{x,v}^{1,\infty}$ estimates for nonlinear terms.
    \begin{lemma}\label{DDFXXEST} We abbreviate $w_{\gamma}(v)$ in \eqref{WWWEIGHTFUC}  to $w$. $ For \,\,any~\alpha, \beta \in \{A,B\} $ and $-3<\gamma\leq1$, it holds that
        \begin{gather}\label{Lwuq000}
            \left|\frac{w}{\sqrt{\mu^{\alpha}_M}}\,Q^{\alpha \beta }(\sqrt{\mu^{\alpha}_M}\frac{g^\alpha}{w},\sqrt{\mu^{\beta}_M} \frac{g^\beta}{w})(v)\right|\lesssim \left \langle v \right\rangle^{\gamma} |  g^{\alpha} |_{L^\infty_v} \, |  g^{\beta} |_{L^\infty_v}.
        \end{gather}
       Furthermore, the following estimate holds
        \begin{gather}\label{Lwuqx100}
            \begin{split}
                \left| D_x\Big[\frac{w}{\sqrt{\mu^{\alpha}_M}}\,Q^{\alpha \beta }(\sqrt{\mu^{\alpha}_M}\frac{g^\alpha}{w},\sqrt{\mu^{\beta}_M} \frac{g^\beta}{w})\Big](v)\right|
                \lesssim \left \langle v \right\rangle^{\gamma} \Big(|  D_x g^{\alpha} |_{L^\infty_v} \, |  g^{\beta} |_{L^\infty_v}+ |  D_x g^{\beta} |_{L^\infty_v} \, |  g^{\alpha} |_{L^\infty_v}\Big),
            \end{split}
        \end{gather}
        and
        \begin{gather}\label{Lwuqv010}
            \begin{split}
                \left|D_v\Big[\frac{w}{\sqrt{\mu^{\alpha}_M}}\,Q^{\alpha \beta }(\sqrt{\mu^{\alpha}_M}\frac{g^\alpha}{w},\sqrt{\mu^{\beta}_M} \frac{g^\beta}{w})\Big](v)\right|
                \lesssim & \left \langle v \right\rangle^{\gamma} \Big(|D_v g^{\alpha} |_{L^\infty_v} \, |g^{\beta} |_{L^\infty_v}+ |D_v g^{\beta} |_{L^\infty_v} \, |  g^{\alpha} |_{L^\infty_v}\Big)\\
                &+\left \langle v \right\rangle^{\gamma} | \left \langle v \right\rangle g^{\alpha} |_{L^\infty_v} \, | \left \langle v \right\rangle g^{\beta} |_{L^\infty_v}.
            \end{split}
        \end{gather}
    \end{lemma}
    \begin{proof}
        Without loss of generality, we take the  weight function $w_{\gamma}(v)=\omega_{\gamma}$ to give the proof, since  it increases the fastest with respect to the velocity variable $v$.  By the conservation of energy $m^{\alpha}|v'|^{2}+m^{\beta}|v_*'|^2 =m^{\alpha}|v|^2+m^{\beta}|v_*|^2$  for $\alpha , \beta \in \{A,B\}$, we obtain
        \begin{gather*}
            w(v)\leq Cw(v_*')w(v') \,,  \, \sqrt{\mu_M^{\alpha}(v')}\sqrt{\mu_M^{\beta} (v_*')} = \sqrt{\mu_M^{\alpha}(v)}\sqrt{\mu_M^{\beta}(v_*)}.
        \end{gather*}
        For the gain term,  one has
        \begin{gather*}
            \begin{split}
                & \left|\frac{w(v)}{\sqrt{\mu_M^{\alpha} (v)}}\iint_{\mathbb{R}^3\times \mathbb{S}^2}|v-v_*|^{\gamma}b^{\alpha \beta}(\theta)\,\sqrt{\mu_{M}^{\alpha}(v')}\frac{g^\alpha (v')}{w(v')}\sqrt{\mu_{M}^{\beta}(v_*')}\frac{g^\beta(v_*')}{w(v_*')}d\sigma dv_* \right|\\
                &\hspace{0.3cm}  \leq C \int_{\mathbb{R}^3}|v-v_*|^{\gamma}\sqrt{\mu_{M}^{\beta}(v_*)} \,| w(v')w(v_*')\frac{g^{\alpha}(v')}{w(v')}\frac{g^\beta(v_*') }{w(v_*')}|\, dv_* \\
                &\hspace{0.3cm} \leq C |  g^{\alpha} |_{L^\infty_v} | w g^\beta |_{L^\infty_v}
                \int_{\mathbb{R}^3}|v-v_*|^{\gamma}\sqrt{\mu_{M}^{\beta}(v_*)}dv_*\\
                 &\hspace{0.3cm}\lesssim  (1+|v|)^{\gamma} | w g_1 |_{L^\infty_v}\,| w g_2 |_{L^\infty_v}.
            \end{split}
        \end{gather*}
        Regarding the loss term, the following holds
        \begin{gather*}
            \begin{split}
                & \left| w(v) g^\alpha(v) \iint_{\mathbb{R}^3\times \mathbb{S}^2}|v-v_*|^{\gamma}b^{\alpha \beta}(\theta)\sqrt{\mu_{M }^{\beta}(v_*)}g^\beta(v_*)d\sigma dv_* \right|\\
                &\leq  |w(v) g^\alpha (v)|\,|g^\beta|_{L^{\infty}_v}  \left|  \int_{\mathbb{R}^3} |v-v_*|^{\gamma}\sqrt{\mu_{M}^{\beta}(v_*)} dv_* \right|
                \lesssim   (1+|v|)^{\gamma} \,| w g^\alpha |_{L^\infty_v} \, | w g^\beta |_{L^\infty_v},
            \end{split}
        \end{gather*}
        which yields
        \begin{gather*}
            \left|\frac{w}{\sqrt{\mu^{\alpha}_M}}\,Q^{\alpha \beta }(\sqrt{\mu^{\alpha}_M}\frac{g^\alpha}{w},\sqrt{\mu^{\beta}_M} \frac{g^\beta}{w})(v)\right|\lesssim \left \langle v \right\rangle^{\gamma} |  g^{\alpha} |_{L^\infty_v} \, |  g^{\beta} |_{L^\infty_v}.
        \end{gather*}
        Since $w$ and $\mu_{M}^{\alpha}$ are the functions of $t,v$, through a direct calculation, we obtain
        \begin{gather*}
            \begin{split}
                D_x\Big[\frac{w}{\sqrt{\mu^{\alpha}_M}}\,Q^{\alpha \beta }(\sqrt{\mu^{\alpha}_M}\frac{g^\alpha}{w},\sqrt{\mu^{\beta}_M} \frac{g^\beta}{w})\Big]
                \lesssim \left \langle v \right\rangle^{\gamma} (|  D_x g^{\alpha} |_{L^\infty_v} \, |  g^{\beta} |_{L^\infty_v}+ |  D_x g^{\alpha} |_{L^\infty_v} \, |  g^{\beta} |_{L^\infty_v}).
            \end{split}
        \end{gather*}
        Moreover, the following is derived:
        \begin{align*}
            &\left|D_v\Big[\frac{w}{\sqrt{\mu^{\alpha}_M}}\,Q^{\alpha \beta }(\sqrt{\mu^{\alpha}_M}\frac{g^\alpha}{w},\sqrt{\mu^{\beta}_M} \frac{g^\beta}{w})\Big]\right| \\
            & \hspace{0.3cm}\lesssim \left \langle v \right\rangle^{\gamma} (|D_v g^{\alpha} |_{L^\infty_v} \, |g^{\beta} |_{L^\infty_v}+ |D_v g^{\alpha} |_{L^\infty_v} \, |  g^{\beta} |_{L^\infty_v})+\left \langle v \right\rangle^{\gamma} | \left \langle v \right\rangle g^{\alpha} |_{L^\infty_v} \, | \left \langle v \right\rangle g^{\beta} |_{L^\infty_v}.
        \end{align*}
        In conclusion, \eqref{Lwuqx100},\eqref{Lwuqv010} hold for all $\alpha,\beta \in \{A, B\}$. This completes the proof of Lemma \ref{DDFXXEST}.
    \end{proof}

    For $\gamma=1$, we derive the equations for $\mathbf{h}=(h^A,h^B)^{T}$ from \eqref{reeqmain}. That is,
    \begin{gather}\label{HDSTLINFEQMAIN}
        \begin{split}
            \partial_t  h^\alpha & + v\cdot\nabla_{x} h^\alpha + \frac{e^{\alpha}}{m^{\alpha}} \nabla_{x}\phi^{\varepsilon}\cdot\nabla_{v}  h^\alpha + \frac{\upsilon^\alpha}{\varepsilon}  h^\alpha= \sum_{j=1}^{6} \mathcal{H}^\alpha_j+ \frac{1}{\varepsilon}(K_{M,\left \langle v \right\rangle^{l}}^{\alpha,\chi}+K_{M,\left \langle v \right\rangle^{l}}^{\alpha,c})\mathbf{h},
        \end{split}
    \end{gather}
    where
        \begin{align}
            &\mathcal{H}^\alpha_1= \varepsilon^{k-1}  \frac{ \left \langle v \right\rangle^{l}}{\sqrt{\mu^{\alpha}_M}}\sum_{\beta=A,B} Q^{ \alpha\beta }(\frac{\sqrt{\mu^{\alpha}_M}}{\left \langle v \right\rangle^{l}}h^\alpha, \frac{\sqrt{\mu^{\beta}_M}}{\left \langle v \right\rangle^{l}} h^\beta),\notag\\
            &\mathcal{H}_2^\alpha=  \frac{\left \langle v \right\rangle^{l}}{\sqrt{\mu^{\alpha}_M}} \sum_{\beta=A,B}\sum_{i=1}^{2k-1} \varepsilon^{i-1}\Big[ Q^{ \alpha \beta}(F_i^{\alpha},\frac{\sqrt{\mu^{\beta}_M}}{\left \langle v \right\rangle^{l}}h^\beta)+Q^{\alpha \beta}(\frac{\sqrt{\mu^{\alpha}_M}}{\left \langle v \right\rangle^{l}}h^\alpha, F_i^{\beta})\Big], \notag\\
            &\mathcal{H}^\alpha_3 = -\frac{e^{\alpha}}{m^{\alpha}}\nabla_{x}\phi^{\varepsilon} \cdot  \frac{\left \langle v \right\rangle^{l}}{\sqrt{\mu^{\alpha}_M}} \nabla_{v}(\frac{\sqrt{\mu^{\alpha}_M}}{\left \langle v \right\rangle^{l}})h^\alpha,\label{HD00HI6}\\
            &\mathcal{H}^\alpha_4 =  -\frac{e^{\alpha}}{m^{\alpha}}\nabla_{x}\phi_R \cdot  \frac{\left \langle v \right\rangle^{l}}{\sqrt{\mu^{\alpha}_M}} \nabla_{v}(\mu^{\alpha}+\sum_{i=1}^{2k-1}\varepsilon^iF_i^{\alpha}),\notag\\
            &\mathcal{H}^\alpha_5 = - \frac{\left \langle v \right\rangle^{l}}{\sqrt{\mu^{\alpha}_M}} \Big[\varepsilon^{k-1}( \partial_t +v\cdot \nabla_{x} ) F^\alpha_{2k-1}+\mathop{\sum}_{i+j \geq 2k-1\atop 0 \leq i,j \leq 2k-1}\varepsilon^{i+j-k}\nabla_x\phi_i\cdot\nabla_vF_j^\alpha\Big],\notag\\
            &\mathcal{H}^\alpha_6 =  \frac{\left \langle v \right\rangle^{l}}{\sqrt{\mu^{\alpha}_M}}  \Big\{\sum_{\beta=A,B} \mathop{\sum}_{i+j \geq 2k\atop 1 \leq i,j \leq 2k-1} \varepsilon^{i+j-k-1}\Big[Q^{\alpha \beta }(F_i^{\alpha},F_j^{\beta})+Q^{\alpha\beta }(F_j^{\alpha},F_i^{\beta})\Big] \Big\}.\notag
        \end{align}
    The characteristic curve of \eqref{HDSTLINFEQMAIN} passing through $(t,x,v)$ is defined as the solution of the equations
    \begin{eqnarray}\label{chracterlinedef}
        \begin{split}
            &\frac{d}{d\tau}X^{\alpha}(\tau;t,x,v)=V^{\alpha}(\tau;t,x,v), \quad X^{\alpha}(t;t,x,v)=x,\\
            &\frac{d}{d\tau}V^{\alpha}(\tau;t,x,v)=\frac{e^{\alpha}}{m^{\alpha}}\phi^{\varepsilon}(\tau,X^{\alpha}(\tau;t,x,v)), \quad V^{\alpha}(t;t,x,v)=v.
        \end{split}
    \end{eqnarray}
    One can verify that $\frac{d}{d\tau}h^\alpha(\tau,X^{\alpha}(\tau;t,x,v),V^{\alpha}(\tau;t,x,v))=\partial_t  h^\alpha + v\cdot\nabla_{x} h^\alpha + \frac{e^{\alpha}}{m^{\alpha}} \nabla_{x}\phi^{\varepsilon}\cdot\nabla_{v}  h^\alpha$.

    \begin{lemma}\label{lemofcharacter} Recall the weighted remainder $\mathbf{R}_{\gamma}$ defined in \eqref{WWDWEIGHTFUC}. For $0\leq T_1 \leq \frac{1}{\varepsilon}$, assuming that
        \begin{eqnarray}\label{SUPPOSELINFGXV}
            \begin{split}
                &\sup_{\tau\in [0,T_1]}\varepsilon^k\Vert \mathbf{R}_{\gamma}(t)\Vert_{W^{1,\infty}_{x} L^{\infty}_{v}}\leq\sqrt{\varepsilon}.
            \end{split}
        \end{eqnarray}
        Consider the vectors $X^{\alpha}(\tau;t,x,v)=(X_1^{\alpha},X_2^{\alpha},X_3^{\alpha})$, $V^{\alpha}(\tau;t,x,v)=(V_1^{\alpha},V_2^{\alpha},V_3^{\alpha})$, with spatial coordinates $x=(x_1,x_2,x_3)$ and velocity coordinates $v=(v_1,v_2,v_3)$. For each $1\leq i,j\leq3$, the assumption \eqref{SUPPOSELINFGXV} implies
        \begin{equation}\label{ASSPDREVYVxXx}
            \begin{split}
                &\sup_{\tau\in [0,T_1]}\Big(\Vert\frac{\partial X_i^{\alpha}}{\partial v_j}\Vert_{L_{x,v}^{\infty}}+\Vert\frac{\partial X_i^{\alpha}}{\partial x_j}\Vert_{L_{x,v}^{\infty}}+\Vert\frac{\partial V_i^{\alpha}}{\partial v_j}\Vert_{L_{x,v}^{\infty}}+\Vert\frac{\partial V_i^{\alpha}}{\partial x_j}\Vert_{L_{x,v}^{\infty}}\Big)\\
                &   \leq  C (1+\varepsilon\mathcal{I}_1) +C \varepsilon^k\Vert R^\alpha_{\gamma}(t)\Vert_{W^{1,\infty}_{x} L^{\infty}_{v}} \leq C.
            \end{split}
        \end{equation}
        Moreover, there exists a sufficiently small $T_0\in [0,T_1]$  such that for all $0\leq\tau \leq t \leq T_0$, $N \geq 1$ and $x_0 \in \mathbb{R}^3$, it holds that
        \begin{eqnarray}\label{shortXVDNN}
            \begin{split}
                &\frac{|t-\tau|^3}{2}\leq\Big|\rm{det}\Big(\frac{\partial X(\tau)}{\partial v}\Big)\Big|\leq2|t-\tau|^3, \qquad |\partial_{v}X(\tau)|\leq2|t-\tau|,\quad \frac{1}{2}\leq\Big|\rm{det}\Big(\frac{\partial V(\tau)}{\partial v}\Big)\Big|\leq2,\\
                &\frac{1}{2}\leq\Big|\rm{det}\Big(\frac{\partial X(\tau)}{\partial x}\Big)\Big|\leq2, \qquad\sup_{|v|\leq N}\Big[\int_{ |x-x_0| \leq CN }(|\partial_{x,v}X(\tau)|^2+|\partial_{v,v}X(\tau)|^2) dx \Big]^{\frac{1}{2}} \leq C_N.
            \end{split}
        \end{eqnarray}
       As the ratio $\frac{e^{\alpha}}{m^{\alpha}}$ remains bounded and the time interval $T_0 \in [0, T_1]$ is sufficiently small, the argument for Lemma \ref{lemofcharacter} follows directly from the proof of Lemma 4.1 in \cite{[ininp]Guo2010CMP}, hence the details are omitted.
    \end{lemma}

    \noindent $\mathbf{The \,\,W_{x,v}^{1,\infty}\,\, estimates \,\,for\,\,the\,\, remainder\,\, term \,\,is \,\,in \,\,the\,\, following\,\, proposition}.$

    \begin{proposition}\label{Maintheo45HSLINF}
        Let $(n^{A},n^{B},\mathbf{u},\theta)  \in C([0,\tau] ; H^{s_0}(\mathbb{{R}}^3)) \cap C^{1}([0,\tau] ; H^{s_0-1}(\mathbb{{R}}^3))$  be the  solution of \eqref{EQF0EPSION}, and $\eqref{IID24}$.  $\mathbf{R}_{\gamma}$ is the weighted remainder term  defined in \eqref{Threedefweight1}.
        Then there exist $\varepsilon_0 >0$  such that for all $0<\varepsilon \leq \varepsilon_0$, it holds that
        \begin{gather}\label{sdpflinf0}
            \sup_{s\in [0,T_L]}(\varepsilon^{\frac{3}{2}}\Vert \mathbf{R}_{\gamma}(s)\Vert_{{L^{\infty}_{x,v} }})\lesssim \varepsilon^{\frac{3}{2}}\Vert \mathbf{R}^{\rm in}_{\gamma}\Vert_{{L^{\infty}_{x,v} }}+\sup_{s\in [0,T_L]}\Vert \mathbf{f}_{R}(s)\Vert_{{L^{2}_{x,v} }}+\varepsilon^{k+\frac{3}{2}},
        \end{gather}
        and
        \begin{gather}\label{sdpflinfd}
            \begin{split}
                &\sup_{s\in [0,T_L]}(\varepsilon^{\frac{3}{2}}\Vert \left \langle v \right\rangle^{2-\gamma}\mathbf{R}_{\gamma}(s)\Vert_{{L^{\infty}_{x,v} }}+\varepsilon^5\Vert\nabla_{x,v} \mathbf{R}_{\gamma}(s)\Vert_{{L^{\infty}_{x,v} }})\\
                \lesssim & \quad  \varepsilon^{\frac{3}{2}}\Vert \left \langle v \right\rangle^{2-\gamma}\mathbf{R}^{\rm in}_{\gamma}\Vert_{{L^{\infty}_{x,v} }}+\sup_{s\in [0,T_L]}\Vert \mathbf{f}_R(s)\Vert_{{L^{2}_{x,v} }}+\varepsilon^5\Vert\nabla_{x,v} \mathbf{R}^{\rm in}_{\gamma}\Vert_{{L^{\infty}_{x,v} }}+\varepsilon^{k+\frac{3}{2}}.
            \end{split}
        \end{gather}
    \end{proposition}
    \begin{proof}[$\mathbf{Proof \,\,for\,\, hard\,\, sphere\,\, case} (\gamma=1)$]
        By Duhamel's
        Principle, the equation \eqref{HDSTLINFEQMAIN} gives the expression of $h^\alpha$ along the characteristics \eqref{chracterlinedef} as
        \begin{gather}\label{linfHADTOT}
            \begin{split}
                h^\alpha(t,x,v) = & \exp \left(-\frac{1}{\varepsilon}\int_{0}^{t}\upsilon^{\alpha}(\tau,  X^{\alpha}(\tau),V^{\alpha}(\tau)) d\tau\right)h^\alpha(0,X^{\alpha}(0;t,x,v),V^{\alpha}(0;t,x,v)) \\
                & + \int_{0}^{t}\exp \left(-\frac{1}{\varepsilon}\int_{s}^{t}\upsilon^{\alpha}(\tau,  X^{\alpha}(\tau),V^{\alpha}(\tau))   d\tau \right) \sum_{i=1}^{6}\mathcal{H}^\alpha_j(s, X^{\alpha}(s),V^{\alpha}(s))ds\\
                &+ \frac{1}{\varepsilon} \int_{0}^{t}\exp \left(-\frac{1}{\varepsilon}\int_{s}^{t}\upsilon^{\alpha}(\tau,  X^{\alpha}(\tau),V^{\alpha}(\tau))  d\tau \right)  K_{M,\left \langle v \right\rangle^{l}}^{\alpha,\chi} \mathbf{h} (s, X^{\alpha}(s),V^{\alpha}(s))ds\\
                &+ \frac{1}{\varepsilon} \int_{0}^{t}\exp \left(-\frac{1}{\varepsilon}\int_{s}^{t}\upsilon^{\alpha}(\tau,  X^{\alpha}(\tau),V^{\alpha}(\tau))   d\tau \right)  K_{M,\left \langle v \right\rangle^{l}}^{\alpha,c} \mathbf{h} (s, X^{\alpha}(s),V^{\alpha}(s))ds.
            \end{split}
        \end{gather}
       Introduce the notation:
        \begin{gather*}\label{TZXSCTUREF0}
            \mathcal{F}^\alpha_0(t,s) =: \exp \left(-\frac{1}{\varepsilon}\int_{s}^{t}\upsilon^{\alpha}(\tau,  X^{\alpha}(\tau),V^{\alpha}(\tau)) d\tau\right).
        \end{gather*}
        Since $\frac{1}{C}\left \langle v \right\rangle^{\gamma}  \leq\upsilon^{\alpha}(t,x,v)\leq C \left \langle v \right\rangle^{\gamma} $, there exists a small constant $\nu_o$ such that
        \begin{equation}\label{nudetjH}
            \inf_{t,x,v} \{\upsilon^{\alpha}(t,x,v)\}\geq 2 \nu_o >0.
        \end{equation}
       Moreover, when $|v|\leq 1$, for sufficiently small $T_0>0$, the relation \eqref{ASSPDREVYVxXx} implies
        \begin{equation*}\label{les1estTZX}
            \sup_{s\in [0,T_0]}|V^{\alpha}(s)|\leq C.
        \end{equation*}
        Then, for $0\leq s\leq t \leq T_0$, the following estimate holds
        \begin{gather*}
            \left|\int_{0}^{t} \mathcal{F}_0^{\alpha}(t,s) \left \langle V^{\alpha}(s) \right\rangle e^{-\frac{\nu_o s}{\varepsilon}} ds \right|\lesssim e^{-\frac{\nu_o t}{\varepsilon}} \int_{0}^{t}\exp \left(-\frac{ \left \langle v \right\rangle (t-s)}{ C\varepsilon}\right)\left \langle v \right\rangle ds \lesssim \varepsilon e^{-\frac{\nu_o t}{\varepsilon}}.
        \end{gather*}

        When $|v|>1$ and $T_0>0$ is sufficiently small, the characteristic velocity satisfies the following uniform bounds
        \begin{equation*}\label{ges1estTZX}
            \frac{|v|}{2}   \leq    |V^{\alpha}(s)|\leq 2 |v|.
        \end{equation*}
       These velocity bounds lead to the refined estimate
        \begin{gather}\label{YDZDDJG}
            \left|\int_{0}^{t} \mathcal{F}_0^{\alpha}(t,s) \left \langle V^{\alpha}(s) \right\rangle e^{-\frac{\nu_o s}{\varepsilon}} ds \right|\lesssim e^{-\frac{\nu_o t}{\varepsilon}} \int_{0}^{t}\exp \left(-\frac{ \left \langle v \right\rangle (t-s)}{ C\varepsilon}\right)\left \langle v \right\rangle ds \lesssim \varepsilon e^{-\frac{\nu_o t}{\varepsilon}}.
        \end{gather}
        \subsection{$L^{\infty}$estimates}
        In the following, we give the estimates in right hand side of \eqref{linfHADTOT} line to line.
        First, the initial data contribution yields the exponential decay estimate
        \begin{gather*}
            \left|\mathcal{F}_0^{\alpha}(t,0) h^\alpha(0,\,X^{\alpha}(0),\,V^{\alpha}(0))\right| \lesssim \varepsilon  e^{-\frac{\nu_o t}{\varepsilon}} \|\mathbf{h}^{in}\|_{L^\infty_{x,v}}.
        \end{gather*}
        Then, from Lemma \ref{DDFXXEST} ($\gamma=1$), we have
        \begin{gather}\label{HHHD111}
            \begin{split}
                |\mathcal{H}_1^\alpha| = & \varepsilon^{k-1} \left| \sum_{\beta=A,B}\frac{\left \langle v \right\rangle^{l}}{\sqrt{\mu^{\alpha}_M}} Q^{\alpha\beta }(\sqrt{\mu^{\alpha}_M}\frac{h^\alpha}{\left \langle v \right\rangle^{l}}, \sqrt{\mu^{\beta}_M} \frac{h^\beta}{\left \langle v \right\rangle^{l}})\right|
                \lesssim  \varepsilon^{k-1} \left \langle v \right\rangle \, \Vert \mathbf{h} \Vert_{L^{\infty}_{x,v}}^2,
            \end{split}
        \end{gather}
        which together with \eqref{YDZDDJG} leads to
        \begin{gather*}
            \begin{split}
                \left|\int_{0}^{t}\mathcal{F}_0^{\alpha}(t,s)  \mathcal{H}^\alpha_1(s, X^{\alpha}(s),V^{\alpha}(s))ds \right|
                \lesssim & \varepsilon^{k-1}  \int_{0}^{t}\mathcal{F}_0^{\alpha}(t,s)  e^{-\frac{\nu_o s}{\varepsilon}}\left \langle v \right\rangle  ds [\sup_{0\leq s \leq t}(e^{\frac{\nu_o s}{2\varepsilon}}\Vert \mathbf{h} (s)\Vert_{L^{\infty}_{x,v}})]^2
                \\
                \lesssim & \varepsilon^k  e^{-\frac{\nu_o t}{\varepsilon}} \sup_{0\leq s \leq t}(e^{\frac{\nu_o s}{2\varepsilon}}\Vert \mathbf{h} (s)\Vert_{L^{\infty}_{x,v}})^2.
            \end{split}
        \end{gather*}
       Next, the construction of $\mathbf{F}_i$ in Lemma \ref{FI2KM1EST} gives
        \begin{gather*}
            \begin{split}
                |\mathcal{H}_2^{\alpha}|
                \lesssim & \, \left \langle v \right\rangle \sum_{\beta=A,B} \sum_{i=1}^{2k-1}\varepsilon^{i-1} \Big(\Big\Vert\frac{\left \langle v \right\rangle^{l}}{\sqrt{\mu^{\beta}_M}} F_i^{\beta} \Big\Vert_{L^\infty_{x,v}}\, \Big\Vert h^\alpha \Big\Vert_{L^\infty_{x,v}}
                +\Big\Vert h^\beta \Big\Vert_{L^\infty_{x,v}} \, \Big\Vert\frac{\left \langle v \right\rangle^{l}}{\sqrt{\mu^{\alpha}_M}} F_i^{\alpha} \Big\Vert_{L^\infty_{x,v}}\Big)
                \\
                \lesssim &  \, \left \langle v \right\rangle \mathcal{I}_1 \,\Vert\mathbf{h} \Vert_{L^\infty_{x,v}}.
            \end{split}
        \end{gather*}
        For this reason, we see
        \begin{gather*}\label{4.19}
            \begin{split}
                \left|\int_{0}^{t}\mathcal{F}_0^{\alpha}(t,s) \mathcal{H}^\alpha_2(s,X^{\alpha}(s),V^{\alpha}(s))ds \right|
                \lesssim & \, \mathcal{I}_1 \int_{0}^{t}\mathcal{F}_0^{\alpha}(t,s) \left \langle v \right\rangle e^{-\frac{\nu_o s}{2\varepsilon}} (e^{\frac{\nu_o s}{2\varepsilon}}\Vert \mathbf{h} (s)\Vert_{L^{\infty}_{x,v}})ds\\
                \lesssim &  \mathcal{I}_1 \int_{0}^{t}\mathcal{F}_0^{\alpha}(t,s) \left \langle v \right\rangle e^{-\frac{\nu_o s}{2\varepsilon}} ds \sup_{0\leq s \leq t}(e^{\frac{\nu_o s}{2\varepsilon}}\Vert \mathbf{h} (s)\Vert_{L^{\infty}_{x,v}})
                \\
                \lesssim & \varepsilon \mathcal{I}_1 e^{-\frac{\nu_o t}{2\varepsilon}} \sup_{0\leq s \leq t}(e^{\frac{\nu_o s}{2\varepsilon}}\Vert \mathbf{h} (s)\Vert_{L^{\infty}_{x,v}}).
            \end{split}
        \end{gather*}
      Moreover, the estimates from $\eqref{ORDERZONG}_1$ and $\eqref{ORDERZONG}_3$ combine to produce
        \begin{eqnarray}\label{phiimpus}
            \begin{split}
                \Delta \phi_0 =&\sum_{\alpha=A,B} \int_{\mathbb{R}^3} e^{\alpha} F^{\alpha}_0 dv-\bar{n}_e, \quad
                \Delta \phi_i =\sum_{\alpha=A,B}  \int_{\mathbb{R}^3} e^{\alpha} F^{\alpha}_i dv, \qquad i\geq 1,\\
                \Delta \phi_R =&\sum_{\alpha=A,B} \int_{\mathbb{R}^3} e^{\alpha}\frac{\sqrt{\mu^{\alpha}_M}}{\left \langle v \right\rangle^{l}} h^{\alpha} \,dv, \\
                \Delta \phi^\varepsilon=&\sum_{\alpha=A,B}\Big[\sum_{i=0}^{2k-1}\varepsilon^{i}\int_{\mathbb{R}^3} e^{\alpha}F^{\alpha}_i dv +\varepsilon^k\int_{\mathbb{R}^3} e^{\alpha} \frac{\sqrt{\mu^{\alpha}_M}}{\left \langle v \right\rangle^{l}}h^{\alpha} \, dv\Big]-\bar{n}_e.
            \end{split}
        \end{eqnarray}
        Then for any $0<\alpha_{0}<1$, it holds that
        \begin{equation}\label{psideddddg}
            | \nabla_{x}\phi_R|_{C^{1,\alpha_{0}}_{x} } \lesssim \Vert\mathbf{h} \Vert_{W^{1,\infty}_{x} L^{\infty}_{v}}, \qquad | \nabla_{x}\phi^{\varepsilon}|_{C^{1,\alpha_{0}}_{x} }    \lesssim 1+\varepsilon \mathcal{I}_1+\varepsilon^k\Vert\mathbf{h} \Vert_{W^{1,\infty}_{x} L^{\infty}_{v}}\leq C.
        \end{equation}
       Together with
       $
       \frac{\langle v \rangle^{l}}{\sqrt{\mu^{\alpha}_M}} \nabla_v ( \frac{\sqrt{\mu^{\alpha}_M}}{\langle v \rangle^{l}} ) \lesssim \langle v \rangle
       $
       and
       $
       \frac{\langle v \rangle^{l}}{\sqrt{\mu^{\alpha}_M}} \nabla_v ( \mu^{\alpha} + \sum_{i=1}^{2k-1} \varepsilon^i F_i^{\alpha} ) \lesssim (1 + \varepsilon \mathcal{I}_1) \langle v \rangle,
       $
       we have
        \begin{align}
                |\mathcal{H}^\alpha_3| =& |\frac{e^{\alpha}}{m^{\alpha}}\nabla_{x}\phi^{\varepsilon} \cdot  \frac{\left \langle v \right\rangle^{l}}{\sqrt{\mu^{\alpha}_M}} \nabla_{v}(\frac{\sqrt{\mu^{\alpha}_M}}{\left \langle v \right\rangle^{l}})h^\alpha|
                \lesssim \left \langle v \right\rangle \Vert\mathbf{h} \Vert_{L^{\infty}_{x,v} },\label{h3h3h3ddq}\\
                |\mathcal{H}^\alpha_4| =& |\frac{e^{\alpha}}{m^{\alpha}}\nabla_{x}\phi_R \cdot  \frac{\left \langle v \right\rangle^{l}}{\sqrt{\mu^{\alpha}_M}} \nabla_{v}(\mu^{\alpha}+\sum_{i=1}^{2k-1}\varepsilon^iF_i^{\alpha})|
                \lesssim (1+\varepsilon\mathcal{I}_1)\left \langle v \right\rangle \Vert\mathbf{h} \Vert_{L^{\infty}_{x,v}}.\notag
        \end{align}
      Consequently, one obtains
        \begin{align*}
            &\left|\int_{0}^{t}\mathcal{F}_0^{\alpha}(t,s) (\mathcal{H}^\alpha_3+\mathcal{H}^\alpha_4)(s,X^{\alpha}(s),V^{\alpha}(s))ds \right|\\
                &\hspace{0.4cm}\lesssim (1+\varepsilon\mathcal{I}_1)\int_{0}^{t}\mathcal{F}_0^{\alpha}(t,s)\left \langle v \right\rangle e^{-\frac{\nu_os}{2\varepsilon}} (e^{\frac{\nu_os}{2\varepsilon}}\Vert \mathbf{h} (s)\Vert_{L^{\infty}_{x,v}})ds\\
                &\hspace{0.4cm}\lesssim \varepsilon(1+\varepsilon\mathcal{I}_1) e^{-\frac{\nu_ot}{2\varepsilon}} \sup_{0\leq s \leq t}(e^{\frac{\nu_os}{2\varepsilon}}\Vert \mathbf{h} (s)\Vert_{L^{\infty}_{x,v}}).
        \end{align*}
        Finally, the estimate of $\mathbf{F}_i$ in Lemma \ref{FI2KM1EST} further implies
        \begin{gather*}
            \begin{split}
                |\mathcal{H}^\alpha_5|
                \lesssim &  \left \langle v \right\rangle \Big[\varepsilon^{k-1} \Big\|\frac{\left \langle v \right\rangle^{l-1}( \partial_t +v\cdot \nabla_{x} ) F^\alpha_{2k-1} }{ \sqrt{\mu^{\alpha}_M}}\Big\|_{L^\infty_{x,v}}\\
                &\hspace{3.5cm}+  \mathop{\sum}_{i+j \geq 2k-1\atop 0 \leq i,j \leq 2k-1}\varepsilon^{i+j-k} | \nabla_{x}\phi_i|_{L^{\infty}_{x}} \Big\Vert \frac{\nabla_vF_j^\alpha}{\left \langle v \right\rangle} \Big\Vert_{L^{\infty}_{x,v} } \Big] \\
                \lesssim &  \varepsilon^{k-1} \left \langle v \right\rangle \mathcal{I}_2,
            \end{split}
        \end{gather*}
        and
        \begin{gather*}
            \begin{split}
                |\mathcal{H}^\alpha_6| = & \Big| \sum_{\beta=A,B} \mathop{\sum}_{i+j \geq 2k\atop 1 \leq i,j \leq 2k-1} \varepsilon^{i+j-k-1}\frac{\left \langle v \right\rangle^{l}}{\sqrt{\mu^{\alpha}_M}} \Big[Q^{\alpha \beta }(F_i^{\alpha},F_j^{\beta})+Q^{\alpha\beta }(F_j^{\alpha},F_i^{\beta})\Big] \Big|\\
                \lesssim &  \mathop{\sum}_{i+j \geq 2k-1\atop 0 \leq i,j \leq 2k-1}\varepsilon^{i+j-k} \left \langle v \right\rangle \Big\Vert\frac{\left \langle v \right\rangle^{l}}{\sqrt{\mu^{\alpha}_M}} \mathbf{F}_i\Big\|_{L^\infty_{x,v}}  \, \Big\|\frac{\left \langle v \right\rangle^{l}}{\sqrt{\mu^{\beta}_M}} \mathbf{F}_j\Big\Vert_{L^\infty_{x,v}}\\
                \lesssim & \varepsilon^{k-1} \left \langle v \right\rangle \mathcal{I}_2.
            \end{split}
        \end{gather*}
        Combining these estimates yields the integral bound
        \begin{gather}\label{HHHDDB}
            \begin{split}
                \left|\int_{0}^{t}\mathcal{F}_0^{\alpha}(t,s)(\mathcal{H}^\alpha_5+\mathcal{H}^\alpha_6)(s,X^{\alpha}(s),V^{\alpha}(s))ds \right|
                \lesssim \varepsilon^{k-1}\mathcal{I}_2 \int_{0}^{t}\mathcal{F}_0^{\alpha}(t,s)\left \langle v \right\rangle ds
                \lesssim \varepsilon^k\mathcal{I}_2.
            \end{split}
        \end{gather}
        By Lemma \ref{kchiestmainle}, we have
        \begin{gather*}\label{4.21}
            \begin{split}
                \frac{1}{\varepsilon} \left|\int_{0}^{t}\mathcal{F}_0^{\alpha}(t,s) K_{M,\left \langle v \right\rangle^{l}}^{\alpha,\chi}  \mathbf{h}(s,X^{\alpha}(s),V^{\alpha}(s))ds \right|
                \lesssim {\delta}^{3+\gamma}e^{-\frac{\nu_ot}{2\varepsilon}} \sup_{0\leq s \leq t}(e^{\frac{\nu_os}{2\varepsilon}}\Vert \mathbf{h} (s)\Vert_{L^{\infty}_{x,v}}).
            \end{split}
        \end{gather*}

        The analysis focuses primarily on the final term in \eqref{linfHADTOT}'s right-hand side. Lemma \ref{LeMK2ker4444} establishes that $K_{M,2,w}^{\alpha,c,(2)}$ demonstrates superior decay properties compared to $K_{M,2,w}^{\alpha,c,(1)}$. Consequently, the decay estimates require consideration only for $K_{M,2,w}^{\alpha,c,(1)}$, that is
        \begin{eqnarray}\label{HDKCLITER1}
            \begin{split}
                &K_{M,\left \langle v \right\rangle^{l}}^{\alpha,c} \mathbf{h}(s,X^{\alpha}(s),V^{\alpha}(s))\\
                =&\sum_{\beta=A,B}\int_{\mathbb{R}^3} [k_{M,1}^{\alpha \beta}+k_{M,2}^{\alpha \beta(1)}+k_{M,2}^{\alpha \beta(2)}] (s,X^{\alpha}(s),V^{\alpha}(s),v_*)\frac{\left \langle V^{\alpha}(s) \right\rangle^{l}}{\left \langle v_* \right\rangle^{l}}h^{\beta}(s,X^{\alpha}(s),v_*) d v_{*}.
            \end{split}
        \end{eqnarray}
      For brevity, we denote:
        \begin{equation}\label{kkd1212He}
            k_{M,w}^{\alpha\beta}:=[k_{M,1}^{\alpha \beta}+k_{M,2}^{\alpha \beta(1)}+k_{M,2}^{\alpha \beta(2)}]\frac{w(v)}{w(v_*)},
        \end{equation}
        where $w(v)=\left \langle v \right\rangle^{l}$ for $\gamma=1$. The ``double characteristic line'' is defined as
        \begin{equation}\label{doublelineqianT}
            \Big[X^{\alpha,\beta}_1(s_1),V^{\alpha,\beta}_1(s_1)\Big]=\Big[X^{\beta}(s_1;s,X^{\alpha}(s;t,x,v),v_*),V^{\beta}(s_1;s,X^{\alpha}(s;t,x,v),v_*)\Big],
        \end{equation}
        and
        \begin{gather*}
            \begin{split}
                \mathcal{F}^\beta_1 (s,s_1)  =\exp \left(-\frac{1}{\varepsilon}\int_{s_1}^s\upsilon^{\beta}(\tau, X^{\alpha,\beta}_1(\tau),V^{\alpha,\beta}_1(\tau))   d\tau \right).
            \end{split}
        \end{gather*}
         $h^{\beta}(s,X^{\alpha}(s),v_*)$ could be formulated as
        \begin{gather}\label{HDliterative1111}
            \begin{split}
                h^\beta(s,X^{\alpha}(s),v_*)
                = & \mathcal{F}^\beta_1(s, 0) h^\beta(0,X^{\alpha,\beta}_1(0),V^{\alpha,\beta}_1(0)) \\
                &+ \sum_{j=1}^6 \int_{0}^s \mathcal{F}_1^\beta(s,s_1) \mathcal{H}^\beta_j(s_1,X^{\alpha,\beta}_1(s_1),V^{\alpha,\beta}_1(s_1))ds_1\\
                & + \frac{1}{\varepsilon} \int_{0}^s \mathcal{F}_1^\beta(s,s_1)  K_{M,\left \langle v \right\rangle^{l}}^{\alpha,\chi}  \mathbf{h} (s_1,X^{\alpha,\beta}_1(s_1),V^{\alpha,\beta}_1(s_1))ds_1\\
                & + \frac{1}{\varepsilon} \int_{0}^s \mathcal{F}_1^\beta(s,s_1)  K_{M,\left \langle v \right\rangle^{l}}^{\alpha,c} \mathbf{h} (s_1,X^{\alpha,\beta}_1(s_1),V^{\alpha,\beta}_1(s_1))ds_1.
            \end{split}
        \end{gather}
      Then, substituting \eqref{HDliterative1111} into \eqref{HDKCLITER1} yields
        \begin{equation}\label{HD00Z222ZZ00}
            \frac{1}{\varepsilon} \left|\int_{0}^{t}\mathcal{F}_0^{\alpha}(t,s) K_{M,\left \langle v \right\rangle^{l}}^{\alpha,c}  \mathbf{h}(s,X^{\alpha}(s),V^{\alpha}(s))ds \right| \leq \sum_{i=0}^{2} \mathcal{J}_i(t) + \mathcal{O}^\alpha(t),
        \end{equation}
        where
        \begin{eqnarray*}
            \begin{split}
                &\mathcal{J}_0(t)=\frac{1}{\varepsilon} \int_{0}^{t} \mathcal{F}_0^\alpha(t,s) \sum_{\beta=A,B}\int_{\mathbb{R}^3}  k_{M,w}^{\alpha\beta}(s,X^{\alpha}(s),V^{\alpha}(s),v_*) \\  &\hspace{5cm}\times\mathcal{F}_1^\beta(s,0) h^\beta(0,X^{\alpha,\beta}_1(0),V^{\alpha,\beta}_1(0))  dv_*ds \\
                &\hspace{0.8cm}\lesssim \frac{t}{\varepsilon}e^{-\frac{\nu_ot}{\varepsilon}} \|\mathbf{h}^{in}\|_{L^{\infty}_{x,v}},
            \end{split}
        \end{eqnarray*}
        \begin{eqnarray*}
            \begin{split}
                &\mathcal{J}_1(t)=\frac{1}{\varepsilon} \int_{0}^{t} \mathcal{F}_0^\alpha(t,s) \sum_{\beta=A,B}\int_{\mathbb{R}^3} k_{M,w}^{\alpha\beta}(s,X^{\alpha}(s), V^{\alpha}(s), v_*)\\
                &\hspace{4.5cm}\times\sum_{j=1}^6 \int_{0}^s \mathcal{F}_1^\beta(s,s_1) \mathcal{H}^\beta_j(s_1,X_1^{\alpha,\beta}(s_1),V_1^{\alpha,\beta}(s_1))ds_1  dv_*ds \\
                & \hspace{0.8cm}\lesssim  \varepsilon^k e^{-\frac{\nu_ot}{\varepsilon}}(\sup_{0\leq s \leq t}e^{\frac{\nu_os}{2\varepsilon}}\Vert \mathbf{h}(s) \Vert_{L^{\infty}_{x,v}})^2+({1+\mathcal{I}_1}) \varepsilon e^{-\frac{\nu_ot}{2\varepsilon}}\sup_{0\leq s \leq t}(e^{\frac{\nu_os}{2\varepsilon}}\Vert  \mathbf{h}(s) \Vert_{L^{\infty}_{x,v}})+\mathcal{I}_2 \varepsilon^k,
            \end{split}
        \end{eqnarray*}
        \begin{eqnarray*}
            \begin{split}
                &\mathcal{J}_2(t)=\frac{1}{\varepsilon^2} \int_{0}^{t} \mathcal{F}_0^\alpha(t,s) \sum_{\beta=A,B}\int_{\mathbb{R}^3} k_{M,w}^{\alpha\beta}(s,X^{\alpha}(s), V^{\alpha}(s), v_*)\\
                &\hspace{3.5cm} \times \int_{0}^s \mathcal{F}_1^\beta(s,s_1) K_{M,\left \langle v \right\rangle^{l}}^{\alpha,\chi}  \mathbf{h} (s_1,X^{\alpha,\beta}_1(s_1),V^{\alpha,\beta}_1(s_1))ds_1 dv_* ds \\
                & \hspace{0.5cm} \lesssim  \delta^{3+\gamma}e^{-\frac{\nu_ot}{2\varepsilon}}\sup_{0\leq s \leq t}( e^{\frac{\nu_os}{2\varepsilon}}\Vert \mathbf{h}(s) \Vert_{L^{\infty}_{x,v}}).
            \end{split}
        \end{eqnarray*}
     The last term in \eqref{HD00Z222ZZ00} can be expressed as
        \begin{gather*}
            \begin{split}
                \mathcal{O}^\alpha(t) = & \frac{1}{\varepsilon^2} \int_{0}^{t} \mathcal{F}^{\alpha}_0(t,s) \sum_{\beta=A,B}\int_{\mathbb{R}^3} k_{M,\left \langle v \right\rangle^{l}}^{\alpha\beta}(s,X^{\alpha}(s),V^{\alpha}(s),v_*)  \int_{0}^s \mathcal{F}_1^\beta(s,s_1) \\
                &\hspace{4cm} \times (K_{M,\left \langle v \right\rangle^{l}}^{\beta,c} \mathbf{h}) (s_1,X^{\alpha,\beta}_1(s_1),V^{\alpha,\beta}_1(s_1))ds_1 dv_* ds\\
                \leq & \frac{1}{\varepsilon^2} \int_{0}^{t} \mathcal{F}^{\alpha}_0(t,s) \sum_{\beta=A,B}\int_{\mathbb{R}^3} k_{M,\left \langle v \right\rangle^{l}}^{\alpha\beta}(s,X^{\alpha}(s),V^{\alpha}(s),v_*) \sum_{\beta'=A,B}\int_{0}^s \mathcal{F}_1^\beta(s,s_1) \\
                &\hspace{0.5cm} \times \int_{\mathbb{R}^3}  k_{M,\left \langle v \right\rangle^{l}}^{\beta\beta^{'}}(s_1,X^{\alpha,\beta}_1(s_1),V^{\alpha,\beta}_1(s_1),v_{**}) |h^{\beta'}| (s_1,X^{\alpha,\beta}_1(s_1),v_{**}) dv_{**}ds_1 dv_* ds.
            \end{split}
        \end{gather*}
        From \eqref{chracterlinedef} and \eqref{ASSPDREVYVxXx}, for sufficiently large $N>0$ and  small $T_0$, the following holds:
        \begin{equation}\label{TZXVvCHA}
            \sup_{0\leq s,t \leq T_0}|V^{\alpha}(s)-v|\leq\frac{N}{2}, \quad \alpha \in \{A,B\}.
        \end{equation}
        By Lemma \ref{LeMK2ker4444}\,($\gamma=1$), it yields that
        \begin{equation*}
            \begin{split}
                &\int_{\mathbb{R}^3} \int_{\mathbb{R}^3} |k_{M,\left \langle v \right\rangle^{l}}^{\alpha\beta}(s,X^{\alpha}(s),V^{\alpha}(s),v_*)k_{M,\left \langle v \right\rangle^{l}}^{\beta\beta^{'}}(s_1,X^{\alpha,\beta}_1(s_1),V^{\alpha,\beta}_1(s_1),v_{**})|dv_{**}\, dv_*\\
                &\hspace{2cm}\leq \frac{C\left \langle V(s) \right\rangle}{1+|V(s)|}  \times \left \langle V_1(s_1) \right\rangle .
            \end{split}
        \end{equation*}

        We estimate  $\mathcal{O}^{\alpha}$ in the following cases.

        \noindent{\bf Case 1}. For $|v| \geq N $.  From \eqref{TZXVvCHA} we get $|V^{\alpha}(s)|\geq \frac{N}{2}$, which leads to
        \begin{equation*}
            \begin{split}
                & \sum_{\beta'=A,B} \int_{\mathbb{R}^3} \int_{\mathbb{R}^3} |k_{M,\left \langle v \right\rangle^{l}}^{\alpha\beta}(s,X^{\alpha}(s),V^{\alpha}(s),v_*)k_{M,\left \langle v \right\rangle^{l}}^{\beta\beta^{'}}(s_1,X^{\alpha,\beta}_1(s_1),V^{\alpha,\beta}_1(s_1),v_{**})|dv_{**} \, dv_* \\
                & \hspace{2cm}\leq  \frac{C}{N} \times \left \langle V^{\alpha}(s) \right\rangle \times \left \langle V^{\alpha,\beta}_1(s_1) \right\rangle .
            \end{split}
        \end{equation*}
        Thus, one obtains the following bound
        \begin{align}
            \mathcal{O}^\alpha(t)\mathbf{1}_{|v|\geq N}
            \leq&   \frac{C}{\varepsilon^2N} \int_{0}^{t} \mathcal{F}^{\alpha}_0(t,s) \left \langle V^{\alpha}(s) \right\rangle \sum_{\beta=A,B}\int_{0}^{s} \mathcal{F}^{\beta}_1(s,s_1)\left \langle V^{\alpha,\beta}_1(s_1) \right\rangle e^{-\frac{\nu_os_1}{2\varepsilon}}ds_1\notag\\
            &\hspace{3.5cm}\times \sup_{0\leq s_1 \leq t}(e^{\frac{\nu_os_1}{2\varepsilon}}\|\mathbf{h}(s_1)\|_{L^{\infty}_{x,v}}) \, \,ds \label{4.25}\\
            \leq & \frac{C}{N }\, e^{-\frac{\nu_ot}{2\varepsilon}}\sup_{0\leq s \leq t}(e^{\frac{\nu_os}{2\varepsilon}}\Vert  \mathbf{h}(s) \Vert_{L^{\infty}_{x,v}})\notag.
        \end{align}

        \noindent{\bf Case 2}.  For either $|v| \leq N,  |v_*| \geq 2N $ or $|v_*| \leq 2 N,  |v_{**}| \geq 3 N $. We note that
        \begin{gather*}\label{Case2FENHUAV}
            \begin{split}
                |V^{\alpha}(s)-v_*|\geq|v_*-v|-|V^{\alpha}(s)-v|\geq |v_*|-|v|-|V^{\alpha}(s)-v|, \\
                |V_1^{\alpha,\beta}(s_1)-v_{**}|\geq|v_{**}-v_*|-|V^{\alpha,\beta}_1(s_1)-v_*|\geq |v_{**}|-|v_*|-|V^{\alpha,\beta}_1(s_1)-v_*|.
            \end{split}
        \end{gather*}
        For sufficiently small $0 < \eta \leq c$, either $|V^{\alpha}(s)-v_*| \geq \frac{N}{2}$ or $|V^{\alpha,\beta}_1(s_1)-v_{**}| \geq \frac{N}{2}$ holds. Consequently, it follows that
        \begin{gather*}\label{CASE2jizhishu}
            \begin{split}
                k_{M,\left \langle v \right\rangle^{l}}^{\alpha\beta}(V^{\alpha}(s),v_*)  &\leq  e^{-\frac{\eta}{8} N^2} k_{M,\left \langle v \right\rangle^{l}}^{\alpha\beta}(V^{\alpha}(s),v_*) e^{\frac{\eta}{8} |V^{\alpha}(s)-v_*|^2}  ,\\
                k_{M,\left \langle v \right\rangle^{l}}^{\beta\beta^{'}}(V^{\alpha,\beta}_1(s_1),v_{**})  &\leq  e^{-\frac{\eta}{8} N^2} k_{M,\left \langle v \right\rangle^{l}}^{\beta\beta^{'}}(V^{\alpha,\beta}_1(s_1),v_{**}) e^{\frac{\eta}{8} |V^{\alpha,\beta}_1(s_1)-v_{**}|^2}.
            \end{split}
        \end{gather*}
        Further, from Lemma \ref{LeMK2ker4444},  since $0\leq \eta \leq c$ is sufficiently small, we obtain
        \begin{eqnarray*}
            \begin{split}
                \int_{\mathbb{R}^3} k_{M,\left \langle v \right\rangle^{l}}^{\alpha\beta}(V^{\alpha}(s),v_*) e^{\frac{\eta}{8} |V^{\alpha}(s)-v_*|^2}d v_* &\leq C \left \langle V^{\alpha}(s) \right\rangle, \\
                \int_{\mathbb{R}^3} k_{M,\left \langle v \right\rangle^{l}}^{\beta\beta^{'}}(V^{\alpha,\beta}_1(s_1),v_{**}) e^{\frac{\eta}{8} |V^{\alpha,\beta}_1(s_1)-v_{**}|^2}d v_{**} &\leq C \left \langle V^{\alpha,\beta}_1(s_1) \right\rangle .
            \end{split}
        \end{eqnarray*}
        This leads to the following bound
        \begin{gather*}
            \begin{split}
                & \frac{1}{\varepsilon^2} \int_{0}^{t} \mathcal{F}^{\alpha}_0(t,s) \sum_{\beta=A,B}\int_{|v| \leq N,  |v_*| \geq 2N } k_{M,\left \langle v \right\rangle^{l}}^{\alpha\beta}(s,X^{\alpha}(s),V^{\alpha}(s),v_*) \sum_{\beta'=A,B}\int_{0}^s \mathcal{F}_1^\beta(s,s_1) \\
                &\hspace{0.3cm} \times \int_{|v_*| \leq 2 N,  |v_{**}| \geq 3 N }  k_{M,\left \langle v \right\rangle^{l}}^{\beta\beta^{'}}(s_1,X^{\alpha,\beta}_1(s_1),V^{\alpha,\beta}_1(s_1),v_{**}) |h^{\beta'}| (s_1,X^{\alpha,\beta}_1(s_1),v_{**}) dv_{**}ds_1 dv_* ds \\
                & \leq \frac{C_{\eta}}{\varepsilon^2}e^{-\frac{\eta}{8} N^2}\int_{0}^{t} \mathcal{F}^{\alpha}_0(t,s) \left \langle V^{\alpha}(s) \right\rangle \sum_{\beta=A,B}\int_{0}^{s} \mathcal{F}^{\beta}_1(s,s_1)\left \langle V^{\alpha,\beta}_1(s_1) \right\rangle e^{-\frac{\nu_os_1}{2\varepsilon}}e^{\frac{\nu_os_1}{2\varepsilon}} \|\mathbf{h}(s_1)\|_{L^{\infty}_{x,v}} \, ds_1\,ds \\
                & \leq  C_{\eta}\,e^{-\frac{\eta}{8} N^2}\, e^{-\frac{\nu_ot}{2\varepsilon}}\sup_{0\leq s \leq t}(e^{\frac{\nu_os}{2\varepsilon}}\Vert  \mathbf{h}(s) \Vert_{L^{\infty}_{x,v}}).
            \end{split}
        \end{gather*}

        \noindent{\bf Case 3}. If  $ |v|\leq N,  |v_*| > 2N$,  it is also included in Case 2; while if  $|v_{**}| > 3N $, either $|v_*| \leq 2N $  or $|v_*| \geq 2N $ is also included in Case 2.
        The remaining case is that $|v| \leq N,   |v_*| \leq 2N ,   |v_{**}| \leq 3N $, that is, $v, v_*,v_{**}$ are all bounded. Then we
        assume that $s-s_1 \leq \varepsilon \lambda $ and  $s-s_1 \geq \varepsilon \lambda $  where the $\lambda$ is small enough.
        \quad \\

        \noindent{\it Case 3a}.  In this case, $|v| \leq N,   |v_{*}| \leq 2N ,   |v_{**}| \leq 3N $ and  $s-s_1 \leq \varepsilon \lambda $, we have
        \begin{gather*}
            \begin{split}
                & \frac{1}{\varepsilon^2}\Big| \int_{0}^{t} \mathcal{F}^{\alpha}_0(t,s) \sum_{\beta=A,B}\int_{\mathbb{R}^3} k_{M,\left \langle v \right\rangle^{l}}^{\alpha\beta}(s,X^{\alpha}(s),V^{\alpha}(s),v_*) \sum_{\beta'=A,B}\int_{s-\lambda\varepsilon}^s \mathcal{F}_1^\beta(s,s_1) \\
                &\hspace{1cm} \times \int_{\mathbb{R}^3}  k_{M,\left \langle v \right\rangle^{l}}^{\beta\beta^{'}}(s_1,X^{\alpha,\beta}_1(s_1),V^{\alpha,\beta}_1(s_1),v_{**}) |h^{\beta'}| (s_1,X^{\alpha,\beta}_1(s_1),v_{**}) dv_{**}ds_1 dv_* ds\Big|\\
                &\leq \frac{C}{\varepsilon^2}\int_{0}^{t} \mathcal{F}^{\alpha}_0(t,s)\left \langle V^{\alpha}(s) \right\rangle\sum_{\beta=A,B}\int_{s-\lambda\varepsilon}^s \mathcal{F}_1^\beta(s,s_1)\left \langle V^{\alpha,\beta}_1(s_1) \right\rangle e^{-\frac{\nu_os_1}{2\varepsilon}} \, ds_1\,ds \sup_{0\leq s \leq t}(e^{\frac{\nu_os}{2\varepsilon}}\Vert  \mathbf{h}(s) \Vert_{L^{\infty}_{x,v}})\\
                &\leq C \lambda  e^{-\frac{\nu_ot}{2\varepsilon}}\sup_{0\leq s \leq t}(e^{\frac{\nu_os}{2\varepsilon}}\Vert  \mathbf{h}(s) \Vert_{L^{\infty}_{x,v}}).
            \end{split}
        \end{gather*}
        \quad \\
        \noindent{\it Case 3b}. In the case $|v|\leq N,  |v_*|\leq 2N,  |v_{**}|\leq 3N, $ and $s-s_1 \geq \varepsilon \lambda $ where $\lambda$ is small enough. It could be bounded by
        \begin{gather}\label{4.30}
            \begin{split}
                &\frac{C}{\varepsilon^2}\int_{0}^{t} \mathcal{F}^{\alpha}_0(t,s) \int_{|v^*|\leq 2N,  |v^{**}| \leq 3N } \sum_{\beta = A,B}\sum_{\beta' = A,B}\int_{0}^{s-\varepsilon\lambda}\mathcal{F}^{\beta}_1(s,s_1) k_{M,\left \langle v \right\rangle^{l}}^{\alpha\beta}(s,X^{\alpha}(s),V^{\alpha}(s),v_*)\\
                & \hspace{2.4cm}\times   k_{M,\left \langle v \right\rangle^{l}}^{\beta \beta'}(s_1,X^{\alpha,\beta}_1(s_1),V^{\alpha,\beta}_1(s_1),v_{**}) \mathbf{h} (s_1,X^{\alpha,\beta}_1(s_1),v_{**})ds_1 dv_{**}dv_* ds .
            \end{split}
        \end{gather}
        Moreover, there exists some smooth  operators $k_{M,N}^{\alpha\beta}(V^{\alpha},v^*),k_{M,N}^{\beta \beta'}(V^{\alpha,\beta}_1,v_{**})\,\alpha ,\beta \in \{A,B\}$  with compact support, such that
        \begin{gather}\label{jingkerbijingg1}
            \begin{split}
                \int_{|V^{\alpha}(s)|\leq 3N,|v_*|\leq 3N}|k_{M,\left \langle v \right\rangle^{l}}^{\alpha\beta}- k_{M,N}^{\alpha\beta}|(s,X^{\alpha}(s),V^{\alpha}(s),v^*) dv_* \leq & \frac{1}{N}, \\
                \int_{V^{\alpha,\beta}_1(s_1)|\leq 3N,|v_{**}|\leq 3N} |k_{M,\left \langle v \right\rangle^{l}}^{\beta \beta'} - k_{M,N}^{\beta \beta'}| (s_1,X^{\alpha,\beta}_1(s_1), V^{\alpha,\beta}_1(s_1),v_{**})dv_{**} \leq & \frac{1}{N}.
            \end{split}
        \end{gather}
       Then, we split $k_{M,\left \langle v \right\rangle^{l}}^{\alpha\beta}(s,X^{\alpha}(s),V^{\alpha}(s),v_*)  k_{M,\left \langle v \right\rangle^{l}}^{\beta \beta'}(s_1,X^{\alpha,\beta}_1(s_1),V^{\alpha,\beta}_1(s_1),v_{**}) $ as
        \begin{gather}\label{SEC4MAINSPLIT}
            \begin{split}
                &   k_{M,\left \langle v \right\rangle^{l}}^{\alpha\beta}(s,X^{\alpha}(s),V^{\alpha}(s),v_*)  k_{M,\left \langle v \right\rangle^{l}}^{\beta \beta'}(s_1,X^{\alpha,\beta}_1(s_1),V^{\alpha,\beta}_1(s_1),v_{**}) \\
                = &    \Big[k_{M,\left \langle v \right\rangle^{l}}^{\beta \beta'}(s_1,X^{\alpha,\beta}_1(s_1),V^{\alpha,\beta}_1(s_1),v_{**}) - k_{M,N}^{\beta \beta'}(s_1,X^{\alpha,\beta}_1(s_1),V^{\alpha,\beta}_1(s_1),v_{**})\Big]\\
                & \hspace{8cm}\times k_{M,N}^{\alpha\beta}(s,X^{\alpha}(s),V^{\alpha}(s),v_*) \\
                & + \Big[k_{M,\left \langle v \right\rangle^{l}}^{\alpha\beta}(s,X^{\alpha}(s),V^{\alpha}(s),v_*)- k_{M,N}^{\alpha\beta}(s,X^{\alpha}(s),V^{\alpha}(s),v_*)\Big]  \\
                &\hspace{7cm}\times k_{M,\left \langle v \right\rangle^{l}}^{\beta \beta'} (s_1,X^{\alpha,\beta}_1(s_1),V^{\alpha,\beta}_1(s_1),v_{**}) \\
                & + k_{M,N}^{\alpha\beta}(s,X^{\alpha}(s),V^{\alpha}(s),v_*)  k_{M,N}^{\beta \beta'}(s_1,X^{\alpha,\beta}_1(s_1),V^{\alpha,\beta}_1(s_1),v_{**}).
            \end{split}
        \end{gather}
       The term can be bounded by combining estimates from \eqref{4.30}, \eqref{jingkerbijingg1}, and \eqref{SEC4MAINSPLIT}. That is
        \begin{align}
                & \frac{C}{ N\varepsilon^2}  \Big(\int_{0}^{t} \mathcal{F}^{\alpha}_0(t,s) \sum_{\beta = A,B} \int_{|v_{*}|\leq 2N} |k_{M,N}^{\alpha\beta }(s,X^{\alpha}(s),V^{\alpha}(s),v_{*})|dv_{*}\int_{0}^{s-\varepsilon \lambda} \mathcal{F}^{\beta}_1(s,s_1)  ds_1 ds \notag\\
                & +  \int_{0}^{t}\mathcal{F}^{\alpha}_0(t,s) \sum_{\beta' = A,B} \int_{|v_{**}|\leq 3N} |k_{M,\left \langle v \right\rangle^{l}}^{\beta \beta'}(s_1,X^{\alpha,\beta }_1(s_1),V^{\alpha,\beta }_1(s_1),v_{**})| \int_{0}^{s-\varepsilon\lambda} \mathcal{F}^{\beta}_1(s,s_1) ds_1 dv_{**} ds \Big)\notag\\
                &\hspace{9cm} \times \sup_{0\leq s \leq t} \|h^{\beta'}(s)\|_{L^{\infty}_{x,v}}\\
                &+\frac{C}{\varepsilon^2} \int_{0}^{t} \mathcal{F}^{\alpha}_0(t,s) \int_{|v^*|\leq 2N,  |v^{**}| \leq 3N }\sum_{\beta = A,B} \int_{0}^{s-\varepsilon\lambda}\mathcal{F}^{\beta}_1(s,s_1) k_{M,N}^{\alpha \beta}(s,X^{\alpha}(s),V^{\alpha}(s),v_*) \notag\\
                &\hspace{2cm} \times \sum_{\beta' = A,B} k^{\beta \beta'}_{M,N}(s_1,X^{\alpha,\beta}_1(s_1),V^{\alpha,\beta}_1(s_1),v_{**}) |h^{\beta'} (s_1,X^{\alpha,\beta}_1(s_1),v_{**})| ds_1 dv_{**}dv_* ds  \notag\\
                &\leq  \frac{C}{ N} e^{-\frac{\nu_ot}{2\varepsilon}}\sup_{0\leq s \leq t}(e^{\frac{\nu_os}{2\varepsilon}}\Vert  \mathbf{h}(s) \Vert_{L^{\infty}_{x,v}}) + \frac{C}{\varepsilon^2} \int_{0}^{t} \mathcal{F}^{\alpha}_0(t,s) \int_{|v^*|\leq 2N,  |v^{**}| \leq 3N }\sum_{\beta = A,B} \int_{0}^{s-\varepsilon\lambda}\mathcal{F}^{\beta}_1(s,s_1) \notag\\
                & \hspace{3.5cm}\times k_{M,N}^{\alpha \beta}(s,X^{\alpha}(s),V^{\alpha}(s),v_*) \sum_{\beta' = A,B} k^{\beta \notag \beta'}_{M,N}(s_1,X^{\alpha,\beta}_1(s_1),V^{\alpha,\beta}_1(s_1),v_{**}) \\
                &\hspace{7.8cm}\times| \mathbf{h} (s_1,X^{\alpha,\beta}_1(s_1),v_{**})| ds_1 dv_{**}dv_* ds. \label{HD44400465F}
            \end{align}
        Let $\zeta_{\alpha\beta} = X^{\alpha,\beta}_1(s_1) = X^{\beta}(s_1; s, X^{\alpha}(s; t, x, v), v_*)$.
            By Lemma \ref{lemofcharacter}, it follows that
        \begin{equation*}
            |\zeta_{\beta\alpha}-X^{\alpha}(s)|=|X^{\beta,\alpha}_1(s_1)-X^{\alpha}(s)|\leq C (s-s_1).
        \end{equation*}
        For $s-s_1 \geq \lambda \varepsilon$, the derivative bound holds:
        $
        \left|\frac{d \zeta_{\beta\alpha}}{d v_*}\right| \geq \frac{\lambda^3\varepsilon^3}{2}.
        $
        Since
        \begin{equation*}
            \Big|k_{M,N}^{\alpha \beta}(s,X^{\alpha}(s),V^{\alpha}(s),v_*)  k^{\beta \beta'}_{M,N}(s_1,X^{\alpha,\beta}_1(s_1),V^{\alpha,\beta}_1(s_1),v_{**})\Big| \leq C_N,
        \end{equation*}
       integrating $v_*$ over the domain $|v_*| \leq 2N$ implies
        \begin{gather*}
            \begin{split}
                &\int_{|v_*| \leq 2 N} |h^{\beta'}(s_1,X^{\alpha,\beta}_1(s_1),v_{**})|dv_*
                \leq  C_N \Big(\int_{|v_*|\leq 2N}  |h^{\beta'}(s_1,X^{\alpha,\beta}_1(s_1),v_{**})|^2 dv_*\Big)^{\frac{1}{2}} \\
                &\hspace{0.3cm}\leq  \frac{C_N}{\lambda^{
                        \frac{3}{2}}\varepsilon^{\frac{3}{2}}} \Big(\int_{|\zeta_{\alpha\beta}-X^{\alpha}(s)|\leq C (s-s_1)N}  |h^{\beta'} (s_1,\zeta_{\alpha\beta},v_{**})|^2 d\zeta_{\alpha\beta} \Big)^{\frac{1}{2}}\\
                &\hspace{0.3cm}\leq  \frac{C_N}{\lambda^{
                        \frac{3}{2}}\varepsilon^{\frac{3}{2}}} \Big(\int_{\mathbb{R}^3}  |h^{\beta'} (s_1,\zeta_{\alpha\beta},v_{**})|^2 d\zeta_{\alpha\beta}\Big)^{\frac{1}{2}}.
            \end{split}
        \end{gather*}
      Then the last term in \eqref{HD44400465F} is further controlled by
        \begin{align}
            & \frac{C_{N,\lambda}}{\varepsilon^{7/2}} \int_{0}^{t} \mathcal{F}^{\alpha}_0(t,s) \sum_{\beta=A,B}\int_{0}^{s-\varepsilon\lambda}\mathcal{F}^{\beta}_1(s,s_1)\notag\\
            &\hspace{3.5cm}\times\sum_{\beta'=A,B} \int_{ |v_{**}|\leq 3N } \Big( \int_{\mathbb{R}^3_{\zeta_{\alpha\beta}}}|h^{\beta'}(s_1,\zeta_{\alpha\beta},v_{**})|^2 d\zeta_{\alpha\beta}\Big)^{\frac{1}{2}}dv_{**}  ds_1 ds\notag
            \\
            \leq & \frac{C_{N,\lambda}}{\varepsilon^{7/2}} \int_0^t \sum_{\beta'=A,B} \int_{0}^{s-\varepsilon \kappa}
            e^{-\frac{\nu(N)(t-s)}{\varepsilon}}e^{-\frac{\nu(2N)(s-s_1)}{\varepsilon}}   \| h^{\beta'}(s_1)\|_{L^{2}_{x,v}}  ds_1 ds \notag.
        \end{align}
       For $|v|\leq N$, $|v_*|\leq 2N$, and $|v_{**}|\leq 3N$,
       $
        |h^{\beta'}| = w \frac{\sqrt{\mu^{\beta'}}}{\sqrt{\mu^{\beta'}_M}} |f^{\beta'}_R| \leq C_N |f^{\beta'}_R|
       $
       implies
       \begin{equation*}
        \begin{split}
            &\frac{C_{N,\lambda}}{\varepsilon^{7/2}} \int_0^{T_0} \sum_{\beta'=A,B} \int_{0}^{s-\varepsilon \lambda}
            e^{-\frac{\nu(N)(t-s)}{\varepsilon}}e^{-\frac{\nu(2N)(s-s_1)}{\varepsilon}} \| h^{\beta'}(s_1)\|_{L^{2}_{x,v}} ds_1 ds \\
            &\leq \frac{C_{N,\lambda}}{\varepsilon^{7/2}} \int_0^{T_0} \sum_{\beta'=A,B} \int_{0}^{s-\varepsilon \lambda}
            e^{-\frac{\nu(N)(t-s)}{\varepsilon}}e^{-\frac{\nu(2N)(s-s_1)}{\varepsilon}} \sup_{0\leq s_1 \leq s} \| f^{\beta'}_R(s_1)\|_{L^{2}_{x,v}} ds_1 ds \\
            &\leq \frac{C_{N,\lambda,T_0}}{\varepsilon^{3/2}} \sup_{0\leq s \leq t}\| \mathbf{f}_R(s)\|_{L^{2}_{x,v}}.
        \end{split}
       \end{equation*}

      Hence, the preceding arguments establish that
        \begin{gather}\label{HSC465}
            \begin{split}
                \sup_{0\leq s \leq T_0}&(e^{\frac{\nu_os}{2\varepsilon}} \Vert  h^{\alpha}(s) \Vert_{L^{\infty}_{x,v}})
                \leq C \sup_{0\leq s \leq T_0}[(1+\frac{s}{\varepsilon})e^{-\frac{\nu_os}{2\varepsilon}}]\|\mathbf{h}^{in}\|_{L^\infty_{x,v}}\\
                &+[C(1+\mathcal{I}_1(T_0))\varepsilon+C\lambda+\frac{C_\lambda}{N}+\delta^{3+1}]\sup_{0\leq s \leq T_0}(e^{\frac{\nu_os}{2\varepsilon}}\Vert  \mathbf{h}(s) \Vert_{L^{\infty}_{x,v}})\\
                &+C\varepsilon^k \sup_{0\leq s \leq T_0}(e^{\frac{\nu_os}{2\varepsilon}}\Vert  \mathbf{h}(s) \Vert_{L^{\infty}_{x,v}})^2 +C\varepsilon^k\mathcal{I}_2(T_0)e^{\frac{\nu_oT_0}{2\varepsilon}}
                +\frac{C_{N,\lambda}}{\varepsilon^{\frac{3}{2}}}e^{\frac{\nu_oT_0}{2\varepsilon}}\sup_{0\leq s \leq T_0}\| \mathbf{f}_R(s)\|_{L^{2}_{x,v}}.
            \end{split}
        \end{gather}
        For  small $\delta >0$, we choose $\lambda$ small and $N$  sufficiently large such that $$ C(1+\mathcal{I}_1(T_0))\varepsilon+C{\lambda}+\frac{C_\lambda}{N}+\delta^4<\frac{1}{2}.$$
        Noticing that $(1+\frac{s}{\varepsilon})e^{-\frac{\nu_os}{2\varepsilon}}$, $\mathcal{I}_1$ and $\mathcal{I}_2$ are uniformly bounded, with assumption \eqref{SUPPOSELINFGXV}, one obtains
        \begin{gather}\label{HD00LAST2}
            \begin{split}
                \sup_{0\leq s \leq T_0}(e^{\frac{\nu_os}{2\varepsilon}}\Vert  h^{\alpha}(s) \Vert_{L^{\infty}_{x,v}})
                \leq & C \|\mathbf{h}^{in}\|_{L^\infty_{x,v}}+\frac{C_{N,\lambda}}{\varepsilon^{\frac{3}{2}}}e^{\frac{\nu_oT_0}{2\varepsilon}}\sup_{0\leq s \leq T_0}\| \mathbf{f}_R(s)\|_{L^{2}_{x,v}}+C\varepsilon^k\mathcal{I}_2(T_0)e^{\frac{\nu_oT_0}{2\varepsilon}}.
            \end{split}
        \end{gather}
       Then multiplying \eqref{HD00LAST2} by $\varepsilon^{3/2} e^{-\nu_0 T_0/(2\varepsilon)}$
       with $s = T_0$ and summing over $\alpha \in \{A,B\}$ yields
        \begin{equation}\label{HSC467}
            \varepsilon^{\frac{3}{2}}\Vert  \mathbf{h}(T_0) \Vert_{L^{\infty}_{x,v}}\leq \frac{1}{4}\Vert \varepsilon^{\frac{3}{2}} \mathbf{h}^{\rm in} \Vert_{L^{\infty}_{x,v}}+C \sup_{0\leq s \leq T_0}\| \mathbf{f}_R(s)\|_{L^{2}_{x,v}}+C\varepsilon^{\frac{2k+3}{2}}.
        \end{equation}
    \end{proof}

    \subsection{$W_x^{1,\infty}$ estimate }Verification of the key hypothesis \eqref{SUPPOSELINFGXV} requires establishing $W_x^{1,\infty}$ estimates for energy closure. Applying $D_x$ to \eqref{HDSTLINFEQMAIN} gives
    \begin{eqnarray}\label{DXLINFT4HS}
        \begin{split}
            &\hspace{0.5cm}\partial_{t}(D_xh^{\alpha})+v\cdot\nabla_{x}(D_xh^{\alpha})+\frac{e^{\alpha}}{m^{\alpha}} \nabla_{x}\phi^{\varepsilon}\cdot\nabla_{v}  (D_xh^{\alpha}) + \frac{\upsilon^\alpha}{\varepsilon}  (D_xh^{\alpha}) \\
            &=-\nabla_{x}(D_x\phi^{\varepsilon})\cdot\nabla_{v}h^{\alpha}-\frac{D_x\upsilon^\alpha}{\varepsilon}h^{\alpha}-D_x\Big[\frac{1}{\varepsilon}(K_{M,\left \langle v \right\rangle^{l}}^{\alpha,\chi}+K_{M,\left \langle v \right\rangle^{l}}^{\alpha,c})\mathbf{h}\Big]
            + \sum_{i=1}^{6}D_x (\mathcal{H}^\alpha_i),
        \end{split}
    \end{eqnarray}
    where $\mathcal{H}^\alpha_i$ is described in \eqref{HD00HI6}. Along the trajectory, the solution $D_xh^{\alpha}$ of the equation \eqref{DXLINFT4HS} is expressed as
    \begin{align}
            D_xh^\alpha(t,x,v) =\, &\mathcal{F}^\alpha_0(t,0)\,(D_xh^\alpha)(0,X^{\alpha}(0;t,x,v),V^{\alpha}(0;t,x,v)) \notag\\
            &-\int_{0}^{t}\mathcal{F}^\alpha_0(t,s) [\nabla_{x}(D_x\phi^{\varepsilon})\cdot\nabla_{v}h^{\alpha}](s, X^{\alpha}(s),V^{\alpha}(s))ds\notag\\
            &- \frac{1}{\varepsilon} \int_{0}^{t}\mathcal{F}^\alpha_0(t,s)(D_x\upsilon^\alpha) \, h^{\alpha} (s, X^{\alpha}(s),V^{\alpha}(s))ds\label{linfHDDXXXADTODXT46}\\
            & + \int_{0}^{t}\mathcal{F}^\alpha_0(t,s) \sum_{i=1}^{6}(D_x\mathcal{H}^\alpha_j)(s, X^{\alpha}(s),V^{\alpha}(s))ds\notag\\
            &- \frac{1}{\varepsilon} \int_{0}^{t}\mathcal{F}^\alpha_0(t,s)  D_x(K_{M,\left \langle v \right\rangle^{l}}^{\alpha,\chi} \mathbf{h}) (s, X^{\alpha}(s),V^{\alpha}(s))ds\notag\\
            &- \frac{1}{\varepsilon} \int_{0}^{t}\mathcal{F}^\alpha_0(t,s)  D_x(K_{M,\left \langle v \right\rangle^{l}}^{\alpha,c} \mathbf{h}) (s, X^{\alpha}(s),V^{\alpha}(s))ds.\notag
    \end{align}
   First, the leading term in \eqref{linfHDDXXXADTODXT46} satisfies the following bound:
    \begin{equation*}
        \left|\mathcal{F}_0^{\alpha}(t,0) D_xh^\alpha(0,\,X(0),\,v)\right| \lesssim \varepsilon  e^{-\frac{\nu_ot}{\varepsilon}} \|D_x\mathbf{h}^{in}\|_{L^\infty_{x,v}}.
    \end{equation*}
    Moreover, through direct computation, \eqref{psideddddg} immediately gives
    \begin{equation*}
        |\nabla_{x}(D_x\phi^{\varepsilon})\cdot\nabla_{v}h^{\alpha}|\lesssim 1+ \varepsilon\mathcal{I}_1+\varepsilon^k(\Vert \mathbf{h}\Vert_{L^{\infty}_{x,v}}+\Vert D_x\mathbf{h}\Vert_{L^{\infty}_{x,v}}).
    \end{equation*}
  which together with \eqref{SUPPOSELINFGXV} implies
    \begin{gather*}
        \begin{split}
            &\left|\int_{0}^{t}\mathcal{F}^\alpha_0(t,s) [\nabla_{x}(D_x\phi^{\varepsilon})\cdot\nabla_{v}h^{\alpha}](s, X^{\alpha}(s),V^{\alpha}(s))ds\right| \\
            &\hspace{0.3cm}\lesssim \, \varepsilon  e^{-\frac{\nu_ot}{2\varepsilon}} [1+\varepsilon\mathcal{I}_1+\varepsilon^k(\Vert \mathbf{h}\Vert_{L^{\infty}_{x,v}}+\Vert D_x\mathbf{h}\Vert_{L^{\infty}_{x,v}})]\sup_{0\leq s \leq t}(e^{\frac{\nu_os}{2\varepsilon}}\Vert \nabla_{v}\mathbf{h}\Vert_{L^{\infty}_{x,v}})\\
            &\hspace{0.3cm} \lesssim \, \varepsilon  e^{-\frac{\nu_ot}{2\varepsilon}} \sup_{0\leq s \leq t}(e^{\frac{\nu_os}{2\varepsilon}}\Vert \nabla_{v}\mathbf{h}\Vert_{L^{\infty}_{x,v}}).
        \end{split}
    \end{gather*}
    Furthermore, \eqref{nuKsuanzidef} shows
    \begin{gather*}
        \begin{split}
            |D_x\upsilon^\alpha|
            \lesssim  \sum_{\beta=A,B} \int_{\mathbb{R}^3\times \mathbb{S}^2}
            |B^{ \alpha \beta}(|v-v_*|, \cos \theta)  \sqrt{\mu^{\beta}(v_*)}| d\sigma dv_*
         \lesssim  \left \langle v \right\rangle.
        \end{split}
    \end{gather*}
   Next, it is easy to get
    \begin{gather*}
        \begin{split}
            \frac{1}{\varepsilon}\left| \int_{0}^{t}\mathcal{F}^\alpha_0(t,s)(D_x\upsilon^\alpha) \, h^{\alpha} (s, X^{\alpha}(s),V^{\alpha}(s))ds\right|
            \lesssim \,  e^{-\frac{\nu_ot}{2\varepsilon}} \sup_{0\leq s \leq t}(e^{\frac{\nu_os}{2\varepsilon}}\Vert \mathbf{h}\Vert_{L^{\infty}_{x,v}}).
        \end{split}
    \end{gather*}

  Now, let us focus on the fourth term in \eqref{linfHDDXXXADTODXT46}. First, applying Lemma~\eqref{DDFXXEST}, we see
    \begin{gather*}
        \begin{split}
            |D_x\mathcal{H}_1^\alpha|
            \lesssim \, \varepsilon^{k-1} \left \langle v \right\rangle(\Vert \mathbf{h}\Vert^2_{L^{\infty}_{x,v}}+\Vert D_x\mathbf{h}\Vert^2_{L^{\infty}_{x,v}}).
        \end{split}
    \end{gather*}
    Consequently, we derive the bound
    \begin{gather*}
        \begin{split}
            &\left| \int_{0}^{t}\mathcal{F}^\alpha_0(t,s)(D_x\mathcal{H}^\alpha_1)(s, X^{\alpha}(s),V^{\alpha}(s))ds\right|
            \lesssim \,\varepsilon^k  e^{-\frac{\nu_ot}{\varepsilon}} \sup_{0\leq s \leq t}[(e^{\frac{\nu_os}{2\varepsilon}}\|  \mathbf{h}\|_{L^{\infty}_{x,v}})^2+(e^{\frac{\nu_os}{2\varepsilon}}\|  D_x\mathbf{h}\|_{L^{\infty}_{x,v}})^2].
        \end{split}
    \end{gather*}
    Moreover, utilizing the exponential decay of $F_i$, and Lemma \ref{FI2KM1EST},
    \begin{gather*}
        \begin{split}
            |D_x\mathcal{H}_2^{\alpha}|
            \lesssim   \, \left \langle v \right\rangle \mathcal{I}_1 \,(\Vert\mathbf{h} \Vert_{L^\infty_{x,v}}+\Vert D_x\mathbf{h} \Vert_{L^\infty_{x,v}}),
        \end{split}
    \end{gather*}
    which leads to
    \begin{gather*}
        \begin{split}
            &\left| \int_{0}^{t}\mathcal{F}^\alpha_0(t,s) (D_x\mathcal{H}^\alpha_2)(s, X^{\alpha}(s),V^{\alpha}(s))ds\right|
            \lesssim \,\varepsilon \mathcal{I}_1 e^{-\frac{\nu_ot}{2\varepsilon}} \sup_{0\leq s \leq t}[e^{\frac{\nu_os}{2\varepsilon}}(\Vert\mathbf{h} \Vert_{L^{\infty}_{x,v}}+\Vert D_x\mathbf{h} \Vert_{L^{\infty}_{x,v}})].
        \end{split}
    \end{gather*}
    From \eqref{psideddddg}, and the assumption \eqref{SUPPOSELINFGXV}, it follows that
    \begin{eqnarray*}
        \begin{split}
            |D_x\mathcal{H}^\alpha_3| =& \frac{e^{\alpha}}{m^{\alpha}}\Big(\Big|\nabla_{x}(D_x\phi^{\varepsilon}) \cdot  \frac{\left \langle v \right\rangle^{l}}{\sqrt{\mu^{\alpha}_M}} \nabla_{v}(\frac{\sqrt{\mu^{\alpha}_M}}{\left \langle v \right\rangle^{l}})h^\alpha\Big|
            +\Big|\nabla_{x}\phi^{\varepsilon} \cdot  \frac{\left \langle v \right\rangle^{l}}{\sqrt{\mu^{\alpha}_M}} \nabla_{v}(\frac{\sqrt{\mu^{\alpha}_M}}{\left \langle v \right\rangle^{l}})(D_xh^\alpha)\Big|\Big)\\
            \lesssim &\left \langle v \right\rangle (\Vert\mathbf{h} \Vert_{L^{\infty}_{x,v}}+\Vert D_x\mathbf{h} \Vert_{L^{\infty}_{x,v}}),\\
            |D_x\mathcal{H}^\alpha_4| =& \frac{e^{\alpha}}{m^{\alpha}}\Big(\Big|\nabla_{x}(D_x\phi_R) \cdot  \frac{\left \langle v \right\rangle^{l}}{\sqrt{\mu^{\alpha}_M}} \nabla_{v}(\mu^{\alpha}+\sum_{i=1}^{2k-1}\varepsilon^iF_i^{\alpha})\Big|\\
            &\hspace{5cm}
            +\Big|\nabla_{x}\phi_R \cdot  \frac{\left \langle v \right\rangle^{l}}{\sqrt{\mu^{\alpha}_M}} \nabla_{v}D_x(\mu^{\alpha}+\sum_{i=1}^{2k-1}\varepsilon^iF_i^{\alpha})\Big|\Big)\\
            \lesssim &(1+\varepsilon\mathcal{I}_1)\left \langle v \right\rangle (\Vert\mathbf{h} \Vert_{L^{\infty}_{x,v}}+\Vert D_x\mathbf{h} \Vert_{L^{\infty}_{x,v}}).
        \end{split}
    \end{eqnarray*}
    Collecting these estimates yields
    \begin{gather*}
        \begin{split}
            &\left| \int_{0}^{t}\mathcal{F}^\alpha_0(t,s) (D_x\mathcal{H}^\alpha_3+D_x\mathcal{H}^\alpha_4)(s, X^{\alpha}(s),V^{\alpha}(s))ds\right| \\
            &\hspace{0.5cm}\lesssim \,\, \varepsilon (1+\varepsilon\mathcal{I}_1) e^{-\frac{\nu_ot}{2\varepsilon}} \sup_{0\leq s \leq t}[e^{\frac{\nu_os}{2\varepsilon}}(\Vert\mathbf{h} \Vert_{L^{\infty}_{x,v}}+\Vert D_x\mathbf{h} \Vert_{L^{\infty}_{x,v}})].
        \end{split}
    \end{gather*}
    \noindent  The bounds of $\mathbf{F}_i$ in Lemma \ref{FI2KM1EST} rigorously proves
    \begin{gather*}
        \begin{split}
            |D_x\mathcal{H}^\alpha_5|
            \lesssim &  \left \langle v \right\rangle \Big(\varepsilon^{k-1} \Big|\frac{\left \langle v \right\rangle^{l-1}( \partial_t +v\cdot \nabla_{x} ) (D_xF^\alpha_{2k-1}) }{ \sqrt{\mu^{\alpha}_M}}\Big|_{L^\infty_{x,v}} \\
            &+  \mathop{\sum}_{i+j \geq 2k-1\atop 0 \leq i,j \leq 2k-1}\varepsilon^{i+j-k} | \nabla_{x}(D_x\phi_i)|_{L^{\infty}_{x}} \Vert \frac{\nabla_vF_j^\alpha}{\left \langle v \right\rangle} \Vert_{L^{\infty}_{x,v} }+| \nabla_{x}\phi_i|_{L^{\infty}_{x}} \Vert \frac{\nabla_v(D_xF_j^\alpha)}{\left \langle v \right\rangle} \Vert_{L^{\infty}_{x,v} } \Big) \\
            \lesssim &  \, \varepsilon^{k-1} \left \langle v \right\rangle \mathcal{I}_2,
        \end{split}
    \end{gather*}
    and
    \begin{gather*}
        \begin{split}
            |D_x\mathcal{H}^\alpha_6|
            \lesssim &  \mathop{\sum}_{i+j \geq 2k-1\atop 0 \leq i,j \leq 2k-1}\varepsilon^{i+j-k} \left \langle v \right\rangle \Big(\Big\Vert\frac{\left \langle v \right\rangle^{l}}{\sqrt{\mu^{\alpha}_M}} D_x\mathbf{F}_i\Big\Vert_{L^\infty_v}  \, \Big\Vert\frac{\left \langle v \right\rangle^{l}}{\sqrt{\mu^{\beta}_M}} \mathbf{F}_j\Big\Vert_{L^\infty_{x,v}}\\
            &\hspace{6cm}+\Big\Vert\frac{\left \langle v \right\rangle^{l}}{\sqrt{\mu^{\alpha}_M}} \mathbf{F}_i\Big\Vert_{L^\infty_{x,v}}  \, \Big\Vert\frac{\left \langle v \right\rangle^{l}}{\sqrt{\mu^{\beta}_M}} D_x\mathbf{F}_j\Big\Vert_{L^\infty_{x,v}}\Big)  \\
            \lesssim & \, \varepsilon^{k-1} \left \langle v \right\rangle \mathcal{I}_2.
        \end{split}
    \end{gather*}
   Then, it follows that
    \begin{gather*}\label{H0DZISHI11}
        \begin{split}
            &\left| \int_{0}^{t}\mathcal{F}^\alpha_0(t,s) (D_x\mathcal{H}^\alpha_5+D_x\mathcal{H}^\alpha_6)(s, X^{\alpha}(s),V^{\alpha}(s))ds\right|
            \lesssim \, \varepsilon^k \mathcal{I}_2.
        \end{split}
    \end{gather*}

    Next, recalling the expressions in Lemma \ref{kchiestmainle} that
    \begin{gather*}
        \begin{split}
            &D_x(K_{M,\left \langle v \right\rangle^{l}}^{\alpha,\chi} \mathbf{h})=(D_xK_{M,\left \langle v \right\rangle^{l}}^{\alpha,\chi})\mathbf{h}+K_{M,\left \langle v
\right\rangle^{l}}^{\alpha,\chi}(D_x\mathbf{h}),
        \end{split}
    \end{gather*}
  with corresponding estimates
    \begin{gather*}
        \begin{split}
            | K_{M,\left \langle v \right\rangle^{l}}^{\alpha,\chi}(D_x\mathbf{h})(s, X^{\alpha}(s),V^{\alpha}(s))  | &\lesssim {\delta}^{3+1} \left \langle V^{\alpha}(s) \right\rangle e^{-c|V^{\alpha}(s)|^2} \|D_x\mathbf{\mathbf{h}} \|_{L^\infty_v},\\
            | (D_xK_{M,\left \langle v \right\rangle^{l}}^{\alpha,\chi})\mathbf{h} (s, X^{\alpha}(s),V^{\alpha}(s))| &\lesssim {\delta}^{3+1} \left \langle V^{\alpha}(s) \right\rangle e^{-c|V^{\alpha}(s)|^2} \|\mathbf{\mathbf{h}} \|_{L^\infty_v},
        \end{split}
    \end{gather*}
  which implies
    \begin{gather*}
        | D_x(K_{M,\left \langle v \right\rangle^{l}}^{\alpha,\chi} \mathbf{h})(s, X^{\alpha}(s),V^{\alpha}(s))  | \lesssim {\delta}^{3+1} \left \langle V^{\alpha}(s) \right\rangle e^{-c|V^{\alpha}(s)|^2} (\|\mathbf{\mathbf{h}} \|_{L^\infty_v}+\|D_x\mathbf{\mathbf{h}} \|_{L^\infty_v}).
    \end{gather*}
    Hence, for hard sphere collisions, we obtain
    \begin{align*}
           & \left| \int_{0}^{t}\mathcal{F}^\alpha_0(t,s) D_x(K_{M,\left \langle v \right\rangle^{l}}^{\alpha,\chi} \mathbf{h})(s, X^{\alpha}(s),V^{\alpha}(s))ds\right|\\
           &\hspace{0.4cm} \lesssim \,\, \delta^{3+1}  e^{-\frac{\nu_ot}{2\varepsilon}} \sup_{0\leq s \leq t}[e^{\frac{\nu_os}{2\varepsilon}}(\|\mathbf{h}\|_{L^{\infty}_{x,v}}+\|D_x\mathbf{h}\|_{L^{\infty}_{x,v}})].
    \end{align*}
    Finally, let us turn to the last term in \eqref{linfHDDXXXADTODXT46}. Noticing that
    \begin{equation*}
        D_x(K_{M,\left \langle v \right\rangle^{l}}^{\alpha,c} \mathbf{h})=(D_xK_{M,\left \langle v \right\rangle^{l}}^{\alpha,c})\mathbf{h}+K_{M,\left \langle v \right\rangle^{l}}^{\alpha,c}(D_x\mathbf{h}),
    \end{equation*}
    which together with Lemma \ref{LeMK2ker4444} and \eqref{ASSPDREVYVxXx} yields
    \begin{align}
            &(D_xK_{M,\left \langle v \right\rangle^{l}}^{\alpha,c}) \mathbf{h}(s, X^{\alpha}(s),V^{\alpha}(s))\notag\\
            =&\sum_{\beta=A,B}\int_{\mathbb{R}^3} (D_xk_{M,w}^{\alpha \beta}) (s,X^{\alpha}(s),V^{\alpha}(s),v_*)h^{\beta}(s,X^{\alpha}(s),v_*) d v_{*} ,\label{K2CDXJFBD}
    \end{align}
    and
    \begin{align}
            &K_{M,\left \langle v \right\rangle^{l}}^{\alpha,c} (D_x\mathbf{h})(s, X^{\alpha}(s),V^{\alpha}(s))\notag\\
            =&\sum_{\beta=A,B}\int_{\mathbb{R}^3} k_{M,w}^{\alpha \beta} (s,X^{\alpha}(s),V^{\alpha}(s),v_*)(D_xh^{\beta})(s,X^{\alpha}(s),v_*) d v_{*} .\label{K2CDXJFBD22f}
    \end{align}
   Define the integral operators
    \begin{gather*}
        \begin{split}
            & \frac{1}{\varepsilon}\int_{0}^{t}\mathcal{F}^\alpha_0(t,s) (D_xK_{M,\left \langle v \right\rangle^{l}}^{\alpha,c})\mathbf{h}(s, X^{\alpha}(s),V^{\alpha}(s))ds:=I_{1x}.\\
            & \frac{1}{\varepsilon}\int_{0}^{t}\mathcal{F}^\alpha_0(t,s) K_{M,\left \langle v \right\rangle^{l}}^{\alpha,c}(D_x\mathbf{h})(s, X^{\alpha}(s),V^{\alpha}(s))ds:=I_{2x},
        \end{split}
    \end{gather*}
    then the estimate for $I_{1x}$ follows by substituting \eqref{HDliterative1111} into \eqref{K2CDXJFBD}
    \begin{equation}\label{xxxccJOX}
        I_{1x}=\frac{1}{\varepsilon} \int_{0}^{t}\mathcal{F}_0^{\alpha}(t,s) (D_xK_{M,\left \langle v \right\rangle^{l}}^{\alpha,c})  \mathbf{h}(s,X^{\alpha}(s),V^{\alpha}(s))ds  = \sum_{i=0}^{2} \mathcal{J}^{\alpha}_{i,x}(t) + \mathcal{O}_{x}^\alpha(t).
    \end{equation}
   The proof of Lemma \ref{LeMK2ker4444} establishes that
    \begin{align*}
            &\mathcal{J}^{\alpha}_{0,x}(t)=\frac{1}{\varepsilon} \int_{0}^{t} \mathcal{F}_0^\alpha(t,s) \sum_{\beta=A,B}\int_{\mathbb{R}^3}     (D_xk_{M,\left \langle v \right\rangle^{l}}^{\alpha\beta})(s,X^{\alpha}(s),V^{\alpha}(s),v_*) \\  &\hspace{5cm}\times\mathcal{F}_1^\beta(s,0) h^\beta(0,X^{\alpha,\beta}_1(0),V^{\alpha,\beta}_1(0))  dv_*ds \\
            &\hspace{1cm}\lesssim \frac{t}{\varepsilon}e^{-\frac{\nu_ot}{\varepsilon}} \|\mathbf{h}^{in}\|_{L^{\infty}_{x,v}},\\
            &\mathcal{J}^{\alpha}_{1,x}(t)=\frac{1}{\varepsilon} \int_{0}^{t} \mathcal{F}_0^\alpha(t,s) \sum_{\beta=A,B}\int_{\mathbb{R}^3} (D_xk_{M,\left \langle v \right\rangle^{l}}^{\alpha\beta})(s,X^{\alpha}(s), V^{\alpha}(s), v_*)\\
            &\hspace{4.5cm}\times\sum_{j=1}^6 \int_{0}^s \mathcal{F}_1^\beta(s,s_1) \mathcal{H}^\beta_j(s_1,X_1^{\alpha,\beta}(s_1),V_1^{\alpha,\beta}(s_1))ds_1  dv_*ds \\
            & \hspace{0.5cm}\lesssim  \varepsilon^k e^{-\frac{\nu_ot}{\varepsilon}}(\sup_{0\leq s \leq t}e^{\frac{\nu_os}{2\varepsilon}}\Vert \mathbf{h}(s) \Vert_{L^{\infty}_{x,v}})^2+({1+\mathcal{I}_1}) \varepsilon e^{-\frac{\nu_ot}{2\varepsilon}}\sup_{0\leq s \leq t}(e^{\frac{\nu_os}{2\varepsilon}}\Vert  \mathbf{h}(s) \Vert_{L^{\infty}_{x,v}})+\mathcal{I}_2 \varepsilon^k,\\
            &\mathcal{J}^{\alpha}_{2,x}(t)=\frac{1}{\varepsilon^2} \int_{0}^{t} \mathcal{F}_0^\alpha(t,s) \sum_{\beta=A,B}\int_{\mathbb{R}^3} (D_xk_{M,\left \langle v \right\rangle^{l}}^{\alpha\beta})(s,X^{\alpha}(s), V^{\alpha}(s), v_*)\\
            &\hspace{3.5cm} \times \int_{0}^s \mathcal{F}_1^\beta(s,s_1) K_{M,\left \langle v \right\rangle^{l}}^{\alpha,\chi}  \mathbf{h} (s_1,X^{\alpha,\beta}_1(s_1),V^{\alpha,\beta}_1(s_1))ds_1 dv_* ds \\
            & \hspace{1cm} \lesssim  \delta^{3+\gamma}e^{-\frac{\nu_ot}{2\varepsilon}}\sup_{0\leq s \leq t}( e^{\frac{\nu_os}{2\varepsilon}}\Vert \mathbf{h}(s) \Vert_{L^{\infty}_{x,v}}).
        \end{align*}
    Moreover, the last term in \eqref{xxxccJOX} is expressed by
    \begin{gather*}
        \begin{split}
            \mathcal{O}_x^\alpha(t) = & \frac{1}{\varepsilon^2} \int_{0}^{t} \mathcal{F}^{\alpha}_0(t,s) \sum_{\beta=A,B}\int_{\mathbb{R}^3} (D_xk_{M,\left \langle v \right\rangle^{l}}^{\alpha\beta})(s,X^{\alpha}(s),V^{\alpha}(s),v_*)  \int_{0}^s \mathcal{F}_1^\beta(s,s_1) \\
            &\hspace{4cm} \times K_{M,\left \langle v \right\rangle^{l}}^{\beta,c} \mathbf{h} (s_1,X^{\alpha,\beta}_1(s_1),V^{\alpha,\beta}_1(s_1))ds_1 dv_* ds.
        \end{split}
    \end{gather*}
    Following the methodology established in Cases 1-3b (Section 4.1) and applying Lemma \ref{LeMK2ker4444} for  $D_x k_{M,\left \langle v \right\rangle^{l}}^{\alpha\beta}$
    estimates, we obtain
    \begin{equation*}
        \begin{split}
            \mathcal{O}_x^\alpha(t) \leq & \frac{C}{N }\, e^{-\frac{\nu_ot}{2\varepsilon}}\sup_{0\leq s \leq t}(e^{\frac{\nu_os}{2\varepsilon}}\Vert  \mathbf{h}(s) \Vert_{L^{\infty}_{x,v}})+C_{\eta}\,e^{-\frac{\eta}{8} N^2}\, e^{-\frac{\nu_ot}{2\varepsilon}}\sup_{0\leq s \leq t}(e^{\frac{\nu_os}{2\varepsilon}}\Vert  \mathbf{h}(s) \Vert_{L^{\infty}_{x,v}})\\
            &+C \lambda  e^{-\frac{\nu_ot}{2\varepsilon}}\sup_{0\leq s \leq t}(e^{\frac{\nu_os}{2\varepsilon}}\Vert  \mathbf{h}(s) \Vert_{L^{\infty}_{x,v}})+\frac{C_{N,\lambda,T_0}}{\varepsilon^{3/2}}  \sup_{0\leq s \leq t}\| \mathbf{f}_R(s)\|_{L^{2}_{x,v}}.
        \end{split}
    \end{equation*}
    Consequently, the bound for $I_{1x}$ can be established as
    \begin{gather*}
        \begin{split}
            I_{1x}  & \leq \, C\frac{t}{\varepsilon}e^{-\frac{\nu_ot}{\varepsilon}}\| \mathbf{h}^{\rm in} \|_{L^\infty_{x,v}}+C\varepsilon^ke^{-\frac{\nu_ot}{\varepsilon}} \sup_{0\leq s \leq t}(e^{\frac{\nu_os}{2\varepsilon}}\|  \mathbf{h}\|_{L^{\infty}_{x,v}})^2\\
            & \hspace{0.4cm} +[C(1+\mathcal{I}_1)\varepsilon+C{\lambda}+\frac{C_{\lambda}}{N}+C\varepsilon^{4}]e^{-\frac{\nu_ot}{2\varepsilon}} \sup_{0\leq s \leq t}(e^{\frac{\nu_os}{2\varepsilon}}\|  \mathbf{h}\|_{L^{\infty}_{x,v}})\\
            & \hspace{0.4cm} + \frac{C_{N,\lambda}}{\varepsilon^{\frac{3}{2}}}\sup_{0\leq s \leq T_0}\| \mathbf{f}_R(s)\|_{L^{2}_{x,v}}+C\varepsilon^k\mathcal{I}_2.
        \end{split}
    \end{gather*}
 The analysis of $I_{2x}$ begins by substituting $v \to v_*$ in \eqref{linfHDDXXXADTODXT46} and inserting into \eqref{K2CDXJFBD22f}, which combined with prior estimates gives
    \begin{equation*}\label{xx45xeeccJOX}
        I_{2x}=\frac{1}{\varepsilon} \int_{0}^{t}\mathcal{F}_0^{\alpha}(t,s) K_{M,\left \langle v \right\rangle^{l}}^{\alpha,c}  (D_x\mathbf{h})(s,X^{\alpha}(s),V^{\alpha}(s))ds  = \sum_{i=0}^{3} \mathcal{Z}^{\alpha}_{i,x}(t) + \sum_{i=1}^{2}\mathcal{D}_{i,x}^\alpha(t),
    \end{equation*}
    where
    \begin{align*}
            &\mathcal{Z}^{\alpha}_{0,x}(t)=\frac{1}{\varepsilon} \int_{0}^{t} \mathcal{F}_0^\alpha(t,s) \sum_{\beta=A,B}\int_{\mathbb{R}^3}     k_{M,\left \langle v \right\rangle^{l}}^{\alpha\beta}(s,X^{\alpha}(s),V^{\alpha}(s),v_*) \\  &\hspace{5cm}\times\mathcal{F}_1^\beta(s,0) D_xh^\beta(0,X^{\alpha,\beta}_1(0),V^{\alpha,\beta}_1(0))  dv_*ds \\
            &\hspace{1cm}\lesssim \frac{t}{\varepsilon}e^{-\frac{\nu_ot}{\varepsilon}} \|D_x\mathbf{h}^{in}\|_{L^{\infty}_{x,v}},\\
            &\mathcal{Z}^{\alpha}_{1,x}(t)=-\frac{1}{\varepsilon} \int_{0}^{t} \mathcal{F}_0^\alpha(t,s) \sum_{\beta=A,B}\int_{\mathbb{R}^3} k_{M,\left \langle v \right\rangle^{l}}^{\alpha\beta}(s,X^{\alpha}(s), V^{\alpha}(s), v_*)\\
            &\hspace{1.3cm}\times \int_{0}^s \mathcal{F}_1^\beta(s,s_1) [\nabla_{x}(D_x\phi^{\varepsilon})\cdot\nabla_{v}h^{\alpha}+\frac{D_x\upsilon^\alpha}{\varepsilon}](s_1,X_1^{\alpha,\beta}(s_1),V_1^{\alpha,\beta}(s_1))ds_1  dv_*ds \\
            & \hspace{1cm}\lesssim \varepsilon  e^{-\frac{\nu_ot}{2\varepsilon}} \sup_{0\leq s \leq t}(e^{\frac{\nu_os}{2\varepsilon}}\Vert \nabla_{v}\mathbf{h}\Vert_{L^{\infty}_{x,v}})+ e^{-\frac{\nu_ot}{2\varepsilon}} \sup_{0\leq s \leq t}(e^{\frac{\nu_os}{2\varepsilon}}\Vert \mathbf{h}\Vert_{L^{\infty}_{x,v}}),\\
            &\mathcal{Z}^{\alpha}_{2,x}(t)=\frac{1}{\varepsilon} \int_{0}^{t} \mathcal{F}_0^\alpha(t,s) \sum_{\beta=A,B}\int_{\mathbb{R}^3} k_{M,\left \langle v \right\rangle^{l}}^{\alpha\beta}(s,X^{\alpha}(s), V^{\alpha}(s), v_*)\\
            &\hspace{4cm}\times\sum_{j=1}^6 \int_{0}^s \mathcal{F}_1^\beta(s,s_1) D_x\mathcal{H}^\beta_j(s_1,X_1^{\alpha,\beta}(s_1),V_1^{\alpha,\beta}(s_1))ds_1  dv_*ds \\
            & \hspace{1cm}\lesssim  \varepsilon^k e^{-\frac{\nu_ot}{\varepsilon}}\sup_{0\leq s \leq t}[(e^{\frac{\nu_os}{2\varepsilon}}\Vert \mathbf{h}(s) \Vert_{L^{\infty}_{x,v}})^2+(e^{\frac{\nu_os}{2\varepsilon}}\Vert D_x\mathbf{h}(s) \Vert_{L^{\infty}_{x,v}})^2]+\mathcal{I}_2 \varepsilon^k\\
            & \hspace{1.3cm}+ ({1+\mathcal{I}_1}) \varepsilon e^{-\frac{\nu_ot}{2\varepsilon}}\sup_{0\leq s \leq t}(e^{\frac{\nu_os}{2\varepsilon}}\Vert  \mathbf{h}(s) \Vert_{L^{\infty}_{x,v}}+e^{\frac{\nu_os}{2\varepsilon}}\Vert  D_x\mathbf{h}(s) \Vert_{L^{\infty}_{x,v}}),\\
            &\mathcal{Z}^{\alpha}_{3,x}(t)=\frac{1}{\varepsilon^2} \int_{0}^{t} \mathcal{F}_0^\alpha(t,s) \sum_{\beta=A,B}\int_{\mathbb{R}^3} k_{M,\left \langle v \right\rangle^{l}}^{\alpha\beta}(s,X^{\alpha}(s), V^{\alpha}(s), v_*)\\
            &\hspace{3.5cm} \times \int_{0}^s \mathcal{F}_1^\beta(s,s_1) D_x(K_{M,\left \langle v \right\rangle^{l}}^{\alpha,\chi}  \mathbf{h}) (s_1,X^{\alpha,\beta}_1(s_1),V^{\alpha,\beta}_1(s_1))ds_1 dv_* ds \\
            & \hspace{1.2cm} \lesssim  \delta^{3+1}e^{-\frac{\nu_ot}{2\varepsilon}}\sup_{0\leq s \leq t}(e^{\frac{\nu_os}{2\varepsilon}}\Vert \mathbf{h}(s) \Vert_{L^{\infty}_{x,v}}+e^{\frac{\nu_os}{2\varepsilon}}\Vert D_x\mathbf{h}(s) \Vert_{L^{\infty}_{x,v}}).
        \end{align*}
    By applying Lemma \ref{LeMK2ker4444} and adapting the methodology from Section 4.1, one obtains
    \begin{gather*}\label{HSP0101DD}
        \begin{split}
            |\mathcal{D}_{1,x}^\alpha(t)|
            \leq  (\frac{C}{N }+C_{\eta}\,e^{-\frac{\eta}{8} N^2}+C \lambda)\, e^{-\frac{\nu_ot}{2\varepsilon}}\sup_{0\leq s \leq t}(e^{\frac{\nu_os}{2\varepsilon}}\Vert  \mathbf{h}(s) \Vert_{L^{\infty}_{x,v}})+\frac{C_{N,\lambda,T_0}}{\varepsilon^{3/2}}  \sup_{0\leq s \leq t}\| \mathbf{f}_R(s)\|_{L^{2}_{x,v}}.
        \end{split}
    \end{gather*}
    The estimation of $\mathcal{D}_{2,x}^\alpha(t)$ follows the $L^{\infty}$ framework in Section 4.1 until the final conversion to $L^2$ norm. Beginning from this critical transition, consider Case 3b with velocity truncations:
   $|v|\leq N,  |v_*|\leq 2N,  |v_{**}|\leq 3N, $ and $s-s_1 \geq \varepsilon \lambda $
    where $\lambda$ is sufficiently small. The differential structure yields
    \begin{gather*}
        \begin{split}
            &\frac{1}{\varepsilon^2}\int_{0}^{t} \mathcal{F}^{\alpha}_0(t,s) \int_{|v^*|\leq 2N,  |v^{**}| \leq 3N } \sum_{\beta = A,B}\int_{0}^{s-\varepsilon\lambda}\mathcal{F}^{\beta}_1(s,s_1) k_{M,N}^{\alpha\beta}(s,X^{\alpha}(s),V^{\alpha}(s),v_*)\\
            & \times \sum_{\beta' = A,B}  k_{M,N}^{\beta \beta'}(s_1,X^{\alpha,\beta}_1(s_1),V^{\alpha,\beta}_1(s_1),v_{**}) (D_x\mathbf{h}) (s_1,X^{\alpha,\beta}_1(s_1),v_{**})ds_1 dv_{**}dv_* ds .
        \end{split}
    \end{gather*}
   Moreover, noting that $\frac{e^\alpha}{m^\alpha}$ remains constant, the variable substitution from $v_*$ to
    \begin{equation*}\label{CGOFVBL2}
        \zeta_{\alpha\beta}=X^{\alpha,\beta}_1(s_1)=X^{\beta}(s_1;s,X^{\alpha}(s;t,x,v),v_*),
    \end{equation*}
    such that
    \begin{equation*}
        |\zeta_{\beta\alpha}-X^{\alpha}(s)|=|X^{\beta,\alpha}_1(s_1)-X^{\alpha}(s)|\leq C (s-s_1).
    \end{equation*}
    Applying Lemma \ref{lemofcharacter} to $X^{\alpha,\beta}_1(s_1)=X^{\beta}(s_1;s,X^{\alpha}(s;t,x,v),v_*)$, with
    \begin{equation*}
        x=X^{\alpha}(s;t,x,v),\quad \tau=s_1, \quad t=s,
    \end{equation*}
   and integrating by parts, one obtains
    \begin{align}
            &\frac{C}{\varepsilon^2}\int_{0}^{t} \mathcal{F}^{\alpha}_0(t,s) \int_{D} \sum_{\beta = A,B}\int_{0}^{s-\varepsilon\lambda}\mathcal{F}^{\beta}_1(s,s_1) k_{M,N}^{\alpha\beta}(s,X^{\alpha}(s),V^{\alpha}(s),v_*) \notag\\
            &\hspace{1cm}\times  \sum_{\beta' = A,B}k_{M,N}^{\beta \beta'}(s_1,\zeta_{\alpha\beta},V^{\alpha,\beta}_1(s_1),v_{**})  D_x\mathbf{h} (s_1,\zeta_{\alpha\beta},v_{**}) \frac{dv_*}{d\zeta_{\alpha\beta}} ds_1 dv_{**}d\zeta_{\alpha\beta} ds \notag\\
            \leq    & -\frac{C}{\varepsilon^2}\int_{0}^{t} \mathcal{F}^{\alpha}_0(t,s) \int_{D} \sum_{\beta = A,B}\int_{0}^{s-\varepsilon\lambda}\mathcal{F}^{\beta}_1(s,s_1) (D_xk_{M,N}^{\alpha\beta})(s,X^{\alpha}(s),V^{\alpha}(s),v_*)\,\label{55d3h}\\
            &\hspace{1cm}\times \sum_{\beta' = A,B} k_{M,N}^{\beta \beta'}(s_1,\zeta_{\alpha\beta},V^{\alpha,\beta}_1(s_1),v_{**})   \mathbf{h} (s_1,\zeta_{\alpha\beta},v_{**}) \frac{dv_*}{d\zeta_{\alpha\beta}} ds_1 dv_{**}d\zeta_{\alpha\beta} ds \notag\\
            & -\frac{C}{\varepsilon^2}\int_{0}^{t} \mathcal{F}^{\alpha}_0(t,s) \int_{D} \sum_{\beta = A,B}\int_{0}^{s-\varepsilon\lambda}\mathcal{F}^{\beta}_1(s,s_1) k_{M,N}^{\alpha\beta}(s,X^{\alpha}(s),V^{\alpha}(s),v_*)\,\label{55d5h}\\
            &\hspace{1cm}\times \sum_{\beta' = A,B} (D_xk_{M,N}^{\beta \beta'})(s_1,\zeta_{\alpha\beta},V^{\alpha,\beta}_1(s_1),v_{**})  \mathbf{h} (s_1,\zeta_{\alpha\beta},v_{**}) \frac{dv_*}{d\zeta_{\alpha\beta}} ds_1 dv_{**}d\zeta_{\alpha\beta} ds \notag\\
            & -\frac{C}{\varepsilon^2}\int_{0}^{t} \mathcal{F}^{\alpha}_0(t,s) \int_{D} \sum_{\beta = A,B}\int_{0}^{s-\varepsilon\lambda}\mathcal{F}^{\beta}_1(s,s_1) k_{M,N}^{\alpha\beta}(s,X^{\alpha}(s),V^{\alpha}(s),v_*)\, \notag\\
            &\hspace{1cm}\times \sum_{\beta' = A,B} k_{M,N}^{\beta \beta'}(s_1,\zeta_{\alpha\beta},V^{\alpha,\beta}_1(s_1),v_{**})   \mathbf{h} (s_1,\zeta_{\alpha\beta},v_{**}) D_x(\frac{dv_*}{d\zeta_{\alpha\beta}}) ds_1 dv_{**}d\zeta_{\alpha\beta} ds\notag\\
            & -\frac{C}{\varepsilon^2}\int_{0}^{t} \mathcal{F}^{\alpha}_0(t,s) \int_{D} \sum_{\beta = A,B}\int_{0}^{s-\varepsilon\lambda}(D_x\mathcal{F}^{\beta}_1)(s,s_1) k_{M,N}^{\alpha\beta}(s,X^{\alpha}(s),V^{\alpha}(s),v_*)\,\label{55d9h}\\
            &\hspace{1cm}\times \sum_{\beta' = A,B} k_{M,N}^{\beta \beta'}(s_1,\zeta_{\alpha\beta},V^{\alpha,\beta}_1(s_1),v_{**})   \mathbf{h} (s_1,\zeta_{\alpha\beta},v_{**}) \frac{dv_*}{d\zeta_{\alpha\beta}} ds_1 dv_{**}d\zeta_{\alpha\beta} ds\notag\\
            &+\frac{C_{N,\lambda,T_0}}{\varepsilon^3}e^{-\frac{\nu_ot}{2\varepsilon}}\sup_{0\leq s \leq  T_0}(e^{\frac{\nu_os}{2\varepsilon}}\Vert \mathbf{h}\Vert_{L^{\infty}_{x,v}})\notag,
        \end{align}
    where
    $ D=\{|v_*|\leq 2N, |v_{**}|\leq 3N\}
    $, $\hat{D}=\{|\zeta-X(s)|\leq C (s-s_1) N, |v_{**}|\leq 3N\}$.\\
    \quad \\
    Then, following
    the same argument of $L^{\infty}$ bounds in Section 4.1, \eqref{55d3h},\eqref{55d5h} and \eqref{55d9h} are bounded by $\frac{C_{N,\lambda}}{\varepsilon^3}\sup_{0\leq s \leq T_0}\| \mathbf{f}_R(s)\|_{L^{2}_{x,v}}$. Next we estimate $|D_x(\frac{dv_*}{\zeta_{\alpha\beta}})|$ to bound the other part. Noting that
    \begin{equation*}
        D_x[\text{det}(\frac{dv_*}{d\zeta_{\alpha\beta}})]=D_x[\frac{1}{\text{det}(\frac{d\zeta_{\alpha\beta}}{dv_*})}]=-\frac{1}{[\text{det}(\frac{d\zeta_{\alpha\beta}}{dv_*})]^2}D_x[\text{det}(\frac{d\zeta_{\alpha\beta}}{dv_*})].
    \end{equation*}
    Lemma \ref{lemofcharacter} implies
    \begin{equation*}
        \frac{1}{4(s-s_1)^6}\leq \frac{1}{\text{det}(\frac{d\zeta_{\alpha\beta}}{dv_*})} \leq \frac{4}{(s-s_1)^6}.
    \end{equation*}
    Let us set
    \begin{equation*}
        X^{\alpha,\beta}_1(s_1)=\Big(X^{\alpha,\beta}_{1,1}(s_1),X^{\alpha,\beta}_{1,2}(s_1),X^{\alpha,\beta}_{1,3}(s_1)\Big), \quad v_*=(v_{*1},v_{*2},v_{*3}),
    \end{equation*}
    then the derivative matrix can be specifically represented as
    \begin{equation*}
        (\frac{d\zeta_{\alpha\beta}}{dv_*}) =   \begin{pmatrix}
            & \partial_{v_{*1}}X^{\alpha,\beta}_{1,1}(s_1) & \partial_{v_{*2}}X^{\alpha,\beta}_{1,1}(s_1)  & \partial_{v_{*3}}X^{\alpha,\beta}_{1,1}(s_1) \\
            & \partial_{v_{*1}}X^{\alpha,\beta}_{1,2}(s_1) & \partial_{v_{*2}}X^{\alpha,\beta}_{1,2}(s_1) & \partial_{v_{*3}}X^{\alpha,\beta}_{1,2}(s_1) \\
            & \partial_{v_{*1}}X^{\alpha,\beta}_{1,3}(s_1) & \partial_{v_{*2}}X^{\alpha,\beta}_{1,3}(s_1)  & \partial_{v_{*3}}X^{\alpha,\beta}_{1,3}(s_1)
        \end{pmatrix}.
    \end{equation*}
    By \eqref{ASSPDREVYVxXx}, the determinant is bounded by
    \begin{equation*}
        |D_x[\text{det}(\frac{d\zeta_{\alpha\beta}}{dv_*})]|\lesssim |\partial_{v_*}X_1^{\alpha,\beta}(s_1)|^2 \, |\partial_x\partial_{v_*}X_1^{\alpha,\beta}(s_1)|\lesssim (s-s_1)^2|\partial_x\partial_{v_*}X_1^{\alpha,\beta}(s_1)|.
    \end{equation*}
     Since $s_1-s\geq \lambda \varepsilon$, it yields that
    \begin{equation*}
        D_x[\text{det}(\frac{dv_*}{d\zeta_{\alpha\beta}})]=-\frac{1}{[\text{det}(\frac{d\zeta_{\alpha\beta}}{dv_*})]^2}D_x[\text{det}(\frac{d\zeta_{\alpha\beta}}{dv_*})]\lesssim \frac{|\partial_x\partial_{v_*}X_1^{\alpha,\beta}(s_1)|}{(s-s_1)^4}\lesssim \frac{|\partial_x\partial_{v_*}X_1^{\alpha,\beta}(s_1)|}{(\lambda\varepsilon)^4}.
    \end{equation*}
    From the above analysis, we conclude that
    \begin{gather*}
        \begin{split}
            &\|D_x[\text{det}(\frac{d\zeta_{\alpha\beta}}{dv_*})]\|^2_{L^{2}_{x,v}(\hat{D})}\leq \frac{C_{\lambda}}{\varepsilon^4}\|\partial_x\partial_{v_*}X_1^{\alpha,\beta}(s_1)\|_{L^{2}_{x,v}(\hat{D})}\\
            &=\frac{C_{\lambda}}{\varepsilon^4}\int_{\hat{D}}|\partial_x\partial_{v_*}X^{\beta}(s_1;s,z,v_*)|^2\text{det}\Big(\frac{\partial X^{\beta}(s_1;s,X^{\alpha}(s,x,v),v_{*})}{\partial X^{\alpha}(s,x,v)}\Big)dz dv_{**}dv_*ds_1ds\\
            &\leq \frac{C_{\lambda,N}}{\varepsilon^4}.
        \end{split}
    \end{gather*}
    Thus, the main integral admits the bound
    \begin{gather*}
        \begin{split}
            &\frac{C}{\varepsilon^2}\Big|\int_{0}^{t} \mathcal{F}^{\alpha}_0(t,s) \int_{D} \sum_{\beta = A,B}\int_{0}^{s-\varepsilon\lambda}\mathcal{F}^{\beta}_1(s,s_1) k_{M,N}^{\alpha\beta}(s,X^{\alpha}(s),V^{\alpha}(s),v_*)\\
            &\hspace{0.5cm}\times  \sum_{\beta' = A,B} k_{M,N}^{\beta \beta'}(s_1,\zeta_{\alpha\beta},V^{\alpha,\beta}_1(s_1),v_{**}) \mathbf{h} (s_1,X^{\alpha,\beta}_1(s_1),v_{**})D_x(\frac{dv_*}{\zeta_{\alpha\beta}})ds_1 dv_{**}d\zeta_{\alpha\beta} ds\Big|\\
            &\leq \frac{C_{\lambda,T_0,N}}{\varepsilon^4}\sup_{0\leq s \leq T_0}\| \mathbf{f}_R(s)\|_{L^{2}_{x,v}}.
        \end{split}
    \end{gather*}
    Consequently, it holds that
    \begin{gather*}\label{HDDD22E4}
        \begin{split}
            \mathcal{D}_{2,x}^\alpha(t)
            \leq & (\frac{C}{N }+C_{\eta}\,e^{-\frac{\eta}{8} N^2}+C \lambda)\, e^{-\frac{\nu_ot}{2\varepsilon}}\sup_{0\leq s \leq t}(e^{\frac{\nu_os}{2\varepsilon}}\Vert D_x \mathbf{h}(s) \Vert_{L^{\infty}_{x,v}})+\frac{C_{N,\lambda,T_0}}{\varepsilon^{4}}  \sup_{0\leq s \leq t}\| \mathbf{f}_R(s)\|_{L^{2}_{x,v}}.
        \end{split}
    \end{gather*}
     In addition, $I_{2x}$ admits the following bound
    \begin{gather*}
        \begin{split}
            I_{2x}  & \leq \, C\frac{t}{\varepsilon}e^{-\frac{\nu_ot}{\varepsilon}}\| D_x\mathbf{h}^{\rm in} \|_{L^\infty_{x,v}}+C\varepsilon^ke^{-\frac{\nu_ot}{\varepsilon}} \sup_{0\leq s \leq t}[(e^{\frac{\nu_os}{2\varepsilon}}\|  \mathbf{h}\|_{L^{\infty}_{x,v}})^2+(e^{\frac{\nu_os}{2\varepsilon}}\|  D_x\mathbf{h}\|_{L^{\infty}_{x,v}})^2]\\
            & \hspace{0.4cm} +[C(1+\mathcal{I}_1)\varepsilon+C{\lambda}+\frac{C_{\lambda}}{N}+C\delta^{4}]e^{-\frac{\nu_ot}{2\varepsilon}} \sup_{0\leq s \leq t}(e^{\frac{\nu_os}{2\varepsilon}}\|  \mathbf{h}\|_{L^{\infty}_{x,v}}+e^{\frac{\nu_os}{2\varepsilon}}\|  D_x\mathbf{h}\|_{L^{\infty}_{x,v}})\\
            & \hspace{0.4cm} + \frac{C_{\lambda,T_0,N}}{\varepsilon^4}\sup_{0\leq s \leq T_0}\| \mathbf{f}_R(s)\|_{L^{2}_{x,v}}+\frac{C_{N,\lambda,T_0}}{\varepsilon^3}e^{-\frac{\nu_ot}{2\varepsilon}}\sup_{0\leq s \leq  T_0}(e^{\frac{\nu_os}{2\varepsilon}}\Vert \mathbf{h}\Vert_{L^{\infty}_{x,v}})\\
            & \hspace{0.4cm}+\varepsilon  e^{-\frac{\nu_ot}{2\varepsilon}} \sup_{0\leq s \leq t}(e^{\frac{\nu_os}{2\varepsilon}}\Vert \nabla_{v}\mathbf{h}\Vert_{L^{\infty}_{x,v}})+C\varepsilon^k\mathcal{I}_2.
        \end{split}
    \end{gather*}

For fixed $\lambda>0$, selecting $N>0$ sufficiently large, we show that
    \begin{align}
            \sup_{0\leq s \leq  T_0}(e^{\frac{\nu_os}{2\varepsilon}}&\Vert D_x\mathbf{h}\Vert_{L^{\infty}_{x,v}}) \leq C (\Vert  \mathbf{h}^{\rm in}\Vert^2_{L^{\infty}_{x,v}}+\Vert D_x\mathbf{h}^{\rm in}\Vert_{L^{\infty}_{x,v}})+C\varepsilon \sup_{0\leq s \leq T_0}(e^{\frac{\nu_os}{2\varepsilon}}\Vert \nabla_{v}\mathbf{h}\Vert_{L^{\infty}_{x,v}})\notag\\
            &+(\frac{C_{\lambda}}{N}+C\lambda+C\varepsilon+\delta^4)\sup_{0\leq s \leq  T_0}(e^{\frac{\nu_os}{2\varepsilon}}\Vert D_x\mathbf{h}\Vert_{L^{\infty}_{x,v}}+e^{\frac{\nu_os}{2\varepsilon}}\Vert \mathbf{h}\Vert_{L^{\infty}_{x,v}})\label{44W1INFXXXCL}\\
            &+\frac{C}{\varepsilon^3}\sup_{0\leq s \leq  T_0}(e^{\frac{\nu_os}{2\varepsilon}}\Vert \mathbf{h}\Vert_{L^{\infty}_{x,v}})+\frac{C_{\lambda,N}}{\varepsilon^4}e^{\frac{\nu_oT_0}{2\varepsilon}}\sup_{0\leq s \leq T_0}\| \mathbf{f}_R(s)\|_{L^{2}_{x,v}}\notag\\
            &+C\varepsilon^k \sup_{0\leq s \leq T_0}[(e^{\frac{\nu_os}{2\varepsilon}}\|  \mathbf{h}\|_{L^{\infty}_{x,v}})^2+(e^{\frac{\nu_os}{2\varepsilon}}\| D_x \mathbf{h}\|_{L^{\infty}_{x,v}})^2]+C\varepsilon^{k-1}e^{\frac{\nu_oT_0}{2\varepsilon}}\notag.
    \end{align}
    \begin{remark}\label{GYKGET6RE}
        From the above estimate, it can be seen that the process of transforming $L^{\infty}$
        into $L^2$ through integration by parts will generate the term $\frac{C_{\lambda,N}}{\varepsilon^4} e^{\frac{\nu_0 T_0}{2\varepsilon}}$. Due to the coefficient $C_{\lambda,N}$ in front of it, we need an additional $\varepsilon$ to multiply with it in order to make it smaller. Hence, we require $\varepsilon^5$. Therefore, the index $k$ of the expansion term must be greater than or equal to $6$.
    \end{remark}
    We note that \eqref{44W1INFXXXCL} contains the term $C\varepsilon \displaystyle{\sup_{0\leq s \leq T_0}}(e^{\frac{\nu_os}{2\varepsilon}}\Vert \nabla_{v}\mathbf{h}\Vert_{L^{\infty}_{x,v}})$. As a result, we need to estimate $\| D_v \mathbf{h}\|_{L^{\infty}_{x,v}}$ to close the energy.

    \subsection{$W_v^{1,\infty}$ estimate } Taking
    $D_v$ on both sides of \eqref{HDSTLINFEQMAIN}, one obtains
    \begin{eqnarray}\label{linfVfHADTVVODXT}
        \begin{split}
            &\hspace{0.5cm}\partial_{t}(D_vh^{\alpha})+v\cdot\nabla_{x}(D_vh^{\alpha})+\frac{e^{\alpha}}{m^{\alpha}} \nabla_{x}\phi^{\varepsilon}\cdot\nabla_{v}  (D_vh^{\alpha}) + \frac{\upsilon^\alpha}{\varepsilon}  (D_vh^{\alpha}) \\
            &=-D_xh^{\alpha}-\frac{D_v\upsilon^\alpha}{\varepsilon}h^{\alpha}-D_v[ \frac{1}{\varepsilon}(K_{M,\left \langle v \right\rangle^{l}}^{\alpha,\chi}+K_{M,\left \langle v \right\rangle^{l}}^{\alpha,c})\mathbf{h}] + \sum_{i=1}^{6}D_v, (\mathcal{H}^\alpha_i),
        \end{split}
    \end{eqnarray}
    where $\mathcal{H}^\alpha_i$ is described in \eqref{HD00HI6}. Along the trajectory, the solution $D_vh^{\alpha}$ of the equation \eqref{linfVfHADTVVODXT} can be expressed as
        \begin{align}
            D_vh^\alpha(t,x,v) =\, &\mathcal{F}^\alpha_0(t,0)\,(D_vh^\alpha)(0,X^{\alpha}(0;t,x,v),V^{\alpha}(0;t,x,v)) \notag\\
            &-\int_{0}^{t}\mathcal{F}^\alpha_0(t,s) (D_xh^\alpha) (s, X^{\alpha}(s),V^{\alpha}(s))ds\notag\\
            &- \frac{1}{\varepsilon} \int_{0}^{t}\mathcal{F}^\alpha_0(t,s) (D_v\upsilon^\alpha) \, h^{\alpha} (s, X^{\alpha}(s),V^{\alpha}(s))ds\notag\\
            & + \int_{0}^{t}\mathcal{F}^\alpha_0(t,s) \sum_{i=1}^{6} (D_v\mathcal{H}^\alpha_j) (s, X^{\alpha}(s),V^{\alpha}(s))ds  \label{linfVHADTVVODXT}\\
            &- \frac{1}{\varepsilon} \int_{0}^{t}\mathcal{F}^\alpha_0(t,s)  D_v (K_{M,\left \langle v \right\rangle^{l}}^{\alpha,\chi} \mathbf{h}) (s, X^{\alpha}(s),V^{\alpha}(s))ds\notag\\
            &- \frac{1}{\varepsilon} \int_{0}^{t}\mathcal{F}^\alpha_0(t,s)  D_v (K_{M,\left \langle v \right\rangle^{l}}^{\alpha,c} \mathbf{h}) (s, X^{\alpha}(s),V^{\alpha}(s))ds\notag.
        \end{align}
    As in the spatial derivative case, the first line and the second line in \eqref{linfVHADTVVODXT} can be estimated
    \begin{equation*}
        \left|\mathcal{F}_0^{\alpha}(t,0) (D_vh^\alpha) (0,X^{\alpha}(0;t,x,v),V^{\alpha}(0;t,x,v))\right| \lesssim \varepsilon  e^{-\frac{\nu_ot}{\varepsilon}} \|D_v\mathbf{h}^{in}\|_{L^\infty_{x,v}},
    \end{equation*}
     and
    \begin{gather*}
        \begin{split}
            &\left|\int_{0}^{t}\mathcal{F}^\alpha_0(t,s) (D_xh^\alpha) (s, X^{\alpha}(s),V^{\alpha}(s))ds\right|
            \lesssim \, \varepsilon  e^{-\frac{\nu_ot}{\varepsilon}} \sup_{0\leq s \leq t}(e^{\frac{\nu_os}{2\varepsilon}}\Vert D_xh^{\alpha}\Vert_{L^{\infty}_{x,v}}).
        \end{split}
    \end{gather*}
    Moreover, from the expression in \eqref{nuKsuanzidef} for hard sphere case ($\gamma=1$), we have
    \begin{gather*}
        \begin{split}
            |D_v\upsilon^\alpha|\leq &\sum_{\beta=A,B} \int_{\mathbb{R}^3\times \mathbb{S}^2}
            |D_vB^{ \alpha \beta}|(|v-v_*|, \cos \theta) \mu^{\beta}(v_*) d\sigma dv_*
             \lesssim  \left \langle v \right\rangle,
        \end{split}
    \end{gather*}
    then the third line in \eqref{linfVHADTVVODXT} is bounded by
    \begin{gather*}
        \begin{split}
            &\frac{1}{\varepsilon}\left|\int_{0}^{t}\mathcal{F}^\alpha_0(t,s) (D_v\upsilon^\alpha) h^{\alpha}(s, X^{\alpha}(s),V^{\alpha}(s))ds\right|
            \lesssim \,  e^{-\frac{\nu_ot}{2\varepsilon}} \sup_{0\leq s \leq t}(e^{\frac{\nu_os}{2\varepsilon}}\Vert h^{\alpha}\Vert_{L^{\infty}_{x,v}}).
        \end{split}
    \end{gather*}
    Now, let us turn to the fourth line in \eqref{linfVHADTVVODXT}. First, from lemma \ref{DDFXXEST}, one gets
    \begin{equation*}
        \sum_{\beta=A,B}\Big|D_v\Big[\frac{ \left \langle v \right\rangle^{l}}{\sqrt{\mu^{\alpha}_M}} Q^{ \alpha\beta }(\frac{\sqrt{\mu^{\alpha}_M}}{\left \langle v \right\rangle^{l}}h^\alpha, \frac{\sqrt{\mu^{\beta}_M}}{\left \langle v \right\rangle^{l}} h^\beta)\Big]\Big|\lesssim \, \left \langle v \right\rangle \Vert \mathbf{h}\Vert_{L^{\infty}_{x,v}} (\Vert \left \langle v \right\rangle\mathbf{h}\Vert_{L^{\infty}_{x,v}}+\Vert D_v\mathbf{h}\Vert_{L^{\infty}_{x,v}}),
    \end{equation*}
    which implies
    \begin{align*}
            &\left| \int_{0}^{t}\mathcal{F}^\alpha_0(t,s) (D_v\mathcal{H}^\alpha_1) (s, X^{\alpha}(s),V^{\alpha}(s))ds\right| \\
            &\hspace{0.4cm}\lesssim \,\varepsilon^k  e^{-\frac{\nu_ot}{\varepsilon}} \sup_{0\leq s \leq t}\{(e^{\frac{\nu_os}{2\varepsilon}}\|  \mathbf{h}\|_{L^{\infty}_{x,v}})[(e^{\frac{\nu_os}{2\varepsilon}}\| \left \langle v \right\rangle  \mathbf{h}\|_{L^{\infty}_{x,v}})+(e^{\frac{\nu_os}{2\varepsilon}}\|  D_v\mathbf{h}\|_{L^{\infty}_{x,v}})]\}.
    \end{align*}
   Next, by the exponential decay property of $F_i$, the following estimate holds:
    \begin{gather*}
        \begin{split}
            |D_v\mathcal{H}_2^{\alpha}|
            \lesssim   \, \left \langle v \right\rangle \mathcal{I}_1 \,(\Vert \left \langle v \right\rangle
            \mathbf{h} \Vert_{L^\infty_{x,v}}+\Vert D_v\mathbf{h} \Vert_{L^\infty_{x,v}}),
        \end{split}
    \end{gather*}
    and
    \begin{align*}
            &\left| \int_{0}^{t}\mathcal{F}^\alpha_0(t,s) (D_v\mathcal{H}^\alpha_2) (s, X^{\alpha}(s),V^{\alpha}(s))ds\right|\\
            &\hspace{0.4cm}\lesssim \,\varepsilon \mathcal{I}_1 e^{-\frac{\nu_ot}{2\varepsilon}} \sup_{0\leq s \leq t}(e^{\frac{\nu_os}{2\varepsilon}}\| \left \langle v \right\rangle  \mathbf{h}\|_{L^{\infty}_{x,v}}+e^{\frac{\nu_os}{2\varepsilon}}\| D_v \mathbf{h}\|_{L^{\infty}_{x,v}}).
    \end{align*}
    Since $D_v[\frac{\left \langle v \right\rangle^{l}}{\sqrt{\mu^{\alpha}_M}} \nabla_{v}(\frac{\sqrt{\mu^{\alpha}_M}}{\left \langle v \right\rangle^{l}})]\leq C \left \langle v \right\rangle^{2}$, it satisfies
    \begin{eqnarray*}
        \begin{split}
            |D_v\mathcal{H}^\alpha_3|
            \lesssim \left \langle v \right\rangle (\Vert\left \langle v \right\rangle \mathbf{h} \Vert_{L^{\infty}_{x,v}}+\Vert D_v\mathbf{h} \Vert_{L^{\infty}_{x,v}}).
        \end{split}
    \end{eqnarray*}
    By \eqref{decayLxvFf} and the assumption \eqref{SUPPOSELINFGXV}, one gets
    \begin{eqnarray*}
        \begin{split}
            |D_v\mathcal{H}^\alpha_4|
            \lesssim (1+\varepsilon\mathcal{I}_1)\left \langle v \right\rangle (\Vert\mathbf{h} \Vert_{L^{\infty}_{x,v}}+\Vert D_x\mathbf{h} \Vert_{L^{\infty}_{x,v}}),
        \end{split}
    \end{eqnarray*}
    and
    \begin{gather*}
        \begin{split}
            &\left| \int_{0}^{t}\mathcal{F}^\alpha_0(t,s) (D_v\mathcal{H}^\alpha_3+D_v\mathcal{H}^\alpha_4)(s, X^{\alpha}(s),V^{\alpha}(s))ds\right| \\
            &\hspace{0.3cm}\lesssim \,\, \varepsilon (1+\varepsilon\mathcal{I}_1) e^{-\frac{\nu_ot}{2\varepsilon}} \sup_{0\leq s \leq t}(e^{\frac{\nu_os}{2\varepsilon}}\| \left \langle v \right\rangle  \mathbf{h}\|_{L^{\infty}_{x,v}}+e^{\frac{\nu_os}{2\varepsilon}}\| D_v \mathbf{h}\|_{L^{\infty}_{x,v}}).
        \end{split}
    \end{gather*}
 Finally, Lemma \ref{FI2KM1EST} provides the estimate of
    $\mathbf{F}_i$, then we have
    \begin{align*}
            &\hspace{1.5cm}|D_v\mathcal{H}^\alpha_5| +|D_v\mathcal{H}^\alpha_6| \lesssim   \, \varepsilon^{k-1} \left\langle v \right\rangle
                \mathcal{I}_2,\\
            &\left| \int_{0}^{t}\mathcal{F}^\alpha_0(t,s) (D_v\mathcal{H}^\alpha_5+D_v\mathcal{H}^\alpha_6)(s, X^{\alpha}(s),V^{\alpha}(s))ds \right|
            \lesssim \, \varepsilon^k \mathcal{I}_2.
     \end{align*}
    For the singular part in the fifth line, we first split
    \begin{equation*}
        (D_vK_{M,\left \langle v \right\rangle^{l}}^{\alpha,\chi} \mathbf{h})=(D_vK_{M,\left \langle v \right\rangle^{l}}^{\alpha,\chi})\mathbf{h}+K_{M,\left \langle v \right\rangle^{l}}^{\alpha,\chi}(D_v\mathbf{h}).
    \end{equation*}
    Employing Lemma \ref{kchiestmainle}\,($\gamma=1$), one has
    \begin{gather*}
        \begin{split}
            |K_{M,\left \langle v \right\rangle^{l}}^{\alpha,\chi}(D_v\mathbf{h})(v) | &\lesssim \, {\delta}^{3+1} \left \langle v \right\rangle e^{-c|v|^2} \| D_v\mathbf{h} \|_{L^\infty_{x,v}}, \\
            |(D_vK_{M,\left \langle v \right\rangle^{l}}^{\alpha,\chi})\mathbf{h}(v) | &\lesssim \, {\delta}^{2+1} \left \langle v \right\rangle e^{-c|v|^2} \| \mathbf{h} \|_{L^\infty_{x,v}}.
        \end{split}
    \end{gather*}
    Under hard sphere conditions, one obtains
    \begin{align*}
            &\frac{1}{\varepsilon}\left| \int_{0}^{t}\mathcal{F}^\alpha_0(t,s) D_v(K_{M,\left \langle v \right\rangle^{l}}^{\alpha,\chi} \mathbf{h})(s, X^{\alpha}(s),V^{\alpha}(s))ds\right|\\
            &\hspace{0.4cm}\lesssim  \,\, \delta^{3}  e^{-\frac{\nu_ot}{2\varepsilon}} \sup_{0\leq s \leq t}(e^{\frac{\nu_os}{2\varepsilon}}\|  \mathbf{h}\|_{L^{\infty}_{x,v}}+e^{\frac{\nu_os}{2\varepsilon}}\|D_v \mathbf{h}\|_{L^{\infty}_{x,v}}).
    \end{align*}

    Finally, the analysis now concentrates on the sixth line, where the regular part admits the decomposition
    \begin{equation*}
        D_v(K_{M,\left \langle v \right\rangle^{l}}^{\alpha,c} \mathbf{h})=(D_vK_{M,\left \langle v \right\rangle^{l}}^{\alpha,c})\mathbf{h}+K_{M,\left \langle v \right\rangle^{l}}^{\alpha,c}(D_v\mathbf{h}),
    \end{equation*}
    which satisfies
    \begin{eqnarray}\label{HDVVVK2CDXJFBD}
        \begin{split}
            &(D_vK_{M,\left \langle v \right\rangle^{l}}^{\alpha,c}) \mathbf{h}(s, X^{\alpha}(s),V^{\alpha}(s))\\
            =&\sum_{\beta=A,B}\int_{\mathbb{R}^3} (D_vk_{M,w}^{\alpha \beta}) (s,X^{\alpha}(s),V^{\alpha}(s),v_*)h^{\beta}(s,X^{\alpha}(s),v_*) d v_{*},
        \end{split}
    \end{eqnarray}
    and
    \begin{eqnarray}\label{DHVVVK2CDXJFBD}
        \begin{split}
            &K_{M,\left \langle v \right\rangle^{l}}^{\alpha,c} (D_v\mathbf{h})(s, X^{\alpha}(s),V^{\alpha}(s))\\
            =&\sum_{\beta=A,B}\int_{\mathbb{R}^3} k_{M,w}^{\alpha \beta} (s,X^{\alpha}(s),V^{\alpha}(s),v_*)(D_vh^{\beta})(s,X^{\alpha}(s),v_*) d v_{*}.
        \end{split}
    \end{eqnarray}
    We introduce
    \begin{gather*}
        \begin{split}
            & \frac{1}{\varepsilon}\int_{0}^{t}\mathcal{F}^\alpha_0(t,s) (D_vK_{M,\left \langle v \right\rangle^{l}}^{\alpha,c})\mathbf{h}(s, X^{\alpha}(s),V^{\alpha}(s))ds:=I_{1v},\\
            & \frac{1}{\varepsilon}\int_{0}^{t}\mathcal{F}^\alpha_0(t,s) K_{M,\left \langle v \right\rangle^{l}}^{\alpha,c}(D_v\mathbf{h})(s, X^{\alpha}(s),V^{\alpha}(s))ds:=I_{2v}.
        \end{split}
    \end{gather*}
  Then, substituting \eqref{HDliterative1111} into \eqref{HDVVVK2CDXJFBD}, leads to
    \begin{equation}\label{xxxdiccJOX}
        I_{1v}=\frac{1}{\varepsilon} \int_{0}^{t}\mathcal{F}_0^{\alpha}(t,s) (D_vK_{M,\left \langle v \right\rangle^{l}}^{\alpha,c})  \mathbf{h}(s,X^{\alpha}(s),V^{\alpha}(s))ds  = \sum_{i=0}^{2} \mathcal{J}^{\alpha}_{i,v}(t) + \mathcal{O}_{v}^\alpha(t).
    \end{equation}
   The expression \eqref{K222XVE} implies
    \begin{equation*}
        |(D_v  K_{M,\left \langle v \right\rangle^{l}}^{\alpha,c} )\mathbf{h} (s,X^{\alpha}(s),V^{\alpha}(s)) |\leq C  \left \langle V^{\alpha}(s) \right\rangle^{\gamma}  | \left \langle v_* \right\rangle\mathbf{h} |_{L^\infty_v}.
    \end{equation*}
    Note that
    \begin{eqnarray*}
        | K_{M,\left \langle v \right\rangle^{l}}^{\alpha,c}\mathbf{h}(v) |+| (D_x  K_{M,\left \langle v \right\rangle^{l}}^{\alpha,c})\mathbf{h}(v) | \leq C  \left \langle v \right\rangle^{\gamma-1}  |\mathbf{h} |_{L^\infty_v},
    \end{eqnarray*}
    which means that $D_v$  introduces an additional polynomial weight $\left \langle v_* \right\rangle$ to the operator $K_{M,\left \langle v \right\rangle^{l}}^{\alpha,c}$. For $|v_*|\leq 1$ and $T_0$ small enough, it holds that
    \begin{equation*}\label{les1estTZX}
        \sup_{s\in [0,T_0]}|V^{\alpha,\beta}_1(s_1)|\leq C.
    \end{equation*}
    When $|v_*|> 1$,  for $0\leq s\leq t \leq T_0$, we have
    \begin{equation*}\label{ges1estTZX}
        \frac{|v_*|}{2} \leq    |V^{\alpha,\beta}_1(s_1)|\leq 2 |v_*|.
    \end{equation*}
   The conversion from $v_*$ to $V^{\alpha,\beta}_1(s_1)$ leads to
    \begin{eqnarray*}
        \begin{split}
            \mathcal{J}^{\alpha}_{0,v}(t)=&\frac{1}{\varepsilon} \int_{0}^{t} \mathcal{F}_0^\alpha(t,s) \sum_{\beta=A,B}\int_{\mathbb{R}^3}     (D_vk_{M,w}^{\alpha\beta})(s,X^{\alpha}(s),V^{\alpha}(s),v_*) \\  &\hspace{1.9cm}\times\mathcal{F}_1^\beta(s,0) h^\beta(0,X^{\alpha,\beta}_1(0),V^{\alpha,\beta}_1(0))  dv_*ds \\
            \lesssim &\frac{t}{\varepsilon}e^{-\frac{\nu_ot}{\varepsilon}} \|\left \langle v \right\rangle\mathbf{h}^{in}\|_{L^{\infty}_{x,v}},\\
            \mathcal{J}^{\alpha}_{1,v}(t)=&\frac{1}{\varepsilon} \int_{0}^{t} \mathcal{F}_0^\alpha(t,s) \sum_{\beta=A,B}\int_{\mathbb{R}^3} (D_vk_{M,w}^{\alpha\beta})(s,X^{\alpha}(s), V^{\alpha}(s), v_*)\\
            &\hspace{0.5cm}\times\sum_{j=1}^6 \int_{0}^s \mathcal{F}_1^\beta(s,s_1) \mathcal{H}^\beta_j(s_1,X_1^{\alpha,\beta}(s_1),V_1^{\alpha,\beta}(s_1))ds_1  dv_*ds \\
            \lesssim & ({1+\mathcal{I}_1}) \varepsilon e^{-\frac{\nu_ot}{2\varepsilon}}\sup_{0\leq s \leq t}(e^{\frac{\nu_os}{2\varepsilon}}\Vert  \left \langle v \right\rangle\mathbf{h}(s) \Vert_{L^{\infty}_{x,v}})+\mathcal{I}_2 \varepsilon^k \\
            &\hspace{2.7cm}+\varepsilon^k e^{-\frac{\nu_ot}{\varepsilon}} \sup_{0\leq s \leq t}[(e^{\frac{\nu_os}{2\varepsilon}}\Vert\left \langle v \right\rangle \mathbf{h}(s) \Vert_{L^{\infty}_{x,v}})^2].
        \end{split}
    \end{eqnarray*}
   The equation \eqref{ZZZzhuanghuagezV} further reveals
    \begin{eqnarray*}
        \begin{split}
            \frac{V^{\alpha}(s)}{v_*}\frac{w_{\gamma}(V^{\alpha}(s))}{w_{\gamma}(v_*)}|k_{M,1}^{\alpha \beta}(V^{\alpha}(s),v_*)|\leq C |v-v_*|^{\gamma}\,\,e^{-c(|V^{\alpha}(s)|^2+|v_*|^2)},
        \end{split}
    \end{eqnarray*}
    and
    \begin{align*}
            &\frac{V^{\alpha}(s)}{v_*}\frac{w_{\gamma}(V^{\alpha}(s))}{w_{\gamma}(v_*)}|k_{M,2}^{\alpha \beta}(V^{\alpha}(s),v_*)|\\
            &\hspace{0.4cm}\leq C \frac{(1+|V^{\alpha}(s)|+|v_*|)^{\gamma-1}}{|V^{\alpha}(s)-v_*|} \, e^{-c(|V^{\alpha}(s)-v_*|^2+\frac{||V^{\alpha}(s)|^2-|v_*|^2|^2}{|V^{\alpha}(s)-v_*|^2})}.
    \end{align*}
    Through conversion chain $v_* \to V^{\alpha,\beta}_1(s_1) \to v_{**}$ and \eqref{kchiestmain}, the final term satisfies
    \begin{eqnarray*}
        \begin{split}
            &\mathcal{J}^{\alpha}_{2,v}(t)=\frac{1}{\varepsilon^2} \int_{0}^{t} \mathcal{F}_0^\alpha(t,s) \sum_{\beta=A,B}\int_{\mathbb{R}^3} (D_vk_{M,w}^{\alpha\beta})(s,X^{\alpha}(s), V^{\alpha}(s), v_*)\\
            &\hspace{2cm} \times \int_{0}^s \mathcal{F}_1^\beta(s,s_1) K_{M,\left \langle v \right\rangle^{l}}^{\alpha,\chi}  \mathbf{h} (s_1,X^{\alpha,\beta}_1(s_1),V^{\alpha,\beta}_1(s_1))ds_1 dv_* ds \\
            & \hspace{1cm} \lesssim \, \delta^{2+\gamma}e^{-\frac{\nu_ot}{2\varepsilon}}\sup_{0\leq s \leq t} (e^{\frac{\nu_os}{2\varepsilon}}\Vert \mathbf{h}(s) \Vert_{L^{\infty}_{x,v}}).
        \end{split}
    \end{eqnarray*}
    The last term in \eqref{xxxdiccJOX} is expressed by
    \begin{gather*}
        \begin{split}
            \mathcal{O}_v^\alpha(t) = & \frac{1}{\varepsilon^2} \int_{0}^{t} \mathcal{F}^{\alpha}_0(t,s) \sum_{\beta=A,B}\int_{\mathbb{R}^3} (D_vk_{M,\left \langle v \right\rangle^{l}}^{\alpha\beta}) (s,X^{\alpha}(s),V^{\alpha}(s),v_*)  \int_{0}^s \mathcal{F}_1^\beta(s,s_1) \\
            &\hspace{4cm} \times (K_{M,\left \langle v \right\rangle^{l}}^{\beta,c} \mathbf{h}) (s_1,X^{\alpha,\beta}_1(s_1),V^{\alpha,\beta}_1(s_1))ds_1 dv_* ds\\
            \leq& \frac{1}{\varepsilon^2} \int_{0}^{t} \mathcal{F}^{\alpha}_0(t,s) \sum_{\beta=A,B}\int_{\mathbb{R}^3} |D_v(k_{M,\left \langle v \right\rangle^{l}}^{\alpha\beta})|(s,X^{\alpha}(s),V^{\alpha}(s),v_*) \sum_{\beta'=A,B}\int_{0}^s \mathcal{F}_1^\beta(s,s_1) \\
            &\hspace{0.5cm} \times \int_{\mathbb{R}^3}  k_{M,\left \langle v \right\rangle^{l}}^{\beta\beta^{'}}(s_1,X^{\alpha,\beta}_1(s_1),V^{\alpha,\beta}_1(s_1),v_{**}) |h^{\beta'}| (s_1,X^{\alpha,\beta}_1(s_1),v_{**}) dv_{**}ds_1 dv_* ds.
        \end{split}
    \end{gather*}
    Building upon the transformation from $v_{*}$ to $v_{**}$, Cases 1-3b in Section 4.1 are reapplied with Lemma~\ref{LeMK2ker4444}'s estimates for $D_v(k_{M,\langle v \rangle^{l}}^{\alpha\beta})$, that is
    \begin{equation*}
        \begin{split}
            |\mathcal{O}_v^\alpha(t)| \leq & \frac{C}{N }\, e^{-\frac{\nu_ot}{2\varepsilon}}\sup_{0\leq s \leq t}(e^{\frac{\nu_os}{2\varepsilon}}\Vert \left \langle v \right \rangle \mathbf{h}(s) \Vert_{L^{\infty}_{x,v}})+C_{\eta}\,e^{-\frac{\eta}{8} N^2}\, e^{-\frac{\nu_ot}{2\varepsilon}}\sup_{0\leq s \leq t}(e^{\frac{\nu_os}{2\varepsilon}}\Vert \left \langle v \right \rangle \mathbf{h}(s) \Vert_{L^{\infty}_{x,v}})\\
            &+C \lambda  e^{-\frac{\nu_ot}{2\varepsilon}}\sup_{0\leq s \leq t}(e^{\frac{\nu_os}{2\varepsilon}}\Vert \left \langle v \right \rangle \mathbf{h}(s) \Vert_{L^{\infty}_{x,v}})+\frac{C_{N,\lambda,T_0}}{\varepsilon^{3/2}}  \sup_{0\leq s \leq t}\| \mathbf{f}_R(s)\|_{L^{2}_{x,v}}.
        \end{split}
    \end{equation*}
    Hence, the $I_{1v}$ bound emerges as:
    \begin{gather*}
        \begin{split}
            I_{1v}  & \leq \, C\frac{t}{\varepsilon}e^{-\frac{\nu_ot}{\varepsilon}}\| \left \langle v \right \rangle \mathbf{h}^{\rm in} \|_{L^\infty_{x,v}}+C\varepsilon^ke^{-\frac{\nu_ot}{\varepsilon}} \sup_{0\leq s \leq t}(e^{\frac{\nu_os}{2\varepsilon}}\| \left \langle v \right \rangle \mathbf{h}\|_{L^{\infty}_{x,v}})^2\\
            & \hspace{0.4cm} +[C(1+\mathcal{I}_1)\varepsilon+C{\lambda}+\frac{C_{\lambda}}{N}+C\delta^{2+1}]e^{-\frac{\nu_ot}{2\varepsilon}} \sup_{0\leq s \leq t}(e^{\frac{\nu_os}{2\varepsilon}}\|  \left \langle v \right \rangle\mathbf{h}\|_{L^{\infty}_{x,v}})\\
            & \hspace{0.4cm} + \frac{C_{N,\lambda}}{\varepsilon^{\frac{3}{2}}}\sup_{0\leq s \leq T_0}\| \mathbf{f}_R(s)\|_{L^{2}_{x,v}}+C\varepsilon^k\mathcal{I}_2.
        \end{split}
    \end{gather*}

    Next, we turn to $I_{2v}$, and
     substitute \eqref{linfVHADTVVODXT} into \eqref{DHVVVK2CDXJFBD} to obtain
    \begin{equation}\label{xx45xccJOX}
        I_{2v}=\frac{1}{\varepsilon} \int_{0}^{t}\mathcal{F}_0^{\alpha}(t,s) K_{M,\left \langle v \right\rangle^{l}}^{\alpha,c}  (D_v\mathbf{h})(s,X^{\alpha}(s),V^{\alpha}(s))ds  = \sum_{i=0}^{3} \mathcal{Z}^{\alpha}_{i,v}(t) + \sum_{i=1}^{2}\mathcal{D}_{i,v}^\alpha(t),
    \end{equation}
    where
    \begin{align*}
            &\mathcal{Z}^{\alpha}_{0,v}(t)=\frac{1}{\varepsilon} \int_{0}^{t} \mathcal{F}_0^\alpha(t,s) \sum_{\beta=A,B}\int_{\mathbb{R}^3}     k_{M,\left \langle v \right\rangle^{l}}^{\alpha\beta}(s,X^{\alpha}(s),V^{\alpha}(s),v_*) \\  &\hspace{5cm}\times\mathcal{F}_1^\beta(s,0) (D_vh^\beta)(0,X^{\alpha,\beta}_1(0),V^{\alpha,\beta}_1(0))  dv_*ds \\
            &\hspace{1cm}\lesssim \frac{t}{\varepsilon}e^{-\frac{\nu_ot}{\varepsilon}} \|D_v\mathbf{h}^{in}\|_{L^{\infty}_{x,v}},\\
            &\mathcal{Z}^{\alpha}_{1,v}(t)=-\frac{1}{\varepsilon} \int_{0}^{t} \mathcal{F}_0^\alpha(t,s) \sum_{\beta=A,B}\int_{\mathbb{R}^3} k_{M,\left \langle v \right\rangle^{l}}^{\alpha\beta}(s,X^{\alpha}(s), V^{\alpha}(s), v_*)\\
            &\hspace{2cm}\times \int_{0}^s \mathcal{F}_1^\beta(s,s_1) (D_xh^\alpha+\frac{D_v\upsilon^\alpha}{\varepsilon}h^\alpha)(s_1,X_1^{\alpha,\beta}(s_1),V_1^{\alpha,\beta}(s_1))ds_1  dv_*ds \\
            & \hspace{1cm}\lesssim  \varepsilon  e^{-\frac{\nu_ot}{\varepsilon}} \sup_{0\leq s \leq t}(e^{\frac{\nu_os}{2\varepsilon}}\Vert D_xh^{\alpha}\Vert_{L^{\infty}_{x,v}}) + e^{-\frac{\nu_ot}{2\varepsilon}} \sup_{0\leq s \leq t}(e^{\frac{\nu_os}{2\varepsilon}}\Vert h^{\alpha}\Vert_{L^{\infty}_{x,v}}),\\
            &\mathcal{Z}^{\alpha}_{2,v}(t)=\frac{1}{\varepsilon} \int_{0}^{t} \mathcal{F}_0^\alpha(t,s) \sum_{\beta=A,B}\int_{\mathbb{R}^3} k_{M,w}^{\alpha\beta}(s,X^{\alpha}(s), V^{\alpha}(s), v_*)\\
            &\hspace{4cm}\times\sum_{j=1}^6 \int_{0}^s \mathcal{F}_1^\beta(s,s_1) (D_v\mathcal{H}^\beta_j) (s_1,X_1^{\alpha,\beta}(s_1),V_1^{\alpha,\beta}(s_1))ds_1  dv_*ds \\
            & \hspace{0.8cm}\lesssim  \varepsilon^k e^{-\frac{\nu_ot}{\varepsilon}}\sup_{0\leq s \leq t}[(e^{\frac{\nu_os}{2\varepsilon}}\Vert \left \langle v \right \rangle \mathbf{h}(s) \Vert_{L^{\infty}_{x,v}})^2+(e^{\frac{\nu_os}{2\varepsilon}}\Vert D_v\mathbf{h}(s) \Vert_{L^{\infty}_{x,v}})^2]+\mathcal{I}_2 \varepsilon^k\\
            & \hspace{1cm}+ ({1+\mathcal{I}_1}) \varepsilon e^{-\frac{\nu_ot}{2\varepsilon}}\sup_{0\leq s \leq t}[e^{\frac{\nu_os}{2\varepsilon}}(\Vert \left \langle v \right \rangle \mathbf{h}(s) \Vert_{L^{\infty}_{x,v}}+\Vert  D_x\mathbf{h}(s) \Vert_{L^{\infty}_{x,v}}+\Vert  D_v\mathbf{h}(s) \Vert_{L^{\infty}_{x,v}})],\\
            &\mathcal{Z}^{\alpha}_{3,v}(t)=\frac{1}{\varepsilon^2} \int_{0}^{t} \mathcal{F}_0^\alpha(t,s) \sum_{\beta=A,B}\int_{\mathbb{R}^3} k_{M,w}^{\alpha\beta}(s,X^{\alpha}(s), V^{\alpha}(s), v_*)\\
            &\hspace{3.5cm} \times \int_{0}^s \mathcal{F}_1^\beta(s,s_1) D_v(K_{M,\left \langle v \right\rangle^{l}}^{\alpha,\chi}  \mathbf{h}) (s_1,X^{\alpha,\beta}_1(s_1),V^{\alpha,\beta}_1(s_1))ds_1 dv_* ds \\
            & \hspace{1cm} \lesssim  \delta^{3+1}e^{-\frac{\nu_ot}{2\varepsilon}}\sup_{0\leq s \leq t} (e^{\frac{\nu_os}{2\varepsilon}}\Vert \mathbf{h}(s) \Vert_{L^{\infty}_{x,v}}+e^{\frac{\nu_os}{2\varepsilon}}\Vert D_v\mathbf{h}(s) \Vert_{L^{\infty}_{x,v}}).
        \end{align*}
   Employing Lemma~\ref{LeMK2ker4444}'s estimates for $D_vk_{M,\langle v \rangle^{l}}^{\alpha\beta}$, similar procedures yield
    \begin{gather}\label{estDsss4444H}
        \begin{split}
            \mathcal{D}_{1,v}^\alpha(t) = & \frac{1}{\varepsilon^2} \int_{0}^{t} \mathcal{F}^{\alpha}_0(t,s) \sum_{\beta=A,B}\int_{\mathbb{R}^3} k_{M,\left \langle v \right\rangle^{l}}^{\alpha\beta}(s,X^{\alpha}(s),V^{\alpha}(s),v_*)  \int_{0}^s \mathcal{F}_1^\beta(s,s_1) \\
            &\hspace{4cm} \times (D_vK_{M,\left \langle v \right\rangle^{l}}^{\beta,c}) \mathbf{h} (s_1,X^{\alpha,\beta}_1(s_1),V^{\alpha,\beta}_1(s_1))ds_1 dv_* ds\\
            \leq & (\frac{C}{N }+C_{\eta}\,e^{-\frac{\eta}{8} N^2}+C \lambda)\, e^{-\frac{\nu_ot}{2\varepsilon}}\sup_{0\leq s \leq t}(e^{\frac{\nu_os}{2\varepsilon}}\Vert \left \langle v \right \rangle \mathbf{h}(s) \Vert_{L^{\infty}_{x,v}})+\frac{C_{N,\lambda,T_0}}{\varepsilon^{3/2}}  \sup_{0\leq s \leq t}\| \mathbf{f}_R(s)\|_{L^{2}_{x,v}}.
        \end{split}
    \end{gather}
    For Case 3b ($|v|\leq N$, $|v_*|\leq 2N$, $|v_{**}|\leq 3N$, $s-s_1 \geq \varepsilon\lambda$), the variable change $v_*\to\zeta_{\alpha\beta}$ with integration by parts gives:
        \begin{align*}
            &\frac{C}{\varepsilon^2}\int_{0}^{t} \mathcal{F}^{\alpha}_0(t,s) \int_{D} \sum_{\beta = A,B}\int_{0}^{s-\varepsilon\lambda}\mathcal{F}^{\beta}_1(s,s_1) k_{M,N}^{\alpha\beta}(s,X^{\alpha}(s),V^{\alpha}(s),v_*) \,\\
            &\hspace{1cm}\times  \sum_{\beta' = A,B}k_{M,N}^{\beta \beta'}(s_1,\zeta_{\alpha\beta},V^{\alpha,\beta}_1(s_1),v_{**})  D_v\mathbf{h} (s_1,\zeta_{\alpha\beta},v_{**}) (\frac{dv_*}{d\zeta_{\alpha\beta}}) ds_1 dv_{**}d\zeta_{\alpha\beta} ds \\
            \leq & -\frac{C}{\varepsilon^2}\int_{0}^{t} \mathcal{F}^{\alpha}_0(t,s) \int_{D} \sum_{\beta = A,B}\int_{0}^{s-\varepsilon\lambda}\mathcal{F}^{\beta}_1(s,s_1) k_{M,N}^{\alpha\beta}(s,X^{\alpha}(s),V^{\alpha}(s),v_*)\,\\
            &\hspace{1cm}\times \sum_{\beta' = A,B} (D_vk_{M,N}^{\beta \beta'})(s_1,\zeta_{\alpha\beta},V^{\alpha,\beta}_1(s_1),v_{**})  \mathbf{h} (s_1,\zeta_{\alpha\beta},v_{**}) (\frac{dv_*}{d\zeta_{\alpha\beta}}) ds_1 dv_{**}d\zeta_{\alpha\beta} ds \\
            &+\frac{C_{N,\lambda,T_0}}{\varepsilon^3}e^{-\frac{\nu_ot}{2\varepsilon}}\sup_{0\leq s \leq  T_0}(e^{\frac{\nu_os}{2\varepsilon}}\Vert \mathbf{h}\Vert_{L^{\infty}_{x,v}})\\
            \leq & \frac{C_{N,\lambda,T_0}}{\varepsilon^3}e^{-\frac{\nu_ot}{2\varepsilon}}\sup_{0\leq s \leq  T_0}(e^{\frac{\nu_os}{2\varepsilon}}\Vert \mathbf{h}\Vert_{L^{\infty}_{x,v}})+\frac{C_{N,\lambda,T_0}}{\varepsilon^{3}}  \sup_{0\leq s \leq t}\| \mathbf{f}_R(s)\|_{L^{2}_{x,v}}.
        \end{align*}
    where $(\frac{dv_*}{d\zeta_{\alpha\beta}})$ is independent of $v_{**}$, and
    \begin{gather*}
        \begin{split}
            \mathcal{D}_{2,v}^\alpha(t)     \leq & (\frac{C}{N }+C_{\eta}\,e^{-\frac{\eta}{8} N^2}+C \lambda)\, e^{-\frac{\nu_ot}{2\varepsilon}}\sup_{0\leq s \leq t}(e^{\frac{\nu_os}{2\varepsilon}}\Vert  D_v\mathbf{h}(s) \Vert_{L^{\infty}_{x,v}})\\
            &+ \frac{C_{N,\lambda,T_0}}{\varepsilon^{3}}  \sup_{0\leq s \leq t}\| \mathbf{f}_R(s)\|_{L^{2}_{x,v}}
            +\frac{C_{N,\lambda,T_0}}{\varepsilon^3}e^{-\frac{\nu_ot}{2\varepsilon}}\sup_{0\leq s \leq  T_0}(e^{\frac{\nu_os}{2\varepsilon}}\Vert \mathbf{h}\Vert_{L^{\infty}_{x,v}}).
        \end{split}
    \end{gather*}
    Therefore, we have
    \begin{align}
            I_{2v}  & \leq \, C\frac{t}{\varepsilon}e^{-\frac{\nu_ot}{\varepsilon}}(\| \left \langle v \right \rangle \mathbf{h}^{\rm in} \|_{L^\infty_{x,v}}+\| D_v \mathbf{h}^{\rm in} \|_{L^\infty_{x,v}})\notag\\
            & \hspace{0.3cm} +(C\varepsilon+C{\lambda}+\frac{C_{\lambda}}{N}+C\delta^{3})e^{-\frac{\nu_ot}{2\varepsilon}} \sup_{0\leq s \leq t}[e^{\frac{\nu_os}{2\varepsilon}}(\|  \left \langle v \right \rangle\mathbf{h}\|_{L^{\infty}_{x,v}}+\|  D_{x,v}\mathbf{h}\|_{L^{\infty}_{x,v}})]\notag\\
            &  \hspace{0.3cm}+ \frac{C_{N,\lambda,T_0}}{\varepsilon^3}e^{-\frac{\nu_ot}{2\varepsilon}}\sup_{0\leq s \leq  T_0}(e^{\frac{\nu_os}{2\varepsilon}}\Vert \mathbf{h}\Vert_{L^{\infty}_{x,v}})+\frac{C_{N,\lambda,T_0}}{\varepsilon^{3}}  \sup_{0\leq s \leq t}\| \mathbf{f}_R(s)\|_{L^{2}_{x,v}}+C\varepsilon^k\mathcal{I}_2 \label{HI2CLT}\\
            &\hspace{0.3cm}+C\varepsilon^ke^{-\frac{\nu_ot}{\varepsilon}} \sup_{0\leq s \leq t}[(e^{\frac{\nu_os}{2\varepsilon}}\| \left \langle v \right \rangle \mathbf{h}\|_{L^{\infty}_{x,v}})^2+(e^{\frac{\nu_os}{2\varepsilon}}\| D_v \mathbf{h}\|_{L^{\infty}_{x,v}})^2].\notag
    \end{align}

    Then, we obtain the $W^{1,\infty}_{x,v}$ estimate as follows
    \begin{gather}\label{44HDW1INFXV}
        \begin{split}
            \sup_{0\leq s \leq  T_0}(e^{\frac{\nu_os}{2\varepsilon}}&\Vert \nabla_{x,v}\mathbf{h}\Vert_{L^{\infty}_{x,v}}) \leq C (\Vert \left \langle v \right\rangle \mathbf{h}^{\rm in}\Vert^2_{L^{\infty}_{x,v}}+\Vert \nabla_{x,v}\mathbf{h}^{\rm in}\Vert_{L^{\infty}_{x,v}})\\
            &+(\frac{C_{\lambda}}{N}+C\lambda+C\varepsilon+\delta^3)\sup_{0\leq s \leq  T_0}(e^{\frac{\nu_os}{2\varepsilon}}\Vert \nabla_{x,v}\mathbf{h}\Vert_{L^{\infty}_{x,v}}+e^{\frac{\nu_os}{2\varepsilon}}\Vert \left \langle v \right \rangle\mathbf{h}\Vert_{L^{\infty}_{x,v}})\\
            &+\frac{C_{N,\lambda,T_0}}{\varepsilon^3}\sup_{0\leq s \leq  T_0}(e^{\frac{\nu_os}{2\varepsilon}}\Vert \mathbf{h}\Vert_{L^{\infty}_{x,v}})+\frac{C_{\lambda,N}}{\varepsilon^4}e^{\frac{\nu_oT_0}{2\varepsilon}}\sup_{0\leq s \leq T_0}\| \mathbf{f}_R(s)\|_{L^{2}_{x,v}}\\
            &+C\varepsilon^k \sup_{0\leq s \leq T_0}[(e^{\frac{\nu_os}{2\varepsilon}}\| \left \langle v \right \rangle \mathbf{h}\|_{L^{\infty}_{x,v}})^2+(e^{\frac{\nu_os}{2\varepsilon}}\| \nabla_{x,v} \mathbf{h}\|_{L^{\infty}_{x,v}})^2]+\varepsilon^{k-1}e^{\frac{\nu_oT_0}{2\varepsilon}}\mathcal{I}_2.
        \end{split}
    \end{gather}
   Multiplying \eqref{44HDW1INFXV} by $\varepsilon^5 e^{-\frac{\nu_oT_0}{2\varepsilon}}$ leads to:
    \begin{gather}\label{44FsinaldT0s0}
        \begin{split}
            \varepsilon^{5}\Vert  \nabla_{x,v}\mathbf{h}(T_0) \Vert_{L^{\infty}_{x,v}}\leq& \frac{1}{4}(\varepsilon^{5}\Vert \left \langle v \right\rangle \mathbf{h}^{\rm in} \Vert_{L^{\infty}_{x,v}}+ \varepsilon^{5}\Vert \nabla_{x,v} \mathbf{h}^{\rm in} \Vert_{L^{\infty}_{x,v}}+\varepsilon^{\frac{3}{2}}\Vert \left \langle v \right\rangle \mathbf{h} \Vert_{L^{\infty}_{x,v}})\\
            &+C(\varepsilon^{\frac{1}{2}}\sup_{0\leq s \leq T_0}\| \mathbf{f}_R(s)\|_{L^{2}_{x,v}}+\varepsilon^{k+1}).
        \end{split}
    \end{gather}
    Analogous to the $\displaystyle{\sup_{0\leq s \leq T_0}\varepsilon^{\frac{3}{2}}\Vert  \mathbf{h}(s) \Vert_{L^{\infty}_{x,v}}}$ estimate, we obtain
    \begin{equation*}
        \varepsilon^{\frac{3}{2}}\Vert \left \langle v \right\rangle \mathbf{h}(T_0) \Vert_{L^{\infty}_{x,v}}\leq \frac{1}{4}\varepsilon^{\frac{3}{2}}\Vert \left \langle v \right\rangle  \mathbf{h}^{\rm in} \Vert_{L^{\infty}_{x,v}}+C \sup_{0\leq s \leq T_0}\| \mathbf{f}_R(s)\|_{L^{2}_{x,v}}+C\varepsilon^{\frac{2k+3}{2}}.
    \end{equation*}
   These estimates collectively establish that
    \begin{gather*}
        \begin{split}
            &\varepsilon^{\frac{3}{2}}\Vert \left \langle v \right\rangle \mathbf{h}(T_0) \Vert_{L^{\infty}_{x,v}}+\varepsilon^{5}\Vert  \nabla_{x,v}\mathbf{h}(T_0) \Vert_{L^{\infty}_{x,v}}\\
            \leq & \, \frac{1}{2}(\varepsilon^{\frac{3}{2}}\Vert \left \langle v \right\rangle \mathbf{h}^{\rm in} \Vert_{L^{\infty}_{x,v}}+\varepsilon^{5}\Vert  \nabla_{x,v}\mathbf{h}^{\rm in} \Vert_{L^{\infty}_{x,v}})+C(\sup_{0\leq s \leq T_0}\| \mathbf{f}_R(s)\|_{L^{2}_{x,v}}+\varepsilon^{k+\frac{1}{2}}).
        \end{split}
    \end{gather*}
    Then, for time $t$ satisfies $T_0< t \leq T_L = \varepsilon^{-\frac{2k-3}{2(2k-1)}}$, there exists a positive integer $n_t$ and $0\leq t_s< T_0$, such that
    $t=n_tT_0+t_s$. Then we can deduce that
    \begin{align}
            &\varepsilon^{\frac{3}{2}}\Vert \left \langle v \right\rangle \mathbf{h}(t) \Vert_{L^{\infty}_{x,v}}+\varepsilon^{5}\Vert  \nabla_{x,v}\mathbf{h}(t) \Vert_{L^{\infty}_{x,v}}\notag\\
            \leq & \, \frac{1}{2}(\varepsilon^{\frac{3}{2}}\Vert \left \langle v \right\rangle \mathbf{h}((n-1)T_0+t_s) \Vert_{L^{\infty}_{x,v}}+\varepsilon^{5}\Vert  \nabla_{x,v}\mathbf{h}((n-1)T_0+t_s) \Vert_{L^{\infty}_{x,v}})\notag\\
            &+C(\sup_{0\leq s \leq T_1}\| \mathbf{f}_R(s)\|_{L^{2}_{x,v}}+\varepsilon^{k+\frac{1}{2}}) \notag\\
            \leq & \, \frac{1}{4}(\varepsilon^{\frac{3}{2}}\Vert \left \langle v \right\rangle \mathbf{h}((n-2)T_0+t_s) \Vert_{L^{\infty}_{x,v}}+\varepsilon^{5}\Vert  \nabla_{x,v}\mathbf{h}((n-2)T_0+t_s) \Vert_{L^{\infty}_{x,v}})\notag\\
            &+(C+\frac{C}{2})(\sup_{0\leq s \leq T_1}\| \mathbf{f}_R(s)\|_{L^{2}_{x,v}}+\varepsilon^{k+\frac{1}{2}})\notag\\
            \leq & \, ...\notag\\
            \leq & \, \frac{1}{2^n}(\varepsilon^{\frac{3}{2}}\Vert \left \langle v \right\rangle \mathbf{h}(t_s) \Vert_{L^{\infty}_{x,v}}+\varepsilon^{5}\Vert  \nabla_{x,v}\mathbf{h}(t_s) \Vert_{L^{\infty}_{x,v}})\label{dddHDdgui0H1}\\
            &+(C+\frac{C}{2}+...+\frac{C}{2^n})(\sup_{0\leq s \leq T_1}\| \mathbf{f}_R(s)\|_{L^{2}_{x,v}}+\varepsilon^{k+\frac{1}{2}})\notag\\
            \lesssim & \, \varepsilon^{\frac{3}{2}}\Vert \left \langle v \right\rangle \mathbf{h}^{\rm in} \Vert_{L^{\infty}_{x,v}}+\varepsilon^{5}\Vert  \nabla_{x,v}\mathbf{h}^{\rm in} \Vert_{L^{\infty}_{x,v}}+\sup_{0\leq s \leq T_L}\| \mathbf{f}_R(s)\|_{L^{2}_{x,v}}+\varepsilon^{k+\frac{1}{2}}\notag.
        \end{align}
    Consequently, the following estimate holds:
    \begin{gather}\label{H444lastHSw1p}
        \begin{split}
            &\sup_{0\leq t \leq T_L}(\varepsilon^{\frac{3}{2}}\Vert \left \langle v \right\rangle \mathbf{h}(t) \Vert_{L^{\infty}_{x,v}}+\varepsilon^{5}\Vert  \nabla_{x,v}\mathbf{h}(t) \Vert_{L^{\infty}_{x,v}})\\
            \lesssim & \, \varepsilon^{\frac{3}{2}}\Vert \left \langle v \right\rangle \mathbf{h}^{\rm in} \Vert_{L^{\infty}_{x,v}}+\varepsilon^{5}\Vert  \nabla_{x,v}\mathbf{h}^{\rm in} \Vert_{L^{\infty}_{x,v}}+\sup_{0\leq t \leq T_L}\| \mathbf{f}_R(t)\|_{L^{2}_{x,v}}+\varepsilon^{k+\frac{1}{2}}.
        \end{split}
    \end{gather}
    By the same iterative process, we also obtain that
    \begin{gather}\label{H44jjDm4HSLinf}
        \begin{split}
            \sup_{0\leq t \leq T_L}(\varepsilon^{\frac{3}{2}}\Vert  \mathbf{h}(t) \Vert_{L^{\infty}_{x,v}})
            \lesssim  \, \varepsilon^{\frac{3}{2}}\Vert  \mathbf{h}^{\rm in} \Vert_{L^{\infty}_{x,v}}+\sup_{0\leq t \leq T_L}\| \mathbf{f}_R(t)\|_{L^{2}_{x,v}}+\varepsilon^{k+\frac{1}{2}}.
        \end{split}
    \end{gather}
    \begin{remark}
        The estimate \eqref{H444lastHSw1p} reveals that for hard sphere interactions $(\gamma = 1)$, the initial data incurs a polynomial weight loss of $\langle v \rangle^{2-1}$. Subsequent analysis in Section~5 demonstrates this generalizes to $\langle v \rangle^{2-\gamma}$ for $-3 < \gamma < 1$, originating from the $D_v(\mathcal{H}^\alpha_3)$ term.
    \end{remark}
    \subsection{The proof of Theorem \ref{maintheorem}}[hard sphere part] Plugging   \eqref{H444lastHSw1p} and \eqref{H44jjDm4HSLinf} into \eqref{MainresultL22}, we derive that
    \begin{align}
            &\frac{d}{dt}(\sum_{\alpha=A,B}\Vert \sqrt{\theta}f^{\alpha}_R\Vert_{L^{2}_{x,v} }^{2}+| \nabla_x\mathbf{\phi}_R|_{L^{2}_{x} }^{2})+\frac{c_0}{2\varepsilon}\Vert(\mathbf{I}-\mathcal{P})\mathbf{f}_R\Vert_{\nu}^2 \notag\\
            &\hspace{0.4cm}\leq  C
            \sqrt{\varepsilon}\Big(\Vert\varepsilon^{\frac{3}{2}}\mathbf{h}^{\rm in}\Vert_{{L^{\infty}_{x,v} }}
		+\sup_{0\leq t \leq T_L}\|
		\mathbf{f}_R(t)\|_{L^{2}_{x,v}}+\varepsilon^{\frac{2k+1}{2}}\Big)\notag\\
		&\hspace{4cm}\times\Big(
            \Vert \mathbf{f}_R\Vert_{{L^{2}_{x,v} }}
            +\varepsilon^{k-3}
            \Vert \mathbf{f}_R\Vert^2_{{L^{2}_{x,v} }}+\varepsilon^{k-2}
            \Vert \mathbf{f}_R\Vert_{{L^{2}_{x,v} }}| \nabla_x\mathbf{\phi}_R|_{{L^{2}_{x} }}\Big)\notag\\
            &\hspace{0.8cm}+C\Big[\frac{1}{(1+t)^q}+\varepsilon\mathcal{I}_1\Big](\Vert \mathbf{f}_R\Vert^2_{{L^{2}_{x,v} }}
			+| \nabla_x\mathbf{\phi}_R|^2_{{L^{2}_{x}
				}})+C\varepsilon^{k-1}\mathcal{I}_2\Vert
				\mathbf{f}_R\Vert_{{L^{2}_{x,v} }}.\label{HDL2LINFZZM}
    \end{align}
    Moreover, applying the Cauchy-Schwarz inequality, we have
    \begin{gather*}
        \begin{split}
            \Vert \mathbf{f}_R\Vert_{{L^{2}_{x,v} }}
            +\varepsilon^{k-3}
            \Vert \mathbf{f}_R\Vert^2_{{L^{2}_{x,v} }}+\varepsilon^{k-2}
            \Vert \mathbf{f}_R\Vert_{{L^{2}_{x,v} }}| \nabla_x\mathbf{\phi}_R|_{{L^{2}_{x} }} \lesssim  \Vert \mathbf{f}_R\Vert^2_{{L^{2}_{x,v} }} +| \nabla_x\mathbf{\phi}_R|^2_{{L^{2}_{x} }}+1,\\
        \end{split}
    \end{gather*}
   with the subsequent estimate
    \begin{gather*}
        \begin{split}
            &C\Big[\frac{1}{(1+t)^q}+\varepsilon\mathcal{I}_1\Big](\Vert \mathbf{f}_R\Vert^2_{{L^{2}_{x,v} }}+| \nabla_x\mathbf{\phi}_R|^2_{{L^{2}_{x} }})+C\varepsilon^{k-1}\mathcal{I}_2\Vert \mathbf{f}_R\Vert_{{L^{2}_{x,v} }}\\
            \lesssim  & \, \Big[\frac{1}{(1+t)^q}+\varepsilon\mathcal{I}_1+\varepsilon^{k-1}\mathcal{I}_2\Big](\Vert \mathbf{f}_R\Vert^2_{{L^{2}_{x,v} }} +| \nabla_x\mathbf{\phi}_R|^2_{{L^{2}_{x} }}+1).
        \end{split}
    \end{gather*}
    Due to $\frac{1}{C}<\theta<C$, it is easy to see
    \begin{equation}\label{ddddjjnl}
        \displaystyle{\sum_{\alpha=A,B}}\Vert \sqrt{\theta}f^{\alpha}_R\Vert_{L^{2}_{x,v} }^{2}\lesssim \Vert \mathbf{f}_R\Vert^2_{{L^{2}_{x,v} }}\lesssim \displaystyle{\sum_{\alpha=A,B}}\Vert \sqrt{\theta}f^{\alpha}_R\Vert_{L^{2}_{x,v} }^{2}.
    \end{equation}
   Consequently, \eqref{HDL2LINFZZM} reduces to
    \begin{equation*}
        \begin{split}
            &\hspace{1cm}\frac{d}{dt}\Big(\sum_{\alpha=A,B}\Vert \sqrt{\theta}f^{\alpha}_R\Vert_{L^{2}_{x,v} }^{2}+| \nabla_x\mathbf{\phi}_R|_{L^{2}_{x} }^{2}+1\Big) \\
            \leq & C
            \sqrt{\varepsilon}\Big(\Vert\varepsilon^{\frac{3}{2}}\mathbf{h}^{\rm in}\Vert_{{L^{\infty}_{x,v} }}+\sup_{0\leq t \leq T_L}\| \mathbf{f}_R(t)\|_{L^{2}_{x,v}}
+\varepsilon^{\frac{2k+1}{2}}\Big)\Big(\sum_{\alpha=A,B}\Vert \sqrt{\theta}f^{\alpha}_R\Vert_{L^{2}_{x,v} }^{2}+| \nabla_x\mathbf{\phi}_R|_{L^{2}_{x} }^{2}+1\Big)\\
            &\,\,\,+C\Big[\frac{1}{(1+t)^q}+\varepsilon\mathcal{I}_1+\varepsilon^{k-1}\mathcal{I}_2\Big]\Big(\sum_{\alpha=A,B}\Vert \sqrt{\theta}f^{\alpha}_R\Vert_{L^{2}_{x,v} }^{2}
+| \nabla_x\mathbf{\phi}_R|_{L^{2}_{x} }^{2}+1\Big).
        \end{split}
    \end{equation*}
Combining with $\int_{0}^{t}(1+s)^{-q}ds <\infty$ leads to
    \begin{equation*}
        \begin{split}
            &\frac{d}{dt}\Big(\sum_{\alpha=A,B}\Vert \sqrt{\theta}f^{\alpha}_R\Vert_{L^{2}_{x,v} }^{2}+| \nabla_x\mathbf{\phi}_R|_{L^{2}_{x} }^{2}+1\Big)
 \leq \Big(\sum_{\alpha=A,B}\Vert \sqrt{\theta}f^{\alpha}_R(0)\Vert_{L^{2}_{x,v} }^{2}+| \nabla_x\mathbf{\phi}_R(0)|_{L^{2}_{x} }^{2}+1\Big)\\
            &\, \, \, \times \exp\Big\{\int_{0}^{t}C\Big[\sqrt{\varepsilon}(\Vert\varepsilon^{\frac{3}{2}}\mathbf{h}^{\rm in}\Vert_{{L^{\infty}_{x,v} }}
+\sup_{0\leq t \leq T_L}\|
\mathbf{f}_R(t)\|_{L^{2}_{x,v}})+\frac{1}{(1+t)^q}+\varepsilon\mathcal{I}_1+\varepsilon^{k-1}\mathcal{I}_2\Big]ds\Big\}\\
            &\leq \Big(\sum_{\alpha=A,B}\Vert \sqrt{\theta}f^{\alpha}_R(0)\Vert_{L^{2}_{x,v} }^{2}+| \nabla_x\mathbf{\phi}_R(0)|_{L^{2}_{x} }^{2}+1\Big)\\
            &\,\, \times \exp[C+Ct\sqrt{\varepsilon}(\Vert\varepsilon^{\frac{3}{2}}\mathbf{h}^{\rm in}\Vert_{{L^{\infty}_{x,v} }}
+\sup_{0\leq t \leq T_L}\| \mathbf{f}_R(t)\|_{L^{2}_{x,v}})+Ct(\varepsilon\mathcal{I}_1+\varepsilon^{k-1}\mathcal{I}_2)].
        \end{split}
    \end{equation*}
Then for the time scale $0\leq t\leq T_L=\varepsilon^{-m}$ ($0<m\leq\frac{2k-3}{4k-4}$), the integrals satisfy
    \begin{eqnarray}\label{Otime11111111}
        \begin{split}
            &\mathcal{I}_1\leq 2\sum_{i=1}^{2k-1}\varepsilon^{i-1}(1+t)^{i-1}\leq C \sum_{i=1}^{2k-1}(\varepsilon+\varepsilon^{1-m})^{i-1}\leq C, \\
            &\mathcal{I}_2 = \sum_{2k\leq i+j \leq 4k-2}\varepsilon^{i+j-2k}(1+t)^{i+j-2} \leq  C (1+\varepsilon^{-m})^{2k-2} \leq C \varepsilon^{-m(2k-2)},
        \end{split}
    \end{eqnarray}
implying the bound
    \begin{equation}\label{OtimeZZZZZ1}
        \mathcal{I}_1t\varepsilon+\mathcal{I}_2 t\varepsilon^{k-1} \leq C(\varepsilon^{1-m}+\varepsilon^{k-1-m(2k-1)})\leq C\varepsilon^{\frac{1}{2}-m}.
    \end{equation}
  For sufficiently small $\varepsilon$, this establishes
    \begin{eqnarray*}
        \begin{split}
            &\sum_{\alpha=A,B}\Vert \sqrt{\theta}f^{\alpha}_R\Vert_{L^{2}_{x,v} }^{2}+| \nabla_x\mathbf{\phi}_R|_{L^{2}_{x} }^{2}
\lesssim \Big(\sum_{\alpha=A,B}\Vert \sqrt{\theta}f^{\alpha}_R(0)\Vert_{L^{2}_{x,v} }^{2}+| \nabla_x\mathbf{\phi}_R(0)|_{L^{2}_{x} }^{2}+1\Big)\\
            & \hspace{5cm}\times \Big(1+\varepsilon^{\frac{1}{2}-m}\Vert\varepsilon^{\frac{3}{2}}\mathbf{h}^{\rm in}\Vert_{{L^{\infty}_{x,v}}}
+\varepsilon^{\frac{1}{2}-m}\sup_{0\leq t \leq T_L}\|
\mathbf{f}_R(t)\|_{L^{2}_{x,v}}\Big).
        \end{split}
    \end{eqnarray*}
The norm equivalence \eqref{ddddjjnl} consequently gives
    \begin{equation}\label{HDFIANL}
        \sup_{0\leq t \leq T_L}(\| \mathbf{f}_R\|_{L^{2}_{x,v}}+| \nabla_x\mathbf{\phi}_R|_{L^{2}_{x} }^{2})
 \lesssim \, 1+\| \mathbf{f}_R(0)\|_{L^{2}_{x,v}}+| \nabla_x\mathbf{\phi}_R(0)|_{L^{2}_{x} }^{2}
+\Vert\varepsilon^{\frac{3}{2}}\mathbf{h}^{\rm
in}\Vert_{{L^{\infty}_{x,v}}}.
    \end{equation}
Therefore, combined with \eqref{H444lastHSw1p} and \eqref{HDFIANL}, this completes the proof for hard sphere case in Theorem~\ref{maintheorem} .
    \begin{remark}\label{FXPINGSHIJIAN}
        The estimate \eqref{OtimeZZZZZ1} indicates the validity time scaling as $O(\varepsilon^{-\frac{2k-3}{2k-1}})$, determined by the coefficients of the truncated expansion's remainder terms. Considering the ansatz
        \begin{equation}\label{NRNRTIME}
            F_{\varepsilon}^\alpha= \sum_{i=0}^{n}\varepsilon^i F^\alpha_i + \varepsilon^{r} F_R^{\alpha},\quad n\geq r\geq 1,
        \end{equation}
        with time scale $T_L=\varepsilon^{-m}$ $(0<m<1)$, substitution into \eqref{Meqeps} yields through comparison with \eqref{Otime11111111} and \eqref{OtimeZZZZZ1}:
        \begin{eqnarray*}
            \begin{split}
                t\varepsilon\mathcal{I}_1 &\leq Ct\varepsilon \sum_{i=1}^{n}(\varepsilon+\varepsilon^{1-m})^{i-1}\leq C\varepsilon^{1-m},\\
                t\varepsilon^{r-1}\mathcal{I}_2 &\leq C(1+\varepsilon^{n-r-m(n-1)}),
            \end{split}
        \end{eqnarray*}
        requiring the constraint $0<m<\frac{n-r}{n-1}$. The case $n=2k-1$, $r=k$ corresponds to $0<m<\frac{1}{2}$.
    \end{remark}

    \section{The  $W^{1,\infty}_{x,v}$ estimate of the remainder term for potential range $-1\leq\gamma<1$}
    Before presenting the proof, the key difficulty for $\gamma \neq 1$ cases is outlined. From \eqref{phiimpus} and \eqref{psideddddg}, the following estimates emerge:
    \begin{itemize}\label{aafxH3}
        \item  $\mathcal{H}^\alpha_3 = -\frac{e^{\alpha}}{m^{\alpha}}\nabla_{x}\phi^{\varepsilon} \cdot  \frac{\left \langle v \right\rangle^{l}}{\sqrt{\mu^{\alpha}_M}} \nabla_{v}(\frac{\sqrt{\mu^{\alpha}_M}}{\left \langle v \right\rangle^{l}})h^\alpha_R$, \, \quad $ |\mathcal{H}^\alpha_3|  \thicksim \left \langle v \right\rangle \Vert\mathbf{h}_R \Vert_{L^{\infty}_{x,v} }$,\\
        \item $\left|\int_{0}^{t}\mathcal{F}_0^{\alpha}(t,s) \mathcal{H}^\alpha_3(s,X(s),V(s))ds \right|
        \lesssim \varepsilon(1+\varepsilon\mathcal{I}_1) e^{-\frac{\nu_ot}{2\varepsilon}} \displaystyle{\sup_{0\leq s \leq t}}(e^{\frac{\nu_os}{2\varepsilon}}\Vert \left \langle v \right\rangle^{1-\gamma}\mathbf{h}_R (s)\Vert_{L^{\infty}_{x,v}})$.
    \end{itemize}
    For $-3 < \gamma < 1$, the preceding analysis shows this term cannot be absorbed by the left-hand side expression. The term fundamentally belongs to the linear part, though under the hard sphere model ($\gamma = 1$), the dominant influence of $\mathcal{F}_0^{\alpha}$ permits its treatment as a nonlinear term. For $\gamma < 1$ cases, inclusion in the linear part becomes necessary for proper analysis. This motivates the introduction of a new weight function $w(t,x,v)$ for the remainder terms, leading to consideration of the linear part operator:
    \begin{equation}\label{ftftyuia}
        \frac{\upsilon^\alpha}{\varepsilon} +\frac{e^{\alpha}}{m^{\alpha}}\nabla_{x}\phi^{\varepsilon} \cdot  \frac{w(t,x,v)}{\sqrt{\mu^{\alpha}_M}} \nabla_{v}(\frac{\sqrt{\mu^{\alpha}_M}}{w(t,x,v)})-\frac{(\partial_t + v\cdot\nabla_{x})w(t,x,v)}{w(t,x,v)}.
    \end{equation}
   The second term in \eqref{ftftyuia} exhibits $\langle v \rangle=(1+|v|)$ velocity growth, presenting two key challenges for the linear component analysis:
   \begin{itemize}
    \item Indefinite sign behavior complicating direct estimates,
    \item Stronger velocity growth compared to the first term.
   \end{itemize}
  This necessitates the equation transformations to generate dominant negative terms with accelerated velocity decay. While the weight functions depending on spatial derivatives may introduce structural complications, a $(t,v)$-dependent weighting strategy offers distinct advantages. The natural exponential function $e^{x}$ proves particularly suitable for this construction due to its differentiation properties:
    \begin{equation*}
        \frac{d}{dt} (e^{g(t,v)}) = e^{g(t,v)} \cdot \frac{d}{dt} g(t,v).
    \end{equation*}
To ensure that $g(t,v)$ yields a sufficiently strong polynomial decay in $v$ for controlling $\mathcal{N}^\alpha_3$, we impose the following conditions on parameters $\kappa_0, \kappa_1 > 0$ and $\kappa_2 \in (1,2]$
    \begin{equation*}
        \omega_{\kappa}=\exp\{{\kappa_0[1+(1+t)^{-\kappa_1}]\left \langle v \right\rangle^{\kappa_2}}\}, \,\, \quad \text{and} \quad  \,\,\,S^\alpha_{R,\kappa}=  \frac{\omega_{\kappa} }{\sqrt{\mu^\alpha_M}} F_R^\alpha.
    \end{equation*}
   Then the remainder equation transforms to:
    \begin{gather}\label{HDPTLINANALFEQMAIN}
        \begin{split}
            \partial_t  S^\alpha_{R,\kappa} & + v\cdot\nabla_{x} S^\alpha_{R,\kappa} + \frac{e^{\alpha}}{m^{\alpha}} \nabla_{x}\phi^{\varepsilon}\cdot\nabla_{v}  S^\alpha_{R,\kappa} + \frac{\upsilon^\alpha}{\varepsilon}  S^\alpha_{R,\kappa}+\frac{e^{\alpha}}{m^{\alpha}}\nabla_{x}\phi^{\varepsilon} \cdot  \frac{\omega_{\kappa}}{\sqrt{\mu^{\alpha}_M}} \nabla_{v}(\frac{\sqrt{\mu^{\alpha}_M}}{\omega_{\kappa}})S^\alpha_{R,\kappa}-\frac{\partial_t\omega_{\kappa}}{\omega_{\kappa}}\\
            &\hspace{1cm}= \sum_{i=1}^{5} \mathcal{N}^\alpha_j+ \frac{1}{\varepsilon}(K_{M,\omega_{\kappa}}^{\alpha,\chi}+K_{M,\omega_{\kappa}}^{\alpha,c})\mathbf{S}_{R,\kappa},
        \end{split}
    \end{gather}
and the main linear part in \eqref{HDPTLINANALFEQMAIN} satisfies
    \begin{align}
            \hat{\upsilon_{\varepsilon}^\alpha}(t,x,v):&=   \frac{\upsilon^\alpha}{\varepsilon} +\frac{e^{\alpha}}{m^{\alpha}}\nabla_{x}\phi^{\varepsilon} \cdot  \frac{\omega_{\kappa}(t,v)}{\sqrt{\mu^{\alpha}_M}} \nabla_{v}(\frac{\sqrt{\mu^{\alpha}_M}}{\omega_{\kappa}(t,v)})-\frac{\partial_t\omega_{\kappa}(t,v)}{\omega_{\kappa}(t,v)}\notag\\
            & \geq \frac{\upsilon^\alpha}{\varepsilon}-\frac{e^{\alpha}}{m^{\alpha}}\Big(\frac{1}{2\theta_M}+2\kappa_0\Big)\left \langle v \right\rangle \Vert\nabla_{x}\phi^{\varepsilon} \Vert_{L^{\infty}_{x,v} }+\kappa_0\kappa_1\frac{\left \langle v \right\rangle^{\kappa_2}}{(1+t)^{1+\kappa_1}}\label{xxbfANLmainlower1}\\
            &\geq \frac{\left \langle v \right\rangle^{\gamma}}{C\varepsilon}- C\left \langle v \right\rangle \Vert\nabla_{x}\phi^{\varepsilon} \Vert_{L^{\infty}_{x,v} } +\frac{\kappa_0\kappa_1}{C}\frac{\left \langle v \right\rangle^{\kappa_2}}{(1+t)^{1+\kappa_1}}\notag.
    \end{align}
    For parameter $0<\iota<1$, Young's inequality gives
    \begin{equation}\label{ANLYYYEQIMP}
        \begin{split}
            \left \langle v \right\rangle &= \Big(\left \langle v \right\rangle^{\gamma}\varepsilon^{\iota-1}(1+t)^{\frac{(1+\kappa_1)(1-\gamma)}{\kappa_2-1}}\Big)^{\frac{\kappa_2-1}{\kappa_2-\gamma}}    \Big(\left \langle v \right\rangle^{\kappa_2}\varepsilon^{\frac{(1-\iota)(\kappa_2-1)}{1-\gamma}}(1+t)^{-(1+\kappa_1)}\Big)^{\frac{1-\gamma}{\kappa_2-\gamma}}\\
            & \leq \frac{1}{\varepsilon}\left \langle v \right\rangle^{\gamma}\varepsilon^{\iota}(1+t)^{\frac{(1+\kappa_1)(1-\gamma)}{\kappa_2-1}} + \left \langle v \right\rangle^{\kappa_2}\varepsilon^{\frac{(1-\iota)(\kappa_2-1)}{1-\gamma}}(1+t)^{-(1+\kappa_1)}.
        \end{split}
    \end{equation}
    Plunging \eqref{ANLYYYEQIMP} into \eqref{xxbfANLmainlower1}, it holds that
    \begin{equation*}
        \hat{\upsilon_{\varepsilon}^\alpha}\geq (\frac{1}{C}-C\varepsilon^{\iota}(1+t)^{\frac{(1+\kappa_1)(1-\gamma)}{\kappa_2-1}}\Vert\nabla_{x}\phi^{\varepsilon} \Vert_{L^{\infty}_{x,v} })\frac{\left \langle v \right\rangle^{\gamma}}{\varepsilon}+(\frac{\kappa_0\kappa_1}{C}-C\varepsilon^{\frac{(1-\iota)(\kappa_2-1)}{1-\gamma}})\frac{\left \langle v \right\rangle^{\kappa_2}}{(1+t)^{1+\kappa_1}}.
    \end{equation*}
  Due to $0<\kappa_0\ll1$, we need to find suitable $\kappa_1$ and $\kappa_2$, such that
    \begin{equation}\label{4455FPFPFPKZ}
        \frac{1}{C}-C\varepsilon^{\iota}(1+t)^{\frac{(1+\kappa_1)(1-\gamma)}{\kappa_2-1}}\Vert\nabla_{x}\phi^{\varepsilon} \Vert_{L^{\infty}_{x,v} }>\frac{2}{3C}.
    \end{equation}

   In fact, for $0<t\leq \varepsilon^{-m}$, $0<m\leq \frac{1}{2} \frac{2k-3}{2k-1}$ and $k\geq 6$, we have
    \begin{equation*}
        C\varepsilon^{\iota}(1+t)^{\frac{(1+\kappa_1)(1-\gamma)}{\kappa_2-1}}\Vert\nabla_{x}\phi^{\varepsilon} \Vert_{L^{\infty}_{x,v} } \leq 2C \varepsilon^{\iota-\frac{m(1+\kappa_1)(1-\gamma)}{\kappa_2-1}} \Vert\nabla_{x}\phi^{\varepsilon} \Vert_{L^{\infty}_{x,v} }.
    \end{equation*}
Then we can choose $1+\kappa_1=\frac{2k+1}{2k-1}$,
\,$\frac{1-\gamma}{\kappa_2-1}=2,\,\,\,$i.e.
$\,\kappa_1=\frac{2}{2k-1}$, $\kappa_2=\frac{3-\gamma}{2}$.
Then for $-1\leq\gamma<1$, it yields
    \begin{equation*}
        \frac{m(1+\kappa_1)(1-\gamma)}{\kappa_2-1}\leq \frac{(2k-3)(2k+1)}{(2k-1)^2}.
    \end{equation*}
 Setting $\iota=\frac{(2k-2)(2k+2)}{4k^2}<1$ gives
    \begin{equation*}
        \iota-\frac{m(1+\kappa_1)(1-\gamma)}{\kappa_2-1}\geq \frac{(2k-2)(2k+2)}{4k^2}-\frac{(2k-3)(2k+1)}{(2k-1)^2}>\frac{1}{16k^4}>0.
    \end{equation*}
 For sufficiently small $\varepsilon$
    \begin{equation*}
        C\varepsilon^{\iota}(1+t)^{\frac{(1+\kappa_1)(1-\gamma)}{\kappa_2-1}}\Vert\nabla_{x}\phi^{\varepsilon} \Vert_{L^{\infty}_{x,v} }
\leq 2C \varepsilon^{\iota-\frac{m(1+\kappa_1)(1-\gamma)}{\kappa_2-1}} \Vert\nabla_{x}\phi^{\varepsilon} \Vert_{L^{\infty}_{x,v} } < \frac{1}{3C},
    \end{equation*}
which implies \eqref{4455FPFPFPKZ}. Furthermore, with $\frac{(1-\iota)(\kappa_2-1)}{1-\gamma}>\frac{1}{4}$ for $-1\leq\gamma<1$:
    \begin{equation*}
        \Big(\frac{\kappa_0\kappa_1}{C}-C\varepsilon^{\frac{(1-\iota)(\kappa_2-1)}{1-\gamma}}\Big)\frac{\left \langle v \right\rangle^{\kappa_2}}{(1+t)^{1+\kappa_1}}
>\frac{1}{4Ck}\frac{\left \langle v \right\rangle^{\kappa_2}}{(1+t)^{1+\kappa_1}}, \quad
k\geq6.
    \end{equation*}

Finally, the weight function
\begin{equation}\label{gghk5sctm}
    \omega_{\gamma} := \kappa_0\exp\Big[(1+(1+t)^{-\frac{2}{2k-1}}) \langle v \rangle^{\frac{3-\gamma}{2}}\Big], \quad -1\leq\gamma<1,
\end{equation}
leads to
    \begin{equation}\label{xxbfdmaisnlower1}
        \begin{split}
            \hat{\upsilon_{\varepsilon}^\alpha}:&=  \frac{\upsilon^\alpha}{\varepsilon} +\frac{e^{\alpha}}{m^{\alpha}}\nabla_{x}\phi^{\varepsilon} \cdot  \frac{\omega_{\gamma}(t,v)}{\sqrt{\mu^{\alpha}_M}} \nabla_{v}(\frac{\sqrt{\mu^{\alpha}_M}}{\omega_{\gamma}(t,v)})-\frac{\partial_t\omega_{\gamma}(t,v)}{\omega_{\gamma}(t,v)}\\
            &\geq \frac{2\left \langle v \right\rangle^{\gamma}}{3C\varepsilon} +\frac{1}{4Ck}\frac{\left \langle v \right\rangle^{\frac{3-\gamma}{2}}}{(1+t)^{\frac{2k+1}{2k-1}}}> \frac{2\left \langle v \right\rangle^{\gamma}}{3C\varepsilon}.
        \end{split}
    \end{equation}
    \begin{remark}
        For the following estimates in this section, we actually replace $\hat{\upsilon_{\varepsilon}^\alpha}$
        with $\frac{\left \langle v \right\rangle^{\gamma}}{C\varepsilon}$, where the coefficient $\frac{1}{\varepsilon}$
        in $\frac{\left \langle v \right\rangle^{\gamma}}{C\varepsilon}$
        is critical, as it can eliminate some singularities in these terms, such as $\frac{1}{\varepsilon}(K_{M,\omega_{\kappa}}^{\alpha,\chi}+K_{M,\omega_{\kappa}}^{\alpha,c})\mathbf{S}_{R,\kappa}$
        as well as the terms $\frac{D_x\upsilon^\alpha}{\varepsilon}$
        and $\frac{D_v\upsilon^\alpha}{\varepsilon}$. It should be emphasized that $\frac{1}{\varepsilon}$ cannot be replaced by any $\frac{1}{\varepsilon^{l_d}}$, even if $\frac{10^{10}}{10^{10}+1} < l_d < 1$, due to the mathematical structure that
        \begin{align*}
           & \left|\int_{0}^{t} \exp \left(-\frac{1}{\varepsilon^{l_d}}\int_{s}^{t}\upsilon^{\alpha}(\tau,  X^{\alpha}(\tau),V^{\alpha}(\tau)) d\tau\right) \left \langle v \right\rangle^{\gamma}  ds \right|\\
             &\hspace{0.3cm} \lesssim  \int_{0}^{t}\exp \left(-\frac{ C\left \langle v \right\rangle^{\gamma}  (t-s)}{ \varepsilon^{l_d}}\right)\left \langle v \right\rangle^{\gamma}  ds\,\, \lesssim \,\varepsilon^{l_d}.
        \end{align*}
    \end{remark}
    \begin{remark}\label{WLIGHTERANALYSIS}
        The weight function $\omega_{\gamma} := \kappa_0 \exp [(1 + (1 + t)^{-\frac{2}{2k-1}}) \left\langle v \right\rangle^{\frac{3 - \gamma}{2}}]$
        can help us eliminate the tricky term $\frac{e^{\alpha}}{m^{\alpha}} \nabla_{x} \phi^{\varepsilon} \cdot \frac{w(t,x,v)}{\sqrt{\mu^{\alpha}_M}} \nabla_{v}(\frac{\sqrt{\mu^{\alpha}_M}}{w(t,x,v)})$.  However, the velocity weighting index $\frac{3 - \gamma}{2}$
        cannot exceed 2, because of the joint equilibrium state $\mathbf{F}_0$, obtained from the conservation of momentum and energy. Otherwise, terms such as $\mathcal{N}^\alpha_i$
        (for $i = 2, 4, 5$) cannot be controlled.
    \end{remark}

    \subsection*{Mathematical Obstacles in Soft Potential Case}
   The term $\langle v \rangle^\gamma$ lacks a positive lower bound as $v \to +\infty$, resulting in some mathematical difficulties.
   \subsubsection*{Local Existence Analysis}
    We use $f(t,x,v)$ to denote a nonlinear term.
     The control of the weighted norm $\displaystyle{\sup_{0\leq s \leq t} }\|\left \langle v \right\rangle^{|\gamma|}f(s)\|_{L^{\infty}_{x,v}}$ is particularly crucial for terms requiring $\varepsilon$-balancing with inherent $\frac{1}{\varepsilon}$ coefficients.
    We denote
    the characteristic structure as
    \begin{gather*}
        \mathcal{T}^\alpha_{0}(t,s) =: \exp \left(-\frac{1}{C\varepsilon}\int_{s}^{t} \left \langle v \right\rangle^{\gamma} d\tau\right).
    \end{gather*}
    In hard potential case, it implies that
    \begin{gather*}
        \left|\int_{0}^{t} \mathcal{T}_{0}^{\alpha}(t,s) f(t,x,v) ds \right|\lesssim  \left|\int_{0}^{t} \mathcal{T}_{0}^{\alpha}(t,s)  ds \right| \sup_{0\leq s \leq t} \|f(s)\|_{L^{\infty}_{x,v}}\lesssim \varepsilon \sup_{0\leq s \leq t} \|f(s)\|_{L^{\infty}_{x,v}}.
    \end{gather*}
    While in soft potential case, since $\left \langle v \right\rangle^{\gamma}$ lacks a positive lower bound, one could only get
    \begin{gather*}
        \left|\int_{0}^{t} \mathcal{T}_{0}^{\alpha}(t,s) f(t,x,v) ds \right|\lesssim  \left|\int_{0}^{t} \mathcal{T}_{0}^{\alpha}(t,s)  ds \right| \sup_{0\leq s \leq t} \|f(s)\|_{L^{\infty}_{x,v}}\lesssim  \sup_{0\leq s \leq t} \|f(s)\|_{L^{\infty}_{x,v}}.
    \end{gather*}
    To extract the small parameter $\varepsilon$, the following estimate is considered
    \begin{align}
        \left|\int_{0}^{t} \mathcal{T}_{0}^{\alpha}(t,s) f(t,x,v) ds \right|
         &\lesssim  \left|\int_{0}^{t} \left \langle v \right\rangle^{\gamma}  \mathcal{T}_{0}^{\alpha} (t,s)  ds \right| \sup_{0\leq s \leq t} \|\frac{f(s)}{\left \langle v \right\rangle^{\gamma}}\|_{L^{\infty}_{x,v}} \notag'\\
         &\lesssim  \varepsilon \sup_{0\leq s \leq t} \|\left \langle v \right\rangle^{|\gamma|}f(s)\|_{L^{\infty}_{x,v}}.\label{resollowerbbd}
    \end{align}
    \subsubsection*{Long-Time Existence Analysis}
    For large time scales $t$, the solution estimates via temporal iterations (see \eqref{dddHDdgui0H1}) are obstructed in soft potential cases by the absence of constant $\nu_0$ satisfying \eqref{nudetjH}. But if the initial
data of the remainder  has a compact support, there is a hope
for solving this problem. We use $supp_{t}\{v|f(t,x,v)\}$ to denote
the support of the function $f(t,x,v)$ with respect to the variable
$v$ at time $t$. The radius of the support is defined as follows:
    \begin{equation}
        R_s(t|f(t,x,v)) =   \rm{sup}\{|v|:v\in \rm{supp}_{t}[f(t,x,v)]\}.
    \end{equation}
Given the finite initial support radius $R_s(0|S_{\gamma}(t,x,v)) < +\infty$, the temporal evolution of the support radius requires careful analysis. Recall the characteristic equations
    \begin{eqnarray*}
        \begin{split}
            &\frac{d}{d\tau}V^{\alpha}(\tau;t,x,v)=\frac{e^{\alpha}}{m^{\alpha}}\phi^{\varepsilon}(\tau,X^{\alpha}(\tau;t,x,v)), \quad V^{\alpha}(\tau;t,x,v)=v,
        \end{split}
    \end{eqnarray*}
      By Lemma \ref{lemofcharacter} and \eqref{psideddddg}, it yields
    \begin{eqnarray}\label{SOFSUPGT59}
        R_s(t|f(t,x,v)) \leq  R_s(0|f(t,x,v))+C t,
    \end{eqnarray}
    which imply bounded propagation speed for $R_s(t|f(t,x,v))$. Because of \eqref{44HDW1INFXV}-\eqref{44FsinaldT0s0}, the exponential decay term $e^{-\langle v \rangle^\gamma T_0/(2\varepsilon)}$ must be small enough, which restricts the velocity domain to $v \sim O(\varepsilon^{-\iota_d/|\gamma|})$. Here the parameter $\iota_d$ is selectable within $[\frac{10^8}{1+10^8},1)$, and for example, we can  take
    $
    \iota_d := \max\left\{\frac{10^8}{1+10^8}, \frac{2k-3}{2k-1}\right\}.
    $
    Consequently, for all $t \in [0, T_\varepsilon]$ with 
    \begin{equation*}
    	T_\varepsilon \sim
    	\begin{cases}
    		& \varepsilon^{-\frac{2k-3}{2(2k-1)}} \,\,\,\,\qquad -1\leq \gamma \leq 1, \\
    		& \varepsilon^{-\frac{2k-3}{(1-\gamma)(2k-1)}}\quad  -3<\gamma<-1.
    	\end{cases}
    \end{equation*}
     In both cases by \eqref{SOFSUPGT59}, it holds that:
    \begin{equation}\label{SOFSUPGT510}
      \varepsilon \lesssim \varepsilon^{\iota_d} \lesssim |R_s(t|f)|^{|\gamma|} .
    \end{equation}
    Therefore, we have
    \begin{equation*}
    	\frac{\langle v \rangle^\gamma T_0}{2\varepsilon} \rightarrow +\infty, \quad e^{-\langle v \rangle^\gamma T_0/(2\varepsilon)} \rightarrow 0^{+}, \quad as \quad \varepsilon \rightarrow 0^{+}.
    \end{equation*}
    which yields an iterable mathematical structure.

    \subsection{$L_{x,v}^{\infty}$ estimates for $\mathbf{S}_{\gamma}$.}
From the notation in \eqref{xxbfdmaisnlower1} and
\eqref{Threedefweight1}, the solution $S^{\alpha}_{\gamma}$ of the
original nonlinear equation \eqref{reeqmain} can be written as
    \begin{gather}\label{SSS5555MAINEQ}
        \begin{split}
            \partial_t  S^\alpha_{\gamma} & + v\cdot\nabla_{x} S^\alpha_{\gamma} + \frac{e^{\alpha}}{m^{\alpha}} \nabla_{x}\phi^{\varepsilon}\cdot\nabla_{v}  S^\alpha_{\gamma} + \hat{\upsilon_{\varepsilon}^\alpha}\,S^\alpha_{\gamma}
            = \sum_{i=1}^{5} \mathcal{N}^\alpha_{i}+ \frac{1}{\varepsilon}(K_{M,\omega_{\gamma}}^{\alpha,\chi}+K_{M,\omega_{\gamma}}^{\alpha,c})\mathbf{S}_{\gamma},
        \end{split}
    \end{gather}
    where   \begin{align}
            &\hat{\upsilon_{\varepsilon}^\alpha}(t,x,v)=    \frac{\upsilon^\alpha}{\varepsilon} +\frac{e^{\alpha}}{m^{\alpha}}\nabla_{x}\phi^{\varepsilon} \cdot  \frac{\omega_{\kappa}(t,v)}{\sqrt{\mu^{\alpha}_M}} \nabla_{v}(\frac{\sqrt{\mu^{\alpha}_M}}{\omega_{\kappa}(t,v)})-\frac{\partial_t\omega_{\kappa}(t,v)}{\omega_{\kappa}(t,v)},\notag\\
            &\mathcal{N}^\alpha_{1}= \varepsilon^{k-1}  \frac{ \omega_{\gamma}}{\sqrt{\mu^{\alpha}_M}}\sum_{\beta=A,B} Q^{ \alpha\beta }(\frac{\sqrt{\mu^{\alpha}_M}}{\omega_{\gamma}}S^\alpha_{\gamma}, \frac{\sqrt{\mu^{\beta}_M}}{\omega_{\gamma}} S^\beta_{\gamma}),\notag\\
            &\mathcal{N}_{2}^\alpha=  \frac{\omega_{\gamma}}{\sqrt{\mu^{\alpha}_M}} \sum_{\beta=A,B}\sum_{i=1}^{2k-1} \varepsilon^{i-1}\Big[ Q^{ \alpha \beta}(F_i^{\alpha},\frac{\sqrt{\mu^{\beta}_M}}{\omega_{\gamma}}S^\beta_\gamma)+Q^{\alpha \beta}(\frac{\sqrt{\mu^{\alpha}_M}}{\omega_{\gamma}}S^\alpha_\gamma, F_i^{\beta})\Big], \notag\\
            &\mathcal{N}^\alpha_{3} =  -\frac{e^{\alpha}}{m^{\alpha}}\nabla_{x}\phi_R \cdot  \frac{\omega_{\gamma}}{\sqrt{\mu^{\alpha}_M}} \nabla_{v}(\mu^{\alpha}+\sum_{i=1}^{2k-1}\varepsilon^iF_i^{\alpha}),\label{NLINEANNM1M2}\\
            &\mathcal{N}^\alpha_{4} = - \frac{\omega_{\gamma}}{\sqrt{\mu^{\alpha}_M}} \Big[\varepsilon^{k-1}( \partial_t +v\cdot \nabla_{x} ) F^\alpha_{2k-1}+\mathop{\sum}_{i+j \geq 2k-1\atop 0 \leq i,j \leq 2k-1}\varepsilon^{i+j-k}\nabla_x\phi_i\cdot\nabla_vF_j^\alpha\Big],\notag\\
            &\mathcal{N}^\alpha_{5} =  \frac{\omega_{\gamma}}{\sqrt{\mu^{\alpha}_M}}  \Big\{\sum_{\beta=A,B} \mathop{\sum}_{i+j \geq 2k\atop 1 \leq i,j \leq 2k-1} \varepsilon^{i+j-k-1}\Big[Q^{\alpha \beta }(F_i^{\alpha},F_j^{\beta})+Q^{\alpha\beta }(F_j^{\alpha},F_i^{\beta})\Big] \Big\}\notag.
            \end{align}
Making use of Duhamel's principle,  one has
    \begin{gather}\label{SGMZONG512}
        \begin{split}
            S^\alpha_{\gamma}(t,x,v) = & \exp \left(-\int_{0}^{t}\hat{\upsilon_{\varepsilon}^\alpha}(\tau,  X^{\alpha}(\tau),V^{\alpha}(\tau)) d\tau\right)S^\alpha_{\gamma}(0,X^{\alpha}(0;t,x,v),V^{\alpha}(0;t,x,v)) \\
            & + \int_{0}^{t}\exp \left(-\int_{s}^{t}\hat{\upsilon_{\varepsilon}^\alpha}(\tau,  X^{\alpha}(\tau),V^{\alpha}(\tau))   d\tau \right) \sum_{i=1}^{5}\mathcal{N}^\alpha_i(s, X^{\alpha}(s),V^{\alpha}(s))ds\\
            &+ \frac{1}{\varepsilon} \int_{0}^{t}\exp \left(-\int_{s}^{t}\hat{\upsilon_{\varepsilon}^\alpha}(\tau,  X^{\alpha}(\tau),V^{\alpha}(\tau))  d\tau \right)  K_{M,\omega_{\gamma}}^{\alpha,\chi} \mathbf{S}_{\gamma} (s, X^{\alpha}(s),V^{\alpha}(s))ds\\
            &+ \frac{1}{\varepsilon} \int_{0}^{t}\exp \left(-\int_{s}^{t}\hat{\upsilon_{\varepsilon}^\alpha}(\tau,  X^{\alpha}(\tau),V^{\alpha}(\tau))   d\tau \right)  K_{M,\omega_{\gamma}}^{\alpha,c} \mathbf{S}_{\gamma} (s, X^{\alpha}(s),V^{\alpha}(s))ds.
        \end{split}
    \end{gather}
    First, from \eqref{xxbfdmaisnlower1}, the inequality
        $\hat{\upsilon}_{\varepsilon}^\alpha > \frac{\langle v \rangle^{\gamma}}{C\varepsilon}$
    holds, then we have
    \begin{equation*}
        \exp\left(-\int_{0}^{t}\hat{\upsilon}_{\varepsilon}^\alpha(\tau, X^{\alpha}(\tau),V^{\alpha}(\tau)) d\tau\right)
        < \exp\left(-\frac{1}{C\varepsilon}\int_{s}^{t}\langle V^{\alpha}(\tau) \rangle^{\gamma} d\tau\right).
    \end{equation*}
Moreover, setting
    \begin{align*}
        \hat{\mathcal{F}^\alpha_{0}}(t,s)&=:\exp \left(-\int_{s}^{t}\hat{\upsilon_{\varepsilon}^\alpha}(\tau,  X^{\alpha}(\tau),V^{\alpha}(\tau)) d\tau\right),\\
        \mathcal{T}^\alpha_{0}(t,s) &=: \exp \left(-\frac{1}{C\varepsilon}\int_{s}^{t}\left \langle V^{\alpha}(\tau) \right\rangle^{\gamma}
d\tau\right).
    \end{align*}
 Then, \eqref{SGMZONG512} is estimated  by
    \begin{gather}\label{SGMZONG512TT}
        \begin{split}
            |S^\alpha_{\gamma}|(t,x,v) \leq & \mathcal{T}^\alpha_{0}(t,0)|S^\alpha_{\gamma}|(0,X^{\alpha}(0;t,x,v),V^{\alpha}(0;t,x,v)) \\
            & + \int_{0}^{t}\mathcal{T}^\alpha_{0}(t,s) \sum_{i=1}^{5}|\mathcal{N}^\alpha_i|(s, X^{\alpha}(s),V^{\alpha}(s))ds\\
            &+ \frac{1}{\varepsilon} \int_{0}^{t}\mathcal{T}^\alpha_{0}(t,s)  |K_{M,\omega_{\gamma}}^{\alpha,\chi} \mathbf{S}_{\gamma} |(s, X^{\alpha}(s),V^{\alpha}(s))ds\\
            &+ \frac{1}{\varepsilon} \int_{0}^{t}\mathcal{T}^\alpha_{0}(t,s) |K_{M,\omega_{\gamma}}^{\alpha,c} \mathbf{S}_{\gamma}| (s, X^{\alpha}(s),V^{\alpha}(s))ds.
        \end{split}
    \end{gather}
  Now let us control the terms on the right-hand side of \eqref{SGMZONG512TT}. First, the estimates \eqref{les1estTZX} and \eqref{ges1estTZX} hold for sufficiently small $T_0$. Consequently, when $0\leq\gamma<1$, there exists a constant $\nu_o>0$ satisfying
    \begin{equation*}
        \left \langle v \right\rangle^{\gamma} \sim \nu_o \gtrsim 1.
    \end{equation*}
    Then for $0\leq s < t \leq T_0$,  the following estimate holds:
    \begin{gather*}
        \left|\int_{0}^{t} \mathcal{T}_{0}^{\alpha}(t,s) \left \langle V^{\alpha}(s) \right\rangle^{\gamma} e^{-\frac{\nu_os}{\varepsilon}} ds \right|\lesssim e^{-\frac{\nu_ot}{\varepsilon}} \int_{0}^{t}\exp \left(-\frac{ \left \langle v \right\rangle^{\gamma}(t-s)}{ C\varepsilon}\right)\left \langle v \right\rangle^{\gamma}  ds \lesssim \varepsilon e^{-\frac{\nu_ot}{\varepsilon}}.
    \end{gather*}
    For soft potentials with compactly supported initial data, the support propagation analysis \eqref{SOFSUPGT59} and \eqref{SOFSUPGT510} implies that when
    $-3 < \gamma < 0$, there exists an $\varepsilon$-dependent parameter $\nu_{\varepsilon}>0$ such that
    \begin{equation}\label{SSP516CZ}
        \left \langle v \right\rangle^{\gamma} \sim \nu_{\varepsilon}>0.
    \end{equation}
    This inequality holds for all $v>0$ within the velocity support set and $t>0$ in the maximal existence time interval  $O(\varepsilon^{-\frac{1}{2}\frac{2k-3}{2k-1}})$, leading to
    \begin{gather*}
        \left|\int_{0}^{t} \mathcal{T}_{0}^{\alpha}(t,s) \left \langle v \right\rangle^{\gamma}e^{-\frac{\nu_{\varepsilon}s}{\varepsilon}} ds \right|\lesssim e^{-\frac{\nu_{\varepsilon}t}{\varepsilon}} \int_{0}^{t}\exp \left(-\frac{ \left \langle v \right\rangle^{\gamma}(t-s)}{ C\varepsilon}\right)\left \langle v \right\rangle^{\gamma}  ds \lesssim \varepsilon e^{-\frac{\nu_{\varepsilon}t}{\varepsilon}}.
    \end{gather*}
    To streamline subsequent expressions, the following notation is introduced :
    \begin{equation*}
        \nu_{\gamma}=
        \begin{cases}
            & \nu_o\,\,\,\,\,\,\,\,\,\,\quad 0\leq \gamma<1, \\
            & \nu_{\varepsilon} \,\,\,\quad -3<\gamma<0 .
        \end{cases}
    \end{equation*}
    Estimates for \eqref{linfHADTOT}'s right-hand side follow from
    \begin{gather*}\label{4.181}
        \left|\mathcal{T}_0^{\alpha}(t,0) S^\alpha_{\gamma}(0,\,X^{\alpha}(0),\,V^{\alpha}(0))\right| \lesssim \varepsilon  e^{-\frac{\nu_{\gamma}t}{\varepsilon}} \|\mathbf{S}^{in}_{\gamma}\|_{L^\infty_{x,v}}.
    \end{gather*}

   Next, we choose the weight function
    \begin{equation*}
        \omega_{\gamma}=e^{\tilde{\kappa}\left \langle v \right\rangle^{\frac{3-\gamma}{2}}},\quad \tilde{\kappa}=\kappa_0 (1+(1+t)^{-\frac{2}{2k-1}}), \quad -1\leq\gamma<1,
    \end{equation*}
    where
    $\kappa_0$
    is a sufficiently small constant. Substituting the weighting function
   $\left \langle v \right\rangle^{l}$ with $\omega_{\gamma}$
    preserves the estimation validity. Following the approach in \eqref{HHHD111}-\eqref{HHHDDB}, one obtains
    \begin{align*}
            |\mathcal{N}_1^\alpha| = & \varepsilon^{k-1} \left| \sum_{\beta=A,B}\frac{\omega_{\gamma}}{\sqrt{\mu^{\alpha}_M}} Q^{\alpha\beta }(\sqrt{\mu^{\alpha}_M}\frac{S^\alpha_{\gamma}}{\omega_{\gamma}}, \sqrt{\mu^{\beta}_M} \frac{S^\beta_{\gamma}}{\omega_{\gamma}})\right|\\
            \lesssim & \varepsilon^{k-1} \left \langle v \right\rangle^{\gamma} \, \Vert \mathbf{S}_{\gamma} \Vert_{L^2_{x,v}}^2,\\
            |\mathcal{N}_2^{\alpha}|= & \Big| \sum_{\beta=A,B} \sum_{i=1}^{2k-1} \varepsilon^{i-1} \frac{\omega_{\gamma}}{\sqrt{\mu^{\alpha}_M}} \Big[Q^{\alpha\beta }(F_i^{\alpha}, \sqrt{\mu^{\beta}_M}\frac{S^\beta_{\gamma}}{\omega_{\gamma}})+Q^{\alpha\beta }(\sqrt{\mu^{\alpha}_M}\frac{S^\alpha_{\gamma}}{\omega_{\gamma}},F_i^{\beta})\Big] \Big| \\
            \lesssim  & \, \left \langle v \right\rangle^{\gamma}  \mathcal{I}_1 \,\Vert\mathbf{S}_{\gamma} \Vert_{L^\infty_{x,v}},\\
            |\mathcal{N}^\alpha_3| =& |\frac{e^{\alpha}}{m^{\alpha}}\nabla_{x}\phi_R \cdot  \frac{\omega_{\gamma}}{\sqrt{\mu^{\alpha}_M}} \nabla_{v}(\mu^{\alpha}+\sum_{i=1}^{2k-1}\varepsilon^iF_i^{\alpha})|\\
            \lesssim &(1+\varepsilon\mathcal{I}_1)\left \langle v \right\rangle^{\gamma} \Vert\mathbf{S}_{\gamma} \Vert_{L^{\infty}_{x,v} },\\
            |\mathcal{N}^\alpha_4| = & \Big|\frac{\omega_{\gamma}}{\sqrt{\mu^{\alpha}_M}} \Big[\varepsilon^{k-1}( \partial_t +v\cdot \nabla_{x} ) F^\alpha_{2k-1}+\mathop{\sum}_{i+j \geq 2k-1\atop 0 \leq i,j \leq 2k-1}\varepsilon^{i+j-k}\nabla_x\phi_i\cdot\nabla_vF_j^\alpha\Big] \Big|\\
            \lesssim  & \varepsilon^{k-1} \left \langle v \right\rangle^{\gamma} \mathcal{I}_2,\\
            |\mathcal{N}^\alpha_5| = & \Big| \sum_{\beta=A,B} \mathop{\sum}_{i+j \geq 2k\atop 1 \leq i,j \leq 2k-1} \varepsilon^{i+j-k-1}\frac{\omega_{\gamma}}{\sqrt{\mu^{\alpha}_M}} \Big[Q^{\alpha \beta }(F_i^{\alpha},F_j^{\beta})+Q^{\alpha\beta }(F_j^{\alpha},F_i^{\beta})\Big] \Big| \\
            \lesssim  &\,\varepsilon^{k-1} \left \langle v \right\rangle^{\gamma} \mathcal{I}_2.
        \end{align*}
   Consequently, it holds that
    \begin{align*}
            &\left|\int_{0}^{t}\mathcal{T}_0^{\alpha}(t,s)\sum_{i=1}^{5}\mathcal{N}^\alpha_i(s,X^{\alpha}(s),V^{\alpha}(s))ds \right| \\
            &\hspace{0.4cm}\lesssim  \,\varepsilon(1+\mathcal{I}_1+\varepsilon\mathcal{I}_1) e^{-\frac{\nu_{\gamma}t}{2\varepsilon}} \sup_{0\leq s \leq t}(e^{\frac{\nu_{\gamma}s}{2\varepsilon}}\Vert \mathbf{S}_{\gamma} (s)\Vert_{L^{\infty}_{x,v}}) \\
            &\hspace{0.8cm}+\varepsilon^k  e^{-\frac{\nu_{\gamma} t}{\varepsilon}} \sup_{0\leq s \leq t}(e^{\frac{\nu_{\gamma} s}{2\varepsilon}}\Vert \mathbf{S}_{\gamma} (s)\Vert_{L^{\infty}_{x,v}})^2+\varepsilon^k\mathcal{I}_2.
    \end{align*}
    Moreover, from Lemma \ref{kchiestmainle}, one obtains
    \begin{gather*}
| K_{M,\omega_{\gamma}}^{\alpha,\chi}\mathbf{S}_{\gamma}
(s,X^{\alpha}(s),V^{\alpha}(s)) | \leq C {\delta}^{3+\gamma} \left
\langle V^{\alpha}(s) \right\rangle^{\gamma} e^{-c|V^{\alpha}(s)|^2}
|\mathbf{S}_{\gamma} |_{L^\infty_v},
    \end{gather*}
   which leads to
    \begin{gather*}
        \begin{split}
            \frac{1}{\varepsilon} \left|\int_{0}^{t}\mathcal{T}_0^{\alpha}(t,s) K_{M,\omega_{\gamma}}^{\alpha,\chi}\mathbf{S}_{\gamma}  (s,X^{\alpha}(s),V^{\alpha}(s))ds \right|
            \lesssim \delta^{3+\gamma}e^{-\frac{\nu_{\gamma}t}{2\varepsilon}} \sup_{0\leq s \leq t}(e^{\frac{\nu_{\gamma}s}{2\varepsilon}}\Vert \mathbf{S}_{\gamma} (s)\Vert_{L^{\infty}_{x,v}}).
        \end{split}
    \end{gather*}
Finally, the regular part $K_{M,w}^{\alpha,c} \mathbf{S}_{\gamma}$ is expressed by
    \begin{eqnarray}\label{SPKC00}
        \begin{split}
            &K_{M,w}^{\alpha,c} \mathbf{S}_{\gamma}(s,X^{\alpha}(s),V^{\alpha}(s))\\
            =&\sum_{\beta=A,B}\int_{\mathbb{R}^3} [k_{M,1}^{\alpha \beta}+k_{M,2}^{\alpha \beta}] (s,X^{\alpha}(s),V^{\alpha}(s),v_*)\frac{w(V^{\alpha}(s))}{w(v_*)}S_{\gamma}^{\beta}(s,X^{\alpha}(s),v_*) d v_{*}\\
            =&\sum_{\beta=A,B}\int_{\mathbb{R}^3} k_{M,\omega_{\gamma}}^{\alpha \beta}(s,X^{\alpha}(s),V^{\alpha}(s),v_*)S_{\gamma}^{\beta}(s,X^{\alpha}(s),v_*) d v_{*}.
        \end{split}
    \end{eqnarray}
   For clarity, we define the following operators
    \begin{gather*}
        \begin{split}
            \hat{\mathcal{F}}^\beta_1 (s,s_1)  =:\exp \left(-\int_{s_1}^{s}\hat{\upsilon_{\varepsilon}^\alpha}(\tau,  X^{\alpha}(\tau),V^{\alpha}(\tau)) d\tau\right),
~~\mathcal{T}^{\beta}_{1}(s,s_1) =: \exp
\left(-\frac{1}{C\varepsilon}\int_{s_1}^{s}\left \langle
V^{\beta}(\tau) \right\rangle^{\gamma} d\tau\right).
        \end{split}
    \end{gather*}
    The bound for $S^\beta_\gamma(s,X(s),v_*)$ is derived as
    \begin{gather}\label{SP00expr}
        \begin{split}
            |S^\beta_{\gamma}|(s,X^{\alpha}(s),v_*)
            \leq & \mathcal{T}^\beta_1(s, 0) |S^\beta_{\gamma}|(0,X^{\alpha,\beta}_1(0),V^{\alpha,\beta}_1(0)) \\
            &+ \sum_{j=1}^5 \int_{0}^s \mathcal{T}_1^\beta(s,s_1) |\mathcal{N}^\beta_j|(s_1,X^{\alpha,\beta}_1(s_1),V^{\alpha,\beta}_1(s_1))ds_1\\
            & + \frac{1}{\varepsilon} \int_{0}^s \mathcal{T}_1^\beta(s,s_1)  |K_{M,\omega_{\gamma}}^{\alpha,\chi}  \mathbf{S}_{\gamma}|(s_1,X^{\alpha,\beta}_1(s_1),V^{\alpha,\beta}_1(s_1))ds_1\\
            & + \frac{1}{\varepsilon} \int_{0}^s \mathcal{T}_1^\beta(s,s_1)  |K_{M,\omega_{\gamma}}^{\alpha,c}
\mathbf{S}_{\gamma}|(s_1,X^{\alpha,\beta}_1(s_1),V^{\alpha,\beta}_1(s_1))ds_1,
        \end{split}
    \end{gather}
which together with \eqref{SPKC00} yields
    \begin{equation}\label{SPJJO00}
        \frac{1}{\varepsilon} \left|\int_{0}^{t}\mathcal{T}_0^{\alpha}(t,s) K_{M,\omega_{\gamma}}^{\alpha,c}  \mathbf{S}_{\gamma}(s,X^{\alpha}(s),V^{\alpha}(s))ds \right| \leq \sum_{i=0}^{2} \hat{\mathcal{J}}_i(t) + \hat{\mathcal{O}}^\alpha(t).
    \end{equation}
    Moreover, applying Lemma \ref{LeMK2ker4444} with $-1 \leq \gamma < 1$, we obtain the kernel estimate:
    \begin{equation*}
        \begin{split}
            &\int_{\mathbb{R}^3} \int_{\mathbb{R}^3} |k_{M,\omega_{\gamma}}^{\alpha\beta}(s,X^{\alpha}(s),V^{\alpha}(s),v_*)k_{M,\omega_{\gamma}}^{\beta\beta^{'}}(s_1,X^{\alpha,\beta}_1(s_1),V^{\alpha,\beta}_1(s_1),v_{**})|dv_{**}\, dv_*\\
            &\hspace{2cm}\leq \frac{C\left \langle V(s) \right\rangle^{\gamma}}{1+|V(s)|}  \times \left \langle V_1(s_1) \right\rangle^{\gamma} .
        \end{split}
    \end{equation*}
   The estimates in \eqref{SPJJO00} follow from the same methodology developed in Section 4, that is, we can show
    \begin{gather*}
        \begin{split}
            \sup_{0\leq s \leq T_0}(e^{\frac{\nu_{\gamma}s}{2\varepsilon}}\Vert  S_{\gamma}^{\alpha}(s) &\Vert_{L^{\infty}_{x,v}})
            \leq C \sup_{0\leq s \leq T_0}[(1+\frac{s}{\varepsilon})e^{-\frac{\nu_{\gamma}s}{2\varepsilon}}]\|\mathbf{S}^{in}_{\gamma}\|_{L^\infty_{x,v}}+C\varepsilon^k\mathcal{I}_2(T_0)e^{\frac{\nu_{\gamma}T_0}{2\varepsilon}}\\
            &+[C(1+\mathcal{I}_1(T_0))\varepsilon+C\lambda+\frac{C_\lambda}{N}+\delta^{3+\gamma}]\sup_{0\leq s \leq T_0}(e^{\frac{\nu_{\gamma}s}{2\varepsilon}}\Vert  \mathbf{S}_{\gamma}(s) \Vert_{L^{\infty}_{x,v}})\\
            &+C\varepsilon^k \sup_{0\leq s \leq T_0}(e^{\frac{\nu_{\gamma}s}{2\varepsilon}}\Vert  \mathbf{S}_{\gamma}(s) \Vert_{L^{\infty}_{x,v}})^2
            +\frac{C_{N,\lambda}}{\varepsilon^{\frac{3}{2}}}e^{\frac{\nu_{\gamma}T_0}{2\varepsilon}}\sup_{0\leq s \leq T_0}\|
\mathbf{f}_R(s)\|_{L^{2}_{x,v}},
        \end{split}
    \end{gather*}
    which leads to
    \begin{gather}\label{CCCCFU0}
        \begin{split}
            \sup_{0\leq s \leq T_0}(e^{\frac{\nu_{\gamma}s}{2\varepsilon}}\Vert  S_{\gamma}^{\alpha}(s) \Vert_{L^{\infty}_{x,v}})
            \leq & C \|\mathbf{S}^{in}_{\gamma}\|_{L^\infty_{x,v}}+\frac{C_{N,\lambda}}{\varepsilon^{\frac{3}{2}}}e^{\frac{\nu_{\gamma}T_0}{2\varepsilon}}\sup_{0\leq s \leq T_0}\| \mathbf{f}_R(s)\|_{L^{2}_{x,v}}+C\varepsilon^k\mathcal{I}_2(T_0)e^{\frac{\nu_{\gamma}T_0}{2\varepsilon}}.
        \end{split}
    \end{gather}
    For the hard potential case, the analysis directly follows from \eqref{HSC465}-\eqref{HSC467}. In contrast, the soft potential case requires more careful consideration. Based on \eqref{SOFSUPGT59}, \eqref{SOFSUPGT510} and \eqref{SSP516CZ}, we observe the asymptotic behavior
    \begin{eqnarray}\label{CSJCHUFAXXM}
        e^{-\frac{\nu_{\gamma}T_0}{2\varepsilon}} \rightarrow 0, \qquad as \quad \varepsilon \rightarrow 0.
    \end{eqnarray}
    Multiplying \eqref{CCCCFU0} by the scaling factor
    $\varepsilon^{\frac{3}{2}}e^{-\frac{\nu_oT_0}{2\varepsilon}}$
    , yields that
    \begin{equation}\label{44L000ses}
        \varepsilon^{\frac{3}{2}}\Vert  \mathbf{S}_{\gamma}(T_0) \Vert_{L^{\infty}_{x,v}}\leq \frac{1}{4}\Vert \varepsilon^{\frac{3}{2}} \mathbf{S}^{\rm in}_{\gamma} \Vert_{L^{\infty}_{x,v}}+C \sup_{0\leq s \leq T_0}\| \mathbf{f}_R(s)\|_{L^{2}_{x,v}}+C\varepsilon^{\frac{2k+3}{2}}.
    \end{equation}

    \subsection{$W_x^{1,\infty}$ estimate }To confirm the crucial assumption \eqref{SUPPOSELINFGXV}, we have to
    consider $W_x^{1,\infty}$ estimates to close the energy. We take $D_x$ of \eqref{SSS5555MAINEQ} to get
    \begin{eqnarray}\label{DXHPSP540}
        \begin{split}
            &\hspace{0.5cm}\partial_{t}(D_xh_R^{\alpha})+v\cdot\nabla_{x}(D_xS_{\gamma}^{\alpha})+\frac{e^{\alpha}}{m^{\alpha}} \nabla_{x}\phi^{\varepsilon}\cdot\nabla_{v}  (D_xS_{\gamma}^{\alpha}) + \hat{\upsilon_{\varepsilon}^\alpha}  (D_xS_{\gamma}^{\alpha}) \\
            &=-\frac{e^{\alpha}}{m^{\alpha}}\nabla_{x}(D_x\phi^{\varepsilon})\cdot\nabla_{v}S_{\gamma}^{\alpha}-(D_x\hat{\upsilon_{\varepsilon}^\alpha}) S_{\gamma}^{\alpha}-D_x[
\frac{1}{\varepsilon}(K_{M,\omega_{\gamma}}^{\alpha,\chi}+K_{M,\omega_{\gamma}}^{\alpha,c})\mathbf{S}_{\gamma}]
        + \sum_{i=1}^{5}D_x (\mathcal{N}^\alpha_i),
        \end{split}
    \end{eqnarray}
    where $\mathcal{N}^\alpha_i$ is described in \eqref{NLINEANNM1M2}. Along the trajectory, the solution $D_xS_{\gamma}^{\alpha}$ of the equation \eqref{DXHPSP540} is bounded by
    \begin{gather}\label{linfHADTOSPDXT}
        \begin{split}
            |D_xS^\alpha_{\gamma}|(t,x,v) \leq\, &\mathcal{T}^\alpha_0(t,0)\,|D_xS^\alpha_{\gamma}|(0,X^{\alpha}(0;t,x,v),V^{\alpha}(0;t,x,v)) \\
            &+\frac{|e^{\alpha}|}{m^{\alpha}}\int_{0}^{t}\mathcal{T}^\alpha_0(t,s) |\nabla_{x}(D_x\phi^{\varepsilon})\cdot\nabla_{v}S_{\gamma}^{\alpha}|(s, X^{\alpha}(s),V^{\alpha}(s))ds\\
            &+ \int_{0}^{t}\mathcal{T}^\alpha_0(t,s)|(D_x\hat{\upsilon_{\varepsilon}^\alpha})\, S_{\gamma}^{\alpha}| (s, X^{\alpha}(s),V^{\alpha}(s))ds\\
            & + \int_{0}^{t}\mathcal{T}^\alpha_0(t,s) \sum_{i=1}^{5}|D_x\mathcal{N}^\alpha_j|(s, X^{\alpha}(s),V^{\alpha}(s))ds\\
            &+ \frac{1}{\varepsilon} \int_{0}^{t}\mathcal{T}^\alpha_0(t,s)  |D_x(K_{M,\left \langle v \right\rangle^{l}}^{\alpha,\chi} \mathbf{S}_{\gamma})| (s, X^{\alpha}(s),V^{\alpha}(s))ds\\
            &+\frac{1}{\varepsilon} \int_{0}^{t}\mathcal{T}^\alpha_0(t,s)  |D_x(K_{M,\left \langle v \right\rangle^{l}}^{\alpha,c} \mathbf{S}_{\gamma})| (s, X^{\alpha}(s),V^{\alpha}(s))ds.
        \end{split}
    \end{gather}
    The first line in right hand side in \eqref{linfHADTOSPDXT} is bounded by
    \begin{equation}\label{SDXK}
        \left|\mathcal{T}_0^{\alpha}(t,0) (D_xS^\alpha_{\gamma})(0,\,X(0),\,v)\right| \lesssim \varepsilon  e^{-\frac{\nu_{\gamma}t}{\varepsilon}} \|D_x\mathbf{S}^{in}_{\gamma}\|_{L^\infty_{x,v}}.
    \end{equation}
   Following the approach in \eqref{psideddddg}, the gradient estimates
    \begin{equation}
        | \nabla_{x}\phi_R|_{C^{1,\alpha_{0}}_{x} } \lesssim \Vert\mathbf{S}_{\gamma} \Vert_{W^{1,\infty}_{x} L^{\infty}_{v}}, \qquad | \nabla_{x}\phi^{\varepsilon}|_{C^{1,\alpha_{0}}_{x} }   \lesssim 1+\varepsilon \mathcal{I}_1+\varepsilon^k\Vert\mathbf{S}_{\gamma} \Vert_{W^{1,\infty}_{x} L^{\infty}_{v}},
    \end{equation}
    can be established. Under assumption \eqref{SUPPOSELINFGXV}, the second line of \eqref{linfHADTOSPDXT} can be controlled as follows
    \begin{gather*}
        \begin{split}
            \left|\int_{0}^{t}\mathcal{T}^\alpha_0(t,s) [\nabla_{x}(D_x\phi^{\varepsilon})\cdot\nabla_{v}S_{\gamma}^{\alpha}](s, X^{\alpha}(s),V^{\alpha}(s))ds\right|
            \lesssim \, \varepsilon  e^{-\frac{\nu_{\gamma}t}{2\varepsilon}} \sup_{0\leq s \leq t}(e^{\frac{\nu_{\gamma}s}{2\varepsilon}}\Vert \nabla_{v}\mathbf{S}_{\gamma}\Vert_{L^{\infty}_{x,v}}).
        \end{split}
    \end{gather*}
    Moreover, the expression from \eqref{NLINEANNM1M2}
    \begin{gather*}
        \begin{split}
  &\hat{\upsilon_{\varepsilon}^\alpha}(t,x,v)=    \frac{\upsilon^\alpha}{\varepsilon}
+\frac{e^{\alpha}}{m^{\alpha}}\nabla_{x}\phi^{\varepsilon} \cdot
\frac{\omega_{\kappa}(t,v)}{\sqrt{\mu^{\alpha}_M}}
\nabla_{v}(\frac{\sqrt{\mu^{\alpha}_M}}{\omega_{\kappa}(t,v)})-\frac{\partial_t\omega_{\kappa}(t,v)}{\omega_{\kappa}(t,v)},
        \end{split}
    \end{gather*}
  leads to the term estimate
    \begin{align*}
            |(D_x\hat{\upsilon_{\varepsilon}^\alpha})\, S_{\gamma}^{\alpha}|(t,x,v)
            \leq &  \frac{|(D_x\upsilon^\alpha)|}{\varepsilon} |S_{\gamma}^{\alpha}|+|\frac{e^{\alpha}}{m^{\alpha}}\nabla_{x}(D_x\phi^{\varepsilon}) \cdot  \frac{\omega_{\kappa}(t,v)}{\sqrt{\mu^{\alpha}_M}} \nabla_{v}(\frac{\sqrt{\mu^{\alpha}_M}}{\omega_{\kappa}(t,v)})S_{\gamma}^{\alpha}|\\
            \lesssim & \frac{\left \langle v \right\rangle^{\gamma}}{\varepsilon}\|  \mathbf{S}_{\gamma}\|_{L^{\infty}_{x,v}}
+\left \langle v \right\rangle^{\gamma}\Big[1+\varepsilon\mathcal{I}_1+\varepsilon^k(\|  \mathbf{S}_{\gamma}\|_{L^{\infty}_{x,v}}
+\|D_x\mathbf{S}_{\gamma}\|_{L^{\infty}_{x,v}})\Big]\\
&\hspace{7.5cm}\times\| \left\langle v \right\rangle^{1-\gamma} \mathbf{S}_{\gamma}\|_{L^{\infty}_{x,v}}\\
            \lesssim & \frac{\left \langle v \right\rangle^{\gamma}}{\varepsilon}\|  \mathbf{S}_{\gamma}\|_{L^{\infty}_{x,v}}+\left \langle v \right\rangle^{\gamma}\| \left \langle v \right\rangle^{1-\gamma} \mathbf{S}_{\gamma}\|_{L^{\infty}_{x,v}}.
    \end{align*}
    Then, the third line in \eqref{linfHADTOSPDXT} is bounded by
    \begin{gather*}
        \begin{split}
             \varepsilon e^{-\frac{\nu_{\gamma}t}{2\varepsilon}} \sup_{0\leq s \leq t}(e^{\frac{\nu_{\gamma}s}{2\varepsilon}}\Vert \left \langle v \right\rangle^{1-\gamma} \mathbf{S}_{\gamma}\Vert_{L^{\infty}_{x,v}})
            +\, e^{-\frac{\nu_{\gamma}t}{2\varepsilon}}  \sup_{0\leq s \leq t}(e^{\frac{\nu_{\gamma}s}{2\varepsilon}}\Vert  \mathbf{S}_{\gamma}\Vert_{L^{\infty}_{x,v}}).
        \end{split}
    \end{gather*}
    Following the methodology in \eqref{HHHD111}-\eqref{HHHDDB} (Section 4), the integral estimate
    \begin{gather*}
        \begin{split}
            &\left| \int_{0}^{t}\mathcal{T}^\alpha_0(t,s) \sum_{i=1}^{5}(D_x\mathcal{N}^\alpha_1) (s, X^{\alpha}(s),V^{\alpha}(s))ds\right| \\
            \lesssim &\,\, \varepsilon (1+\mathcal{I}_1) e^{-\frac{\nu_ot}{2\varepsilon}} \sup_{0\leq s \leq t}[e^{\frac{\nu_os}{2\varepsilon}}(\|  \mathbf{S}_{\gamma}\|_{L^{\infty}_{x,v}}+\|D_x\mathbf{S}_{\gamma}\|_{L^{\infty}_{x,v}})]+\varepsilon^k \mathcal{I}_2\\
            &+\varepsilon^k  e^{-\frac{\nu_{\gamma}t}{\varepsilon}} \sup_{0\leq s \leq t}[(e^{\frac{\nu_{\gamma}s}{2\varepsilon}}\|   \mathbf{S}_{\gamma}\|_{L^{\infty}_{x,v}})^2+(e^{\frac{\nu_{\gamma}s}{2\varepsilon}}\|  D_x\mathbf{S}_{\gamma}\|_{L^{\infty}_{x,v}})^2],
        \end{split}
    \end{gather*}
    is obtained. Lemma \ref{kchiestmainle} yields the bound
    \begin{gather}\label{SDXMO}
        \begin{split}
            &\left| \int_{0}^{t}\mathcal{T}^\alpha_0(t,s) D_x (K_{M,\omega_{\gamma}}^{\alpha,\chi} \mathbf{S}_{\gamma}) (s, X^{\alpha}(s),V^{\alpha}(s))ds\right|
            \lesssim \,\, \delta^{3+\gamma}  e^{-\frac{\nu_{\gamma}t}{2\varepsilon}} \sup_{0\leq s \leq t}(e^{\frac{\nu_{\gamma}s}{2\varepsilon}}\|  \mathbf{S}_{\gamma}\|_{L^{\infty}_{x,v}}).
        \end{split}
    \end{gather}
Finally, proceeding with the differentiation, we apply the product rule to obtain
    \begin{equation*}
        D_x (K_{M,\omega_{\gamma}}^{\alpha,c} \mathbf{S}_{\gamma})= (D_xK_{M,\omega_{\gamma}}^{\alpha,c}) \mathbf{S}_{\gamma}+K_{M,\omega_{\gamma}}^{\alpha,c} (D_x\mathbf{S}_{\gamma}),
    \end{equation*}
  whereupon we define the decomposed integrals by
    \begin{gather*}
        \begin{split}
            & \frac{1}{\varepsilon}\int_{0}^{t}\mathcal{T}^\alpha_0(t,s) (D_xK_{M,\omega_{\gamma}}^{\alpha,c}) \mathbf{S}_{\gamma}(s, X^{\alpha}(s),V^{\alpha}(s))ds:=\hat{I}_{1x},\\
            & \frac{1}{\varepsilon}\int_{0}^{t}\mathcal{T}^\alpha_0(t,s) K_{M,\omega_{\gamma}}^{\alpha,c} (D_x\mathbf{S}_{\gamma})(s, X^{\alpha}(s),V^{\alpha}(s))ds
:=\hat{I}_{2x}.
        \end{split}
    \end{gather*}
   The final line in \eqref{linfHADTOSPDXT} can be expressed in the form of an integral kernel that
    \begin{eqnarray}\label{XXXK2CDXJFBD}
        \begin{split}
            &(D_xK_{M,\omega_{\gamma}}^{\alpha,c}) \mathbf{S}_{\gamma}(s, X^{\alpha}(s),V^{\alpha}(s))\\
            =&\sum_{\beta=A,B}\int_{\mathbb{R}^3} (D_xk_{M,\omega_{\gamma}}^{\alpha \beta}) (s,X^{\alpha}(s),V^{\alpha}(s),v_*)S_{\gamma}^{\beta}(s,X^{\alpha}(s),v_*) d v_{*},
        \end{split}
    \end{eqnarray}
    and
    \begin{eqnarray}\label{DXXXK2CDXJFBD}
        \begin{split}
            &K_{M,\omega_{\gamma}}^{\alpha,c} (D_x\mathbf{S}_{\gamma})(s, X^{\alpha}(s),V^{\alpha}(s))\\
            =&\sum_{\beta=A,B}\int_{\mathbb{R}^3} k_{M,\omega_{\gamma}}^{\alpha \beta} (s,X^{\alpha}(s),V^{\alpha}(s),v_*)(D_xS_{\gamma}^{\beta})(s,X^{\alpha}(s),v_*) d v_{*}.
        \end{split}
    \end{eqnarray}
 Substituting $v_*$ for $v$ in \eqref{SGMZONG512} and inserting into \eqref{XXXK2CDXJFBD} yields
    \begin{equation*}
        \hat{I}_{1x}=\frac{1}{\varepsilon}\int_{0}^{t}\mathcal{T}_0^{\alpha}(t,s) (D_xK_{M,\omega_{\gamma}}^{\alpha,c})  \mathbf{S}_{\gamma}(s,X^{\alpha}(s),V^{\alpha}(s))ds  = \sum_{i=0}^{2} \hat{\mathcal{J}}_{i,x}(t) + \hat{\mathcal{O}}_{x}^\alpha(t).
    \end{equation*}
The expressions for $\hat{\mathcal{J}}_{i,x}$ and $\hat{\mathcal{O}}_{x}^\alpha$ mirror those in \eqref{SPJJO00}, with $K_{M,\omega_{\gamma}}^{\alpha,c}$ replaced by $D_xK_{M,\omega_{\gamma}}^{\alpha,c}$. Applying Lemma \ref{LeMK2ker4444} and the methodology from Section 5.1 gives the estimate
    \begin{gather*}
        \begin{split}
            \hat{I}_{1x}    & \leq \, C\frac{t}{\varepsilon}e^{-\frac{\nu_{\gamma}t}{\varepsilon}}\| \mathbf{S}^{\rm in}_{\gamma} \|_{L^\infty_{x,v}}+C\varepsilon^ke^{-\frac{\nu_{\gamma}t}{\varepsilon}} \sup_{0\leq s \leq t}(e^{\frac{\nu_{\gamma}s}{2\varepsilon}}\|  \mathbf{S}_{\gamma}\|_{L^{\infty}_{x,v}})^2\\
            & \hspace{0.4cm} +[C(1+\mathcal{I}_1)\varepsilon+C{\lambda}+\frac{C_{\lambda}}{N}+C\delta^{3+\gamma}]e^{-\frac{\nu_{\gamma}t}{2\varepsilon}} \sup_{0\leq s \leq t}(e^{\frac{\nu_{\gamma}s}{2\varepsilon}}\|  \mathbf{S}_{\gamma}\|_{L^{\infty}_{x,v}})\\
            & \hspace{0.4cm} + \frac{C_{N,\lambda}}{\varepsilon^{\frac{3}{2}}}\sup_{0\leq s \leq T_0}\| \mathbf{f}_R(s)\|_{L^{2}_{x,v}}+C\varepsilon^k\mathcal{I}_2.
        \end{split}
    \end{gather*}
 Analogously, replacing $v$ with $v_*$ in \eqref{linfHADTOSPDXT} and substituting into \eqref{DXXXK2CDXJFBD} produces
    \begin{equation}\label{xxxXccJOX}
        \hat{I}_{2x}=\frac{1}{\varepsilon}\int_{0}^{t}\mathcal{T}_0^{\alpha}(t,s) (D_xK_{M,\omega_{\gamma}}^{\alpha,c})  \mathbf{S}_{\gamma}(s,X^{\alpha}(s),V^{\alpha}(s))ds  = \sum_{i=0}^{3} \hat{\mathcal{Z}}^{\alpha}_{i,x}(t) + \sum_{i=1}^{2}\hat{\mathcal{D}}_{x}^\alpha(t).
    \end{equation}
    The estimation procedure in \eqref{xxxccJOX} leads to
    \begin{gather*}
        \begin{split}
            \hat{I}_{2x}    & \leq \, C\frac{t}{\varepsilon}e^{-\frac{\nu_{\gamma}t}{\varepsilon}} \| D_v \mathbf{S}^{\rm in}_{\gamma} \|_{L^\infty_{x,v}}+\varepsilon  e^{-\frac{\nu_{\gamma}t}{2\varepsilon}} \sup_{0\leq s \leq t}[e^{\frac{\nu_{\gamma}s}{2\varepsilon}}(\Vert  \nabla_{v}\mathbf{S}_{\gamma}\Vert_{L^{\infty}_{x,v}}+\Vert \left \langle v \right\rangle^{1-\gamma} \mathbf{S}_{\gamma}\Vert_{L^{\infty}_{x,v}})]\\
            & \hspace{0.3cm} +(C\varepsilon+C{\lambda}+\frac{C_{\lambda}}{N}+C\delta^{3+\gamma})e^{-\frac{\nu_{\gamma}t}{2\varepsilon}} \sup_{0\leq s \leq t}[e^{\frac{\nu_{\gamma}s}{2\varepsilon}}(\|  \mathbf{S}_{\gamma}\|_{L^{\infty}_{x,v}}+\|  D_{x}\mathbf{S}_{\gamma}\|_{L^{\infty}_{x,v}})]\\
            & \hspace{0.3cm} + \frac{C_{N,\lambda,T_0}}{\varepsilon^3}e^{-\frac{\nu_{\gamma}t}{2\varepsilon}}\sup_{0\leq s \leq  T_0}(e^{\frac{\nu_{\gamma}s}{2\varepsilon}}\Vert \mathbf{S}_{\gamma}\Vert_{L^{\infty}_{x,v}})+\frac{C_{N,\lambda,T_0}}{\varepsilon^{4}}  \sup_{0\leq s \leq t}\| \mathbf{f}_R(s)\|_{L^{2}_{x,v}}+C\varepsilon^k\mathcal{I}_2\\
            &\hspace{0.3cm}+C\varepsilon^ke^{-\frac{\nu_{\gamma}t}{\varepsilon}} \sup_{0\leq s \leq t}[(e^{\frac{\nu_{\gamma}s}{2\varepsilon}}\|  \mathbf{S}_{\gamma}\|_{L^{\infty}_{x,v}})^2+(e^{\frac{\nu_{\gamma}s}{2\varepsilon}}\| D_x \mathbf{S}_{\gamma}\|_{L^{\infty}_{x,v}})^2].
        \end{split}
    \end{gather*}
    Therefore, the combination of \eqref{SDXK}-\eqref{SDXMO} with the estimates for $\hat{I}_{x1}$ and $\hat{I}_{x2}$ establishes that
    \begin{gather}\label{55W1INFXXX}
        \begin{split}
            \sup_{0\leq s \leq  T_0}(e^{\frac{\nu_{\gamma}s}{2\varepsilon}}&\Vert D_x\mathbf{S}_{\gamma}\Vert_{L^{\infty}_{x,v}}) \leq C (\Vert  \mathbf{S}^{\rm in}_{\gamma}\Vert^2_{L^{\infty}_{x,v}}+\Vert D_x\mathbf{S}^{\rm in}_{\gamma}\Vert_{L^{\infty}_{x,v}})\\
            &+(\frac{C_{\lambda}}{N}+C\lambda+C\varepsilon+\delta^{3+\gamma})\sup_{0\leq s \leq  T_0}(e^{\frac{\nu_{\gamma}s}{2\varepsilon}}\Vert D_x\mathbf{S}_{\gamma}\Vert_{L^{\infty}_{x,v}}+e^{\frac{\nu_{\gamma}s}{2\varepsilon}}\Vert \mathbf{S}_{\gamma}\Vert_{L^{\infty}_{x,v}})\\
            &+\frac{C}{\varepsilon^3}\sup_{0\leq s \leq  T_0}(e^{\frac{\nu_{\gamma}s}{2\varepsilon}}\Vert \mathbf{S}_{\gamma}\Vert_{L^{\infty}_{x,v}})+\frac{C_{\lambda,N}}{\varepsilon^4}e^{\frac{\nu_{\gamma}T_0}{2\varepsilon}}\sup_{0\leq s \leq T_0}\| \mathbf{f}_R(s)\|_{L^{2}_{x,v}}\\
            &+C\varepsilon^k \sup_{0\leq s \leq T_0}[(e^{\frac{\nu_{\gamma}s}{2\varepsilon}}\| \mathbf{S}_{\gamma}\|_{L^{\infty}_{x,v}})^2+(e^{\frac{\nu_{\gamma}s}{2\varepsilon}}\| D_x \mathbf{S}_{\gamma}\|_{L^{\infty}_{x,v}})^2]+C\varepsilon^{k-1}e^{\frac{\nu_{\gamma}T_0}{2\varepsilon}}\\
            &+\varepsilon  e^{-\frac{\nu_{\gamma}t}{2\varepsilon}} \sup_{0\leq s \leq t}[e^{\frac{\nu_{\gamma}s}{2\varepsilon}}(\Vert  \nabla_{v}\mathbf{S}_{\gamma}\Vert_{L^{\infty}_{x,v}}+\Vert \left \langle v \right\rangle^{1-\gamma} \mathbf{S}_{\gamma}\Vert_{L^{\infty}_{x,v}})].
        \end{split}
    \end{gather}

    \subsection{$W_v^{1,\infty}$ estimate }Applying the differential operator $D_v$ to both sides of \eqref{SSS5555MAINEQ} yields:
    \begin{eqnarray}\label{VVDXLIVNFT4V}
        \begin{split}
            &\hspace{0.5cm}\partial_{t}(D_vS_{\gamma}^{\alpha})+v\cdot\nabla_{x}(D_vS_{\gamma}^{\alpha})+\frac{e^{\alpha}}{m^{\alpha}} \nabla_{x}\phi^{\varepsilon}\cdot\nabla_{v}  (D_vS_{\gamma}^{\alpha}) + \hat{\upsilon_{\varepsilon}^\alpha}  (D_vS_{\gamma}^{\alpha}) \\
            &=-D_xS_{\gamma}^{\alpha}-(D_v\hat{\upsilon_{\varepsilon}^\alpha})S_{\gamma}^{\alpha}-D_v[ \frac{1}{\varepsilon}(K_{M,\left \langle v \right\rangle^{l}}^{\alpha,\chi}+K_{M,\left \langle v \right\rangle^{l}}^{\alpha,c})\mathbf{S}_{\gamma}]+ \sum_{i=1}^{5}D_v, (\mathcal{N}^\alpha_i),
        \end{split}
    \end{eqnarray}
    where $\mathcal{N}^\alpha_i$ is described in \eqref{NLINEANNM1M2}. Along the trajectory, the solution $D_vS_{\gamma}^{\alpha}$ of the equation \eqref{VVDXLIVNFT4V} is expressed as
        \begin{align}
            |D_vS^\alpha_{\gamma}|(t,x,v) \leq\, &\mathcal{T}^\alpha_0(t,0)\,|D_vS^\alpha_{\gamma}|(0,X^{\alpha}(0;t,x,v),V^{\alpha}(0;t,x,v)) \notag\\
            &+\int_{0}^{t}\mathcal{T}^\alpha_0(t,s) |D_xS_{\gamma}^{\alpha}|(s, X^{\alpha}(s),V^{\alpha}(s))ds\notag\\
            & + \int_{0}^{t}\mathcal{T}^\alpha_0(t,s) \sum_{i=1}^{5}|D_v\mathcal{N}^\alpha_j|(s, X^{\alpha}(s),V^{\alpha}(s))ds\label{SlinfHSADTODXT}\\
            &+ \int_{0}^{t}\mathcal{T}^\alpha_0(t,s)|(D_v\hat{\upsilon_{\varepsilon}^\alpha})\, S_{\gamma}^{\alpha}| (s, X^{\alpha}(s),V^{\alpha}(s))ds\notag\\
            &+ \frac{1}{\varepsilon} \int_{0}^{t}\mathcal{T}^\alpha_0(t,s)  |D_v(K_{M,\left \langle v \right\rangle^{l}}^{\alpha,\chi} \mathbf{S}_{\gamma})| (s, X^{\alpha}(s),V^{\alpha}(s))ds\notag\\
            &+\frac{1}{\varepsilon} \int_{0}^{t}\mathcal{T}^\alpha_0(t,s)  |D_v(K_{M,\left \langle v \right\rangle^{l}}^{\alpha,c} \mathbf{S}_{\gamma})| (s, X^{\alpha}(s),V^{\alpha}(s))ds \notag.
        \end{align}
    The first line in right hand side in \eqref{SlinfHSADTODXT} is bounded by
    \begin{equation*}
        \left|\mathcal{T}_0^{\alpha}(t,0) D_vS^\alpha_{\gamma}(0,\,X^{\alpha}(0),\,V^{\alpha}(0))\right| \lesssim \varepsilon  e^{-\frac{\nu_{\gamma}t}{\varepsilon}} \|D_v\mathbf{S}^{in}_{\gamma}\|_{L^\infty_{x,v}}.
    \end{equation*}
    Following Subsection 4.2's derivation, we estimate the second and third terms as:
    \begin{gather}\label{ktDVV555}
        \begin{split}
            &\left|\int_{0}^{t}\mathcal{T}_0^{\alpha}(t,s) [-D_xS_{\gamma}^{\alpha}+\sum_{i=1}^{5}(D_v\mathcal{N}^\alpha_i)](s,X^{\alpha}(s),V^{\alpha}(s))ds \right|  \\
            & \hspace{0.4cm}\lesssim  \,\varepsilon(1+\mathcal{I}_1+\varepsilon\mathcal{I}_1) e^{-\frac{\nu_{\gamma}t}{2\varepsilon}} \sup_{0\leq s \leq t}[e^{\frac{\nu_{\gamma}s}{2\varepsilon}}(\Vert \left \langle v \right\rangle\mathbf{S}_{\gamma} (s)\Vert_{L^{\infty}_{x,v}}+\Vert D_{x,v}\mathbf{S}_{\gamma} (s)\Vert_{L^{\infty}_{x,v}})] \\
            &\hspace{0.8cm}+\varepsilon^k  e^{-\frac{\nu_{\gamma} t}{\varepsilon}} \sup_{0\leq s \leq t}[(e^{\frac{\nu_{\gamma} s}{2\varepsilon}}\Vert \left \langle v \right\rangle \mathbf{S}_{\gamma} (s)\Vert_{L^{\infty}_{x,v}})^2+(e^{\frac{\nu_{\gamma} s}{2\varepsilon}}\Vert D_v\mathbf{S}_{\gamma} (s)\Vert_{L^{\infty}_{x,v}})^2]+\varepsilon^k\mathcal{I}_2.
        \end{split}
    \end{gather}
   Next, for $-1\leq\gamma<1$, recalling \eqref{xxbfdmaisnlower1}, we have
    \begin{equation*}
        \begin{split}
            D_v(\hat{\upsilon_{\varepsilon}^\alpha}) &= \frac{D_v\upsilon^\alpha}{\varepsilon} +\frac{e^{\alpha}}{m^{\alpha}}\nabla_{x}\phi^{\varepsilon} \cdot  D_v[\frac{\omega_{\gamma}(t,v)}{\sqrt{\mu^{\alpha}_M}} \nabla_{v}(\frac{\sqrt{\mu^{\alpha}_M}}{\omega_{\gamma}(t,v)})]-D_v[\frac{\partial_t\omega_{\gamma}(t,v)}{\omega_{\gamma}(t,v)}]\\
            & \lesssim \frac{\left \langle v \right\rangle^{\gamma}}{\varepsilon} + \left \langle v \right\rangle^{2}+ \frac{\left \langle v \right\rangle^{\frac{3-\gamma}{2}-1}}{(1+t)^{1+\kappa_1}} \lesssim \frac{\left \langle v \right\rangle^{\gamma}}{\varepsilon} + \left \langle v \right\rangle^{2},
        \end{split}
    \end{equation*}
    where \eqref{SUPPOSELINFGXV} has been used, then one gets
    \begin{gather*}
        \begin{split}
            |(D_x\hat{\upsilon_{\varepsilon}^\alpha})\, S_{\gamma}^{\alpha}|(t,x,v)
            \leq &  \frac{\left \langle v \right\rangle^{\gamma}}{\varepsilon}\|  \mathbf{S}_{\gamma}\|_{L^{\infty}_{x,v}}+\left \langle v \right\rangle^{\gamma}\| \left \langle v \right\rangle^{2-\gamma} \mathbf{S}_{\gamma}\|_{L^{\infty}_{x,v}}.
        \end{split}
    \end{gather*}
    Hence the fourth line in \eqref{SlinfHSADTODXT} is bounded by
    \begin{gather*}
        \begin{split}
            &\left| \int_{0}^{t}\mathcal{T}^\alpha_0(t,s)(D_v\hat{\upsilon_{\varepsilon}^\alpha}) \, S_{\gamma}^{\alpha} (s, X^{\alpha}(s),V^{\alpha}(s))ds\right| \\
            &\hspace{0.4cm}\lesssim \, \varepsilon e^{-\frac{\nu_{\gamma}t}{2\varepsilon}} \sup_{0\leq s \leq t}(e^{\frac{\nu_{\gamma}s}{2\varepsilon}}\Vert \left \langle v \right\rangle^{2-\gamma} \mathbf{S}_{\gamma}\Vert_{L^{\infty}_{x,v}})
            +\, e^{-\frac{\nu_{\gamma}t}{2\varepsilon}}  \sup_{0\leq s \leq t}(e^{\frac{\nu_{\gamma}s}{2\varepsilon}}\Vert
\mathbf{S}_{\gamma}\Vert_{L^{\infty}_{x,v}}),
        \end{split}
    \end{gather*}
 which together with Lemma \ref{kchiestmainle} yields
    \begin{gather}\label{mwDVV111}
        \begin{split}
            &\left| \int_{0}^{t}\mathcal{T}^\alpha_0(t,s) D_v (K_{M,\left \langle v \right\rangle^{l}}^{\alpha,\chi} \mathbf{S}_{\gamma}) (s, X^{\alpha}(s),V^{\alpha}(s))ds\right|
            \lesssim \,\, \delta^{2+\gamma}  e^{-\frac{\nu_{\gamma}t}{2\varepsilon}} \sup_{0\leq s \leq t}(e^{\frac{\nu_os}{2\varepsilon}}\|  \mathbf{S}_{\gamma}\|_{L^{\infty}_{x,v}}).
        \end{split}
    \end{gather}
 Finally, applying the product rule to the derivative operator yields
    \begin{equation*}
        D_v (K_{M,\omega_{\gamma}}^{\alpha,c} \mathbf{S}_{\gamma})= (D_vK_{M,\omega_{\gamma}}^{\alpha,c}) \mathbf{S}_{\gamma}+K_{M,\omega_{\gamma}}^{\alpha,c} (D_v\mathbf{S}_{\gamma}),
    \end{equation*}
   which naturally decomposes into two integral terms:
    \begin{gather*}
        \begin{split}
            & \frac{1}{\varepsilon}\int_{0}^{t}\mathcal{T}^\alpha_0(t,s) (D_vK_{M,\omega_{\gamma}}^{\alpha,c}) \mathbf{S}_{\gamma}(s, X^{\alpha}(s),V^{\alpha}(s))ds:=\hat{I}_{1v},\\
            & \frac{1}{\varepsilon}\int_{0}^{t}\mathcal{T}^\alpha_0(t,s) K_{M,\omega_{\gamma}}^{\alpha,c} (D_v\mathbf{S}_{\gamma})(s, X^{\alpha}(s),V^{\alpha}(s))ds
:=\hat{I}_{2v},
        \end{split}
    \end{gather*}
   and the expression above can be rewritten as the form of integral kernels:
    \begin{eqnarray}\label{VVVK2CDXJFBD2}
        \begin{split}
            &(D_vK_{M,\omega_{\gamma}}^{\alpha,c}) \mathbf{S}_{\gamma}(s, X^{\alpha}(s),V^{\alpha}(s))\\
            =&\sum_{\beta=A,B}\int_{\mathbb{R}^3} (D_vk_{M,\omega_{\gamma}}^{\alpha \beta}) (s,X^{\alpha}(s),V^{\alpha}(s),v_*)S_{\gamma}^{\beta}(s,X^{\alpha}(s),v_*) d v_{*},
        \end{split}
    \end{eqnarray}
    and
    \begin{eqnarray}\label{DVVVK2553CDXJFBD2}
        \begin{split}
            &K_{M,\omega_{\gamma}}^{\alpha,c} (D_v\mathbf{S}_{\gamma})(s, X^{\alpha}(s),V^{\alpha}(s))\\
            =&\sum_{\beta=A,B}\int_{\mathbb{R}^3} k_{M,\omega_{\gamma}}^{\alpha \beta} (s,X^{\alpha}(s),V^{\alpha}(s),v_*)(D_vS_{\gamma}^{\beta})(s,X^{\alpha}(s),v_*) d v_{*}.
        \end{split}
    \end{eqnarray}
    Plunging \eqref{SGMZONG512} into \eqref{VVVK2CDXJFBD2}, we get
    \begin{equation}\label{vvvccJOX}
        \hat{I}_{1v}=\frac{1}{\varepsilon}\int_{0}^{t}\mathcal{T}_0^{\alpha}(t,s) (D_vK_{M,\omega_{\gamma}}^{\alpha,c})  \mathbf{S}_{\gamma}(s,X^{\alpha}(s),V^{\alpha}(s))ds  = \sum_{i=0}^{2} \hat{\mathcal{J}}_{i,v}(t) + \hat{\mathcal{O}}_{v}^\alpha(t).
    \end{equation}
    The operators $\hat{\mathcal{J}}_{i,v}$ and $\hat{\mathcal{O}}_{v}^\alpha$ maintain the form in \eqref{SPJJO00} with the replacement:
    $
        K_{M,\omega_{\gamma}}^{\alpha,c} \to D_vK_{M,\omega_{\gamma}}^{\alpha,c},
    $
    then it holds that
    \begin{align*}
            \hat{I}_{1v}    & \leq \, C\frac{t}{\varepsilon}e^{-\frac{\nu_{\gamma}t}{\varepsilon}}\|\left \langle v \right\rangle \mathbf{S}^{\rm in}_{\gamma} \|_{L^\infty_{x,v}}+C\varepsilon^ke^{-\frac{\nu_{\gamma}t}{\varepsilon}} \sup_{0\leq s \leq t}[(e^{\frac{\nu_{\gamma}s}{2\varepsilon}}\| \left \langle v \right\rangle \mathbf{S}_{\gamma}\|_{L^{\infty}_{x,v}})^2]\\
            & \hspace{0.4cm} +[C(1+\mathcal{I}_1)\varepsilon+C{\lambda}+\frac{C_{\lambda}}{N}+C\delta^{3+\gamma}]e^{-\frac{\nu_{\gamma}t}{2\varepsilon}} \sup_{0\leq s \leq t}(e^{\frac{\nu_{\gamma}s}{2\varepsilon}}\| \left \langle v \right\rangle \mathbf{S}_{\gamma}\|_{L^{\infty}_{x,v}})\\
            & \hspace{0.4cm} + \frac{C_{N,\lambda}}{\varepsilon^{\frac{3}{2}}}\sup_{0\leq s \leq T_0}\| \mathbf{f}_R(s)\|_{L^{2}_{x,v}}+C\varepsilon^k\mathcal{I}_2.
    \end{align*}
   We perform the substitution $v \mapsto v_*$ in \eqref{SlinfHSADTODXT} and insert it into \eqref{DVVVK2553CDXJFBD2}, to obtain
    \begin{equation}
        \hat{I}_{2v}=\frac{1}{\varepsilon}\int_{0}^{t}\mathcal{T}_0^{\alpha}(t,s) (D_vK_{M,\omega_{\gamma}}^{\alpha,c})  \mathbf{S}_{\gamma}(s,X^{\alpha}(s),V^{\alpha}(s))ds  = \sum_{i=0}^{3} \hat{\mathcal{Z}}_{i,v}(t) + \sum_{i=1}^{2}\hat{\mathcal{D}}_{v}^\alpha(t).
    \end{equation}
   Following the procedure from \eqref{xx45xccJOX}--\eqref{HI2CLT} gives
    \begin{gather*}
        \begin{split}
            I_{2v}  & \leq \, C\frac{t}{\varepsilon}e^{-\frac{\nu_{\gamma}t}{\varepsilon}} \| D_v \mathbf{S}^{\rm in}_{\gamma} \|_{L^\infty_{x,v}}\\
            & \hspace{0.3cm} +(C\varepsilon+C{\lambda}+\frac{C_{\lambda}}{N}+C\delta^{2+\gamma})e^{-\frac{\nu_{\gamma}t}{2\varepsilon}} \sup_{0\leq s \leq t}[e^{\frac{\nu_{\gamma}s}{2\varepsilon}}(\|  \left \langle v \right \rangle^{2-\gamma} \mathbf{S}_{\gamma}\|_{L^{\infty}_{x,v}}+\|  D_{x,v}\mathbf{S}_{\gamma}\|_{L^{\infty}_{x,v}})]\\
            &  \hspace{0.3cm}+ \frac{C_{N,\lambda,T_0}}{\varepsilon^3}e^{-\frac{\nu_{\gamma}t}{2\varepsilon}}\sup_{0\leq s \leq  T_0}(e^{\frac{\nu_{\gamma}s}{2\varepsilon}}\Vert \mathbf{S}_{\gamma}\Vert_{L^{\infty}_{x,v}})+\frac{C_{N,\lambda,T_0}}{\varepsilon^{3}}  \sup_{0\leq s \leq t}\| \mathbf{f}_R(s)\|_{L^{2}_{x,v}}+C\varepsilon^k\mathcal{I}_2\\
           & \hspace{0.3cm}+C\varepsilon^ke^{-\frac{\nu_{\gamma}t}{\varepsilon}} \sup_{0\leq s \leq t}[(e^{\frac{\nu_{\gamma}s}{2\varepsilon}}\| \left \langle v \right \rangle \mathbf{S}_{\gamma}\|_{L^{\infty}_{x,v}})^2+(e^{\frac{\nu_{\gamma}s}{2\varepsilon}}\| D_v \mathbf{S}_{\gamma}\|_{L^{\infty}_{x,v}})^2].
        \end{split}
    \end{gather*}
     Therefore, combining \eqref{ktDVV555}--\eqref{mwDVV111} with the estimates of $\hat{I}_{1v}$ and $\hat{I}_{2v}$ leads to:
    \begin{gather}\label{55W1INFvvv}
        \begin{split}
            \sup_{0\leq s \leq  T_0}(e^{\frac{\nu_{\gamma}s}{2\varepsilon}}&\Vert D_v\mathbf{S}_{\gamma}\Vert_{L^{\infty}_{x,v}}) \leq C (\Vert \left \langle v \right\rangle \mathbf{S}^{\rm in}_{\gamma}\Vert^2_{L^{\infty}_{x,v}}+\Vert D_v\mathbf{S}^{\rm in}_{\gamma}\Vert_{L^{\infty}_{x,v}})+C\varepsilon \sup_{0\leq s \leq T_0}(e^{\frac{\nu_{\gamma}s}{2\varepsilon}}\Vert \nabla_{v}\mathbf{S}_{\gamma}\Vert_{L^{\infty}_{x,v}})\\
            &+(\frac{C_{\lambda}}{N}+C\lambda+C\varepsilon+\delta^{\gamma+2})\sup_{0\leq s \leq  T_0}(e^{\frac{\nu_{\gamma}s}{2\varepsilon}}\Vert D_{x,v}\mathbf{S}_{\gamma}\Vert_{L^{\infty}_{x,v}}+e^{\frac{\nu_{\gamma}s}{2\varepsilon}}\Vert \left \langle v \right\rangle^{2-\gamma}\mathbf{S}_{\gamma}\Vert_{L^{\infty}_{x,v}})\\
            &+\frac{C}{\varepsilon^3}\sup_{0\leq s \leq  T_0}(e^{\frac{\nu_{\gamma}s}{2\varepsilon}}\Vert \mathbf{S}_{\gamma}\Vert_{L^{\infty}_{x,v}})+\frac{C_{\lambda,N}}{\varepsilon^3}e^{\frac{\nu_{\gamma}T_0}{2\varepsilon}}\sup_{0\leq s \leq T_0}\| \mathbf{f}_R(s)\|_{L^{2}_{x,v}}\\
            &+C\varepsilon^k \sup_{0\leq s \leq T_0}[(e^{\frac{\nu_{\gamma}s}{2\varepsilon}}\| \left \langle v \right \rangle \mathbf{S}_{\gamma}\|_{L^{\infty}_{x,v}})^2+(e^{\frac{\nu_{\gamma}s}{2\varepsilon}}\| D_v \mathbf{S}_{\gamma}\|_{L^{\infty}_{x,v}})^2]+C\varepsilon^{k-1}e^{\frac{\nu_{\gamma}T_0}{2\varepsilon}}.
        \end{split}
    \end{gather}

	In summary, combining \eqref{55W1INFvvv} with \eqref{55W1INFXXX} establishes the following estimates for $W^{1,\infty}_{x,v}$:
    \begin{gather}\label{44W1INFXV}
        \begin{split}
            \sup_{0\leq s \leq  T_0}(e^{\frac{\nu_{\gamma}s}{2\varepsilon}}&\Vert \nabla_{x,v}\mathbf{S}_{\gamma}\Vert_{L^{\infty}_{x,v}}) \leq C (\Vert \left \langle v \right\rangle \mathbf{S}^{\rm in}_{\gamma}\Vert^2_{L^{\infty}_{x,v}}+\Vert \nabla_{x,v}\mathbf{S}^{\rm in}_{\gamma}\Vert_{L^{\infty}_{x,v}})\\
            &+(\frac{C_{\lambda}}{N}+C\lambda+C\varepsilon+\delta^{2+\gamma})\sup_{0\leq s \leq  T_0}(e^{\frac{\nu_{\gamma}s}{2\varepsilon}}\Vert \nabla_{x,v}\mathbf{S}_{\gamma}\Vert_{L^{\infty}_{x,v}}+e^{\frac{\nu_{\gamma}s}{2\varepsilon}}\Vert \left \langle v \right \rangle^{2-\gamma} \mathbf{S}_{\gamma}\Vert_{L^{\infty}_{x,v}})\\
            &+\frac{C}{\varepsilon^3}\sup_{0\leq s \leq  T_0}(e^{\frac{\nu_{\gamma}s}{2\varepsilon}}\Vert \mathbf{S}_{\gamma}\Vert_{L^{\infty}_{x,v}})+\frac{C_{\lambda,N}}{\varepsilon^4}e^{\frac{\nu_{\gamma}T_0}{2\varepsilon}}\sup_{0\leq s \leq T_0}\| \mathbf{f}_R(s)\|_{L^{2}_{x,v}}\\
            &+C\varepsilon^k \sup_{0\leq s \leq T_0}[(e^{\frac{\nu_{\gamma}s}{2\varepsilon}}\| \left \langle v \right \rangle \mathbf{S}_{\gamma}\|_{L^{\infty}_{x,v}})^2+(e^{\frac{\nu_{\gamma}s}{2\varepsilon}}\| \nabla_{x,v} \mathbf{S}_{\gamma}\|_{L^{\infty}_{x,v}})^2]+C\varepsilon^{k-1}e^{\frac{\nu_{\gamma}T_0}{2\varepsilon}}.
        \end{split}
    \end{gather}
Further, multiplying \eqref{44W1INFXV} by $\varepsilon^5$ and applying \eqref{CSJCHUFAXXM} yields:
    \begin{align*}
            \varepsilon^{5}\Vert  \nabla_{x,v}\mathbf{S}_{\gamma}(T_0) \Vert_{L^{\infty}_{x,v}}\leq& \frac{1}{4}\Big[\varepsilon^{5}\Vert \left \langle v \right\rangle \mathbf{S}^{\rm in}_{\gamma} \Vert_{L^{\infty}_{x,v}}+ \varepsilon^{5}\Vert \nabla_{x,v} \mathbf{S}^{\rm in}_{\gamma} \Vert_{L^{\infty}_{x,v}}\\
            &\hspace{3.5cm}+\varepsilon^{\frac{3}{2}}(\Vert  \mathbf{S}_{\gamma} \Vert_{L^{\infty}_{x,v}}+\Vert \left \langle v \right \rangle^{2-\gamma} \mathbf{S}_{\gamma}\Vert_{L^{\infty}_{x,v}})\Big]\\
            &+\varepsilon^{\frac{1}{2}}\sup_{0\leq s \leq T_0}\| \mathbf{f}_R(s)\|_{L^{2}_{x,v}}+\varepsilon^{k+1}.
    \end{align*}
Following the estimation procedure for $\varepsilon^{\frac{3}{2}}\Vert \mathbf{S}_{\gamma}(T_0) \Vert_{L^{\infty}_{x,v}}$ in Section 5.1 gives:
    \begin{eqnarray*}
        \varepsilon^{\frac{3}{2}}\Vert \left \langle v \right \rangle^{2-\gamma} \mathbf{S}_{\gamma}(T_0) \Vert_{L^{\infty}_{x,v}}\leq \frac{1}{4} \varepsilon^{\frac{3}{2}} \Vert \left \langle v \right \rangle^{2-\gamma} \mathbf{S}^{\rm in}_{\gamma} \Vert_{L^{\infty}_{x,v}}+C \sup_{0\leq s \leq T_0}\| \mathbf{f}_R(s)\|_{L^{2}_{x,v}}+C\varepsilon^{\frac{2k+3}{2}}.
    \end{eqnarray*}
Consequently, the combined estimate becomes
    \begin{gather*}
        \begin{split}
            &\varepsilon^{\frac{3}{2}}\Vert \left \langle v \right\rangle^{2-\gamma} \mathbf{S}_{\gamma}(T_0) \Vert_{L^{\infty}_{x,v}}+\varepsilon^{5}\Vert  \nabla_{x,v}\mathbf{S}_{\gamma}(T_0) \Vert_{L^{\infty}_{x,v}}\\
            & \hspace{0.3cm}\leq  \, \frac{1}{2}(\varepsilon^{\frac{3}{2}}\Vert \left \langle v \right\rangle^{2-\gamma} \mathbf{S}^{\rm in}_{\gamma} \Vert_{L^{\infty}_{x,v}}+\varepsilon^{5}\Vert  \nabla_{x,v}\mathbf{S}^{\rm in}_{\gamma} \Vert_{L^{\infty}_{x,v}})+C(\sup_{0\leq s \leq T_0}\| \mathbf{f}_R(s)\|_{L^{2}_{x,v}}+\varepsilon^{k+\frac{1}{2}}).
        \end{split}
    \end{gather*}
For $t\in(T_0,T_S]$, there exist positive integer $n_t$ and $t_s\in[0,T_0)$ such that $t=n_tT_0+t_s$. Applying the iterative scheme analogous to \eqref{dddHDdgui0H1} leads to:
    \begin{gather*}\label{444HSw1p}
        \begin{split}
            &\sup_{0\leq t \leq T_S}(\varepsilon^{\frac{3}{2}}\Vert \left \langle v \right\rangle^{2-\gamma} \mathbf{S}_{\gamma}(t) \Vert_{L^{\infty}_{x,v}}+\varepsilon^{5}\Vert  \nabla_{x,v}\mathbf{S}_{\gamma}(t) \Vert_{L^{\infty}_{x,v}})\\
            &\hspace{0.4cm}\lesssim  \, \varepsilon^{\frac{3}{2}}\Vert \left \langle v \right\rangle^{2-\gamma} \mathbf{S}^{\rm in}_{\gamma} \Vert_{L^{\infty}_{x,v}}+\varepsilon^{5}\Vert  \nabla_{x,v}\mathbf{S}^{\rm in}_{\gamma} \Vert_{L^{\infty}_{x,v}}+\sup_{0\leq t \leq T_L}\| \mathbf{f}_R(t)\|_{L^{2}_{x,v}}+\varepsilon^{k+\frac{1}{2}}.
        \end{split}
    \end{gather*}

    \section{Shortening of the validity time for potential range $-3<\gamma<-1$}
        This section considers the solution's validity time for $-3 < \gamma < -1$. The existence proof follows the methodology in Section 5. As established in \eqref{WWWEIGHTFUC}, the weight function becomes $\gamma$-independent and takes the form
        \begin{equation*}
            \omega_{-1} = e^{\tilde{\kappa}\left \langle v \right\rangle^{2}}, \quad
            \tilde{\kappa} = \kappa_0 [1 + (1+t)^{-\frac{2}{2k-1}}], \quad
            0 < \kappa_0 = \kappa_0(m^A,m^B) \ll 1.
        \end{equation*}
        Remark \eqref{WLIGHTERANALYSIS} indicates that the weight's exponent $\frac{3-\gamma}{2}$ cannot exceed 2, which makes it impossible to maintain the $O(\varepsilon^{-\frac{1}{2}\frac{2k-3}{2k-1}})$ validity time for $k\geq 6$. Substituting $\omega_{-1} = \exp(\tilde{\kappa}\left \langle v \right\rangle^{2})$, we have
        \begin{align*}
            \hat{\upsilon_{\varepsilon}^\alpha} &:= \frac{\upsilon^\alpha}{\varepsilon} + \frac{e^{\alpha}}{m^{\alpha}}\nabla_{x}\phi^{\varepsilon} \cdot \frac{\omega_{-1}(t,x)}{\sqrt{\mu^{\alpha}_M}} \nabla_{v}[\frac{\sqrt{\mu^{\alpha}_M}}{\omega_{-1}(t,x)}] - \frac{\partial_t\omega_{-1}(t,v)}{\omega_{-1}(t,v)} \\
            &\geq \frac{\upsilon^\alpha}{\varepsilon} - \frac{e^{\alpha}}{m^{\alpha}}(\frac{1}{2\theta_M} + 2\kappa_0)\left \langle v \right\rangle \|\nabla_{x}\phi^{\varepsilon}\|_{L^{\infty}_{x,v}} + \frac{2\kappa_0}{2k-1}\frac{\left \langle v \right\rangle^{2}}{(1+t)^{\frac{2k+1}{2k-1}}} \\
            &\geq \frac{\left \langle v \right\rangle^{\gamma}}{C\varepsilon} - C\left \langle v \right\rangle \|\nabla_{x}\phi^{\varepsilon}\|_{L^{\infty}_{x,v}} + \frac{2\kappa_0}{C(2k-1)}\frac{\left \langle v \right\rangle^{2}}{(1+t)^{\frac{2k+1}{2k-1}}}.
        \end{align*}
   For the parameter $0<\iota<1$, we apply Young's inequality to obtain
    \begin{equation*}\label{YYYEQIMP}
        \begin{split}
            \left \langle v \right\rangle &= [\left \langle v \right\rangle^{\gamma}\varepsilon^{\iota-1}(1+t)^{\frac{(2k+1)(1-\gamma)}{2k-1}}]^{\frac{1}{2-\gamma}} [\left \langle v \right\rangle^{2}\varepsilon^{\frac{1-\iota}{1-\gamma}}(1+t)^{-\frac{2k+1}{2k-1}}]^{\frac{1-\gamma}{2-\gamma}}\\
            & \leq \frac{1}{\varepsilon}\left \langle v \right\rangle^{\gamma}\varepsilon^{\iota}(1+t)^{\frac{(2k+1)(1-\gamma)}{2k-1}} + \left \langle v \right\rangle^{2}\varepsilon^{\frac{1-\iota}{1-\gamma}}(1+t)^{-\frac{2k+1}{2k-1}}.
        \end{split}
    \end{equation*}
   Consequently, we have
    \begin{equation*}
        \hat{\upsilon_{\varepsilon}^\alpha}\geq [\frac{1}{C}-C\varepsilon^{\iota}(1+t)^{\frac{(2k+1)(1-\gamma)}{2k-1}}\Vert\nabla_{x}\phi^{\varepsilon} \Vert_{L^{\infty}_{x,v} }]\frac{\left \langle v \right\rangle^{\gamma}}{\varepsilon}+[\frac{2\kappa_0}{C(2k-1)}-C\varepsilon^{\frac{1-\iota}{1-\gamma}}]\frac{\left \langle v \right\rangle^{2}}{(1+t)^{\frac{2k+1}{2k-1}}}.
    \end{equation*}
   For sufficiently small $\varepsilon$, it follows that
    \begin{equation*}
        \frac{2\kappa_0}{C(2k-1)}-C\varepsilon^{\frac{1-\iota}{1-\gamma}}\geq \frac{\kappa_0}{C(2k-1)}>0.
    \end{equation*}
   The validity time scales as $O(\varepsilon^{-m_{\gamma}})$ for $-3 < \gamma < -1$, abbreviated as $\varepsilon^{-m_{\gamma}}$, with the exponent satisfying $ 0 \leq m_\gamma \leq \frac{2k-3}{2(2k-1)}$. This leads to the derivation
    \begin{equation*}
        C\varepsilon^{\iota}(1+t)^{\frac{(2k+1)(1-\gamma)}{2k-1}}\Vert\nabla_{x}\phi^{\varepsilon} \Vert_{L^{\infty}_{x,v} } \leq 2C \varepsilon^{\iota-\frac{m_{\gamma}(2k+1)(1-\gamma)}{2k-1}} \Vert\nabla_{x}\phi^{\varepsilon} \Vert_{L^{\infty}_{x,v} } .
    \end{equation*}
    For $\iota=\frac{(2k-2)(2k+2)}{4k^2}<1$ and $m_\gamma = \frac{2k-3}{(1-\gamma)(2k-1)}$, a direct computation shows
    \begin{equation*}
        \iota-\frac{m_{\gamma}(2k+1)(1-\gamma)}{2k-1}=\frac{(2k-2)(2k+2)}{4k^2}-\frac{(2k-3)(2k+1)}{(2k-1)^2}\geq \frac{1}{16k^4}, \quad k \geq 6,
    \end{equation*}
    which implies
    \begin{align*}
        T &= O(\varepsilon^{-y}), \quad y =
        \begin{cases}
            \frac{2k-3}{2(2k-1)}, & \gamma \in [-1,1], \\
            \frac{2k-3}{(1-\gamma)(2k-1)}, & \gamma \in (-3,-1),
        \end{cases}
    \end{align*}
    for arbitrary $k\geq 6$ in \eqref{MAINEXPFPHIKKK}, thus establishing Theorem \ref{maintheorem}.

    \begin{remark}
        When we have chosen the expansion form of \eqref{MAINEXPFPHIKKK}, it can be seen from the above analysis that as $k$ increases, the validity time of the solution is always monotonically increasing. Furthermore, as $k\rightarrow +\infty$, the validity time will approach the upper bound, as shown in Figure2 at the beginning of the article. In addition, the argument process in this article can be used to analyze other expansions similarly, such as
        \begin{equation}
            F_{\varepsilon}^\alpha= \sum_{i=0}^{n}\varepsilon^i F^\alpha_i + \varepsilon^{r} F_R^{\alpha},\quad n\geq r\geq 1.
        \end{equation}
        The validity time may be different. However, once the ratio of $\frac{r}{n}$ is fixed, the validity time is determined, which always increase monotonically as $k$ increases.
    \end{remark}

 \subsection*{Acknowledgements}
This research was supported by the National Natural Science
Foundation of China (Grant No. 12271356 and 12331007).

\subsection*{Conflict of interest}
The authors have no conflicts to disclose.

\subsection*{Data availablity statement}
Data sharing is not applicable to this article as no new data were created or analyzed in this study.


\begin{thebibliography}{10}

        \bibitem{Appleton1933} E. V, Appleton,  \emph{Fine-Structure of the Ionosphere}. Nature. 131 (3319), 872-873.

        \bibitem{[1]Aoki2003JSP} K. Aoki, C. Bardos and C. Takata, \emph{Knudsen layer for gas mixture}.  J. Stat.
        Phys. 112 (2003), 629-655.



        \bibitem{[4]BGLJSP 1991}  C. Bardos, F. Golse and F, C.D. Levermore, \emph{Fluid dynamic limits of kinetic equations. I. Formal derivations}. J. Statist. Phys. 63 (1991) no. 1-2, 323-344.

        \bibitem{[5]BGLCPAM1993}  C. Bardos, F. Golse and C.D. Levermore, \emph{ Fluid dynamic limits of kinetic equations. II. Convergence proofs for the Boltzmann equation.} Comm. Pure Appl. Math. 46 (1993) no. 5, 667-753.


        \bibitem{[7]Bardos2012CPDE} C. Bardos and X.F. Yang, \emph{The Classification of well-posed kinetic boundary layer for hard sphere gas mixtures}. Comm. Partial Diff. Eqs. 37 (2012), 1286-1314.

        \bibitem{[8]CGQSSD} Q.G. Chen, F.M. Huang, T.H. Li, W.Q. Wang  and
        Y. Wang, \emph{Global Finite-Energy Solutions of the Compressible
            Euler–Poisson Equations for General Pressure Laws
            with Large Initial Data of Spherical Symmetry}. Comm. Math. Phys. (2024) 405:77.

        \bibitem{[9]Boltzmann1872AWW} L. Boltzmann, \emph{Weitere Studien \"uber das W\"armegleichgewicht unter Gasmolek\"ulen. Sitzungs} Akad.Wiss.
        Wien 66 (1872), 275-370; translated as: Further studies on the thermal equilibrium of gas molecules. Kinetic theory, vol. 2, 88-174. Pergamon, London, 1966.


        \bibitem{[60]Briant2016ARM} M. Briant, Esther S. Daus \emph{The Boltzmann Equation for a Multi-species
            Mixture Close to Global Equilibrium}. Arch. Rational Mech. Anal. 222 (2016) 1367–1443


        \bibitem{[12]Caflisch1980CPAM} R.E. Caflisch, \emph{The fluid dynamic limit of the nonlinear Boltzmann equation}.
        Comm. Pure Appl. Math. 33(1980) no.5, 651-666.




        \bibitem{[DL]VPB} R.J.Duan and S.Q. Liu, \emph{The Vlasov–Poisson–Boltzmann system for a disparate mass binary mixture}. J. Stat. Phys. (2017) 169:614–684.

        \bibitem{[FQ]ARC} Z.D. FANG and K.L. QI, \emph{From the Boltzmann equation for gas mixture to the two-fluid incompressible hydrodynamic system}. arxiv:2408.03570vl[math.AP].

        \bibitem{[61]Grad1958TG} H. Grad, \emph{Principles of the kinetic theory of gases.}.  In: Handbuch der Physik
        Bd. 12, Thermodynamik der Gase, pp. 205-249. Springer Verlag, Berlin, 1958.


        \bibitem{[ivR]Guo1998} Y. Guo \emph{Smooth irrotational flows in the large to the Euler-Poisson system in $\mathcal{R}^{3+1}$}. Comm. Math.
        Phys. 195 (1998), 249–265 .

        \bibitem{[17]Guo2003Invention} Y. Guo, \emph{The Vlasov-Maxwell-Boltzmann system near Maxwellians}.  Invent. math. 153 (2003), 593-630.



        \bibitem{[ininp]Guo2010CMP} Y. Guo and J. Jang, \emph{Global Hilbert expansion for the Vlasov-Poisson-Boltzmann system}.
        Comm. Math. Phys. 299 (2010), pp. 469-501.

        \bibitem{[iJMP]Guo} Y. Guo, A.D. Ionescu and B. Pausader, \emph{Global solutions of certain plasma fluid models in three dimension}.
        J. Math. Phys. 55, 123102 (2014).

        \bibitem{[19]Guo2010ARMA} Y. Guo, \emph{Decay and Continuity of the Boltzmann Equation in Bounded Domains}. Arch. Rational Mech. Anal. 197 (2010), 713-809

        \bibitem{[20]Guo2021ARMA} Y. Guo,  F.M. Huang and Y. Wang, \emph{Hilbert expansion of the Boltzmann equation with specular boundary
            condition in half-space}. Arch. Rational Mech. Anal. 241 (2021), 231-309.

        \bibitem{[21]Guo2009KRM} Y. Guo, J. Jang and N. Jiang, \emph{Local Hilbert expansion for the Boltzmann equation}.
        Kinet. Relat. Models. 2 (2009) no. 1, 205-214.

        \bibitem{[22]Guo2010CPAM} Y. Guo, J. Jang and N. Jiang, \emph{Acoustic limit for the Boltzmann equation in optimal scaling}. Comm. Pure Appl. Math. 63 (2010) no. 3, 337-361.



        \bibitem{[23]Hilbert} D. Hilbert, \emph{ Mathematical problems}. Bull. Amer. Math. Soc. (N.S.) 37 (2000) no. 4, 407-436.

        \bibitem{Hanson172} W. B. Hanson and R. J. Moffett. \emph{Ionospheric Ion Composition and Density Measurements from the OGO-5 Satellite}. Journal of Geophysical Research, vol. 77, no. 22, 1972, pp. 4187-4203.


        \bibitem{[29]Jiang2021} N. Jiang, Y.L.Luo and S.J. Tang, \emph{Compressible Euler limit from Boltzmann equation with complete diffusive boundary condition in half-space}, arxiv:3714535[math.AP]

        \bibitem{[30]Jiang2021} N. Jiang, Y.L. Luo and S.J. Tang,   \emph{Compressible Euler limit from Boltzmann equation with Maxwell reflection boundary condition in half-space}, arxiv:2101.11199[math.AP]


        \bibitem{[32]Jiang2018IUM} N. Jiang, C.J. Xu and H. Zhao, \emph{Incompressible Navier-Stokes-Fourier limit from the Boltzmann equation:
            classical solutions}, Indiana Univ. Math. J. 67 (2018), no. 5, 1817-1855.

        \bibitem{[2025X]Jiang} N. Jiang, Y.J. Lei and H.J. Zhao, \emph{On the Vlasov-Poisson-Boltzmann limit of the Vlasov-Maxwell-Boltzmann system}, Journal of Functional Analysis 287 (2024) 110529.

        \bibitem{[61]Jiang} N. Jiang, Y-L. Luo and S. Tang, \emph{Grad-Caflish type decay estimate of Pseudo-inverse of linearized Boltzmann operator and appliciation to Hilbert expansion of compressible Euler scaling}, arxiv


        \bibitem{[33]Lachowicz1987MMAS} M. Lachowicz,   \emph{On the initial layer and the existence theorem for the nonlinear Boltzmann equation}. Math. Methods Appl. Sci. 9 (1987), no. 3, 342-366.


        \bibitem{[lc]TM1986} T. Makino, \emph{On a Local Existence Theorem for the Evolution
            Equation of Gaseous Stars}. Patterns and Waves-Qualitative Analysis of
        Nonlinear Differential Equations. (1986) 459-479.

        \bibitem{[37]Maxwell1867JCM} J.C. Maxwell,  \emph{On the dynamical theory of gases}. Philos. Trans. Roy. Soc. London Ser. A 157 (1867), 49-88. Reprinted in The scientific letters and papers of James Clerk Maxwell, vol. II, 1862-1873, 26-78. Dover, New York, 1965.

        \bibitem{[62]A.Majda} A. Majda,     \emph{Compressible Fluid Flow and Systems of Conservation Laws in Several Space Variables}. Springer Science ,  vol 53 New York, 1984.


        \bibitem{[40]Raymond2009BOOK} L. Saint-Raymond, \emph{Hydrodynamic Limits of the Boltzmann Equation}. Lecture Notes in Math. Springer-Verlag, Berlin, 2009.


        \bibitem{[42]Raymond2009JMPA} L. Saint-Raymond and  F. Golse, \emph{The incompressible Navier-Stokes limit of the Boltzmann equation for hard cutoff potentials}. J. Math. Pures Appl. 91 (2009) no. 5, 508-552.


        \bibitem{[DA]Sobook} Y. Sone, \emph{Molecular Gas Dynamics. Theory, Techniques, and Applications}.  Birkh\"auser Boston,
        Boston, 2007.

        \bibitem{[44]Strain2008ARMA} M. Strain and Y. Guo, \emph{Exponential decay for soft potentials near Maxwellian}. Arch. Rational Mech. Anal. 187 (2008), 287-339.





        \bibitem{YJWangSIMA} Y.J.Wang, \emph{The diffusive limit of the Vlasov-Boltzmann system
            for binary fluids}. SIAM J. Math. Anal. Vol. 43, No. 1(2011), pp. 253–301.

        \bibitem{[49]WangJDE2013} Y.J.Wang, \emph{ Decay of the two-species Vlasov-Poisson-Boltzmann system}. J. Differential Equations 254 (2013) no. 5, 2304-2340.

		\bibitem{Wu20252pp}  G.F.Wang and T.F.Wu, \emph{Classical solutions to the Boltzmann equations for  gas mixture with unequal molecular masses}. preprint.

        \bibitem{[50]Wu2023JDE} T.F.Wu and X.F.Yang, \emph{Hydrodynamic limit of Boltzmann equations for gas
            mixture}. J. Differential Equations 377 (2023) 418-468.

        \bibitem{Wu2025pp} T.F.Wu, L.J.Xiong and X.F.Yang, \emph{Acoustic Limit of the Boltzmann equations for  gas mixture}. preprint.

		
    \end{thebibliography}
\end{document}